\documentclass[11pt]{amsart}

\usepackage[utf8]{inputenc}
\usepackage[T1]{fontenc}
\usepackage[english]{babel}

\usepackage{amsthm}
\usepackage{amsmath}
\usepackage{amssymb}
\usepackage{mathrsfs}
\usepackage{bbold}
\usepackage{graphicx}
\usepackage{listings}
\usepackage{mathrsfs}
\usepackage{enumerate}
\usepackage{pict2e}
\usepackage{stmaryrd}
\usepackage{mathtools}
\usepackage{extarrows}

\usepackage[all,cmtip]{xy}

\usepackage{algorithm}  

\usepackage{listings}

\lstdefinelanguage{Maple}%
{morekeywords={and,assuming,break,by,catch,description,do,done,%
elif,else,end,error,export,fi,finally,for,from,global,if,%
implies,in,intersect,local,minus,mod,module,next,not,od,%
option,options,or,proc,quit,read,return,save,stop,subset,then,%
to,try,union,use,uses,while,xor,with,permute,assume,simplify},%
sensitive=true,%
morecomment=[l]\#,%
morestring=[b]",%
morestring=[d]"%
}[keywords,comments,strings]%

\usepackage{color}
\definecolor{marin}{rgb}   {0.,   0.1,   0.5} 
\definecolor{rouge}{rgb}   {0.8,   0.,   0.} 
\definecolor{sepia}{rgb}   {0.4,   0.25,   0.} 
\definecolor{mag}{rgb}   {0.3,   0,   0.3} 
\usepackage[colorlinks,citecolor=sepia,linkcolor=marin,urlcolor=mag ,bookmarksopen,bookmarksnumbered]{hyperref}

\newtheorem{theorem}{Theorem}
\newtheorem*{informaltheorem}{Informal Theorem}
\newtheorem*{assumption}{Assumption}
\newtheorem{lemma}{Lemma}[section]
\newtheorem{proposition}[lemma]{Proposition}
\newtheorem{corollary}[lemma]{Corollary}

\theoremstyle{definition}
\newtheorem{definition}[lemma]{Definition}
\newtheorem{remark}[lemma]{Remark}

\newcommand{\Irr}{\mathcal{I}rr}
\newcommand{\last}{\mathrm{last}}

\newcommand{\M}{\mathcal M}
\newcommand{\E}{\mathcal E}
\newcommand{\J}{\mathcal J}
\newcommand{\Rc}{\mathcal R}

\newcommand{\Z}{\mathbb{Z}}

\newcommand{\hb}{\mathbf{h}}
\newcommand{\kb}{\mathbf{k}}

\newcommand{\kt}{\kappa_\tau}

\newcommand{\T}{\mathbb{T}}
\newcommand{\R}{\mathbb{R}}

\newcommand{\eps}{\varepsilon}

\makeatletter
\newcommand{\dotDelta}{{\vphantom{\Delta}\mathpalette\d@tD@lta\relax}}
\newcommand{\d@tD@lta}[2]{%
  \ooalign{\hidewidth$\m@th#1\mkern-1mu\cdot$\hidewidth\cr$\m@th#1\Delta$\cr}%
}
\makeatother

\begin{document}

\title[Long time dynamics for KdV and BO]{ Long time dynamics for generalized Korteweg-de Vries and Benjamin-Ono equations}

\author{Joackim Bernier }
\address{Institut  de  Math\'ematiques  de  Toulouse  ;  UMR5219,  Universit\'e  de  Toulouse  ;  CNRS,  Universit\'e  Paul Sabatier, F-31062 Toulouse Cedex 9, France} 
\email{joackim.bernier@math.univ-toulouse.fr}

 \author{ Beno\^it Gr\'ebert}
\address{Laboratoire de Math\'ematiques Jean Leray, Universit\'e de Nantes, UMR CNRS 6629\\
2, rue de la Houssini\`ere \\
44322 Nantes Cedex 03, France}
\email{benoit.grebert@univ-nantes.fr}



\maketitle

\begin{abstract}

We provide {an} accurate description of the long time dynamics of the solutions of the generalized  Korteweg-De Vries \eqref{gKdV} and Benjamin-Ono   \eqref{gBO} equations on the one dimension torus,  without external parameters, and that are issued from almost any (in probability and in density) small  and smooth initial data. We stress out that these two equations have unbounded  nonlinearities.\\ In particular, we prove a long-time stability result in Sobolev norm:  given a large constant $r$ and a sufficiently small parameter $\eps$, for generic initial datum $u(0)$ of size $\eps$,  we control the Sobolev norm of the solution $u(t)$ for times of order $\eps^{-r}$. These results are obtained by putting the system in rational normal form : we conjugate, up to some high order remainder terms, the vector fields of these equations to integrable ones on large open sets surrounding the origin in high Sobolev regularity.

%
\end{abstract}

\setcounter{tocdepth}{1} 
\tableofcontents

\section{Introduction}

During the last decades remarkable advances have been realized in the perturbation theory of Hamiltonian partial differential equations. On the one hand, some extensions of the KAM theory succeed to prove the existence of plenty of invariant tori for many systems (let us cite the pioneering  works \cite{Kuk87,Way90}, the works concerning  Korteweg-De Vries \cite{KP} and Benjamin-Ono equations   \cite{LY11} and a recent review paper \cite{Ber19}); but these invariant tori correspond to exceptional initial data and, most of time, they are only finite dimensional. 
On the other hand, considering only solutions in neighborhoods of elliptic equilibrium points, some extensions of the Birkhoff normal form theory enable a quite precise description of the dynamics (typically that it behaves like an integrable system) or, at least, some properties of stability for very long times ($\eps^{-r}$ where $\eps\ll 1$ is the size of the perturbation and $r\gg 1$ can be chosen arbitrarily large). These techniques have been designed for many kinds of models (see e.g. \cite{Bam03,BG06,BDGS07,GIP,Del12,FGL13,BD17,FI19,BMP20,FI20}) but, by nature, the system has to be non-resonant (i.e. the eigenvalues of the linearized systems have to satisfy some kind diophantine conditions).

For resonant systems, the dynamics can be much more complex. For example, some migrations of the energy to arbitrarily small spacial scales have been exhibited (for the nonlinear wave equation on the one dimensional torus \cite{Bou96}, for the nonlinear Schr\"odinger's equation on the two dimensional torus \cite{CKSTT10,CF12,GK15}, for the half-wave equation \cite{GG12}) but also some energy exchanges (for a cubic nonlinear Schr\"odinger's equation on the circle \cite{GV}, for a quintic nonlinear Schr\"odinger's equation on the one dimensional torus \cite{GT12,HP17}). 

However a large class of resonant Hamiltonian systems (among them the  nonlinear Schr\"odinger equation, the Korteweg-De Vries equation and the Benjamin-Ono equation on the one dimensional torus) enjoy a special property: they have no third order resonant term and the fourth order resonant terms are integrable. In that case, after two steps of resonant normal form, the new Hamiltonian system is integrable up to order four. It turns out that this fact can be used to obtain  stability results for non trivial times (see e.g. \cite{BFP20} and perhaps less obviously \cite{BFF20}). But an other consequence interests us even more in this article: the integrable  terms of order four provide a nonlinear correction to the linear frequencies and thus the initial data can be used as parameters to make the system become non-resonant. 
Taking advantage of this property, in \cite{KuP}, Kuksin and P\"oschel proved, using a KAM approach, the existence of quasi-periodic solutions to the nonlinear Schr\"odinger equation on the one dimensional torus $\mathbb{T} = \mathbb{R}/\mathbb{Z}$
\begin{equation}
\tag{NLS}\label{nls}
i\partial_t u=-\partial_x^2 u+ \varphi(|u|^2) u, 
\end{equation}
where $\varphi\in \mathcal{C}^{\infty}( \R; \R)$ is an analytic function on a neighborhood of the origin satisfying $\varphi'(0) \neq 0$. But as we said before, these solutions are exceptional, actually they correspond to finite dimensional invariant tori. Yet these nonlinear corrections to the frequencies has also be used by several authors to study the dynamics and the stability of the solutions of \eqref{nls} living outside of these invariant tori. First, in \cite{Bam99}, with a geometrical approach, Bambusi proved the stability in $H^1$ (the energy space), on exponentially long times, of the odd solutions of \eqref{nls} essentially finitely supported in Fourier. Then Bourgain, in \cite{Bou00}, proved the stability of the small generic solutions of \eqref{nls} in $H^s$, $s\gg 1$. His proof relies on a very local normal form construction: somehow, he linearized the system around some non-resonant initial data (i.e. making the system become non-resonant). Finally, in \cite{KillBill}, introducing a new normal form process based on rational Hamiltonians, and called \emph{rational normal form}, up to some high order remainder term, we conjugate \eqref{nls} to an integrable system on large open set surrounding the origin. It provides a more uniform stability result than \cite{Bou00} and it enables an accurate description of the typical dynamics of \eqref{nls}.  To the best of our knowledge, the nonlinear Schr\"odinger equations on $\T$ constitute the only family of equations for which this kind of result have been established (the Schr\"odinger-Poisson equation on $\T$ is also considered in \cite{KillBill}).

In this paper, extending our technic of rational normal forms, we make this kind of result available for two other important classes of dispersive partial differential equations: the generalized Korteweg-de-Vries equation on $\mathbb{T} = \mathbb{R}/\mathbb{Z}$
\begin{equation}
\label{gKdV}
\tag{gKdV}
\partial_t u = \partial_x (-\partial_x^2 u + f(u))
\end{equation}
and the generalized Benjamin-Ono equations on $\mathbb{T}$
\begin{equation}
\label{gBO}
\tag{gBO}
\partial_t u = \partial_x (|\partial_x| u + f(u))
\end{equation}
where $f\in \mathcal{C}^\infty(\mathbb{R};\mathbb{R})$ is a smooth function, analytic in a neighborhood of the origin, and $u(t):\mathbb{T}\to \mathbb{R}$ satisfies\footnote{This assumption is usual for both KdV and BO. equations but it is not really necessary: we can work in a moving frame to avoid it.}
$$
\int_{\mathbb{T}} u(t,x) \ \mathrm{d}x = 0.
$$
For these two classes of PDEs, as a dynamical consequence of our construction of rational normal form we obtain, not only a stability result in Sobolev norm, but a fairly accurate description of the dynamics for long times and for almost any (in a sense to be defined) small initial datum. Note that it answers (partially) to the problem Problem 5.18. of the survey paper of Guan and Kuksin on KdV \cite{GK13}.

Let us introduce some notations to state the result that we prove in this article for \eqref{gKdV} and \eqref{gBO}.\\

With a given function $u \in L^2(\mathbb{T})$ of zero average, we associate the Fourier coefficients $(u_a)_{a\in\Z^*} \in \ell^2$ defined by 
$$
u_k = \int_{\T} u(x) e^{- 2i\pi kx} \mathrm{d} x.
$$
In the remainder of the paper we identify the function  with its sequence of Fourier coefficients $u = (u_k)_{k\in\Z^*}$, 
and  we consider the Sobolev spaces ($s\geq0$)
\begin{equation}
\label{Hs}
\dot H^s = \{\,  u = (u_k)_{k\in\Z^*} \in L^2(\mathbb{T}) \quad \mid \quad \|u\|_{\dot{H}^s}^2 := \sum_{k \in \Z^*}  |k|^{2s} |u_k|^2 < +\infty \, \}.
\end{equation}
Let $F$ be the primitive of $f$ vanishing in $0$ and denote by  $(a_m)_{m\geq 3}$  the sequence of the Taylor coefficients of $F$ at the origin :
\begin{equation}
\label{def:am}
F(y) \mathop{=}_{y\to 0} \sum_{m =3}^\infty a_m y^m.
\end{equation}

To prove that the set of functions we describe in this paper is not empty, we need the following non degeneracy assumptions on the nonlinearly.
\begin{assumption} 
To deal with {\rm \ref{gKdV}}, we have to assume that
\begin{equation}
\label{assump:gKdV}
\tag{$\mathcal{A}_{\mathrm{gKdV}}$}
\forall k \in \mathbb{N}^*, \ 4 \, \pi^2 \, k^2\, a_4  + a_3^2 \neq 0
\end{equation}
whereas to deal with {\rm \ref{gBO}}, we just have to assume that
\begin{equation}
\label{assump:gBO}
\tag{$\mathcal{A}_{\mathrm{gBO}}$}
a_4\neq 0.
\end{equation}
\end{assumption}

Our main result gives an approximation of the flows of \ref{gBO} and \ref{gKdV} during very long times, for initial datum in open subsets of $\dot H^s$ surrounding the origin. 
\begin{theorem}
\label{mainPDE} Let $\mathcal{E}\in \{\mathrm{\ref{gBO},\ref{gKdV}}\}$ and assume $\mathcal{A}_{\mathcal{E}}$. For all $r>0$, there exists $s_0\equiv s_0(r)$ and for all $s> s_0$ there exists an open set $\mathcal{V}_{r,s}^{\mathcal{E}}\subset \dot H^s$ such that 
for all $u\in \mathcal{V}_{r,s}^{\mathcal{E}}$, while $|t|\leq \|u\|_{\dot{H}^s}^{-r}$, the solution of $\mathcal{E}$ initially equals to $u$, and denoted $\Phi_t^\E(u)$, exists and satisfies
\begin{equation}
\label{desc:norm}
 \|\Phi_t^\E(u)\|_{\dot H^s}\leq 2 \|u\|_{\dot H^s}
\end{equation}
and there exist $C^1$ functions $\theta_k:\R_+\mapsto \R$, $k\in\Z^*,$ satisfying the following estimates
\begin{equation}
\label{desc:dynamique}
 \|\Phi_t^\E(u) -\sum_{k\in \mathbb{Z}^*} u_k e^{i\theta_k(t)} e^{2i\pi kx}\|_{\dot H^{s-1}}\leq  \|u\|^{3/2}_{\dot H^s},
\end{equation} 
\begin{equation}
\label{desc:modulation}
\forall k \in \mathbb{Z}^*, \ |\dot\theta_k(t)-\omega^{\E}_{k}(u)|\leq |k| \|u\|^{5/2}_{\dot H^s}.
\end{equation}
where $\omega^{\E}_{-k} = -\omega^{\E}_{k}$ and if $k>0$
$$
\omega^{\mathrm{\ref{gKdV}}}_{k}(u) = (2\pi k)^{3}  + 12 \pi \,k \, a_4\, \| u\|_{L^2}^2 -  (4 \, \pi^2 \, k^2\, a_4  + a_3^2) \frac{6|u_k|^2}{k\pi}
$$
$$
\omega^{\mathrm{\ref{gBO}}}_{k}(u) = (2\pi k)^{2} +  12 \pi \,k \, a_4\, \| u\|_{L^2}^2 -  24\, \pi k \,a_4 |u_k|^2-  36\, a_3^2 \sum_{\ell=1}^{\infty}  \frac{k}{\max(k,\ell)} |u_\ell|^2 ,
$$
Moreover, the open set $\mathcal{V}_{r,s}^{\mathcal{E}}$ is invariant by translation of the angles
\begin{equation}
\label{def:inv_trans_ang}
\sum_{k\in \mathbb{Z}^*} u_k e^{2i\pi kx} \in \mathcal{V}_{r,s}^{\mathcal{E}}  \iff \sum_{k\in \mathbb{Z}^*} |u_k| e^{2i\pi kx} \in \mathcal{V}_{r,s}^{\mathcal{E}},
\end{equation}
 is asymptotically dense, i.e.
\begin{equation}
\label{def:asympt_dense}
\exists \, \varepsilon_{r,s}>0,\forall v \in \dot{H}^s, \ \| v \|_{\dot{H}^s} \leq \varepsilon_{r,s} \ \Rightarrow \ \exists u \in  \mathcal{V}_{r,s}^{\mathcal{E}}, \ \| v - u \|_{\dot{H}^s} \leq \frac{\| v  \|_{\dot{H}^s}}{|\log \| v  \|_{\dot{H}^s}|}
\end{equation}
and asymptotically of full measure : if $u \in \dot{H}^s$ is a random function with real Fourier coefficients, i.e.
$$
u(x) := 2 \sum_{k=1}^\infty \sqrt{I_k} \cos(2\pi k x),
$$
whose actions, denoted $I_k$, are independent and uniformly distributed in $(\beta,1)\,k^{-2s-\nu}$, where $\beta \in [0,1/2]$ and $1< \nu \leq 9$, then there exists $\varepsilon_{r,s,\nu}>0$ such that  
\begin{equation}
\label{def:asympt_full_measure}
\forall \varepsilon\leq \varepsilon_{r,s,\nu}, \ \mathbb{P}(\varepsilon u \in \mathcal{V}_{r,s}^{\mathcal{E}}) \geq 1-\varepsilon^{\frac1{35}}.
\end{equation}
\end{theorem}
We comment this result and in particular the relevance of the explicit constants, in section \ref{comments}.

The defocusing cubic nonlinear Schr\"odinger equation (\eqref{nls} with $\phi\equiv1$), the Korteweg-De Vries and modified Korteweg-De Vries equation (\eqref{gKdV} with respectively $f(x)=3x^2$ and $f(x)=2x^3$) and the Benjamin-Ono equation (\eqref{gBO} with $f(x)=x^2$)  are certainly the most celebrated examples of integrable PDEs.  Moreover they are all integrable in the strongest possible sense: they admit global Birkhoff coordinates on the space $L^2(\T)$ (see \cite{KP} for KdV,  \cite{KST} for mKdV,  \cite{GK14} for  NLS and \cite{GK19} for BO). These are coordinates which allow to integrate the corresponding PDE by quadrature in the manner of a generalized Fourier transform. In particular for these integrable versions the $H^s$-norm is almost conserved by the flow  for all time and without restriction on the initial datum. Thus our result can be interpreted as a reminiscence of this integrability. Besides, we will make crucial use of the fact that the generalized equations  \eqref{nls}, \eqref{gKdV} and \eqref{gBO} are still integrable up to order four (see Lemma \ref{int-order4}). We notice that, surprisingly, our assumption \eqref{assump:gBO} does not allow us to include the integrable version of the Benjamin-Ono equation in our Theorem. In fact, we will see that  the resonant normal form of the Benjamin-Ono equation contains no 6th-order integrable terms and these play an important role in our construction\footnote{In fact the results of \cite{GK19} suggest that the normal form of the Benjamin-Ono equation is equal to its development of order 4 without remainder }.  In any case, we are not claiming that our assumptions are necessary to obtain the dynamical consequences described in the Theorem \ref{mainPDE}. Rather, let us say that we need some non-degeneracy of the resonant normal form and that different sets of assumptions may suffice, of which \eqref{assump:gKdV} and \eqref{assump:gBO} are the most natural. 

\medskip

Before ending this introduction, note that the result for  \eqref{gKdV} and \eqref{gBO}  is not a simple transposition of the method developed in  \cite{KillBill}. We encounter here two new major problems:
\begin{itemize}
\item the symplectism generates vector fields which are unbounded (in \eqref{poisson1} there is an extra $\partial_x$) and it is therefore all the more difficult to control the dynamics they generate.
\item the odd order terms in the nonlinearity (especially those associated with $a_3$ and $a_5$) generate new kinds of nonlinear interactions (a large part of the technicalities of this paper are devoted to control them).
\end{itemize}

For more details on the new difficulties we have had to face, the reader can refer to  sections \ref{unbounded} and \ref{complications}.

\vspace{2em}
\noindent{\bf Acknowledgments.} This work was initiated after a discussion between one of us (J.B.) and P. G\'erard. It seemed improbable at that time to be able to overcome all the problems we could imagine, but this discussion convinced us to insist.  J.B. also thanks J.F. Coulombel for pointing out some relevant references. Finally both of us give special thanks to E. Faou for many enthusiastic discussions at the beginning of this work.\\
During the preparation of this work, J.B. was supported by ANR project "NABUCO", ANR-17-CE40-0025 and B.G benefited from the support of the Centre Henri Lebesgue ANR-11-LABX-0020-0 and was supported by ANR-15-CE40-0001-02  "BEKAM"  and ANR-16-CE40-0013 "ISDEEC" of the Agence Nationale de la Recherche. \\
Some computations in this paper were performed by using Maple\texttrademark.

\section{Strategy of proof and overview of difficulties}
Our result is based on a normal form approach. After a commented statement of the normal form result in section \ref{sec:NF}, we detail in section \ref{sec:dif} some of the difficulties we had to face in the proofs and we finish in section \ref{comments} with technical comments on Theorem \ref{mainPDE}.
\subsection{Hamiltonian setting and normal form result}\label{sec:NF}.

The phase space $\dot H^s$ (see \ref{Hs}) is equipped  with the Poisson bracket 
\begin{equation}\label{poisson1}
\{ P,Q\}(u) = \int_{\mathbb{T}} \nabla P(u) \partial_x \nabla Q(u) \ \mathrm{d}x 
\end{equation}
which reads in Fourier variables
\begin{equation}\label{poisson2}
\{ P,Q\} (u)= \sum_{k\in \mathbb{Z}^{*}} (\partial_{u_{-k}} P) (2i\pi k)  (\partial_{u_k} Q)
\end{equation}

Equations \eqref{gKdV} and \eqref{gBO} are Hamiltonian systems associated with the following Hamiltonians
$$
\left\{\begin{array}{llccc}
 H_{\mathrm{\ref{gKdV}}}(u) &=& \displaystyle \int_{\mathbb{T}} \frac12 u (-\partial_x^2) u &+& F\circ u \ \ \mathrm{d}x \vspace{0.2cm} \\
H_{\mathrm{\ref{gBO}}}(u) &=& \displaystyle  \int_{\mathbb{T}}  \frac12 u |\partial_x| u &+& F\circ u \ \ \mathrm{d}x
\end{array}\right.
$$
where $F$ is the primitive of $f$ vanishing in $0$. Recalling that $(a_k)_{k\geq 3}$ are the Taylor coefficients of $F$ at the origin (see \eqref{def:am}) and that by assumption $f$ is analytic in a neighborhood of the origin, the Hamiltonian $H_{\mathcal{E}}$ writes on a neighborhood of the origin in $\dot{H}^s$
\begin{equation}\label{HE}
H_{\mathcal{E}}(u) = Z_2^{\mathcal{E}}(I)  + \sum_{m \geq 3} a_m  \int_{\mathbb{T}} (u(x))^m \ \mathrm{d}x.
\end{equation}
where $I_k = |u_k|^2$, $k\in\Z^*$ and
\begin{equation}\label{def:Z2}
Z_2^{\mathcal{E}}(I) = \sum_{k\in \mathbb{Z}^{*}} \frac{|2\pi k|^{\alpha_\E}}2 I_k
\end{equation}
with 
\begin{equation}
\label{def:alpha}
\alpha_{\mathrm{\ref{gBO}}} = 1 \quad \mathrm{and}\quad \alpha_{\mathrm{\ref{gKdV}}} = 2.
\end{equation}
We obtain Theorem \ref{mainPDE} as a dynamical consequence of a normal form result saying that the Hamiltonian system associated to   \eqref{gKdV} (resp. \eqref{gBO}) is integrable up to remainder terms.
We  now give an informal version of our normal form result that allow us to explain our strategy. A more precise version (but more technical!) is given in Theorem \ref{thm:KtA}.
\begin{informaltheorem} 
\label{mainth0}
Let $H$ equals $H_{\mathrm{\ref{gKdV}}}$ (resp. $H_{\mathrm{\ref{gBO}}}$), assume \eqref{assump:gKdV} (resp. \eqref{assump:gBO} and 
let  $r\gg 7$.  For all $s\geq  s_0(r)=O(r^2)$,  for all $N \geq N_0(r,s)$  and for all $0<\gamma<\gamma_0(r,s)$ there exists a large open set $\mathcal{C}_{r,s,\gamma,N}^{\mathcal{E}}$ and $\tau$ a symplectic change of variable close to the identity, defined on $\mathcal{C}_{r,s,\gamma,N}^{\mathcal{E}}$ and taking values in $\dot H^s$, that puts  $H$ in normal form up to order $r$:
$$
H \circ \tau(u)  = Z^\E(I) + R^\E_1(u)+R^\E_2(u)
$$
where 
$Z^\E(I)$ is a smooth function  of the actions and $R^\E(u)=R^\E_1(u)+R^\E_2(u)$ is a remainder term  in the following sense
there exist universal positive constant $\mu, \nu$ such that
\begin{align}\label{NR1}
\left|\{\|\cdot\|_{\dot H^s}^2, R^{\E}_1\}(u)\right|&\lesssim_{s,r} N^{\mu r}(N^{-s+\nu r}+ \|u\|_{\dot H^s}^{r+1})\\
\label{XR1} \| \partial_x \nabla_u R^{\E}_1(u)\|_{\dot H^{s-1}}&\lesssim_{s,r} N^{\mu r}(N^{-s+\nu r}+ \|u\|_{\dot H^s}^r);
\end{align}
 there exist universal positive constant $\mu',\nu'$ such that
\begin{equation}
\label{XR2} \| \partial_x \nabla_u R^{\E}_2(u)\|_{\dot H^{s}}\lesssim_{s,r} N^{\mu' r^2}\gamma^{-\nu'r} \|u\|_{\dot H^s}^r .
\end{equation}
\end{informaltheorem}
(This statement is informal in the sense that \eqref{XR1} is in fact valid only for a part of $R^\E_1$ and that we do not precise the meaning of {\it large} open set and of {\it close to} the identity.)

In this Theorem, and in all the paper, $N$ is a troncation parameter in Fourier modes: we remove all the monomials of order greater than 5 containing\footnote{The exact truncation we use is in fact more complicated: it is given by Definition \ref{J} which in turn is related to the third largest  index by Corrolary \ref{cor:origin}.} at least three indices of size greater than $N$. The remainder term, as already noticed, see for instance \cite{Bam03,Bou00, BG06, G}, has thus a vector field  of order $N^{-s}\|u\|_{\dot H^s}^4$. This justifies the term $N^{-s}$ in \eqref{NR1}, \eqref{XR1} and will lead us to chose $N\sim \|u\|_{\dot H^s}^{-\frac rs}$. The parameter $\gamma$ is related to the control of the so called small divisors and thus measures the size of the set  $\mathcal{C}_{r,s,\gamma,N}^{\mathcal{E}}$. The set $\mathcal{V}_{r,s}$ of Theorem \ref{mainPDE} will be essentially (but not exactly!) constructed as the union over $N\geq N_0(r,s)$ and over $0<\gamma<\gamma_0(r,s)$ of $\mathcal{C}_{r,s,\gamma,N}^{\mathcal{E}}$.\\
 The informal Theorem distinguishes two types of remainders:  $R^{\E}_1$ will be generated by the resonant normal form while $R^{\E}_2$ will be the product of the rational normal form ( see section \ref{sec:bnf} just below). We notice that \eqref{XR2}, saying that the Hamiltonian vector field of $R^{\E}_2$ is controlled  in $\dot H^s$-norm, implies the control of $|\{\|\cdot\|_{\dot H^s}^2, R^{\E}_2\}(u)|$ and $\| \partial_x \nabla_u R^{\E}_2(u)\|_{\dot H^{s-1}}$. The estimate of 
$|\{\|\cdot\|_{\dot H^s}^2, R^{\E}\}(u)|$ will be used to obtain the stability result \eqref{desc:norm} while the estimate of $\| \partial_x \nabla_u R^{\E}(u)\|_{\dot H^{s-1}}$ will be used to obtain the leading terms of the dynamics, i.e. \eqref{desc:dynamique} and \eqref{desc:modulation} with $\omega_k^{\E}=\partial_{I_k} (Z_2^\E+Z_4^\E)$ where $Z_4^\E$ is the homogeneous part of degree  4 (in $u$) of $Z$ and is given by \eqref{Z4-KdV}, \eqref{Z4-BO}.

\subsection{Overview of difficulties}\label{sec:dif}

As we said in the introduction, after getting a normal-form result for \eqref{nls}, it was logical to want to extend it to \eqref{gKdV} and \eqref{gBO} which are, like \eqref{nls}, hamiltonian perturbations of integrable nonlinear PDEs. In this section, we want to recall the main steps of such an undertaking, but above all we want to emphasize the new problems we have encountered in their implementation for these two equations.

\subsubsection{A first obstruction to bypass: the vector fields are unbounded}\label{unbounded}
Before even tackling the generalization of rational normal form, we are faced with a major problem: \eqref{nls} is a semi-linear equation while \eqref{gKdV} and \eqref{gBO} are not.  In particular the vectorfield $\partial_x \nabla P(u)$ of a smooth Hamiltonian $P$ is not bounded as a map from $\dot H^s$ to $\dot H^s$, even when $P$ is a polynomials. Under these conditions, how can we define the Lie transforms and thus construct our changes of variable? So let's start by explaining how we deal with this problem:
\begin{itemize}
\item First we see, using a commutator Lemma (a bit in the spirit of pseudo-differential calculus: the commutator gains one derivative), that even if the vector field of a regular polynomial does not send $\dot H^s$ into $\dot H^s$, it generates a flow that preserves the $H^s$ norm. This can be seen in the Proposition \ref{prop-Ns} and in particular in \eqref{k1k2}.
\item Then our  changes of variable are Lie transforms associated with very particular Hamiltonians $\chi$: solutions of homological equations. This particularity means that all the  monomials $u^k$, $k\in(\Z^*)^n$, appearing in $ \chi$  are divided by a so called small divisor that reads here $D_k=k_1|k_1|^{\alpha_{\mathcal{E}}}+\dots+k_r|k_r|^{\alpha_{\mathcal{E}}}$.  But it turns out that a small divisor is not always small! In fact we show in the Lemma \ref{lem:origin} that either $D_k$ is large with respect to the largest index of $k$ or the third largest index (i.e. the third largest number among $|k_1|,\cdots,|k_n|$) is large. This crucial Lemma, which will be reused many times, is the key that led us to tie ourselves up in this work. The first consequence is the Lemma \ref{lem-chi} that ensures that our Lie transforms are well defined.
\end{itemize}

\subsubsection{The principle of rational normal forms inherited from \cite{KillBill}}\label{sec:bnf}  
We will be rather brief in this paragraph since this is an approach that has already been implemented for \eqref{nls} in \cite{KillBill}. We note\footnote{This was brought to our attention by R. De La Llave.} that the idea of using rational normal forms rather than polynomials goes back to Moser \cite{Mos60} and Glimm \cite{Gli64} in finite dimensional context.\\
The purpose of a Birkhoff normal form  is to iteratively kill all the non-integrable terms in the non-linear part of the Hamiltonian. To do this we first average the Hamiltonian along the linear flow generated by $Z^\E_2$. This is the resonant normal form step (resonant because unfortunately the linear frequencies, $2\pi k|k|^{\alpha_\E}$, $k\in\Z^*$, are resonant).  Despite these linear resonances, this step allows us to get rid of degree three terms and to keep order four terms only those that are integrable, i.e. depending only on actions. This last point is crucial in order to continue the procedure: the resonant  terms of degree 4 must be integrable terms. This was true for \eqref{nls} and it is still true for \eqref{gKdV} or \eqref{gBO}, we can see it as a consequence of the fact that these three equations are perturbations of integrable equations. So after two steps of resonant normal form we end up with a Hamiltonian in the form 
$$Z_2^{\E}(I)+Z_4^{\E}(I)+O(u^5).$$
Of course we still have terms of degree 5 and more which are resonant but  not integrable. The idea then is to average, not along the flow of $Z_2^{\E}$, but along the flow of $Z_2^{\E}+Z_4^{\E}$.\\ This approach is similar to the one used by Kuksin-P\"oschel in \cite{KuP}: we use the integrable part of the nonlinearity to destroy the linear resonances by modulating linear frequencies with terms dependent on actions and thus directly related to the initial conditions. In this sense it is a fundamentally different approach from earlier work on Birkhoff's normal forms where external parameters were used to get rid of linear resonances. Here we use internal parameters: the initial datum.\\ This approach has a high cost: this time the small divisors, which are derivatives with respect to actions of $Z_2^{\E}+Z_4^{\E}$, are linear functions of the actions and the new Hamiltonians are no longer polynomials but rational fractions. It is thus a question of justifying their existence by controlling very precisely the cancellation places of these small divisors and by following step by step the type of rational fractions we generate. This step is similar to the one we implemented in \cite{KillBill} although here we have a somewhat more natural presentation (see section \ref{ratHam}). As in this previous work, it turns out that $Z_4$ is not enough to solve all our problems. In fact the small divisor of the monomials $u^k$, $k\in (\mathbb{Z}^*)^n$, associated with $Z_4$ is too degenerate: it is controlled by $\mu_{\mathrm{min}}(k)^{2s}$ where $\mu_{\mathrm{min}}(k)$ denotes the smallest number among $|k_1|,\cdots,|k_n|$. So the averaging step costs potentially $2s$ derivatives which  of course is not acceptable.
And even if we try to compensate this denominator by the smallness of the numerator, for some combinations, involving various small divisors, that appear from order 10 of the process, we can no longer control the Hamiltonian in $\dot H^s$.
So, as in \cite{KillBill}, we have to go and look for the integrable terms of degree 6. Fortunately it turns out that the small divisor associated with $Z_6^\E$ is less degenerate and so we can complete the normal form procedure in $\dot H^s$. The reader can get an idea of the open sets on which we control the rational Hamiltonians in the Definition \ref{U}.
\subsubsection{Complications specific to \eqref{gKdV} and \eqref{gBO}}\label{complications}
In this paragraph we want to highlight the technical difficulties (but not only!) that are specific to the case of \eqref{gKdV} and \eqref{gBO}. Their resolution required new ideas. The first major obstruction is described in section \ref{unbounded}, here is a (non exhaustive) list of the other novelties:
\begin{itemize}
\item The reader who is somewhat familiar with normal forms may be surprised by the length of the section devoted to the resonant normal form. The normal form at order  6 for  \eqref{gKdV} is already well known (see \cite{KP} ). This said, this calculation was only formal and here we need to control very precisely the remainders to obtain the dynamical consequences we are looking for. This is what is obtained in the Proposition \ref{prop-Ns} and is of course related to the problems of unbounded vector fields mentioned in the section \ref{unbounded}.
\item The second problem related to the resonant normal form lies in the explicit calculation of the terms of order 6 within the framework of the generalized equations. Already in \cite{KP} the formal computation of $Z_6$ in a more restricted framework (the integrable KdV equation), leads to an appendix of nine pages. Here in the more general context (and including \eqref{gBO}) a by hand calculation would require many more pages. This naturally led us to use a computer software (in this case we have chosen Maple) but also to use graphs to describe the computation and allow to follow the details more easily (see section \ref{formal}).
\item In Section 2, probability estimates are based on relatively basic lemmas, in particular Lemme \ref{lem:proba:poly}. Nevertheless, to deal with the more degenerate cases, some refinements were required. For example, in the case of \eqref{gBO}, if $a_3^2/(\pi a_4)$ is a Liouville number (i.e. is not diophantine), the continued fraction theory must be used (see Lemme \ref{lem:continuedfractions}). 
\item We  mention that, as in \cite{KillBill}, the truncation in Fourier modes is particularly important to have non resonance conditions that are stable when actions are slightly moving (see section \ref{sec:stab}). 
\item We have already mentioned above the difficulty of the formal computation of $Z_6$. In fact in Theorem \ref{thm-BNF} we do not calculate all the terms of $Z_6$, it would be tedious and fortunately useless. For example the resolution of terms of order 5 leads to terms of order 6 which are certainly integrable but not polynomial. These terms are grouped under the notation $Z_6^{\mathrm{fr}}$ (see \eqref{def:hom:chi_6}) and were not present in \cite{KillBill}. 
We consider these terms as remainder terms and therefore we do not include them in the $Z_6$ we use to average (otherwise it would generate too complicated rational fractions).  But because of this, at each step of the normal form to the order $r\geq7$, which is supposed to kill the non-integrable terms of order $r$, these terms will make new terms of order $r$ appear. It seems then that we are at an impasse. Yet it turns out that the averaging step is in a certain sense regularizing. It is this phenomenon, well known to specialists (the resolution of a homological equation gives a little extra regularity), that allows us in section \ref{sub:sub:trans} to get rid of these new terms by a transmutation procedure in which the Lemma \ref{lem-chi} again plays a crucial role (see section \ref{sub:sub:trans} to understand this unusual procedure quite difficult to explain in the introduction).

\end{itemize}

\subsection{Comments on Theorem \ref{mainPDE}}\label{comments}
 For simplicity and to make disappear many irrelevant constants, we did not try to optimized many of the explicit constants in Theorem \ref{mainPDE}. 
\begin{itemize}
\item In order to have estimates quite uniform in $\nu$ in the proof, we choose arbitrarily to fix the upper bound $\nu\leq 9$. Nevertheless, we could have fixed any other upper bound of the type $1<\nu\leq \nu_0$.  
\item The denominator $|\log \| v  \|_{\dot{H}^s}|$ in  \eqref{def:asympt_dense} is not optimal. Following the proof, for any fixed $\tau>0$, it could be replaced by $|\log \| v  \|_{\dot{H}^s}|^{\tau}$ ($\varepsilon_{r,s}$ would depend on $\tau$).
\item The exponent in $1/35$ in \eqref{def:asympt_full_measure} is not optimal. It seems possible to improve it paying more attention to the exponents of $\gamma$.
\item It is proven that $s_0$ can be chosen equal to $25\cdot 10^7 \, r^2$. Of course, we did not try to optimized this huge constant. It is only relevant to note that in this construction, the minimal value of $s$ growths at least like $r^2$.
\item The exponent $3/2$ (resp. $5/2$) in \eqref{desc:dynamique} (resp. \eqref{desc:modulation}) could be chosen arbitrarily close to $2$ (resp. $3$) nevertheless the exponent $1/35$ in \eqref{def:asympt_full_measure} would become arbitrarily small.
\end{itemize}

\subsection{Outline of the work} In Section \ref{sec:res}, we put \ref{gBO} and \ref{gKdV} in resonant normal form: we remove the terms of $H_\E$ that do not commute with $Z_2^\E$. In Section \ref{sec:sd}, we define the small divisors we will need to kill the non-integrable resonant terms, we introduce some open sets where they are under control and we prove that these open sets are large. In Section \ref{ratHam}, we introduce a class of rational Hamiltonians, we prove its stability by Poisson bracket and we develop some tools to control them. In Section \ref{section6}, we put  \ref{gBO} and \ref{gKdV} in rational normal form : expanding the Hamiltonian in the previous class, we remove the non-integrable resonant terms. Finally, as a corollary of this normal form result, in Section \ref{sec:dyn} we prove the Theorem \ref{mainPDE}.

\subsection{Notations} We introduce some convenient notations to work with Fourier coefficients and multi-indices.

\medskip

\noindent \emph{$*$ Monomials.} As explained in introduction, we always identify functions in $L^2(\T)$ with their sequence of Fourier coefficients, i.e. if $u\in L^2(\T)$ then
$$
u=(u_k)_{k\in \mathbb{Z}^*}  \quad \mathrm{where} \quad u_k = \int_{\T} u(x) e^{- 2i\pi kx} \mathrm{d} x.
$$
Note that since $u$ is real valued, we always have
$$
u_{-k} = \bar{u}_k.
$$
If $\ell \in (\mathbb{Z}^*)^n$, for some $n\geq 1$, is a multi-index and $u\in L^2(\T)$ then the \emph{monomial} $u^\ell$ is defined by
$$
u^{\ell} = u^{(\ell_1,\dots,\ell_n)} = u_{\ell_1} \dots u_{\ell_n}.
$$
If $k\in \mathbb{Z}^*$ is an integer then $I_k$ denotes the \emph{action} of index $k$. It is the monomial defined by
$$
I_k = u_{-k}\, u_k = |u_k|^2.
$$
We extend the multi-index notation for the actions : if $\ell \in (\mathbb{Z}^*)^n$ then $I^{\ell} = I_{\ell_1} \dots I_{\ell_n}$.

\medskip

\noindent \emph{$*$ Norms.} If $S$ is a set and $1\leq p \leq \infty$, we equip $\mathbb{C}^S$ with the usual $\ell^p$ norm : 
$$
\forall S, \forall b\in \mathbb{C}^S, \ \| b \|_{\ell^p} = \left\{ \begin{array}{cll} \displaystyle \big(\sum_{j\in S} |b_j|^p\big)^{\frac1p} & \mathrm{if} \ p<\infty, \\
														\displaystyle \sup_{j\in S} |b_j| & \mathrm{if} p=\infty.
\end{array}\right.
$$
Sometimes, we also use the following shortcut
$$
 |\cdot|_1=\| \cdot \|_{\ell^1} . 
$$
Similarly, we endow the real valued measurable functions on $\T$ with the usual Lebesgue norms denoted $\| \cdot\|_{L^p}$. Note that, since we identify the functions in $L^2$ with their sequences of Fourier coefficients, their $\ell^p$ norms refer to the norms of their Fourier coefficients. In particular, by Parseval we have
$$
\forall u \in L^2(\T), \ \|u\|_{L^2} = \|u\|_{\ell^2}.
$$

\medskip

\noindent \emph{$*$ Sets of multi-indices.}
In all this paper we consider the following sets of indices :
$$
\mathcal{D}_n=\{ k\in (\mathbb{Z}^*)^n \ | \ |k_1|\geq \dots \geq |k_n| \},  \quad \Irr_n=\{ k\in \mathcal{D}_n \ | \ \forall j\in \llbracket 1,n-1\rrbracket, \ k_j\neq -k_{j+1}  \}
$$
$$
\mathcal{M}_n = \{ k\in (\mathbb{Z}^*)^n \ | \ k_1+\dots+k_n = 0\}, \quad  \mathcal{R}_n^{\mathcal{E}} = \{ k\in \mathcal{M}_n \ | \ k_1|k_1|^{\alpha_{\mathcal{E}}}+\dots+k_n|k_n|^{\alpha_{\mathcal{E}}} = 0\}.
$$
where $n\geq 0$, $\mathcal{E}\in \{\mathrm{\ref{gBO},\ref{gKdV}}\}$ and we denote 
$$
\mathcal{M}= \bigcup_{n\geq 3} \mathcal{M}_n ,\hspace{1cm} \mathcal{D}= \bigcup_{n\geq 2} \mathcal{D}_n , \hspace{1cm} \mathcal{R}^{\mathcal{E}}= \bigcup_{n\geq 2} \mathcal{R}_n^{\mathcal{E}}, \hspace{1cm} \Irr =  \bigcup_{n\geq 0} \Irr_n.
$$
Being given $k\in (\mathbb{Z}^*)^n$, for some $n\geq 1$, the \emph{irreducible part} of $k$, denoted $\Irr(k)$, is the element of $\Irr$ such that there exists $\ell \in (\mathbb{N}^*)^m$, for some $m\geq 0$, such that, for all $u\in L^2(\T)$, we have
$$
u^k = I^\ell u^{\Irr(k)}.
$$
To handle efficiently multi-indices, we introduce some natural but quite unusual notations.
\begin{itemize}
\item If $S$ is a set then with a small abuse of notation we denote by $\emptyset$ the element of $S^{0}$. In other words, every $0$-tuple is denoted $\emptyset$.
\item To avoid the use of two many parentheses, very often, we identify naturally sets of the form $(S_1^{S_2})^{S_3}$ with sets of the form $S_1^{S_2\times S_3}$. In other words, if $e\in (S_1^{S_2})^{S_3}$, $s_2\in S_2$ and $s_3\in S_3$ then $(e_{s_2})_{s_3}$ and $e_{s_2,s_3}$ denote the same thing. 
\item If $S$ is a set, $n\geq 0$ and $x \in S^n$ then $\# x$ denotes the \emph{length} of $x$, i.e.
$$
\# x = n. 
$$
\item If $S$ is a set, $n\geq 1$ and $x\in S^n$ then $x_{\last}$ denotes the last coordinate of $x$, i.e. 
$$
x_{\last}= x_n = x_{\# x}.
$$
\item If $S$ is a set, $n\geq 1$ and $x\in S^n$, $\mathrm{ss}(x)$ denotes the set of its \emph{subsequences}, i.e. $y\in \mathrm{ss}(x)$ if and only if there exists an increasing sequence $1\leq \sigma_1 < \dots < \sigma_m \leq n$ such that $y = (x_{\sigma_1},\dots, x_{\sigma_m})$. Furthermore, $\mathrm{ss}_m(x)$ denotes the set of the subsequences of $x$ of length $m$.
\end{itemize}

\medskip

\noindent \emph{ $*$ A specific set of multi-indices.} To estimate the probability to draw a non-resonant initial datum (i.e. such that the dynamical consequences of Theorem \ref{mainPDE} hold) it is crucial to take into account the multiplicities of the irreducible multi-indices. However, the set we have defined just before is not well suited to manage them. Consequently, in Section \ref{sec:sd} (and exceptionally in Section \ref{sec:dyn}), we use the following set of multi-indices
\begin{equation}
\label{def:MImult}
\mathcal{MI}_{\mathrm{mult}} = \bigcup_{n\geq 2} \{ (m,k) \in (\mathbb{Z}^*)^{n}\times \mathbb{N}^n \ | \ m\cdot k = 0, \ \ k_1>\dots >k_n >0 \ \ \mathrm{and} \ \  |m|_1\geq 5 \}.
\end{equation}
If $\ell \in \Irr_n \cap \mathcal{M}$ is an irreducible multi-index of length $n\geq 5$, then there exists an unique $(m,k) \in \mathcal{MI}_{\mathrm{mult}}$ such that we have
\begin{equation}
\label{rel:l->(m,k)}
\forall u \in L^2(\T), \ u^\ell = \left(\prod_{m_j > 0} u_{k_j}^{m_j} \right) \left( \prod_{m_j < 0} u_{-k_j}^{-m_j}\right).
\end{equation}
This relation provides a correspondance between the elements of $\Irr \cap \mathcal{M}$ and those of $\mathcal{MI}_{\mathrm{mult}}$.

\begin{remark} \label{Il faut faire attention} We warn the reader that in Section \ref{sec:sd}, we define many objects (e.g. the small divisors) indexed by multi-indices in $\mathcal{MI}_{\mathrm{mult}}$ but in the other sections they are indexed by multi-indices in $\Irr \cap \mathcal{M}$. Implicitly, if an object is indexed by an element $\ell$ of $\Irr \cap \mathcal{M}$, it just refer to the same object indexed by the multi-index $(m,k)$ given by \eqref{rel:l->(m,k)}.
\end{remark}

\medskip

\noindent \emph{$*$ Other notations.} If $P$ is a property then $\mathbb{1}_P = 1$ if $P$ is true while $\mathbb{1}_P = 0$ if $P$ is false. Similarly, if $S$ is a set $\mathbb{1}$ denotes the characteristic function of $S$ : $\mathbb{1}_S(x)=1$ if $x\in S$ and $\mathbb{1}_S(x)=0$ else.

If $p$ is a parameter or a list of parameters and $x,y\in \mathbb{R}$ then we denote $x\lesssim_p y$ if there exists a constant $c(p)$, depending continuously on $p$, such that $x\lesssim c(p) \, y$. Similarly, we denote $x\gtrsim_p y$ if $y\lesssim_p x$ and $x\simeq_p y$ if $x\lesssim_p y\lesssim_p x$.
%

\section{The resonant normal form}
\label{sec:res}

In this section, being given $r>0$, we put our Hamiltonian systems in resonant normal form up to order $r$. In particular, we compute explicitly the integrable terms of order $4$ and partially those of order $6$. Furthermore, we expand analytically the remainder terms.

Before stating the main Theorem of this section. Let us precise how we ensure that our Hamiltonian are real valued and that our entire series converge.
\begin{definition}[reality condition] If $S \subset \M$ is symmetric, i.e. $-S=S$, then a family $b\in \mathbb{C}^S$ satisfies the \emph{reality condition} if
$$
\forall k \in S, \ \overline{b_k} = b_{-k}.
$$  
\end{definition}

\begin{lemma}
\label{lem:jaiunbeaurayon}
 Let $s\geq 1$, $M>0$ and $b\in \mathbb{C}^\M$. If $(b_k M^{-\#k})_{k\in \M}$ is bounded and $b$ satisfies the reality condition then the entire series
\begin{equation}
\label{monpremierpetitP}
P(u):=\sum_{k\in\M }b_k u^k
\end{equation}
converges and define a smooth real valued function on $B_s(0,\frac1{c_sM})$, the ball centered at the origin and radius $\frac1{c_sM}$ in ${\dot{H}^s}$ where $c_s=\big(\sum_{j\in\Z^*}\frac1{j^{2s}}\big)^{\frac12}\leq \frac{\pi}{\sqrt3}$.
\end{lemma}
\begin{proof} By Cauchy-Schwarz, the homogeneous polynomials $P_n(u)=\sum_{k\in\M_n}b_k u^k$ satisfies
\begin{equation*}
|P_n(u)|\lesssim M^n\big(\sum_{j\in\Z^*}\frac1{|j|^{s}}|j|^s|u_j|\big)^n\lesssim \big( c_s M \| u\|_{\dot H^s}\big)^n
\end{equation*}
\end{proof}

The following Theorem is the main result of this section. In the first subsection below, we justify that the remainder terms are small (which is not so obvious a priori). Somehow we prove that they do not contribute to the growth of the $\dot{H}^s$-norm. Then, in the second subsection, we put the system in resonant normal form without paying attention to the explicit expression of the fourth and sixth order integrable terms while the last subsection is devoted to their algebraic computation and to the conclusion of the proof of the following theorem.

\begin{theorem}[Resonant normal form] \label{thm-BNF} 
Being given $\mathcal{E}\in \{ \mathrm{\ref{gKdV}}, \mathrm{\ref{gBO}}\}$, $r\geq 6$, $s\geq 1$, $N\gtrsim_r 1$ and $\eps_0 \lesssim_{r,s} N^{-3}$, there exist two symplectic maps, $\tau^{(0)},\tau^{(1)}$, preserving the $L^2$ norm, making the following diagram to commute
\begin{equation}
\label{monbeaudiagramme_roidesforets}
\xymatrixcolsep{5pc} \xymatrix{  B_s(0,\eps_0) \ar[r]^{ \tau^{(0)} } \ar@/_1pc/[rr]_{\mathrm{id}_{\dot{H}^s}} & B_s(0,2\eps_0)  \ar[r]^{ \tau^{(1)} }  & 
 \dot{H}^s(\mathbb{T})  } 
\end{equation}
and close to the identity
\begin{align}\label{tau}
\forall \sigma\in \{0,1\}, \ \|u\|_{\dot{H}^s} <2^{\sigma}\eps_0 \ \Rightarrow \  \|\tau^{(\sigma)}(u)-u\|_{\dot{H}^s} \lesssim_r N^{3} \|u\|_{\dot{H}^s}^2
\end{align}
such that, on $B_s(0,2\eps_0)$, $H_{\mathcal{E}} \circ \tau^{(1)}$ writes
$$
H_{\mathcal{E}} \circ \tau^{(1)} = Z_2^{\mathcal{E}} + Z_4^{\mathcal{E}} + Z_{6,\leq N^3}^{\mathcal{E}} + \mathrm{Res}_{\leq N^3} + \mathrm{R}^{(\mu_3>N)} + \mathrm{R}^{(I_{> N^3})}+ \mathrm{R}^{(or)}
$$
where $Z_2^{\mathcal{E}}$ is given by \eqref{def:Z2}, $Z_4^{\mathcal{E}}$ is an integrable Hamiltonian of order $4$ given by the formulas
\begin{equation}\label{Z4-KdV}
Z_4^{\mathrm{\ref{gKdV}}}(I) = 3 \, a_4\, \| u\|_{L^2}^4 - \sum_{k=1}^{+\infty} \left(6\, a_4+ \frac{3\, a_3^2}{2\, \pi^2 \, k^2} \right) I_k^2
\end{equation}
\begin{equation}\label{Z4-BO}
Z_4^{\mathrm{\ref{gBO}}}(I) = 3 \, a_4\, \| u\|_{L^2}^4 - \sum_{k=1}^{+\infty} \left(6\, a_4+ \frac{9\, a_3^2}{ \pi \, k} \right) I_k^2- \sum_{0<p<q} \frac{18\, a_3^2}{\pi\, q} I_p I_q,
\end{equation}
 $Z_{6,\leq N^3}^{\mathcal{E}}$ is an integrable Hamiltonian that can be written
\begin{equation}
\label{premieredefZ6}
Z_{6,\leq N^3}^{\mathcal{E}}(I)=\sum_{0<p\leq q\leq\ell\leq N^3}c_{p,q}^\E(\ell)I_pI_qI_\ell \ \ \mathrm{with} \ \ |c_{p,q}^\E(\ell)|\lesssim \ell
\end{equation}
and $c_{p,q}^\E(\ell)\in \mathbb{R}$. If $0<2\,p<q$, the coefficient $c_{p,p}^\E(q)$ is explicitly given by 
$$
c_{p,p}^{\mathrm{\ref{gKdV}}}(q) =180 \, a_6 -\,{\frac {3\, a_{3}^{4}}{{\pi}^{6} {p}^{2} (p^2-q^2)^2}}
-\,{\frac {24\, a_{3}^2\,a_{4}}{{\pi}^{4} (p^2-q^2)^2}}-\,{\frac {48\, {p}^{2}\,a_4^{2}}{{\pi}^{2} (p^2-q^2)^2}} 
 -\,{\frac {60\, a_{3}\,a_{5}}{{\pi}^{2}{p}^{2}}}
$$
$$
c_{p,p}^{\mathrm{\ref{gBO}}}(q)  =180 \, a_6-{\frac {288\,(p-q)\, a_{4}^{2}}{\pi \left( p+2\,q \right)  \left( 3\,p-2\,q \right) }
} 
 -\frac{360 \, a_3 \, a_5 \, (p+q)}{\pi \,p\, q} + {\frac { 108\, \left( {p}^{2}+6\,p\,q-6\,{q}^{2} \right) }{ \pi^2 \,
p\, \left( p-q \right) {q}^{2}}}a_{3}^{2} a_{4}.
$$
 The four remaining Hamiltonians are some analytic functions with coefficients satisfying the reality condition and are of the form 
\begin{align}\label{Res}
\mathrm{Res}_{(\leq N^3)}(u) &= \sum_{\substack{k \in \mathcal{R}^{\mathcal{E}}\cap \mathcal D \\ 5\leq \#k \leq r \\ |k_1| \leq N^3 \\ \mathrm{Irr}(k)\neq \emptyset \ \mathrm{if} \  \#k  = 6 }} c_k \, u^k \ \ \mathrm{with} \  |c_k| \lesssim_{\#k}   N^{3{ \#k -9}},\\
\label{Rmu}
\mathrm{R}^{(\mu_3>N)}(u) &=  \sum_{\substack{k \in \mathcal{M}\cap \mathcal D \\ 4\leq  \#k  \leq r \\ k_3 \geq N }} c_k\, u^k \ \ \mathrm{with} \ \ |c_k| \lesssim_{\#k }   N^{3{ \#k -9}} ,\\
\label{RI}
\mathrm{R}^{(I_{> N^3})}(u) &= \sum_{\ell = N^3+1}^{\infty} \sum_{\substack{k \in \mathcal{M}\cap \mathcal D \\ 3\leq  \#k  \leq r-2 }} c_{\ell,k} \, I_\ell \, u^k \ \ \mathrm{with} \ \  |c_{\ell,k}| \lesssim_{ \#k }   N^{3( \#k +2) -9},\\
\label{Ror}
\mathrm{R}^{(or)}(u) &= \sum_{\substack{k \in \mathcal{M} \\ \#k  \geq r+1 }} c_k \, u^k \ \ \mathrm{with} \ \  |c_k| \leq \rho^{\#k} N^{3 \#k -9}  \ \mathrm{and} \ \rho \lesssim_r 1.
\end{align}
\end{theorem}

\subsection{Analysis of the remainder terms} We would like to control directly the vector field generated by the remainder terms of the Theorem \ref{thm-BNF} (i.e. $\mathrm{R}^{(\mu_3>N)},\ \mathrm{R}^{(I_{> N^3})}$ and $\mathrm{R}^{(or)}$). Unfortunately, due to the symplectic structure that imposes a lost of one derivative, in general their vector field $X_R:=\partial_x \nabla R$ does not map $\dot H^s$ into $\dot H^s$.
Nevertheless, we can control the $H^s$-norm of the flow generated by such a Hamiltonian, this is the purpose of the Proposition \ref{prop-Ns}. Actually, we prove a result even stronger, if these remainder Hamiltonians are composed with a symplectic change of coordinates then their Poisson brackets with the $H^s$-norm are small (such a refinement will be crucial to prove Theorem \ref{mainPDE}, see subsection \ref{ma jolie sous section}).

\begin{proposition}\label{prop-Ns}
Let $s\geq2$. We assume that $\tau$ is a symplectic change of variable defined on an open set included in the ball $B_s(0,1)$, taking values in $\dot H^s$ and that there exists $\kappa_\tau\geq1$ such that
\begin{itemize}
\item[(A1)]
$
\| \tau(u) \|_{\dot{H}^s} \leq \kappa_\tau\|u\|_{\dot{H}^s}
$
\item[(A2)]
$
\|(\mathrm{d}\tau (u))^{-1}\|_{\mathscr{L}(\dot{H}^s)} \leq \kappa_\tau.
$

\item[(A3)]
$\forall k\in \mathbb{Z}^*, \ |k|>N  \Rightarrow \  \tau(u)_k = u_k
$
\end{itemize}
Then

\begin{itemize}
\item[(i)] Let $R=\sum_{n\geq r+1}\sum_{k\in\M_n}b_k u^k$, with $r\geq 1$, be a analytic Hamiltonian of order $r+1$ with coefficients $b$ satisfying the estimate $|b_{\#k}|\leq M^{\#k}$ for some $M>0$ then 
$$\big|\{\|\cdot\|_{\dot{H}^s}^2,R\circ\tau\}(u)\big|\lesssim_{s,r,\kappa_\tau} N (M \|u\|_{\dot{H}^s})^{r+1} \quad \text{for } \|u\|_{\dot H^s}\leq (2\kappa_\tau c_0 M)^{-1}$$
where $c_0$ is a universal constant.
\item[(ii)] Let $P_n$ be a homogeneous polynomials of degree $n$ of the form
$$P_n(u)=\sum_{\substack{k\in\M_n\\\mu_3(k)\geq K}}b_k u^k,$$
 where $\mu_3(k)$ denotes the third largest number among $|k_1|,\cdots, |k_n|$, $K\leq N$ and $b$ are some bounded coefficients then 
$$
\big|\{\|\cdot\|_{\dot{H}^s}^2,P_n\circ\tau\}(u)\big|\lesssim_{s,n,\kappa_\tau}N K^{-s+2}\| b\|_{\ell^\infty}\|u\|_{\dot H^s}^n.
$$
\item[(iii)] Let $P_{n+2}$ be a homogeneous polynomials of degree $n+2$ of the form
$$P_{n+2}(u)=\sum_{\substack{k\in\M_n\\ j\geq N}}b_{(j,-j,k)} I_ju^k,$$
where $b$ are some bounded coefficients then
$$\big|\{\|\cdot\|_{\dot{H}^s}^2,P_{n+2}\circ\tau\}(u)\big|\lesssim_{s,n,\kappa_\tau}N^{-2(s-1)} \|b\|_{\ell^\infty} \|u\|^{n+2}_{\dot{H}^s}.$$

\end{itemize}
\end{proposition}

\proof 

We recall that the Poisson bracket of two Hamiltonians $F$ and $G$ reads
$$\{ F,G\} (u)=( \nabla F, \partial_x \nabla G)= \sum_{k\in \mathbb{Z}^{*}} (\partial_{u_{-k}} P) (2i\pi k)  (\partial_{u_k} Q)$$ 
where $(\cdot,\cdot)$ denotes the canonical scalar product on $L^2$: $( u,v)=\int_\T uv \mathrm{d}x=\sum_{k\in\Z^*}u_{-k}v_k.$\\
Then we note that,  since $\tau$ is symplectic, we have
 \begin{align*}
 \{\|\cdot\|_{\dot{H}^s}^2,Q\circ\tau(u)\}=( \nabla\|\cdot\|_{\dot{H}^s}^2(u),\partial_x\nabla (Q\circ\tau)(u))&=( \nabla\|\cdot\|_{\dot{H}^s}^2(u),\partial_x(\mathrm{d}\tau (u))^{\star}\nabla Q(\tau(u)))\\ \nonumber &=( \nabla\|\cdot\|_{\dot{H}^s}^2(u),(\mathrm{d}\tau (u))^{-1}\partial_x\nabla Q(\tau(u))).
 \end{align*}
Consequently, since $\nabla\|\cdot\|_{\dot{H}^s}^2(u) = 2 (2\pi)^{-2s} |\partial_x|^{2s} u$, it follows of the Cauchy Schwarz inequality and the assumption [A2] that 
\begin{equation}
\label{geodiff}
 |\{\|\cdot\|_{\dot{H}^s}^2,Q\circ\tau(u)\}| \leq 2 \kappa_{\tau} \| u\|_{\dot{H}^s}  \| \partial_x\nabla Q(\tau(u)) \|_{\dot{H}^s}.
\end{equation}

First, we focus on (ii) and we consider 
$$
P_n(u)=\sum_{\substack{k\in\M_n\\ \mu_3(k)\geq K}}b_k u^k.
$$ We decompose $P$ in $n+1$ parts, 
\begin{equation}\label{PP}P=P_n^{(0)}+P_n^{(1)}+\cdots+P_n^{(n)} \end{equation}
 where $$P_n^{(i)}(u)= \sum_{\substack{k\in\M_n\ \mu_3(k)\geq K\\ \sharp\{j\ \mid\ |k_j|> nN\}=i}} b_k u^k.$$ 
We begin with the control of $P_n^{(0)}+P_n^{(1)}$.
First we notice that when $\sharp\{j\mid |k_j|> nN\}=1$ and $k\in\M_n$ we have $\max(|k_1|,\cdots,|k_n|)\leq (n-1)nN$ and thus the operator $\partial_x$ is controlled by $2\pi n^2N$ when applied to $\nabla (P_n^{(0)}+P_n^{(1)})$.
Thus using  \eqref{geodiff}, we get
$$
|\{\|\cdot\|_{\dot{H}^s}^2,(P_n^{(0)}+P_n^{(1)})\circ\tau(u)\}|\leq \kt 4\pi n^2N\| u\|_{\dot H^s} \| \nabla (P_n^{(0)}+P_n^{(1)})(\tau(u))\|_{\dot H^s} .
$$
By symmetry on the estimate of $b_k$ we have
$$
| \partial_{u_\ell}(P_n^{(0)}+P_n^{(1)})|\leq n \|b\|_{\ell^\infty} \sum_{\substack{k\in(\Z^*)^{n-1}\\(\ell,k)\in\M_n,\ \mu_3(k,\ell)\geq K}}|u_{k_1}|\cdots|u_{k_{n-1}}|
$$
 and thus,  ordering the two first indices of $k$ (second inequality) and using the zero momentum condition (third inequality), we get successively
\begin{align*}\| \nabla (P_n^{(0)}+P_n^{(1)})\|^2_{\dot H^s}
&\leq (n\|b\|_{\ell^\infty})^2\sum_{\ell\in\Z^*}\ell^{2s}\big| \sum_{\substack{k\in(\Z^*)^{n-1}\\(\ell,k)\in\M_n,\ \mu_3(k,\ell)\geq K}}|u_{k_1}|\cdots|u_{k_{n-1}}|\big|^2\\
&\leq(n\|b\|_{\ell^\infty})^2\sum_{\ell\in\Z^*}\ell^{2s}\big| n^2\sum_{\substack{k\in(\Z^*)^{n-1}\\(\ell,k)\in\M_n,\ |k_1|,|k_2|\geq K}}|u_{k_1}|\cdots|u_{k_{n-1}}|\big|^2\\
&\leq (n\|b\|_{\ell^\infty})^2n^{4+2s}\sum_{\ell\in\Z^*}\big| \sum_{\substack{k\in(\Z^*)^{n-1}\\(\ell,k)\in\M_n,\ |k_1|,|k_2|\geq K}}|k_1|^s|u_{k_1}|\cdots|u_{k_{n-1}}|\big|^2\\
&= (n\|b\|_{\ell^\infty})^2n^{4+2s} \|v \, w \,  \nu^{n-3} \|_{L^2}^2
\end{align*}
where in the last line the functions $\nu$, $w$ and $v$ are respectively defined through their Fourier coefficients by $\nu_j=|u_j|$, $w_j=|j|^s|u_j| \mathbb{1}_{|j|\geq K}$ and
  $v_j=|u_j| \mathbb{1}_{|j|\geq K}$ for $j\in\Z^*$. Consequently, since $\|\cdot\|_{L^\infty} \leq \|\cdot\|_{\ell^1}$ we have
  \begin{align*}
  \|v \, w \,  \nu^{n-3} \|_{L^2} \leq  \|w\|_{L^2 } \|\nu\|_{L^\infty }^{n-3} \|v\|_{L^\infty } \leq \|u\|_{\dot{H}^s} \|v\|_{\ell^1}\|\nu\|^{n-3}_{\ell^1}\leq
 c_0^{n-2} K^{-s+1}\|u\|^{n-1}_{\dot H^s}
  \end{align*}
  where we used that $\|v\|_{\ell^1}\leq c_0 \|u\|_{\dot H^s}$ and $\sum_{|j|\geq K}|u_j|\leq c_0 K^{-s+1}\|u\|_{\dot H^s}$ (with $c_0\leq \frac{\pi}{\sqrt3}$).
 Therefore we conclude
\begin{equation*}|\{\|\cdot\|_{\dot{H}^s}^2,(P_n^{(0)}+P_n^{(1)})\circ\tau(u)\}|\lesssim_{\kt}   n^{5+s}Nc_0^{n} K^{-s+1} \|b \|_{\ell^\infty}\| \tau(u)\|^n_{\dot H^s}.\end{equation*}
and then using assumption [A1], 
\begin{equation}\label{P0P1}|\{\|\cdot\|_{\dot{H}^s}^2,(P_n^{(0)}+P_n^{(1)})\circ\tau(u)\}|\lesssim_{\kt}   n^{5+s}NK^{-s+1}\|b \|_{\ell^\infty} (\kt c_0)^n\| u\|^n_{\dot H^s}.\end{equation}

Now let us estimate $\{\|\cdot\|_{\dot{H}^s}^2,P_n^{(i)}\circ\tau(u)\}$ for $i\geq2$. We have
\begin{align*}|\{\|\cdot\|_{\dot{H}^s}^2,P_n^{(i)}\circ\tau(u)\}|
&=|\{\|\cdot\|_{\dot{H}^s}^2, \sum_{\substack{k\in\M_n,\ \mu_3(k)\geq K\\ \sharp\{j \ \mid \ |k_j|> nN\}=i}}b_k\tau(u)^k\}|\\
&=   |\{\|\cdot\|_{\dot{H}^s}^2,\sum_{\substack{k\in\M_n,\ \mu_3(k)\geq K\\ |k_{\sigma_k(i)}|> nN\geq |k_{\sigma_k(i+1)}|}} b_k\tau(u)^k\}|
\end{align*}
where $\sigma_k$ is a permutation that makes the modulus of the indices $k_1,\cdots,k_n$ become non-increasing :
$$
|\sigma_k(1)|\geq \dots \geq |\sigma_k(n)|.
$$ 
Then using assumption [A3] we decompose this sum in two parts:
\begin{equation}
\label{Pi}
|\{\|\cdot\|_{\dot{H}^s}^2,P_n^{(i)}\circ\tau(u)\}|\leq  \Sigma_1+\Sigma_2
\end{equation}
where 
\begin{align*}
\Sigma_1=&\Big| \sum_{\substack{k\in\M_n,\ \mu_3(k)\geq K\\ |k_{\sigma_k(i)}|> nN\geq |k_{\sigma_k(i+1)}|}}\prod_{j=1}^i u_{\sigma_k(j)} \{\|\cdot\|_{\dot{H}^s}^2,b_k\prod_{j=i+1}^n\tau(u)_{\sigma_k(j)}\}\Big|\\
\Sigma_2=&\| b\|_{\ell^\infty}\sum_{\substack{k\in\M_n,\ \mu_3(k)\geq K\\ |k_{\sigma_k(i)}|> nN\geq |k_{\sigma_k(i+1)}|}}\prod_{j=i+1}^n|\tau(u)_{\sigma_k(j)}|\ |\{\|\cdot\|_{\dot{H}^s}^2,\prod_{j=1}^{i}u_{\sigma_k(k_j)}\}|.
\end{align*}
To estimate $\Sigma_1$ we notice that by ordering the first indices of $k$ we have
$$\Sigma_1\leq n^i \sum_{|k_1|,\cdots,|k_i|\geq nN}\prod_{j=1}^{i}|u_{k_j}||\{\|\cdot\|_{\dot{H}^s}^2,\sum_{\substack{|k_{i+1}|,\cdots|k_{n}|\leq  nN\\k_{i+1}+\cdots+k_{n}=-k_1-\cdots-k_i}}b_k\prod_{j=i+1}^n\tau(u)_{k_j}\}|.$$
Then, as in the control of  $\{\|\cdot\|_{\dot{H}^s}^2,(P_n^{(0)}+P_n^{(1)})\circ\tau(u)\}$,
 we use again \eqref{geodiff} to get
 \begin{multline*}
 \Sigma_1 \leq   n^{i}\sum_{|k_1|,\cdots,|k_i|\geq nN}\prod_{j=1}^{i}|u_{k_j}|4\pi \kt nN \|u\|_{\dot{H}^s} \\
 \times \|b\|_{\ell^\infty}(n-i)\left(\sum_{|\ell|\leq nN}\ell^{2s}\big| \hspace{-1cm}\sum_{\substack{|k_{i+1}|,\cdots,|k_{n-1}|\leq  nN\\k_{i+1}+\cdots+k_{n-1}=-\ell-k_1-\cdots-k_i}}|\tau(u)_{k_{i+1}}|\cdots|\tau(u)_{k_{n-1}}|\big|^2\right)^{\frac12}.
 \end{multline*}
We observe that by Jensen and symmetry we have
 \begin{equation*}
 \begin{split}
 &\sum_{|\ell|\leq nN}\ell^{2s}\big| \hspace{-1cm}\sum_{\substack{|k_{i+1}|,\cdots,|k_{n-1}|\leq  nN\\k_{i+1}+\cdots+k_{n-1}=-\ell-k_1-\cdots-k_i}}|u_{k_{i+1}}|\cdots|u_{k_{n-1}}|\big|^2 \\
  &\leq  \sum_{\ell \in \mathbb{Z}^*}\ell^{2s}\big| \sum_{k_{i+1}+\cdots+k_{n-1}=-\ell}|u_{k_{i+1}}|\cdots|u_{k_{n-1}}|\big|^2\\
  &\leq (n-i-1)^{2s} \sum_{\ell \in \mathbb{Z}^*} \big| \sum_{k_{i+1}+\cdots+k_{n-1}=-\ell} |k_{i+1}|^{2s}|u_{k_{i+1}}|\cdots|u_{k_{n-1}}|\big|^2\\
  &\leq (n^s \|u\|_{\dot{H}^s} \| u\|_{\ell^1}^{n-i-2})^2.
 \end{split}
 \end{equation*} 
Using that $\| \cdot \|_{\ell^1}\leq c_0 \|\cdot\|_{\dot H^s}$ where $c_0\leq \pi/\sqrt{3}$ and the assumption [A1], we get 
\begin{align*}
\Sigma_1&\lesssim_{\kt}  \|b\|_{\ell^\infty} c_0^{n-i} \kt^{n} n^{s+i+2} N\|u\|^{n-i}_{\dot H^{s}}\sum_{|k_1|,\cdots,|k_i|\geq nN}\prod_{j=1}^{i}|u_{k_j}| \\
&\lesssim_{\kt} \|b\|_{\ell^\infty} c_0^{n}\kt^{n}  n^{s+i+2} N (nN)^{-(s-1)i}\|u\|^{n}_{\dot H^{s}}
\end{align*}
where we used that $\sum_{|j|\geq K}|u_j|\leq c_0K^{-s+1}\|u\|_{\dot H^s}$. So since $s\geq 2$ and $i\geq2$ we conclude
\begin{equation}\label{S1}
\Sigma_1\lesssim_{\kt}  \|b\|_{\ell^\infty} c_0^{n}n^{s+2}  \kt^{n} N^{-2(s-2)}\|u\|_{\dot H^s}^n.
\end{equation}


We now estimate $\Sigma_2$, we have:
\begin{align*}\Sigma_2&\leq 2\pi \|b\|_{\ell^\infty} \sum_{\substack{k\in\M_n,\ \mu_3(k)\geq K\\ |k_{\sigma_k(i)}|> nN\geq |k_{\sigma_k(i+1)}|}}
|(k_{\sigma_k(1)}|k_{\sigma_k(1)}|^{2s}+\cdots+k_{\sigma_k(i)}|k_{\sigma_k(i)}|^{2s})|\tau(u)^{k}|
\}|\\
&\leq 2\pi \|b\|_{\ell^\infty} n^3\sum_{\substack{k\in\M_n,\ |k_1|,|k_2|\geq nN,\ |k_3|\geq K\\|k_1|\geq|k_2|\geq|k_3|\geq |k_4|,\cdots,|k_n| }}(\big|k_1|k_1|^{2s}+k_2|k_2|^{2s}\big|+(i-2)|k_3|^{2s+1})|\tau(u)^{k}|
\end{align*}
Then we notice that applying the Young inequality, we have 
\begin{equation}
\label{ilyaquecaquisert}
\sum_{\ell \in \M_n} |\ell_1|^s |\ell_2|^s |\ell_3|^{s-1} |u^\ell| \leq c_0^{n-2} \| u\|_{\dot{H}^s}^n.
\end{equation}
Consequently, to estimate $\Sigma_2$, we just have to control each term by \eqref{ilyaquecaquisert} and [A1].
\begin{itemize}
\item Since we have 
$$
|k_3|^{2s+1}\leq |k_3|^{s-1}|k_2|^{s}|k_1|^2 \leq N^{-s+2} |k_1|^{s}|k_2|^{s}|k_3|^{s-1},$$
by \eqref{ilyaquecaquisert} and [A1], we get
$$
\sum_{\substack{k\in\M_n,\ |k_1|,|k_2|\geq nN\\|k_1|\geq|k_2|\geq|k_3|\geq |k_4|,\cdots,|k_n| }}|k_3|^{2s+1} 
|\tau(u)^k| \leq N^{-s+2}c_0^{n-2}\kt^n \|u\|^n_{\dot H^s}.
$$
\item If $k_1k_2>0$ then, by the zero momentum condition, 
$$
|k_1|,|k_2|\leq n|k_3| \ \mathrm{and} \ |k_1|\leq n|k_2|
$$
 and thus 
 $$
 \big|k_1|k_1|^{2s}+k_2|k_2|^{2s}\big|\leq 2 n^{s+1} |k_3|^{s-1} |k_1|^{s} |k_2|^2 \leq 2n^3 N^{-s+2}  |k_1|^{s}|k_2|^{s}|k_3|^{s-1}.
 $$  
 Then as above, by \eqref{ilyaquecaquisert} and [A1], we get
$$
\sum_{\substack{k\in\M_n,\ |k_1|,|k_2|\geq nN,\ k_1k_2>0\\|k_1|\geq|k_2|\geq|k_3|\geq |k_4|,\cdots,|k_n| }}\hspace{-1cm}\big|k_1|k_1|^{2s}+k_2|k_2|^{2s}\big||\tau(u)^k| 
\leq 2n^3 N^{-s+2}c_0^{n-2}\kt^n \|u\|^n_{\dot H^s}.
$$
\item If $k_1k_2<0$ then, since $\partial_x (x |x|^{2s}) = (2s+1) |x|^{2s}$, by the mean value inequality we get
$$
|k_1|k_1|^{2s}+k_2|k_2|^{2s}|\leq(2s+1) \big| |k_1|-|k_2|\big| |k_1|^{2s}\leq (2s+1)|k_1|^sn^s|k_2|^s n |k_3|
$$
 where we used the zero momentum condition. Consequently, we have
 $$
 |k_1|k_1|^{2s}+k_2|k_2|^{2s}|\leq  (2s+1) n^{s+1} K^{-s+1} |k_1|^s |k_2|^s  |k_3|^{s-1}
 $$
and so  by \eqref{ilyaquecaquisert} and [A1], we get
\begin{align}\label{k1k2}
\sum_{\substack{k\in\M_n,\ |k_1|,|k_2|\geq nN,\ |k_3|\geq K,\ k_1k_2<0\\|k_1|\geq|k_2|\geq|k_3|\geq |k_4|,\cdots,|k_n| }}\hspace{-1,5cm}\big|k_1|k_1|^{2s}+k_2|k_2|^{2s}\big||\tau(u)^k| 
&\leq K^{-s+1}(2s+1) n^{s+1} c_0^{n-2}\kt^n \|u\|^n_{\dot H^s}.
\end{align}
\end{itemize}
Combining these three estimates yields for $K\leq N$ and $1\leq c_0 \leq \pi/\sqrt{3}$
\begin{equation}
\label{S2}\Sigma_2\lesssim_s \|b\|_{\ell^{\infty}} K^{-s+2} n^{s+4} c_0^{n}\kt^n \|u\|^n_{\dot H^s}.
\end{equation}
Inserting \eqref{S1}, \eqref{S2} in \eqref{Pi} we get
\begin{equation}\label{Pi2}
|\{\|\cdot\|_{\dot{H}^s}^2,P_n^{(i)}\circ\tau(u)\}|\lesssim_{s,\kt} K^{-s+2} n^{s+4} \|b\|_{\ell^{\infty}} (c_0\kt)^{n} \|u\|^n_{\dot H^s},\quad \text{for }i\geq2.
\end{equation}
Finally inserting \eqref{Pi2} and \eqref{P0P1} in \eqref{PP} yields
\begin{equation}\label{ouf}\big|\{\|\cdot\|_{\dot{H}^s}^2,P_n\circ\tau\}(u)\big|\lesssim_{s,\kt}N n^{s+5}K^{-s+2}\|b \|_{\ell^\infty} (c_0\kt \|u\|_{\dot{H}^s})^n.\end{equation}

\bigskip

Using \eqref{ouf} we can now easily prove the different assertions of the Proposition:
\begin{itemize}
\item To prove  assertion (i) we take $K=1$ in \eqref{ouf} and we get for $c_0\kt M\|u\|_{\dot{H}^s}\leq \frac12$ (where we recall that by assumption $|b_k|\leq M^{\#k}$)
\begin{align*}
\big|\{\|\cdot\|_{\dot{H}^s}^2,R\circ \tau \}(u)\big|&\lesssim_{s,\kt} N \sum_{n\geq r+1} n^{s+5} (c_0\kt M\|u\|_{\dot{H}^s})^n\\
&\lesssim_{s,\kt,r} N(c_0\kt M\|u\|_{\dot{H}^s})^{r+1}.\end{align*}
\item Assertion (ii) is a direct consequence of \eqref{ouf}.
\item To prove (iii) we just notice that 
$\{\|\cdot\|_{\dot{H}^s}^2,I_j u^k\}= I_j \{\|\cdot\|_{\dot{H}^s}^2, u^k\}$ and $$I_j\leq j^{-2s}\|u\|_{\dot{H}^s}^2\leq N^{-2s}\|u\|_{\dot{H}^s}^2.$$
\end{itemize}

\endproof

\subsection{The resonant normal form process}
In this subsection, we aim at putting \ref{gKdV} and \ref{gBO} in resonant normal without paying attention to the explicit expression of the fourth and sixth order integrable terms.

In order to realize this process, we will have to solve some homological equations of the form
\begin{equation}
\label{mapremiereequationhomologique}
\{ \chi,Z_2^{\mathcal{E}}\} + \sum_{k\in M_r \setminus \Rc_r^\E} b_k u^k =0.
\end{equation}
A natural solution is obtained by observing that if $k\in \M$ then
\begin{equation}
\label{resolvator}
\{ u^k , Z_2^{\mathcal{E}} \} = - i\, \Omega_{\E}(k) \,u^k
\end{equation}
where $\Omega_{\E}(k)$ is given by the following definition.
\begin{definition}[Denominators $\Omega_\E$] If $\E\in \{ \mathrm{\ref{gBO},\ref{gKdV}}\}$ and $k\in \M$, we set
$$
\Omega_{\E}(k) := 2^{-1} (2\pi)^{1+\alpha_{\E}} (k_1 |k_1|^{\alpha_{\E}} + \dots + k_\last |k_\last|^{\alpha_\E}).
$$
\end{definition}
In view of \eqref{resolvator}, a natural solution of the homological equation \eqref{mapremiereequationhomologique} is
$$
\chi(u) = \sum_{k\in M_r \setminus \Rc_r^\E} \frac{b_k}{i\Omega_\E(k)} u^k.
$$
Following the classical strategy to put our system in resonant normal form (see for instance \cite{Bam03, BG06, G, KillBill}), we will have to consider the change of variable induced by the Hamiltonian flow generated by $\chi$ at time $t=1$. However, a priori, the Hamiltonian vector field generated by $\chi$, $X_\chi := \partial_x \nabla \chi$, does not map $\dot{H}^s$ into $\dot{H}^s$. Consequently, a priori this flow does not make sense. To overcome this issue, we only solve homological equations of the form \eqref{mapremiereequationhomologique} where the coefficients $b_k$ are supported in the following sets of indices.
\begin{definition}[$\J_{n,N}^\E$ sets] If $N\geq2$ and $n\geq 3$, we set
\begin{equation}
\label{J}\J_{n,N}^\E:=\{k\in \M_n \setminus \Rc_n^\E\mid \left| \frac{\max(|k_1|,\cdots,|k_n|)}{k_1|k_1|^{\alpha_{\mathcal{E}}}+\dots+k_n|k_n|^{\alpha_{\mathcal{E}}}} \right| \leq N\}.
\end{equation}
As usual, we also set $\J_{N}^\E = \bigcup_{n\geq 3} \J_{n,N}^\E$.
\end{definition}
As stated in the following Lemma, if the coefficients $b_k$ are supported in these sets of indices, the Lie transforms are well defined.
\begin{lemma}\label{lem-chi} Let $N\geq2$, $n\geq 3$, $s\geq 1$.
If $\chi$ is an homogeneous polynomial of degree $n$ of the form
\begin{equation}
\label{monpetitchi}
\chi(u)=\sum_{k\in\J_{n,N}^\E}\frac{b_k}{i\Omega_\E(k)} u^k
\end{equation}
 where $b$ is bounded and satisfies the reality condition then its vector field $X_\chi:= \partial_x \nabla \chi$ maps ${\dot{H}^s}$ into itself and, for all $u\in \dot{H}^s$, we have
$$
 \| X_\chi(u)\|_{\dot{H}^s}\lesssim_n   \| b\|_{\ell^\infty}  N \|u\|_{\dot{H}^s}^{n-1}.
$$
As a consequence, there exists $\eps_0 \gtrsim_{n,s} [N \| b\|_{\ell^\infty}  ]^{-\frac{1}{n-2}}$ such that $\chi$ generates a Hamiltonian flow $\Phi_\chi^t$ on $B_s(0,\eps_0)$, for $0\leq t\leq 1$, that is close to the identity: if $u\in B_s(0,\eps_0)$ and $0\leq t\leq 1$ we have
\begin{equation}\label{antalgiques}
\| \Phi_\chi^t(u)-u\|_{\dot{H}^s} \lesssim_{n,s}   \| b\|_{\ell^\infty} N   \| u \|_{\dot{H}^s}^{n-1}.
\end{equation}
\end{lemma}
\proof  First, note that since $b$ satisfies the reality condition, $\Omega_\E$ is odd and thanks to the $i$ denominator in \eqref{monpetitchi}, $\chi$ is real valued. 

By symmetry on the estimate of $b_k$, we have
$$
| \partial_{u_\ell}\chi|(u)\leq n \| b \|_{\ell^\infty} \sum_{\substack{k\in\J_{n,N}^\E\\  \ell = k_n}}\frac{|u_{k_1}|\cdots|u_{k_{n-1}}|}{|\Omega_{\E}(k)|}.
$$
Consequently, we have
\begin{align*}
\| X_\chi(u)\|^2_{\dot{H}^s}&= 4\pi^2 \sum_{\ell\in\Z^*}|\ell|^{2s+2}|\partial_{u_\ell} \chi(u)|^2\\
&\leq 4\pi^2n^2 \| b \|_{\ell^\infty}^2 \sum_{\ell\in\Z^*}|\ell|^{2s+2}\Big|  \sum_{\substack{k\in\J_{n,N}^\E\\  \ell = k_n}}\frac{|u_{k_1}|\cdots|u_{k_{n-1}}|}{|\Omega_{\E}(k)|}    \Big|^2\\
&\stackrel{\eqref{J}}{\leq} 4\pi^2n^2 \| b \|_{\ell^\infty}^2  N^2\sum_{\ell\in\Z^*}|\ell|^{2s}\Big| \sum_{\substack{k\in(\Z^*)^{n-1}\\k_1+\cdots+k_{n-1}=-\ell}} |u_{k_1}|\cdots|u_{k_{n-1}}|          \Big|^2
\end{align*}
Observing that by Jensen
$$
|\ell|^{s} = |k_1+\dots+k_{n-1}|^{s}  \leq (n-1)^{s-1} (|k_1|^{s}+\dots+|k_{n-1}|^{s})
$$
and applying the Young convolutional inequality and the Cauchy-Schwarz inequality, we get
$$
\| X_\chi(u)\|^2_{\dot{H}^s} \leq 4\pi^2n^{2s+2} \| b \|_{\ell^\infty}^2  N^2 \| u\|_{\dot{H}^s}^2 \| u\|_{\ell^1}^{2(n-2)} \leq 4\pi^2n^{2s+2} \| b \|_{\ell^\infty}^2  N^2 \| u\|_{\dot{H}^s}^{2(n-1)} (\frac{\pi^2}3)^{n-2}
$$

Now we turn to the Lie transform $\Phi_\chi^t$. Noticing that $X_\chi(u)$ is an homogeneous polynomial, the previous estimates natural yield
$$
\|\mathrm{d}X_\chi(u) \|_{\mathscr{L}(\dot{H}^s)}^2 \leq 4\pi^2n^{2s+4} \| b \|_{\ell^\infty}^2  N^2 \| u\|_{\dot{H}^s}^{2(n-1)} (\frac{\pi^2}3)^{n-2}.
$$

Therefore the vector field $X_{\chi}(u)$ is locally Lipschitz in ${\dot{H}^s}$ and we deduce from  the Cauchy-Lipschitz Theorem that the flow $\Phi_{\chi}^t$ is locally well defined on $\dot H^s$. Furthermore as long as $\| \Phi_{\chi}^t(u) \|_{\dot{H}^s} \leq 2 \|u\|_{\dot{H}^s}$ we have
$$
\|\Phi_\chi^t(u)-u\|_{\dot{H}^s}\leq \left| \int_0^t  \|X_\chi(\Phi_\chi^{\mathfrak{t}}(u))\|_{\dot{H}^s}\mathrm{d}\mathfrak{t} \right| \leq \kappa_{n,s} t \| b\|_{\ell^\infty}N \| u\|_{\dot{H^s}}^{n-1}
$$
where $\kappa_{n,s}$ is a constant depending only on $n$ and $s$. Thus we conclude by a bootstrap argument that, if $\|u\|_{\dot{H}^s} < (\kappa_{n,s} \| b\|_{\ell^\infty}N)^{\frac1{n-2}}=:\eps_0$ then $\Phi_{\chi}^t(u)$ is well defined for $0\leq t\leq 1$ and satisfies \eqref{antalgiques}. 
\endproof

A priori we could fear that the unsolved terms (i.e. those associated with indices in $\M \setminus \J_{n,N}^\E$) contribute to the dynamics and will have to be solved by an other way. Hopefully, the following Lemma and its corollary prove that they are remainder terms in the sense of Proposition \ref{prop-Ns}.
\begin{lemma}
\label{lem:origin}
 If $\mathcal{E}\in \{\mathrm{\ref{gBO},\ref{gKdV}}\}$ and $k\in \mathcal{M}_n$ with $n\geq 3$ and satisfies $k_1+k_2\neq 0$ then we have 
$$
\max \left( (n-2)^{\alpha_{\mathcal{E}}} |k_3|^{1+\alpha_{\mathcal{E}}} , \left| \sum_{j=1}^n k_j |k_j|^{\alpha_{\mathcal{E}}}  \right| \right)\geq \frac{|k_1|^{\alpha_{\mathcal{E}}}}2.
$$
\end{lemma}
\begin{proof} Without loss of generality, we assume that $k_1$ is positive.

First, let us observe that if  $k_2$ is positive then, since $k\in \mathcal{M}_n$, we have
$$
k_1\leq k_1+k_2 = -(k_3+\dots+k_n) \leq (n-2) |k_3|.
$$
Now, if $k_2$ is negative we have
$$
k_1|k_1|^{\alpha_{\mathcal{E}}}+k_2|k_2|^{\alpha_{\mathcal{E}}} = k_1^{1+\alpha_{\mathcal{E}}}-|k_2|^{1+\alpha_{\mathcal{E}}} = (k_1+k_2)(\sum_{j=0}^{\alpha_{\mathcal{E}}} k_1^{\alpha_{\mathcal{E}} - j} |k_2|^j).
$$
But, by assumption, we have $k_1+k_2\neq 0$. As a consequence, we have
$$
R_{\mathcal{E}}:= \left| \sum_{j=1}^n k_j |k_j|^{\alpha_{\mathcal{E}}}  \right| \geq k_1|k_1|^{\alpha_{\mathcal{E}}}+k_2|k_2|^{\alpha_{\mathcal{E}}} - (n-2) |k_3|^{1+\alpha_{\mathcal{E}}} \geq k_1^{\alpha_{\mathcal{E}}} - (n-2) |k_3|^{1+\alpha_{\mathcal{E}}}.
$$
As a consequence, if $R_{\mathcal{E}} \leq k_1^{\alpha_{\mathcal{E}}}/2$, we have $(n-2)|k_3|^{1+\alpha_{\mathcal{E}}} \geq k_1^{\alpha_{\mathcal{E}}}/2$.

\end{proof}
\begin{corollary}
\label{cor:origin}
 Let $k\in \mathcal{D}_n\cap(\mathcal{M}_n \setminus \mathcal{R}_n^{\mathcal{E}})$ for some $n\geq 3$. If $N>2$ is such that
\begin{equation}
\label{eq:hp:cor:lem:or}
\left| \frac{k_1}{k_1|k_1|^{\alpha_{\mathcal{E}}}+\dots+k_n|k_n|^{\alpha_{\mathcal{E}}}} \right| \geq N
\end{equation}
then there exists $k'\in \mathcal{M}_{n-2}$ such that
\begin{equation}
\label{case:relou1}
u^k = I_a u^{k'} \ \mathrm{where} \ a\geq N
\end{equation}
or
\begin{equation}
\label{case:relou2}
|k_3|^{1+\alpha_{\mathcal{E}}} \geq \frac{N^{\alpha_{\mathcal{E}}}}{2(n-2)^{\alpha_{\mathcal{E}}}}.
\end{equation}
\end{corollary}
\begin{proof}
By assumption, $k$ is non-resonant, i.e. $k_1|k_1|^{\alpha_{\mathcal{E}}}+\dots+k_n|k_n|^{\alpha_{\mathcal{E}}}\in \mathbb{Z}^*$. Consequently, by \eqref{eq:hp:cor:lem:or}, we have $|k_1|\geq N$.

 If $k_1+k_2 = 0$ then $u^k$ is of the form $u^k = I_{k_1} u^{(k_3,\dots,k_n)}$. Consequently, \eqref{case:relou1} is satisfied. Else if, $k_1+k_2 \neq 0$, since $N> 2$, applying Lemma \ref{lem:origin}, we have \eqref{case:relou2}.
\end{proof}

In the Birkhoff normal form process, naturally, we generate Hamiltonians obtained by computing Poisson brackets with the Hamiltonian $\chi$ we used to generate the change of coordinates (see \eqref{monpetitchi}) . A priori, due to the unbounded operator $\partial_x$ in the Poisson bracket, the coefficients of these new Hamiltonians may be unbounded. However, since the coefficients of $\chi$ are supported in some sets $\J_{n,N}^\E$, we can prove in the following Lemma that, up to a factor $N$, they are still bounded.
\begin{lemma} \label{poisson}
Let $N\geq2$, $r\geq 3$, $n\geq 3$. If $P,\chi$ are homogeneous polynomials of degree $n$ (resp. $r$) of the form
$$
P(u)= \sum_{k\in \M_n} c_k u^k \quad \mathrm{and} \quad \chi(u)=\sum_{k\in\J_{n,N}^\E}\frac{b_k}{i\Omega_\E(k)} u^k
$$
where $c$ and $b$ are bounded and satisfy the reality condition then $\{ \chi,P\}$ is an homogeneous polynomial of the form
\begin{equation}
\label{nemo}
\{  \chi,P\}(u) = \sum_{k\in \M_{n+r-2}} d_k \, u^k
\end{equation}
where $d$ satisfies the reality condition and is bounded :
\begin{equation} \label{cab} \| d\|_{\ell^\infty}\leq 2 (2\pi)^{-\alpha_{\E}}Nnr \| c\|_{\ell^\infty} \|b\|_{\ell^\infty}. 
\end{equation}
\end{lemma}
\begin{proof}
We note that $\chi$ writes
$$
\chi(u) = \sum_{k\in \M_r} \widetilde{b}_k u^k
$$
where $ \widetilde{b}_k = 0$ if $k\notin J^\E_{r,N}$ and $ \widetilde{b}_k  = b_k/i\Omega_\E(k)$ else. Consequently, $\{  \chi,P\}$ is of the form \eqref{nemo}, where
$$
d_k=2i\pi\ell  \sum_{i=1}^n \sum_{i'=1}^r c_{k_1\cdots k_{i-1},-\ell,k_{i},\cdots,k_{n-1}}   \widetilde{b}_{k_n\cdots k_{n+i'-2},\ell,k_{n+i'-1},\cdots,k_{n+r-2}}.
$$
where $\ell=k_1+\cdots+k_{n-1}=-k_n-\cdots-k_{r+n-2}$. Thus, by definition of $\J_{r,N}^\E$, using that $ \widetilde{b}_k=0$ if $k\notin J^\E_{r,N}$ we get that $d$ satisfies the estimate \eqref{cab}.
\end{proof}

In the following proposition, we realize the Birkhoff normal form process. In particular, we pay lot of attention to the estimate of the coefficients of the Hamiltonian. The proof of the Theorem \ref{thm-BNF} (in the next subsection) will rely on this proposition and its proof.
\begin{proposition}
\label{ellefaitlegrosduboulot}
Being given $\mathcal{E}\in \{ \mathrm{\ref{gKdV}}, \mathrm{\ref{gBO}}\}$, $r\geq 2$, $s\geq 1$, $N\gtrsim_r 1$ and $\eps_0 \lesssim_{r,s} N^{-3}$, there exist two symplectic maps, $\tau^{(0)},\tau^{(1)}$, preserving the $L^2$ norm, making the diagram \eqref{monbeaudiagramme_roidesforets} to commute
and close to the identity (i.e. satisfying \eqref{tau}), such that, on $B_s(0,2\eps_0)$, $H_{\mathcal{E}} \circ \tau^{(1)}$ writes
\begin{equation}
\label{eq:firstnormalform}
H_{\mathcal{E}} \circ \tau^{(1)} = Z_2^{\mathcal{E}} + \sum_{k\in \M} c_k u^k.
\end{equation}
where $c$ satisfies the reality condition and is such that
\begin{itemize}
\item[i)] $c_k = 0$ if $3\leq \#k \leq r$ and $k\in \J_{3N^3}^\E$
\item[ii)] $|c_k| \lesssim_{\#k} N^{3\#k -9}$ if $3\leq \#k \leq r$ and $k\notin \J_{3N^3}^\E$
\item[iii)] $|c_k|\leq \rho^{\#k}  N^{3\#k - 9}$ for $k\in \M$ and some $\rho \lesssim_r 1$.
\end{itemize}
\end{proposition}
The index $3N^3$ in $\J_{3N^3}^\E$ will be crucial in the formal computation of the sixth order integrable term $Z^\E_{6,N^3}$ (we refer to Remark \ref{3N^3} for a more detailed explanation). Before proving this proposition, let us explain in details where does the exponent $3\# k - 9$ come from (we will often use similar technics to get some explicit exponent, nevertheless we will not explain anymore how we get them, we will just check that they work). 
\begin{remark} \label{about3k-9} Somehow the bound $|c_k| \lesssim_{\#k} N^{3\#k -9}$ is natural to get a class of Hamiltonian stable by the changes of coordinates of the Birkhoff normal form process. Indeed, it will generates new terms of the form
$$
\big\{ \sum_{k\in \J_{r,3N^3}^\E} \frac{c_k}{i\Omega(k)}u^k, \sum_{k\in \M_n} c_k u^k\big\} = \sum_{k\in \M_{n+r-2}} \widetilde{c}_k u^k.
$$ 

By Lemma \ref{poisson}, we know that the coefficients $\widetilde{c}_k$ satisfy the estimate
$$
\widetilde{c}_k \lesssim_{\#k} N^3 \| (c_k)_{\#k = r} \|_{\ell^\infty} \| (c_k)_{\#k = n} \|_{\ell^\infty} \lesssim_{\#k} N^3 N^{3 r -9} N^{3n-9}.
$$
Consequently, since $\#k =  n+r-2$, we deduce that
$$
\widetilde{c}_k \lesssim_{\#k}  N^{3 (n+r-2) -9} \simeq_{\#k} N^{3 \#k -9}
$$
which is the same estimate we had for $c_k$ : the class is stable.

Now let us explain how we got this bound. Assume that we are looking for a bound a the form $|c_k|\lesssim_{\# k} N^{\alpha \#k - \beta}$ with $\alpha,\beta\geq 0$ such that the bound is satisfied by $H_{\E}$ and it is stable by the Birkhoff normal form process (i.e. $\widetilde{c}_k$ satisfies the same bound). 

The coefficients of $H_{\E}$ are independent of $N$, consequently a priori the best estimate we know is $|c_k|\lesssim_{\# k} 1$. Consequently, $\alpha$ and $\beta$ have to satisfy $\alpha n - \beta \geq 0$ for $n\geq 3$. Of course, since $\alpha \geq 0$, it is enough that it is satisfied for $n=3$, i.e. to have
\begin{equation}
\label{premierechoseaverifier}
3 \, \alpha  - \beta \geq 0.
\end{equation}

Now, if we want $\widetilde{c}_k$ to satisfy the same estimate as $c_k$, $\alpha$ and $\beta$ have to be such that
$$
N^3 N^{\alpha r - \beta} N^{\alpha n - \beta} \leq  N^{\alpha (r+n-2) - \beta}.
$$
Consequently, they have to satisfy the estimate 
$$
3 +  \alpha r - \beta + \alpha n - \beta \leq \alpha (r+n-2) - \beta
$$
which is equivalent to
\begin{equation}
\label{secondechoseaverifier}
2\alpha -\beta \leq -3.
\end{equation}
Finally,  we observe that $\alpha = 3$ and $\beta = -9$ is the sharpest possible choice to ensure that both \eqref{premierechoseaverifier} and \eqref{secondechoseaverifier} are satisfied.

\end{remark}
\begin{proof}[Proof of Proposition \ref{ellefaitlegrosduboulot}.] We prove this Proposition by induction on $r$. 

First, we note if $r=2$ (i.e. initially) it is satisfied. Indeed, we have assumed that the nonlinearity $f$ is analytic. Consequently, $H_\E$ is of the form
$$
H_\E = Z_2^{\mathcal{E}} + \sum_{k\in \M} c_k u^k 
$$
where $c_k=a_{\# k}\in \mathbb{R}$ satisfies $|c_k| \lesssim \rho^{\#k}$ for some $\rho>0$ depending only on $f$. A fortiori, $c$ also satisfies the reality condition and $|c_k| \lesssim N^{3\#k - 9}$ for $k\in \M$.

Now, we assume that the result of Proposition \ref{ellefaitlegrosduboulot} holds at the index $r-1\geq 2$ and we aim at proving that it holds at the index $r$. For $n\geq 3$, we denote by $P_n$ the homogeneous term of degree $n$ of the nonlinearity :
\begin{equation}
\label{monpetitPn}
P_n = \sum_{k\in \M_n } c_k u^k.
\end{equation}
Following the general strategy of  Birkhoff normal forms (see for instance \cite{Bam03, BG06, G, KillBill}), we aim at killing the non-resonant terms of $P_{r}$. Consequently, in view of \eqref{resolvator}, we set 
\begin{equation}
\label{monpetitchir}
\chi_r = \sum_{k\in \J_{r,3N^3}^\E } \frac{ c_k}{ i\Omega(k)} u^k.
\end{equation}
Roughly speaking $\chi_r $ is the solution of the homological equation
$$
\{\chi_r , Z_2^\E\} + P_r = \mathrm{(resonant \ terms)} +  \mathrm{(remainder \ terms)}.
$$
We refer the reader to the proof of Theorem \ref{thm-BNF} in the next section for details about this decomposition and to Remark \ref{3N^3} for choice of $3N^3$ (it is crucial in the formal computation of the sixth order integrable term $Z^\E_{6,N^3}$). Here, more precisely, $\chi_r $ is the solution of the homological equation
\begin{equation}
\label{PBchirZ2}
\{\chi_r , Z_2^\E\} + P_r =  \sum_{k\in \M_r \setminus \J_{r,3N^3}^\E} c_k u^k =:R_r.
\end{equation}
 Using the induction hypothesis we have $N^3 \|(c_k)_{\#k = r}\|_{\ell^\infty} \lesssim_r N^{3} N^{3 r - 9} \simeq_r N^{3(r-2)}$ and, so applying Lemma \ref{lem-chi}, provided that $\eps_0$ satisfies an estimates of the form $ \eps_0 \lesssim_{r,s} N^{-3}$, $-\chi_r$ generates an Hamiltonian flow $\Phi_{-\chi_r}^t$, $0\leq t \leq 1$ mapping $B_s(0,(3/2)\varepsilon_0)$ into $B_s(0, \varepsilon_0)$ and close to the identity :
$$
\forall t \in (0,1), \ \| \Phi_{-\chi_r}^t(u) - u\|_{\dot{H}^s} \lesssim_{r,s}  N^3  \|(c_k)_{\#k = r}\|_{\ell^\infty}  \|u\|_{\dot{H}^s}^{r-1} \lesssim_{r,s}  N^{3(r-2)}   \|u\|_{\dot{H}^s}^{r-1} \leq N^3 \|u\|_{\dot{H}^s}^{2}.
$$
Similarly, $\chi_r$ generates an Hamiltonian flow $\Phi_{\chi_r}^t$, $0\leq t \leq 1$ mapping $B_s(0,2\varepsilon_0)$ into $B_s(0,3\varepsilon_0)$ and close to the identity, i.e. $\| \Phi_{\chi_r}^t(u) - u\|_{\dot{H}^s}  \leq N^3 \|u\|_{\dot{H}^s}^{2}$.
Note that by construction, $\Phi_{\chi_r}^t$ and $\Phi_{-{\chi_r}}^t$ are symplectic and $\Phi_{\chi_r}^t\circ \Phi_{-\chi_r}^t(u) = u$. Furthermore, since $\J_{r,3N^3}^\E \subset \M_r$, $\chi$ commutes with $\| \cdot\|_{L^2}^2$. Consequently, applying the Noether's theorem, $\Phi_{\chi_r}^t$ and $\Phi_{-{\chi_r}}^t$ preserve the $L^2$ norm.

Provided that $\eps_0$ satisfies an estimates of the form $ \eps_0 \lesssim_{r,s} N^{-3}$, since $\tau^{(0)}$ and $\tau{(1)}$ are close to the identity (i.e. they satisfy \eqref{tau}), without loss of generality, we can assume that $\tau^{(0)}$ maps $B_s(0, \varepsilon_0)$ into $B_s(0,(3/2)\varepsilon_0)$ and $\tau^{(1)}$ maps $B_s(0,3 \varepsilon_0)$ into $\dot{H}^s$ (and that the decomposition \eqref{eq:firstnormalform} holds on $B_s(0,3 \varepsilon_0)$). Thus it makes sense to consider $\tau^{(1)}\circ \Phi_{\chi}^t$ and $\Phi_{-\chi_r}^t \circ \tau^{(0)}$, and we have $\tau^{(1)}\circ \Phi_{\chi_r}^t\circ \Phi_{-\chi_r}^t \circ \tau^{(0)}(u)=u$. Note that of course $\tau^{(1)}\circ \Phi_{\chi_r}^t$ and $\Phi_{-\chi_r}^t \circ \tau^{(0)}$ are symplectic, preserve the $L^2$ norm and are close to the identity (see \eqref{tau}).
 
Consequently, now, we only have to focus on the Taylor expansion of $H_\E \circ \tau^{(1)}\circ \Phi_{\chi}^1$. First, we recall that, since $\Phi_{\chi}^1$ is a Hamiltonian flow, for any $j\geq 2$, we have on $B_s(0,2 \varepsilon_0)$
\begin{equation}
\label{monpetitdevdetaylor}
H_\E \circ \tau^{(1)}\circ \Phi_{\chi}^1 =\sum_{\ell=0}^j \frac{1}{\ell !} \mathrm{ad}_{\chi_r}^\ell(H_\E \circ \tau^{(1)} ) + \int_0^1  \frac{(1-t)^j}{j !} \mathrm{ad}_{\chi_r}^{j+1} (H_\E \circ \tau^{(1)} )\circ \Phi_{\chi}^t \mathrm{d}t
\end{equation}
where $\mathrm{ad}_{\chi_r} := \{ \chi_r, \cdot\}$. We aim at proving that the remainder term goes to $0$ as $j$ goes to $+\infty$ and to control the convergence of the entire series. Recalling that by induction hypothesis $H_\E \circ \tau^{(1)} = Z_2^\E + \sum_{n\geq 3} P_n$ (see \eqref{monpetitPn}), we have to estimate the coefficients of $\frac1{\ell !} \mathrm{ad}_{\chi_r}^\ell P_n$.

\noindent \emph{$*$ Estimation of $\frac1{\ell !} \mathrm{ad}_{\chi_r}^\ell P_n$.} Considering the definition of $P_n$ in \eqref{monpetitPn} and $\chi_r$ in \eqref{monpetitchir},  applying iteratively the Lemma \ref{poisson} we deduce that $(\ell !)^{-1} \mathrm{ad}_{\chi_r}^\ell P_n$ is an homogeneous polynomial of degree $n+\ell(r-2)$ of the form
\begin{equation}
\label{defdnl}
\frac1{\ell !} \mathrm{ad}_{\chi_r}^\ell P_n = \sum_{k\in \M_{n+\ell(r-2)}} d_k(n,\ell) \, u^k
\end{equation}
where $d(n,\ell)\in \mathbb{C}^{\M_{n+\ell(r-2)}}$ satisfies the reality condition and the estimate
$$
\| d(n,\ell) \|_{\ell^\infty} \leq  \frac1{\ell !}   \| (c_k)_{\# k =n}\|_{\ell^\infty} (N^3 r \|(c_k)_{\#k = r}\|_{\ell^\infty})^{\ell} \prod_{i=1}^\ell n + (i-1)(r-2).
$$
Consequently, using the induction hypothesis and recalling that $n+\ell(r-2)$ is the degree of $(\ell !)^{-1} \mathrm{ad}_{\chi_r}^\ell P_n$, we have
\begin{equation*}
\begin{split}
\| d(n,\ell) \|_{\ell^\infty} &\leq  \rho^n N^{3n-9}  (\rho^r N^{(3r-9)} N^3 )^{\ell}  r^\ell \prod_{i=1}^\ell \frac{n + (i-1)(r-2)}i \\
				&= (r\rho^2)^{\ell} (\rho N^3)^{n+\ell(r-2)} N^{-9}  \prod_{i=1}^\ell \frac{n + (i-1)(r-2)}i.
\end{split}
\end{equation*}
Then let us estimate the product.  For $\ell\geq n$ we write 
\begin{align*} \prod_{i=1}^\ell \frac{n+(i-1)(r-2)}{i}=& \prod_{i=1}^n \frac{n+(i-1)(r-2)}{i} \prod_{i=n+1}^\ell \frac{n+(i-1)(r-2)}{i} \\
\leq& (r-1)^n \prod_{i=1}^n \frac ni \ (r-1)^{\ell-n}\prod_{i=n+1}^\ell\frac{i-1}i\leq r^\ell\frac{n^n}{n!}\end{align*}
while for $\ell\leq n$ we have  
$$\prod_{i=1}^\ell \frac{n+(i-1)(r-2)}{i}\leq  (r-1)^\ell \prod_{i=1}^\ell \frac ni\leq r^\ell \frac{n^n}{n!} .$$
Therefore using that $ \frac{n^n}{n!} \leq e^{n-1}$, for $n\geq 1$, (it results of an elementary Maclaurin–Cauchy test), we get
\begin{equation}
\label{est_dnl_final}
\| d(n,\ell) \|_{\ell^\infty} \leq e^{n-1} (r\rho)^{2\ell}  (\rho N^3)^{n+\ell(r-2)} N^{-9}.
\end{equation}
\noindent \emph{$*$ Estimation of $\frac1{\ell !} \mathrm{ad}_{\chi_r}^\ell Z_2^\E$.} Recalling that by construction (see \eqref{PBchirZ2}), we have
$$
\{\chi_r, Z_2^\E\} = -\sum_{k\in M_r} \mathbb{1}_{k\in \J_{r,3N^3}^\E} c_k u^k.
$$
The previous analysis proves that $ \mathrm{ad}_{\chi_r}^\ell Z_2^\E$ is an homogeneous polynomial of degree $\ell (r-2)+2$ of the form
$$
\frac1{\ell !} \mathrm{ad}_{\chi_r}^\ell Z_2^\E = \sum_{k\in \M_{\ell (r-2)+2}} \widetilde{d}(\ell) \, u^k
$$
where $\widetilde{d}(\ell)$ satisfies the reality condition and the estimate
\begin{equation}
\label{est_tidnl_final}
\| \widetilde{d(\ell)} \|_{\ell^\infty} \leq \ell^{-1} e^{r-1} (r\rho)^{2(\ell-1)}  (\rho N^3)^{r+(\ell-1)(r-2)} N^{-9}.
\end{equation}

\noindent \emph{$*$ Convergence of the remainder term.} We recall that by induction hypothesis, $H_\E \circ \tau^{(1)}$ writes on $B_s(0,3\eps_0)$
$$
H_\E \circ \tau^{(1)} = Z_2^{\mathcal{E}} + \sum_{k\in \M} c_k u^k = Z_2^{\mathcal{E}} + \sum_{n\geq 3} P_n.
$$
Here implicitly, $3\eps_0\leq (c_s \rho N^3)^{-1}$, with $c_s\leq \pi/\sqrt3$ to ensure that the above entire series is absolutely convergent (see Lemma \ref{lem:jaiunbeaurayon}). Consequently, on $B_s(0,2\eps_0)$, the remainder term of the Taylor expansion \eqref{monpetitdevdetaylor} writes
\begin{multline*}
\mathrm{Rem}^{(j)} = \int_0^1  \frac{(1-t)^j}{j !} \mathrm{ad}_{\chi_r}^{j+1} (H_\E \circ \tau^{(1)} )\circ \Phi_{\chi}^t \mathrm{d}t =  \int_0^t  \frac{(1-t)^j}{j !} (\mathrm{ad}_{\chi_r}^{j+1} Z_2^\E)\circ \Phi_{\chi}^t \mathrm{d}t \\+ \sum_{n\geq 3} \int_0^t  \frac{(1-t)^j}{j !} (\mathrm{ad}_{\chi_r}^{j+1} P_n)\circ \Phi_{\chi}^t \mathrm{d}t=: R_2^{(j)} + \sum_{n\geq 3} R_n^{(j)}.
\end{multline*}
We aim at proving that provided $N^3 \eps_0$ is small enough then for $u\in B_s(0,2\eps_0)$, $\mathrm{Rem}^{(j)}(0)$ goes to $0$  as $j$ goes to $+\infty$.

By definition of $d(n,j+1)$ (see \eqref{defdnl}), for $u\in \dot{H}^s$, we have
\begin{multline*}
\left|\frac1{(j+1)!}\mathrm{ad}_{\chi_r}^{j+1} P_n(u)\right| =  \! \! \sum_{k\in \M_{n+(j+1)(r-2)}}  \! \! |d_k(n,j+1) \, u^k| \leq \|d(n,j+1)\|_{\ell^\infty} (c_s \|u\|_{\dot{H}^s})^{n+(j+1)(r-2)}\\
\mathop{\leq}^{\eqref{est_dnl_final}}  e^{n-1} (r\rho)^{2(j+1)}  (c_s \|u\|_{\dot{H}^s} \rho N^3)^{n+(j+1)(r-2)} N^{-9}.
\end{multline*}
Consequently, since $\Phi_{\chi_r}^{t}$ maps $B_s(0,2\eps_0)$ in $B_s(0,3\eps_0)$, for $u\in B_s(0,2\eps_0)$, we have
$$
|R_n^{(j)}(u)|=\left| \int_0^t  \frac{(1-t)^j}{j !} (\mathrm{ad}_{\chi_r}^{j+1} P_n)( \Phi_{\chi}^t(u) )\mathrm{d}t \right| \leq e^{n-1} (r\rho)^{2(j+1)}  (3 c_s \eps_{0} \rho N^3)^{n+(j+1)(r-2)} N^{-9}.
$$
Therefore, provided that $ e 3\, c_s \eps_{0} \rho N^3 < 1$ and $(r\rho)^{\frac{2}{r-2}} 3\, c_s \eps_{0} \rho N^3<1$, for $u\in B_s(0,2\eps_0)$ the series $\sum_{n\geq 3} R_n^{(j)}(u)$ is absolutely convergent and goes to $0$ as $j$ goes to $+\infty$.

Similarly, using the estimate \eqref{est_tidnl_final} of $\widetilde{d}(\ell)$, for $u\in B_s(0,2\eps_0)$, we have
$$
|R_2^{(j)}(u)| \leq (j+1)^{-1} e^{r-1} (r\rho)^{2j}  (3\, c_s \eps_{0} \rho N^3)^{r+j(r-2)} N^{-9}.
$$
Consequently, provided that $(r\rho)^{\frac{2}{r-2}} 3\, c_s \eps_{0} \rho N^3\leq1$, for $u\in B_s(0,2\eps_0)$, $R_2^{(j)}(u)$ goes to $0$ as $j$ goes to $+\infty$.

\noindent \emph{$*$ Description and convergence of the series.} We have proven that, provided that $\eps_0 N^3$ is small enough with respect to $r^{-1}$, on $B_s(0,2\varepsilon_0)$, the remainder term of the Taylor expansion \eqref{monpetitdevdetaylor} goes to $0$ as $j$ goes to $+\infty$. Consequently, if $u\in B_s(0,2\eps_0)$, we have
$$
H_\E \circ \tau^{(1)}\circ \Phi_{\chi}^1(u)  =\sum_{\ell=0}^{+\infty} \frac{1}{\ell !} \mathrm{ad}_{\chi_r}^\ell(H_\E \circ \tau^{(1)} )(u) = \sum_{\ell=0}^{+\infty} \frac{1}{\ell !} \mathrm{ad}_{\chi_r}^\ell Z_2^\E(u)  + \sum_{\ell=0}^{+\infty} \sum_{n\geq 3}\frac{1}{\ell !} \mathrm{ad}_{\chi_r}^\ell P_n(u).
$$
First, to order the terms of these series as we want, let us check that they are absolutely convergent. Indeed, realizing the same kind of estimates we did to control the remainder terms, for $u\in B_s(0,2\eps_0)$, we have
$$
\left| \frac{1}{\ell !} \mathrm{ad}_{\chi_r}^\ell P_n(u) \right| \leq N^{-9} (2\,e\, c_s \eps_0  \rho N^3)^n    (2\, c_s\, (r\rho)^{\frac{2}{r-2}} \eps_0  \rho N^3)^{\ell(r-2)} 
$$ 
$$
\left| \frac{1}{(\ell+1) !} \mathrm{ad}_{\chi_r}^{\ell+1} Z_2(u) \right| \leq N^{-9} (2\,e\, c_s \eps_0  \rho N^3)^r    (2\, c_s\, (r\rho)^{\frac{2}{r-2}} \eps_0  \rho N^3)^{\ell(r-2)}  .
$$
Consequently, provided that $ 2\, e \, c_s \eps_{0} \rho N^3 < 1$ and $2\, (r\rho)^{\frac{2}{r-2}}  c_s \eps_{0} \rho N^3<1$, the series are absolutely convergent.
Therefore, defining, for $m\geq 3$, a homogeneous polynomial of degree $m$, denoted $Q_m$, by
$$
Q_m(u) = \sum_{n+\ell(r-2)=m} \frac{1}{\ell !} \mathrm{ad}_{\chi_r}^\ell P_n(u) + \frac{\mathbb{1}_{n=r} }{(\ell+1) !} \mathrm{ad}_{\chi_r}^{\ell+1} Z_2(u) = \sum_{k\in \M_m} \widetilde{c}_k  u^k
$$
where
$$
\widetilde{c}_k = \sum_{n+\ell(r-2)=\#k} d(n,\ell) + \mathbb{1}_{n=k} \widetilde{d}(\ell),
$$
we have on $B_s(0,2\eps_0)$
$$
H_\E \circ \tau^{(1)}\circ \Phi_{\chi}^1 = Z_2^\E + \sum_{m\geq 3} Q_m.
$$ 
Finally, let us check the properties $i),ii)$ and $iii)$.
\begin{itemize}
\item If $m< r$ then the only solution of the equation $n+\ell(r-2)=m$ is $n=m$ and $\ell=0$. Consequently, we have $Q_m = P_m$ and so $\widetilde{c}_k = c_k$  if $3\leq\#k<r$. Therefore $i)$ and $ii)$ are satisfied if $3\leq\#k<r$.
\item If $m=r$ then the only solution of the equation $n+\ell(r-2)=m$ is $n=r$ and $\ell=0$. Consequently, we have $Q_r = P_r + \{\chi_r,Z_2^\E\}$. Therefore, by construction of $\chi_r$ (see \eqref{monpetitchir} and \eqref{PBchirZ2}), if $\#k=r$, we have $\widetilde{c}_k = \mathbb{1}_{k\notin \J_{n,3N^3}^\E}c_k$ and so $i)$ and $ii)$ are satisfied.
\item Finally, we have to establish the general control $iii)$ on $\widetilde{c}_k$. Using the estimate \eqref{est_dnl_final} (resp. of \eqref{est_tidnl_final}) of  $d(n,\ell)$ (resp.  $\widetilde{d}(\ell)$), we have
\begin{multline*}
|\widetilde{c}_k| \leq 2\sum_{n+\ell(r-2)=\#k}  e^{n-1} (r\rho)^{2\ell}  (\rho N^3)^{\#k} N^{-9} \leq N^{-9} (e\, \rho N^3)^{\#k} \sum_{0\leq \ell \leq \frac{\#k}{r-2}} (r^2\rho^2 e^{2-r})^\ell \\
\leq N^{-9} (\widetilde{\rho} N^3)^{\# k}
\end{multline*}
where $\widetilde{\rho} = e^{\frac1{r-2}} \rho \max(r^2\rho^2,e)$ (we have used that the number of terms of the sum above is smaller than or equal to $e^{\frac{\#k}{r-2}}$).
\end{itemize}
\end{proof}

\subsection{Proof of Theorem \ref{thm-BNF} and formal computations}\label{formal}
This section is devoted to the proof of Theorem \ref{thm-BNF} and more particularly to the computation of the fourth order integrable terms and some of the sixth order integrable terms of the resonant normal form. 

Nevertheless, before entering into this proof let us introduce two preparatory lemmas. First, let us prove that the third and fourth order resonant terms are integrable.
\begin{lemma}\label{int-order4} For all $\mathcal{E}\in \{\mathrm{\ref{gBO},\ref{gKdV}}\}$, we have $\mathcal{R}_3^{\mathcal{E}} = \emptyset$ and if $k\in\mathcal{R}_4^{\mathcal{E}}$ then $\Irr(k)=\emptyset$.
\end{lemma}
\begin{proof} Let $k\in \mathcal{R}_n^{\mathcal{E}}$ with $n\in \{3,4\}$, i.e. $k\in (\mathbb{Z}^*)^n$ satisfies
\begin{equation}
\label{mon_petit_system}
\left\{ \begin{array}{ccc} k_1+\dots+k_n=0, \\
				 k_1 |k_1|^{\alpha_\E} + \dots + k_n |k_n|^{\alpha_\E} = 0.
\end{array}\right. 
\end{equation}
We aim at proving that $n=4$ and $\Irr \, k = \emptyset$. We note that if $k$ is solution of \eqref{mon_petit_system} then $-k$ is solution of \eqref{mon_petit_system} and it is clear that $\Irr \, k=\emptyset$ if and only if $\Irr (-k)=\emptyset$. Consequently, without loss of generality, we assume that
\begin{equation}
\label{plusduncotequedelautre}
\# \{ j \ | k_j >0 \} \geq \# \{ j \ | k_j <0 \}.
\end{equation}
Furthermore since $k_1+\dots+k_n = 0$, there exists $i,j$ such that $k_i<0<k_j$. Consequently, we deduce that $\# \{ j \ |  \ k_j <0 \}\in \{1,2\}$.

\noindent $\bullet$ First, let us prove by contradiction that $\# \{ j \ |  \ k_j <0 \}\neq 1$. By symmetry of \eqref{mon_petit_system} by permutation of the coordinates, without loss of generality, we assume that $k_n <0$ and $k_1,\dots,k_{n-1}>0$.

Therefore, since $-k_n = k_1+\dots+k_{n-1}$, we have
\begin{multline*}
\|  k_1 e_1 + \dots+ k_{n-1} e_{n-1} \|_{1+\alpha_\E} := (k_1^{1+\alpha_{\mathcal{E}}}+\dots+k_{n-1}^{1+\alpha_{\mathcal{E}}})^{\frac1{1+\alpha_\E}} = k_1+\dots+k_{n-1} \\= \|  k_1 e_1 \|_{1+\alpha_\E} + \dots+ \|k_{n-1} e_{n-1} \|_{1+\alpha_\E}
\end{multline*}
where $(e_1,\dots,e_{n-1})$ denotes the canonical basis of $\mathbb{R}^{n-1}$. As a consequence, the vectors $(k_j e_j )_{j=1\dots n-1}$ satisfy the equality case of the Minkowski inequality. Consequently, they should be all collinear which is impossible since, by assumption, $k_j \neq 0$ for all $j\in \llbracket1,n\rrbracket$.

\noindent $\bullet$ We have proven that $\# \{ j \ |  \ k_j <0 \}=2$. Consequently, we deduce of \eqref{plusduncotequedelautre} that $n\neq 3$ and so $n=4$. Without loss of generality, we assume that $k_1,k_2>0$ and we denote $h=(k_1,k_2)$ and $\ell=-(k_3,k_4)$. Consequently, \eqref{mon_petit_system} writes 
\begin{equation}
\label{mon_petit_system_le_retour}
\left\{ \begin{array}{ccc} h_1 + h_2 &=& \ell_1+\ell_2 \\
				h_1^{1+\alpha_{\mathcal{E}}}+h_2^{1+\alpha_{\mathcal{E}}}&=&\ell_1^{1+\alpha_{\mathcal{E}}}+\ell_2^{1+\alpha_{\mathcal{E}}}
\end{array}
\right. .
\end{equation}
To prove that $\Irr k=\emptyset$, we just have to prove that if $(h,\ell)$ is solution of \eqref{mon_petit_system_le_retour} then $h=\ell$ or $h=(\ell_2,\ell_1)=:\ell^{\mathrm{sym}}$. 

If $(h,\ell)$ is a solution of \eqref{mon_petit_system_le_retour} such that $k_1=k_2$ and $\ell_1=\ell_2$ then by the equation, we deduce that $k_1=\ell_1$ and so $h=\ell$. 

Consequently, by symmetry, without loss of generality, we assume that $\ell$ is fixed, that it satisfies $\ell_1\neq \ell_2$ and we consider $h$ as the unknown of the system \eqref{mon_petit_system_le_retour}. First, we observe that $h=\ell$ and $h=\ell^{\mathrm{sym}}$ are two distinct triviales solutions of \eqref{mon_petit_system_le_retour}. Then, we observe that the solution of \eqref{mon_petit_system_le_retour} belong to the intersection between a straight line and a sphere for the $\| \cdot\|_{1+\alpha_{\mathcal{E}}}$ norm. Consequently, since by Minkowski, the norm $\| \cdot\|_{1+\alpha_{\mathcal{E}}}$ is strictly convex on $\mathbb{R}^2$, the number of solution of \eqref{mon_petit_system_le_retour} is not larger than $2$. Therefore, $h=\ell$ and $h=\ell^{\mathrm{sym}}$ are the only solutions of \eqref{mon_petit_system_le_retour}.
\end{proof}

Now let us prove that $\mathcal{M}_3 =\J_{3,1}^\E$, i.e. that we have killed all the cubic terms in the resonant normal form process.
\begin{lemma}\label{lemma3} If $k\in \mathcal{M}_3 $ and $\mathcal{E}\in \{\mathrm{\ref{gBO},\ref{gKdV}}\}$ then
$$
\left| \frac{k_1}{k_1|k_1|^{\alpha_{\mathcal{E}}}+k_2|k_2|^{\alpha_{\mathcal{E}}}+k_3|k_3|^{\alpha_{\mathcal{E}}}} \right| \leq 1.
$$
\end{lemma}
\begin{proof} Since $k\in \mathcal{M}_3$, we have $k_1=-(k_2+k_3)$. Then, up to some natural symmetries, we just have to deal with the following cases.

\begin{itemize}
\item If $\mathcal{E}=\mathrm{\ref{gKdV}}$ then $\left| \frac{k_2+k_3}{k_2^3+k_3^3 - (k_2+k_3)^3} \right| = \frac1{3|k_2k_3|} \leq \frac13$.
\item If $\mathcal{E}=\mathrm{\ref{gBO}}$ and $k_2>0, k_3>0$ then $\left| \frac{k_2+k_3}{k_2^2+k_3^2 - (k_2+k_3)^2} \right| = \frac1{2k_2} + \frac1{2k_3} \leq 1.$
\item If $\mathcal{E}=\mathrm{\ref{gBO}}$ and $k_2>-k_3>0$ then $\left| \frac{k_2+k_3}{k_2^2-k_3^2 - (k_2+k_3)^2} \right| = \frac1{-2k_3} \leq \frac12.$
\end{itemize}
\end{proof}

Now, we focus on the proof of Theorem \ref{thm-BNF}. Naturally, it relies on Proposition \ref{ellefaitlegrosduboulot} and its proof, where we have realized the Birkhoff normal form process.
\begin{proof}[Proof of Theorem \ref{thm-BNF}]
Until the last step, to get convenient notations, we omit the index $\mathcal{E}$. We adopt the same notations as in the proof of Proposition \ref{ellefaitlegrosduboulot}.
Furthermore, during this proof, if $k\in \M$, we denote by $\mu_n(k)$ the $n^{\mathrm{th}}$ largest index among $|k_1|,\dots,|k_\last|$.
\medskip

\noindent \emph{$\bullet$ Step 1 : Identification of the non integrable terms.} \
In this proposition, we have proven that on $B_{s}(0,2\eps_0)$ we have the decomposition
$$
H \circ \tau^{(1)} = Z_2^{\mathcal{E}} +  \sum_{\substack{k\in \M \setminus \J_{3N^3} \\ \#k\leq  r}} c_k u^k + \sum_{\substack{k\in \M\\ \#k\geq r+1}} c_k u^k
$$
where $c$ satisfies $ii)$ and $iii)$. Naturally, we just have to set
$$
\mathrm{R}^{(or)}(u) := \sum_{\substack{k\in \M\\ \#k\geq r+1}} c_k u^k.
$$
By applying Lemma \ref{lemma3}, we know that $ \J_{3,3N^3} = \M_3$, consequently there are no third order terms in the resonant Hamiltonian, i.e.
$$
\sum_{\substack{k\in \M \setminus \J_{3N^3} \\ \#k\leq  r}} c_k u^k = \sum_{\substack{k\in \M \setminus \J_{3N^3} \\ 4\leq \#k\leq  r}} c_k u^k.
$$
By applying Corollary \ref{cor:origin}, for $n\geq 4$, we get a partition of $\M_{n} \setminus \J_{n,3N^3}$
$$
\M_{n} \setminus \J_{n,3N^3} = P^{(1)}_n \cup P^{(2)}_n  \cup P^{(3)}_n
$$
where the sets $P^{(j)}_n$ are symmetric (i.e. $P^{(j)}_n=-P^{(j)}_n$) and satisfy
\begin{itemize}
\item if $k\in \Rc_n$ then $k\in P^{(1)}_n$,
\item if $k\in \M_{n} \setminus \J_{n,3N^3}$ and $\mu_3(k)\geq (\frac{(3N^3)^{\alpha}}{2(n-2)^{\alpha}})^{\frac1{1+\alpha}}$ then $k\in P^{(2)}_n$,
\item if $k\in (\M_{n} \setminus \J_{n,3N^3}) \setminus (P^{(1)}_n \cup P^{(2)}_n)$ and $-\ell,\ell$ are two coordinates of $k$ for some $\ell \geq 3N^3$ then $k\in P^{(3)}_n$.
\end{itemize} 

Then, observe that if $k\in P^{(3)}$ then $\#k \geq 5$. Indeed, if $k\in P^{(3)}_4$ then using the zero momentum condition, we deduce that $u^k$ is integrable and so $k$ should belong to $P^{(1)}_4$.
Now, we denote by $P^{(1,3)}$ the set of the indices $k\in P^{(1)}$ such that $-\ell,\ell$ are two coordinates of $k$ for some $\ell \geq N^3$.

Consequently, we set 
$$
\mathrm{R}^{(I_{> N^3})}(u) := \sum_{\substack{k\in P^{(3)}\cup P^{(1,3)} \\ 5 \leq \# k \leq r}} c_k u^k.
$$
Note that in Theorem \ref{thm-BNF}, we assume that the indices of the coefficients of the polynomials $\mathrm{Res}_{\leq N^3}, \mathrm{R}^{(\mu_3>N)}, \mathrm{R}^{(I_{> N^3})}$ are ordered (i.e. they belong to $\mathcal{D}$). Here we do not pay attention to this property because due to the symmetry of $k\mapsto u^k$ by permutation, up to multiplication of the coefficients by a factor like $\# k !$, the indices can always be easily ordered.

We denote by $P^{(1,2)}$ the set of the indices $k\in P^{(1)}\setminus P^{(1,3)}$ such that $\mu_1(k) > N^3$. Note that, since $P^{(1,3)}\cap P^{(1,2)} = \emptyset$, if $k\in P^{(1,2)}$ then $|(\Irr \, k)_1| > N^3$ (by construction an irreducible part is ordered). Consequently, applying Lemma \ref{lem:origin} to $\Irr \, k$, we deduce that we have $\mu_3(k) \geq (\frac{(N^3)^{\alpha}}{2(n-2)^{\alpha}})^{\frac1{1+\alpha}}$.

Finally, observing that if $N$ is large enough with respect to $r$ we have $(\frac{(N^3)^{\alpha}}{2(n-2)^{\alpha}})^{\frac1{1+\alpha}}\geq N$ for $n\in \llbracket 4,r\rrbracket$, we set
$$
\mathrm{R}^{(\mu_3>N)}(u) := \sum_{\substack{k\in P^{(2)}\cup P^{(1,2)} \\ 4 \leq \# k \leq r}} c_k u^k
$$
and
$$
\mathrm{Res}_{(\leq N^3)}(u)  := \sum_{\substack{k\in P^{(1)}\setminus (P^{(1,2)}\cup P^{(1,3)}) \\ 5 \leq \# k \leq r \\ \mathrm{Irr}(k)\neq \emptyset \ \mathrm{if} \  \#k  = 6  }} c_k u^k.
$$

Consequently, since by Lemma \ref{int-order4} the fourth order resonant terms are integrable, we have proven that
$$
H \circ \tau^{(1)} = Z_2+ Z_4 + Z_{6,\leq N^3} + \mathrm{Res}_{\leq N^3} + \mathrm{R}^{(\mu_3>N)} + \mathrm{R}^{(I_{> N^3})}+ \mathrm{R}^{(or)}
$$
where $Z_4$ and $Z_{6,\leq N^3}$ are two integrable Hamiltonians such that $Z_4$ contains all the fourth order integrable terms of $H \circ \tau^{(1)}$ and $Z_{6,\leq N^3}$ contains all the sixth order integrable terms of $H\circ \tau^{(1)}$ associated with monomials of indices smaller than or equal to $N^3$. The rest of the proof is devoted to the explicit computation of $Z_4$ and (a part of) $Z_{6,\leq N^3}$.

\medskip

\noindent \emph{$\bullet$ Step 2 : Setting of the formal computations.} We recall that, in the proof the Proposition \ref{ellefaitlegrosduboulot}, the change of coordinate generated by $\chi_r$ (to kill the $r^{\mathrm{th}}$ order term associated with indices in $\J_{r,3N^3}$) preserves the lower order terms. Consequently, $Z_4$ contains all the fourth order integrable terms of $H \circ  \Phi_{\chi_3}^1$ and $Z_{6,\leq N^3}$ contains all the sixth order integrable terms of $H\circ \Phi_{\chi_3}^1 \circ \Phi_{\chi_4}^1\circ \Phi_{\chi_5}^1$  associated with monomials of indices smaller than or equal to $N^3$. Actually, by an elementary argument of degree, we observe that the sixth order terms of $H\circ \Phi_{\chi_3}^1 \circ \Phi_{\chi_4}^1\circ \Phi_{\chi_5}^1$ and $H\circ \Phi_{\chi_3}^1 \circ \Phi_{\chi_4}^1$ are the same. Consequently, $Z_{6,\leq N^3}$ contains all the sixth order integrable terms of $H\circ \Phi_{\chi_3}^1 \circ \Phi_{\chi_4}^1$  associated with monomials of indices smaller than or equal to $N^3$.
 
 First, we have to determine $\chi_3$ and $\chi_4$ explicitly. To get convenient notations, we denote 
 $$
 \mathcal{L}_m(u) := \int_{\mathbb{T}} u^m \ \mathrm{d}x = \sum_{k\in \M_m} u^k
 $$
 in such a way that, on a neighborhood of the origin, we have
 $$
 H = Z_2(I)  + \sum_{m \geq 3} a_m \, \mathcal{L}_m.
 $$
 We recall that formally, we have
$$
H\circ \Phi_\chi^1(u)  \ \mathop{=}_{u\to 0} \ \sum_{k = 0}^{\infty} \frac1{k !} \ \mathrm{ad}_{\chi}^k H(u).
$$
Consequently, we have
\begin{equation*}
\begin{split}
H \circ \Phi_{\chi_3}^1 \mathop{=}_{u\to 0}& Z_2 + a_3 \, \mathcal{L}_3 + \{\chi_3 , Z_2\} 
									+ a_4 \, \mathcal{L}_4 +\{\chi_3 , a_3 \, \mathcal{L}_3\} + \frac12 \{\chi_3,\{\chi_3,Z_2\}\} \\
									&+ P_5(u) +a_6 \, \mathcal{L}_6 + \{\chi_3 ,a_5 \, \mathcal{L}_5 \}+ \frac12 \{\chi_3,\{\chi_3,a_4 \, \mathcal{L}_4 \}\}\\ &+ \frac16 \{\chi_3,\{\chi_3,\{\chi_3,a_3 \, \mathcal{L}_3\}\}\} + \frac1{24} \{\chi_3,\{\chi_3,\{\chi_3,\{\chi_3,Z_2\}\}\}\} + \mathcal{O}(u^7)
\end{split}
\end{equation*}
where $P_5$ is a homogeneous polynomial of degree $5$. Since, by Lemma \ref{lemma3}, $\J_{3,3N^3}=\M_3$,  by construction $\chi_3$ is the solution of the homological equation
\begin{equation}
\label{eq_chi3}
\tag{$\mathcal{H}_3$}
a_3\, \mathcal{L}_3 + \{ \chi_3 , Z_2 \} = 0,
\end{equation}
i.e.
$$
\chi_3 = a_3 \sum_{k_1 + k_2 + k_3 =0} \frac{u^k}{i\, \Omega(k)}.
$$
Consequently, we have
\begin{equation*}
\begin{split}
H\circ \Phi_{\chi_3}^1 \mathop{=}_{u\to 0}& Z_2 + a_4 \, \mathcal{L}_4 +\frac12  \{\chi_3 , a_3 \, \mathcal{L}_3\} 
									+ P_5(u) +a_6 \, \mathcal{L}_6 + \{\chi_3 ,a_5 \, \mathcal{L}_5 \}\\&+ \frac12 \{\chi_3,\{\chi_3,a_4 \, \mathcal{L}_4 \}\}+ \frac18 \{\chi_3,\{\chi_3,\{\chi_3,a_3 \, \mathcal{L}_3\}\}\}+\mathcal{O}(u^7).
\end{split}
\end{equation*}
Therefore, $Z_4$ is the integrable part (i.e. depending only on the actions) of $a_4\, \mathcal{L}_4 + \frac12 \{ \chi_3 , a_3\, \mathcal{L}_3  \}$.

Then, $\chi_4$ is constructed to solve a homological equation restricted to indices in $\J_{4,3N^3}$ as explained in the proof of Proposition  \ref{ellefaitlegrosduboulot} :
\begin{equation}
\label{eq_chi4}
\tag{$\mathcal{H}_4$}
\Pi_{\J_{4,3N^3}}[a_4\, \mathcal{L}_4 + \frac12 \{ \chi_3 , a_3\, \mathcal{L}_3  \}] +  \{ \chi_4 , Z_2 \}  =0.
\end{equation}
Moreover, by a straightforward calculation, we have
\begin{equation}
\label{chi3L3andc}
 \{ \chi_3 , a_3\, \mathcal{L}_3  \} =   9\, a_3^2 \! \! \! \! \! \! \sum_{\substack{ k_1+k_2+k_3+k_4=0 \\ k_1+k_2 \neq 0 }} \! \! \! \! \!c_{k_1,k_2} \, u^k \ \ \mathrm{where} \ \ c_{k_1,k_2} = \frac{2\pi (k_1+k_2)}{\Omega(k_1,k_2,-k_1-k_2)}.
\end{equation}
Consequently, following the construction of the Proposition \ref{ellefaitlegrosduboulot}, we have
$$
\chi_4 =   \sum_{\substack{ k\in\J_{4,3N^3}  }}  \frac{2\, a_4+ 9\, a_3^2 \, c_{k_1,k_2}}{2\, i\, \Omega(k)}u^k\ .
$$
Therefore the sixth order term of $H \circ \Phi_{\chi_3}^1\circ \Phi_{\chi_4}^1$, denoted $P_6$, is
\begin{multline}
\label{givemeZ6}
P_6 = a_6 \, \mathcal{L}_6 + \{\chi_3 ,a_5 \, \mathcal{L}_5 \}+ \frac12 \{\chi_3,\{\chi_3,a_4 \, \mathcal{L}_4 \}\} \\ + \frac18 \{\chi_3,\{\chi_3,\{\chi_3,a_3 \, \mathcal{L}_3\}\}\}+ \frac12\{\chi_4,  a_4\, \mathcal{L}_4 + \frac12 \{ \chi_3 , a_3\, \mathcal{L}_3  \} + Z_4 \}.
\end{multline}
Note that there are three kinds of terms in $P_6$ : the  original terms coming from  $a_6 \, \mathcal{L}_6 $, those that come from the composition by the Lie transformation $ \Phi_{\chi_3}^1$ and those that come from the composition by the Lie transformation $ \Phi_{\chi_4}^1$.

We recall that $Z_{6,\leq N^3}$ is just the integrable part of $P_6$ projected on actions with index smaller than $N^3$ and we can write
\begin{equation}
\label{Z66}
Z_{6,\leq N^3}(I)=\sum_{0<p\leq q\leq\ell\leq N^3}c_{p,q}^\E(\ell)I_pI_qI_\ell.
\end{equation}
 Finally, we also notice that in view of Lemma \ref{lemma3}, the three terms in \eqref{givemeZ6} involving Poisson brackets with $\chi_3$ cannot be responsable of the growth of $c_{p,q}(\ell)$ with respect to $\ell$. So the only contributing term to this growth in \eqref{Z66} is the last one which involves a Poisson bracket with $\chi_4$. 

\medskip

\noindent \emph{$\bullet$ Step 3 : Computation of $Z_4$.} We recall that, by construction, it is the integrable part (i.e. depending only on the actions) of $a_4\, \mathcal{L}_4 + \frac12 \{ \chi_3 , a_3\, \mathcal{L}_3  \}$. To determine it from the explicit expression of $\mathcal{L}_4$ and $\{ \chi_3 , a_3\, \mathcal{L}_3  \}$ (computed just above), we use the Poincar\'e's formula: $\sum_{A \cup B \cup C} = \sum_A +  \sum_B +  \sum_C -  \sum_{B\cap C} -  \sum_{C\cap A} -  \sum_{A\cap B} + \sum_{A\cap B\cap C}$.  For example, it is clear that the integrable terms of $\mathcal{L}_4 = \sum_{k\in \mathcal{M}_4} u^k$ are obtained when $k_1 = -k_2$ or $k_1 = -k_3$ or $k_1 = -k_4$. All these cases being symmetric, by the Poincar\'e's formula, we know that the integrable terms of $\mathcal{L}_4$ contains three times the terms such that $k_1=-k_2$ minus three times those such that $k_1=-k_2=-k_3$ plus those such that $k_1=-k_2=-k_3=-k_4$. Observing that since $k_1+k_2+k_3+k_4=0$ there does not exist any term of this last kind, we deduce that the integrable part of $\mathcal{L}_4$ is
$$
3 \, a_4\sum_{k_1,k_2\in\Z^*}I_{k_1}I_{k_2} - 3 \, a_4 \sum_{k\in \mathbb{Z}^{*}}I_k^2.
$$
Proceeding similarly to determine the integrable part of $\{ \chi_3 , a_3\, \mathcal{L}_3  \}$, we deduce that 
\begin{equation*}
\begin{split}
Z_4(I) &= 3 \, a_4\sum_{k_1,k_2\in\Z^*}I_{k_1}I_{k_2} - 3 \, a_4 \sum_{k\in \mathbb{Z}^{*}}I_k^2 + \frac{18\, a_3^2}2\, \sum_{k_1+ k_2 \neq 0} c_{k_1,k_2} \,I_{k_1} I_{k_2} -  \frac92\,  a_3^2\, \sum_{k \in \mathbb{Z}^{*}} c_{k,k} \, I_{k}^2 \\
&= 3 \, a_4\, \| u\|_{L^2}^4  + 3\, \sum_{k=1}^\infty    (3\, a_3^2 \, c_{k,k}  - 2a_4)\, I_{k}^2 +  9\, a_3^2 \sum_{ |k_1| \neq |k_2|} c_{k_1,k_2} \,I_{k_1} I_{k_2}
\end{split}
\end{equation*}
Taking into account the symmetries of $c$, i.e. $c_{k_1,k_2} = c_{k_2,k_1} = c_{-k_1,-k_2}$, we deduce that
$$
Z_4(I) = 3 \, a_4\, \| u\|_{L^2}^4  + 3\, \sum_{k=1}^\infty    (3\, a_3^2 \, c_{k,k}  - 2a_4)\, I_{k}^2 +  36\, a_3^2 \sum_{ 0< k_1 < k_2} (c_{k_1,k_2} + c_{-k_1,k_2}) \,I_{k_1} I_{k_2}.
$$
Consequently, to get the formula \eqref{Z4-KdV} (resp. \eqref{Z4-BO}) for $Z_4^{\mathrm{\ref{gKdV}}}$ (resp. $Z_4^{\mathrm{\ref{gBO}}}$), we just have to compute $c_{k_1,k_2} + c_{k_1,-k_2}$ when $0<k_1<k_2$.
\begin{itemize}
\item[$*$] \emph{Case $\E = \mathrm{\ref{gKdV}}$.} We have
$$
-2\pi^2 (c_{k_1,k_2} + c_{-k_1,k_2}) = \frac{(k_1+k_2)}{  (k_1+k_2)^3 -k_1^3-k_2^3 } + \frac{(k_1+k_2)}{  (k_1-k_2)^3-k_1^3+k_2^3 } = \frac1{3 k_1 k_2} - \frac1{3 k_1 k_2}=0.
$$ 
\item[$*$] \emph{Case $\E = \mathrm{\ref{gBO}}$.} We have
$$
2\pi  (c_{k_1,k_2} + c_{-k_1,k_2}) = \frac{2(k_1+k_2)}{ k_1^2 + k_2^2 - (k_1+k_2)^2  } + \frac{2(k_1-k_2)}{  k_1^2 - k_2^2 + (k_1-k_2)^2 } = -\frac1{k_1}-\frac1{k_2}+ \frac1{k_1} = -\frac1{k_2}.
$$
\end{itemize}


\medskip

\noindent \emph{$\bullet$ Step 4 : Computation of the brackets in \eqref{givemeZ6}.}  

\noindent \emph{$*$ Step 4.1 : $\{\chi_4,  a_4\, \mathcal{L}_4 + \frac12 \{ \chi_3 , a_3\, \mathcal{L}_3  \} + Z_4 \}$.} 

First, we notice that, since a Poisson bracket between an irreducible monomial and a polynomial in the actions cannot be a polynomial in the actions, the polynomials
$$
 \left\{\chi_4,  a_4\, \mathcal{L}_4 + \frac12 \{ \chi_3 , a_3\, \mathcal{L}_3  \} + Z_4 \right\} \ \mathrm{and} \ \left\{ \chi_4,     \! \! \!\sum_{k_1+k_2+k_3+k_4=0 } h_{k_1,k_2} u^k \right\}
$$
have the same integrable part where, by formula \eqref{chi3L3andc}, we have set
$$
h_{k_1,k_2} = \left\{ \begin{array}{lll} a_4+(9/2) \, a_3^2 \, c_{k_1,k_2} & \mathrm{if} \ k_1\neq k_2\\ 0 & \mathrm{else} \end{array}\right.
$$

Consequently, since we are only interested in the computation of the integrable terms of $ \{\chi_4,  a_4\, \mathcal{L}_4 + \frac12 \{ \chi_3 , a_3\, \mathcal{L}_3  \} + Z_4 \}$, we only compute $\displaystyle \big\{ \chi_4,     \! \! \!\sum_{k_1+k_2+k_3+k_4=0 } h_{k_1,k_2} u^k \big\}$.
\begin{remark}\label{3N^3}
At this stage we can justify our choice to restrict the resolution of the homological equation \eqref{eq_chi4} to $\J_{4,3N^3}$, i.e. why we took  $3N^3$. We have to remember that we want to compute $Z_{6,\leq N^3}$ so we have to be sure to consider all the integrable terms of order six with indices smaller than $N^3$  in 
\begin{align*} \big\{ \chi_4, &    \! \! \!\sum_{k_1+k_2+k_3+k_4=0 } h_{k_1,k_2} u^k \big\}= \big\{ \sum_{\substack{ k\in\J_{4,3N^3}  }}  \frac{2a_4+ 9\, a_3^2 \, c_{k_1,k_2}}{2i\Omega(k)}u^k , \! \! \!\sum_{k_1+k_2+k_3+k_4=0 } h_{k_1,k_2} u^k \big\}\\
&=\sum_{j\in\Z^*}2i\pi j \sum_{(k_1,k_2,k_3,j)\in\J_{4,3N^3}}\frac{2a_4+ 9\, a_3^2 \, c_{k_1,k_2}}{2i\Omega(k_1,k_2,k_3,j)}\sum_{k_4+k_5+k_6+j=0} h_{k_4,k_5} \ u^k+\text{other terms}.
\end{align*}
Now the point is that the restriction $(k_1,k_2,k_3,j)\in\J_{4,3N^3}$ has to allow $\max (|k_1|,|k_2|,|k_3|)= N^3$. In the worst case $|j|\geq \max (|k_1|,|k_2|,|k_3|)$ but in that case $3\max (|k_1|,|k_2|,|k_3|)\geq |j|$ by the zero momentum condition. On the other hand $(k_1,k_2,k_3,j)\in\J_{4,3N^3}$ means $|j|\leq 3 N^3 \Omega (k_1,k_2,k_3,j)$ and since $\Omega (k_1,k_2,k_3,j)\geq1$ we are sure to consider all the $|j|$ up to $3 N^3$ and thus all the $k$ with $\max (|k_1|,|k_2|,|k_3|)\leq N^3$.
\end{remark}
Now we are sure that we are not missing terms for $Z_{6,\leq N^3}$ and thus, instead of computing $\displaystyle \big\{  \chi_4,     \! \! \!\sum_{k_1+k_2+k_3+k_4=0 } h_{k_1,k_2} u^k \big\}$, we can just compute $\displaystyle \big\{ \tilde \chi_4,     \! \! \!\sum_{k_1+k_2+k_3+k_4=0 } h_{k_1,k_2} u^k \big\}$ where 
$$
\tilde\chi_4 =   \sum_{\substack{ k\in\M_4\setminus\mathcal R_{4}  }}  \frac{2a_4+ 9\, a_3^2 \, c_{k_1,k_2}}{2i\Omega(k)}u^k=\sum_{\substack{ k \in \M_4 \cap \Irr  }}\frac{h_{k_1,k_2}}{i\Omega(k)} u^k
$$
and then to restrict the integrable part to indices smaller than $N^3$. Note that we have used that the quartic resonant terms are integrable (see Lemma \ref{int-order4}). Consequently, this sum only holds on indices $k\in \M_4$ such that $k_{j_1}+ k_{j_2} \neq 0$ for all $j_1\neq j_2$.\\
By a straightforward calculation, we get
$$
\left\{ \tilde\chi_4,     \! \! \!\sum_{k_1+k_2+k_3+k_4=0 } h_{k_1,k_2} u^k \right\} =  \sum_{\substack{k_1+k_2+k_3+k_4+k_5+k_6=0\\ k_1+k_2+k_3 \neq 0 \\ k_1 + k_2 \neq 0 \\ k_1 + k_3\neq 0\\ k_2 + k_3 \neq 0}} b_k \, u^k
$$
where 
$$b_k =  4 \, d_{k_1,k_2,k_3} \, w_k, \ \ \
d_{k_1,k_2,k_3} = \frac{2\pi (k_1+k_2+k_3)}{\Omega(k_1,k_2,k_3,-k_1-k_2-k_3)},
$$
$$
w_k = h_{k_1,k_2} h_{k_4,k_5}  + h_{k_1,-k_1-k_2-k_3} h_{k_4,k_5} + h_{k_1,k_2} h_{k_4,-k_4-k_5-k_6} + h_{k_1,-k_1-k_2-k_3} h_{k_4,-k_4-k_5-k_6}.
$$

\noindent \emph{Step 4.2: Computation of $ \{\chi_3 ,a_5 \, \mathcal{L}_5 \}$.} 
By a straightforward calculation, we have
$$
 \{\chi_3 ,a_5 \, \mathcal{L}_5 \} = 15\, a_3 \, a_5\, \sum_{\substack{k_1+\dots+k_6=0\\k_1+k_2\neq 0}} c_{k_1,k_2}\, u^k.
$$

\noindent \emph{Step 4.3 : Computation of $\{\chi_3,\{\chi_3,\{\chi_3,a_3 \, \mathcal{L}_3\}\}\}$.} This computation is elementary but quite heavy, especially to take into account the symmetries and to count the multiplicities. To help the reader, we provide the diagrams we have realized to follow and check it. We could make these diagram computations become rigorous but, since they are quite natural and not fundamental, we believe that it would be uselessly heavy.

First, let us present informally what are our diagram and how we compute their Poisson brackets. 

We represent $\mathcal{L}_n$ by a regular simplex. Their vertices, represented by crosses, refer to the indices of the modes whereas its simple edges refer to the zero momentum condition. To denote that, furthermore, we have solved an homological equation with $Z_2$ and excluded the resonant terms, we draw some double edges. For example, we denote
$$
\mathcal{L}_3 = \begin{array}{lll}  \includegraphics[scale = 0.75]{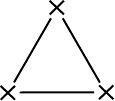} \end{array} \quad \mathrm{and} \quad \chi_3 = a_3 \begin{array}{lll} \includegraphics[scale = 0.75]{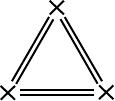} \end{array}.
$$
To compute the Poisson bracket of two diagrams $A$ and $B$, we just add the diagrams we get by connecting $A$ and $B$ by replacing a cross of $A$ and a cross of $B$ by a circle with a dark face on the $A$ side. Somehow, the circles refer to the old indices and the dark face are just a way to remember which diagram was on which side of the Poisson bracket. For example, we have
\begin{equation}
\label{monpremierdiagramcomplique}
a_3^{-1} \{\chi_3,\mathcal{L}_3\} = \left\{\begin{array}{lll} \includegraphics[scale = 0.75]{chi3.pdf} \end{array},\begin{array}{lll} \includegraphics[scale = 0.75]{L3.pdf} \end{array}\right\}   = 3\cdot 3\begin{array}{lll} \includegraphics[scale = 0.75]{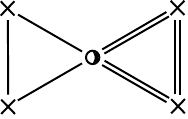} \end{array}.
\end{equation}
The factors $3$ come from the fact that for each diagram, we have $3$ choices of crosses and that all of them are equivalent.

 Now let us explain informally how we get the expression of a polynomial Hamiltonian from its diagram and to highlight this process on the elementary example of $\{\chi_3,\mathcal{L}_3\}$ whose diagram is given by \eqref{monpremierdiagramcomplique}.

\noindent $\blacktriangleright$ First, we index the crosses of the diagram (for example, from the top to the bottom and from the left to the right). Consequently, we get a polynomial of the form
 \begin{equation}
 \label{pedago}
 \sum_{k\in (\mathbb{Z}^*)^n} \beta_k u^k
 \end{equation}
 where $n$ is the number of vertices and $\beta_k$ are some coefficients we have to determined. For example, for $\{\chi_3,\mathcal{L}_3\}$, we have $n=4$.
 
\noindent $\blacktriangleright$ We index the circle and denote by $m$ the number of circles. If $j\in \llbracket 1,m\rrbracket$ is the index of a circles, $n+2j$ is the index of its dark face while $n+2j-1$ is the index of its white face.
 \item We extend $k$ into a vector of length $n+2m$, denoted $\ell$ (i.e. $k_j=\ell_j$) and we write the zero momentum conditions we read on the simplexes : if $(j_1,\dots,j_p)$ are the indices of the vertices of a simplexe we write
 $$
 \ell_{j_1}+\dots+ \ell_{j_p} = 0.
 $$
 Furthermore, we write the expression coming from the connections :
 $$
\forall j\in \llbracket 1,m\rrbracket,\ \ell_{n+2j-1} = - \ell_{n+2j}.
 $$
 From all these equation, we deduce that $\ell$ is a linear function of $k$ denoted $\ell(k)$ and that the sum \eqref{pedago} can be restricted to $k\in \M_4$.
 For example, for $\{\chi_3,\mathcal{L}_3\}$, we have the system
$$
\ell_5=-\ell_6, \quad \ell_1+\ell_2+\ell_6=0, \quad \ell_3+\ell_4+\ell_5=0
$$
which is equivalent to
$$
\ell_6 = -(k_1+k_2),   \quad \ell_5 = -( k_3+k_4),  \quad k_1+k_2+k_3+k_4=0.
$$

\noindent $\blacktriangleright$ Then we restrict the sum \eqref{pedago} to ensure that the coefficients of $\ell$ do no vanish (because they are old indices of modes). Consequently the sum \eqref{pedago} becomes 
\begin{equation}
\label{pedago2}
\sum_{\substack{k\in \M_n\\ \forall j, \ \ell_j(k) \neq 0} } \beta_k u^k.
\end{equation}
 For example, for $\{\chi_3,\mathcal{L}_3\}$, we just add the restriction $k_1+k_2\neq 0$.
 
\noindent $\blacktriangleright$ Let denote by $\mathbb{S}$ the set of the double simplexes. More precisely, $\{j_1,\dots,j_p\} \in \mathbb{S}$ if $j_1,\dots,j_p$ are the indices of the vertices of a simplexe represented by double edges. To ensure the non-resonance conditions, we restrict the sum \eqref{pedago2} to the indices $k$ such that if $\mathbb{s} = \{j_1,\dots,j_p\} \in \mathbb{S}$ then $\Omega_\mathbb{s}(k) := \Omega(\ell_{j_1}(k),\dots,\ell_{j_p}(k)) \neq 0$. Consequently \eqref{pedago2} becomes
 \begin{equation}
\label{pedago3}
\sum_{\substack{k\in \M_n\\ \forall j, \ \ell_j(k) \neq 0 \\ \forall \mathbb{s},\ \Omega_\mathbb{s}(k)\neq 0} } \beta_k u^k.
\end{equation}
For $\{\chi_3,\mathcal{L}_3\}$, we have $\mathbb{S} = \{ \{3,4,5\}\}$. However, since there are no cubic resonances (see Lemma \ref{int-order4}) the condition $\Omega_\mathbb{s}(k)\neq 0$ is trivial.
 
 \noindent $\blacktriangleright$ Then, we determine the coefficient $\beta_k$. We have to take into account the coefficients $2i\pi \ell$ coming from the Poisson brackets and the denominators coming from the resolution of the homological equation. Consequently, naturally, we set
 $$
 \beta_k = \left(\prod_{j=1}^m (-2 i \pi \ell_{n+2j}) \right) \left(\prod_{\mathbb{s}\in \mathbb{S}} i\Omega_{\mathbb{s}}(k) \right)^{-1}.
 $$
 For $\{\chi_3,\mathcal{L}_3\}$, we have 
 $$\beta_k = \frac{-2i\pi \ell_6(k)}{i \Omega(\ell_{3}(k),\dots,\ell_{5}(k)) } = \frac{2\pi (k_1+k_2)}{ \Omega(k_3,k_4,-k_3-k_4)}.$$
  
  \noindent $\blacktriangleright$ Finally, we recombine the denominators coming from the homological equations and the coefficients coming from Poisson brackets (to get coefficients like $c_{k_1,k_2}$). For example, for $\{\chi_3,\mathcal{L}_3\}$, we have $\beta_k = c_{k_3,k_4}$. Consequently, we get the same result as in \eqref{chi3L3andc}. Notice that, since, in practice, we only compute Poisson brackets with $\chi_3$, this last step is straightforward (actually, we skip the previous step).

\medskip

Now, we are going to apply this technic to deal with more intricate terms. For example a direct computation leads to
$$
\{\chi_3,\{\chi_3,\mathcal{L}_3\}\} = 54\, a_3^2 \sum_{\substack{k_1+\dots+k_5=0 \\ k_1+k_2\neq 0 \\ k_4+k_5\neq 0}} c_{k_1,k_2} c_{k_4,k_5} u^k + 54\, a_3^2  \sum_{\substack{k_1+\dots+k_5=0 \\ k_1+k_2\neq 0 \\ k_1+k_2+k_3\neq 0}} c_{k_1,k_2} c_{k_1+k_2,k_3} u^k
$$
which is highlighted through the following graphical computation
$$
\frac{\{\chi_3,\{\chi_3,\mathcal{L}_3\}\}}{9 \, a_3^2} = \left\{\begin{array}{lll} \includegraphics[scale = 0.75]{chi3.pdf} \end{array},\begin{array}{lll} \includegraphics[scale = 0.75]{chi3L3.pdf} \end{array}\right\}   = 6 \begin{array}{lll}\vspace{-0.3cm}  \includegraphics[scale = 0.75]{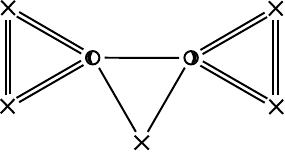} \end{array} + 6 \begin{array}{lll}\vspace{-0.3cm}  \includegraphics[scale = 0.75]{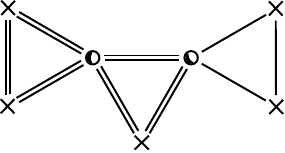} \end{array}.
$$

Finally, an elementary computation leads to
\begin{multline}
\label{chi33L3}
  \{\chi_3,\{\chi_3, \{ \chi_3 ,  \mathcal{L}_3  \} \}\}  = 162\,a_3^3 \! \! \! \sum_{\substack{ k_1+k_2+k_3+k_4+k_5+k_6=0 \\ k_1+k_2\neq 0 \\ k_3+k_4 \neq 0 \\ k_5+k_6\neq 0}} \! \! \!  c_{k_1,k_2}c_{k_3,k_4} \left( c_{k_5,k_6} + c_{k_1+k_2,k_3+k_4}\right) u^k
 \\+  324\,a_3^3 \! \! \! \! \sum_{\substack{ k_1+k_2+k_3+k_4+k_5+k_6=0 \\ k_1+k_2+k_3\neq 0\\ k_1+k_2\neq 0 \\ k_5+k_6 \neq 0 }} c_{k_1,k_2} c_{k_1+k_2,k_3}(c_{k_1+k_2+k_3,k_4}+3 \, c_{k_5,k_6})\, u^k.
\end{multline}
which is highlighted through the following graphical computations
$$
\frac{\{\chi_3,\{\chi_3,\{\chi_3,\mathcal{L}_3\}\}\}}{54 \, a_3^3} =  \left\{\begin{array}{lll} \includegraphics[scale = 0.75]{chi3.pdf} \end{array},\begin{array}{lll}\vspace{-0.3cm}  \includegraphics[scale = 0.75]{chi3chi3L3_1.pdf} \end{array} \right\} +\left\{\begin{array}{lll} \includegraphics[scale = 0.75]{chi3.pdf} \end{array},\begin{array}{lll}\vspace{-0.3cm}  \includegraphics[scale = 0.75]{chi3chi3L3_2.pdf} \end{array} \right\}   
$$
with
$$
\frac13 \left\{\begin{array}{lll} \includegraphics[scale = 0.75]{chi3.pdf} \end{array},\begin{array}{lll}\vspace{-0.3cm}  \includegraphics[scale = 0.75]{chi3chi3L3_2.pdf} \end{array} \right\}   =  \! \! \! \! \begin{array}{lll}  \includegraphics[scale = 0.75]{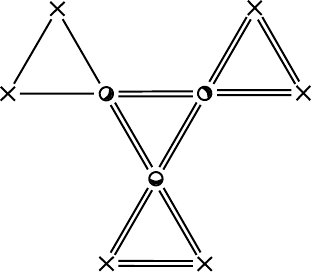} \end{array}  \! \! \! +2 \begin{array}{lll}\vspace{-0.3cm}  \includegraphics[scale = 0.75]{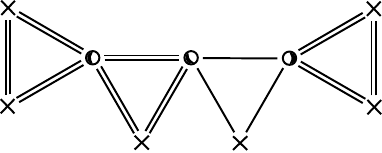} \end{array} + 2 \begin{array}{lll}\vspace{-0.3cm}  \includegraphics[scale = 0.75]{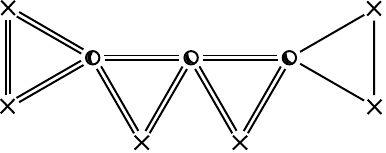} \end{array}
$$
and
$$
\frac13\left\{\begin{array}{lll} \includegraphics[scale = 0.75]{chi3.pdf} \end{array},\begin{array}{lll}\vspace{-0.3cm}  \includegraphics[scale = 0.75]{chi3chi3L3_1.pdf} \end{array} \right\} = \begin{array}{lll}  \includegraphics[scale = 0.75]{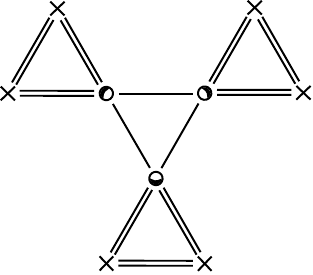} \end{array} + 4 \begin{array}{lll}\vspace{-0.3cm}  \includegraphics[scale = 0.75]{chi3chi3chi3L3_1.pdf} \end{array}.
$$

\noindent \emph{Step 4.4: Computation of $\{\chi_3,\{\chi_3,a_4 \, \mathcal{L}_4 \}\}$.} To follow and check more easily the formal computation of $\{\chi_3,\{\chi_3,a_4 \, \mathcal{L}_4 \}\}$, we use the same technique as before. Here, to highlight its symmetries, $\mathcal{L}_4$ is represented by a tetrahedron.

First, we have
$$
\{\chi_3,a_4 \, \mathcal{L}_4 \}\} = 12\, a_3\, a_4 \sum_{\substack{k_1+\dots+k_5=0\\ k_1+k_2\neq 0 }} c_{k_1,k_2} u^k 
$$
which is highlighted through the following graphical computation
$$
 \frac{\{\chi_3,\mathcal{L}_4\}}{a_3} = \left\{\begin{array}{lll} \includegraphics[scale = 0.75]{chi3.pdf} \end{array},\begin{array}{lll} \includegraphics[scale = 0.75]{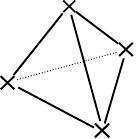} \end{array}\right\}   = 12\begin{array}{lll} \includegraphics[scale = 0.75]{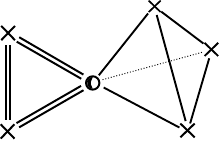} \end{array}.
$$
Then, we get
\begin{equation*}
\begin{split}
\{\chi_3,\{\chi_3,a_4 \, \mathcal{L}_4 \}\} &=108 \, a_3^2 \, a_4 \! \! \! \! \sum_{\substack{k_1+\dots+k_6=0\\k_1+k_2\neq 0\\k_5+k_6\neq 0}} \! \! \! \!  c_{k_1,k_2}c_{k_5,k_6}\,  u^k+72 \, a_3^2 \, a_4 \! \! \! \!  \sum_{\substack{k_1+\dots+k_6=0\\k_1+k_2+k_3\neq 0\\k_1+k_2\neq 0}} \! \! \! \! c_{k_1,k_2} c_{k_1+k_2,k_3} \,  u^k 
\end{split}
\end{equation*}
which is highlighted through the following graphical computation
$$
 \frac{\{\chi_3,\mathcal{L}_4\} }{12\, a_3^2} =  \left\{\begin{array}{lll} \includegraphics[scale = 0.75]{chi3.pdf} \end{array},\begin{array}{lll} \includegraphics[scale = 0.75]{chi3L4.pdf} \end{array}\right\}  = 9 \begin{array}{lll} \includegraphics[scale = 0.75]{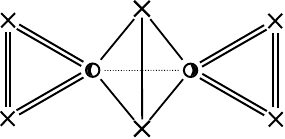} \end{array} + 6 \begin{array}{lll} \includegraphics[scale = 0.75]{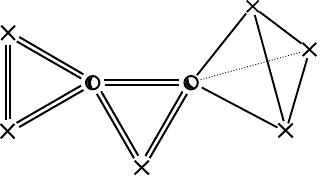} \end{array}.
$$

\begin{algorithm}
\caption{Projection of $\{\chi_3,\{\chi_3,\{\chi_3,a_3 \, \mathcal{L}_3\}\}\}$ for \ref{gBO} with Maple 2019}
\label{algo1}
\begin{lstlisting}[language=Maple,mathescape=true]  
with(combinat):
$\mathrm{E}:=$permute([p, p, -p, -p, q, -q]):
assume(0 < p, 0 < q, 0 < q - 2$\cdot$p):
$\alpha$:=1:
$c:=$($\ell_1$,$\ell_2$) $\rightarrow$ 2$\cdot$$(2\cdot\pi)^{-\alpha}\displaystyle \frac{\ell_1+\ell_2}{\ell_1|\ell_1|^{\alpha}+\ell_2|\ell_2|^{\alpha}-(\ell_1+\ell_2)|\ell_1+\ell_2|^{\alpha} }$:
$b:=0:$
for K in $\mathrm{E}$ do 
  $k_1$:=K[1]; $k_2$:=K[2];  $k_3$:=K[3];  $k_4$:=K[4];  $k_5$:=K[5];  $k_6$:=K[6]; 
  if $k_1+k_2\neq0$ and $k_5+k_6\neq0$ then
    $b:=b+324\cdot a_3^4\cdot c(k_1, k_2)\cdot c(k_1 + k_2, k_3)\cdot (c(k_1 + k_2 + k_3, k_4) + 3\cdot c(k_5, k_6))$;
    if $k_3+k_4\neq0$  then
      $b:=b+162\cdot a_3^4\cdot c(k_1,k_2)\cdot c(k_3,k_4)\cdot(c(k_5,k_6)+c(k_1+k_2,k_3+k_4))$;             
    end if;    
  end if;
end do:
$b:=$simplify($b$):
\end{lstlisting}
\end{algorithm} 

\noindent \emph{Step 5: Specialization.} At the previous steps, we have computed explicitly the terms of the expansion \eqref{givemeZ6} of $P_6$. Now, in order to determine theirs terms associated with the monomials $I_p^2 I_q$, $0<2p<q$ (i.e. $c_{p,p}(q)$, see \eqref{Z66}), we use a formal computation software (here Maple 2019). Just below, in Algorithm \ref{algo1}, we exhibit the Maple source code we have implemented to compute the projection\footnote{i.e. the coefficients $c_{p,p}(q)$ associated with the monomials $I_p^2 I_q$, $0<2p<q$ (see \eqref{Z66}).} of $\{\chi_3,\{\chi_3,\{\chi_3,a_3 \, \mathcal{L}_3\}\}\}$ (whose explicit formula is given in \eqref{chi33L3}). It is straightforward to modify this source code to compute the projections of the other terms, consequently we do not detail the other scripts we have implemented.

\end{proof}

\section{Control of the small divisors}
\label{sec:sd}

To deal with the small divisors of the rational normal form process, we have to introduce some relevant quantities.

In order to take into account the multiplicity of the multi-indices, in this section we do not consider multi-indices in $\Irr \cap \M$ but in $\mathcal{MI}_{\mathrm{mult}}$ (defined in \eqref{def:MImult}). However, using the correspondance \eqref{rel:l->(m,k)} all the objects we define also make sense if they are indexed by elements of $\Irr \cap \M$. See Remark \ref{Il faut faire attention} for details.

Using the explicit formulas given by the Theorem \ref{thm-BNF}, we define the small divisors associated with $Z_4$.
\begin{definition}[Small divisors associated with $Z_4$]
\label{def:Delta4} If $(m,k) \in \mathcal{MI}_{\mathrm{mult}}$ we set
$$
\Delta^{(4),\mathcal{E}}_{m,k}(I) := \sum_{j=1}^{\#k} m_j \, k_j \, \partial_{I_{k_j}} Z_4^{\mathcal{E}}(I) = \sum_{p = 1}^{\infty} (\delta^{\mathcal{E}}_{m,k})_{p} I_{p}
$$
where
\begin{equation}\label{deltaKdV}
(\delta^{\mathrm{\ref{gKdV}}}_{m,k})_{p} = - \sum_{j=1}^{\#k} m_j \, k_j \,  \left(12\, a_4+ \frac{3\, a_3^2}{ \pi^2 \, k_j^2}\right) \mathbb{1}_{k_j=p},
\end{equation}
\begin{equation}\label{deltaBO}
(\delta^{\mathrm{\ref{gBO}}}_{m,k})_{p} = -12\, a_4 \, \sum_{j=1}^{\#k} m_j\, k_j\,  \mathbb{1}_{k_j=p} - \frac{18\, a_3^2}{\pi} \left( \sum_{k_j\geq p} m_j + \frac1{p}\sum_{k_j< p} m_j\, k_j \right).
\end{equation}
\end{definition}

\begin{remark}
\label{rem:useful} Note that, since $(k,m)$ satisfies the zero momentum condition (i.e. $k\cdot m = 0$), the expansion associated with $\Delta^{(4),\mathcal{E}}_{m,k}$ is finite :
\begin{equation}
\label{ca servira dans les fractions rationnelles}
\forall p> k_1, \ (\delta^{\mathcal{E}}_{m,k})_{p} = 0.
\end{equation}

\end{remark}

\begin{definition}[Smallest effective index]
\label{def:kappa}
 We denote $\kappa_{m,k}^{\mathcal{E}}$ the smallest index $p$ such that $\Delta^{(4),\mathcal{E}}_{m,k}$ really depends on $I_{p}$ :
$$
\kappa_{m,k}^{\mathcal{E}} := \inf \{ p\in \mathbb{N}^{*} \ | \ (\delta^{\mathcal{E}}_{m,k})_{p} \neq 0 \} \in \mathbb{N}^{*}.
$$
\end{definition}

\begin{remark} 
\label{rem:pasbete}
It is clear that this infimum is a minimum. In other word, we have $\kappa_{m,k}^{\mathcal{E}}< +\infty$. Indeed, this is clear for \ref{gKdV} and, since $(k,m)$ satisfies the zero momentum condition (i.e. $k\cdot m = 0$) and since for \ref{gBO}, $a_4\neq 0$, we have
\begin{equation}
\label{eq:bienvu}
(\delta^{\mathrm{\ref{gBO}}}_{m,k})_{k_1} = -12 \, a_4 \, m_1\, k_1 \neq 0.
\end{equation}
\end{remark}

In the following lemma, proven in Appendix \ref{app:proof:diot}, we establish a better upper bound on $\kappa_{k,m}^{\mathcal{E}}$.
\begin{lemma} \label{lem:diot}
We have $\kappa_{m,k}^{\mathrm{\ref{gKdV}}}=k_{\last}$
and if $a_3= 0 \ \mathrm{or} \ m_1+\dots+m_{\last} = 0$ then we have $\kappa_{m,k}^{\mathrm{\ref{gBO}}} = k_{\last}$ else we have $\kappa_{m,k}^{\mathrm{\ref{gBO}}}\leq 2 \#k-1$.
\end{lemma}

Now we are focusing on the small divisors associated with $Z_{6,\leq N}^{\mathcal{E}}$. Following the notations of Theorem \ref{thm-BNF} we write for all $N>0$
\begin{equation}
\label{def:coeff_Z_6}
 Z_{6,\leq N}^{\mathcal{E}}(I) = \sum_{0<p\leq q \leq \ell\leq N} c_{p,q}^{\mathcal{E}}(\ell) I_p I_q I_\ell.
\end{equation}

For $(m,k) \in \mathcal{MI}_{\mathrm{mult}}$, we introduce the small divisors associated with $Z_{6,\leq N}^{\mathcal{E}}$ 
\begin{equation}
\label{def_Delta6_N}
 \Delta_{m,k,N}^{(6),\mathcal{E}}(I) = \sum_{j=1}^{\# k} m_j k_j \partial_{I_{k_j}} Z_{6,\leq N}^{\mathcal{E}}(I).
\end{equation}
which are  homogeneous polynomials of degree 2.
We are also introducing a second type of small divisors of degree 2 that are not homogeneous 
\begin{equation}
\label{def_Delta46}
\Delta_{m,k,N}^{(4,6),\mathcal{E}}(I) = \Delta_{m,k}^{(4),\mathcal{E}} (I) + \Delta_{m,k,N}^{(6),\mathcal{E}}(I) .
\end{equation}

In this section, we aim at studying the following open subsets of $\dot{H}^{s}$.
\begin{definition}[Open subsets]\label{U} 
$ \mathcal{U}_{\gamma,N,r}^{\mathcal{E},s} = \mathcal{U}_{\gamma,N,r}^{(4),\mathcal{E},s} \cap \mathcal{U}_{\gamma,N,r}^{(4,6),\mathcal{E},s} $
where
$$
\mathcal{U}_{\gamma,N,r}^{(4),\mathcal{E},s}=\! \! \! \bigcap_{\substack{(m,k)\in \mathcal{MI}_{\mathrm{mult}}  \\ 5\leq|m|_1\leq r \\ k_1\leq N}} \! \! \!  \left\{ u \in \dot{H}^s \ | \ |\Delta^{(4), \mathcal{E}}_{m,k}(I) | \ {>}\ \gamma N^{-5  |m|_1} \| u \|_{\dot{H}^s}^2 (\kappa_{m,k}^{\mathcal{E}})^{-2s} \right\},
$$
$$
\mathcal{U}_{\gamma,N,r}^{(4,6),\mathcal{E},s}=\! \! \! \bigcap_{\substack{(m,k)\in \mathcal{MI}_{\mathrm{mult}}  \\ 5\leq|m|_1\leq r \\ k_1\leq N}} \! \! \!  \left\{ u \in \dot{H}^s \ | \ |\Delta^{(4,6), \mathcal{E}}_{m,k,N}(I) | \ {>}\ \gamma N^{-21  |m|_1} \| u \|_{\dot{H}^s}^2 \max((\kappa_{m,k}^{\mathcal{E}})^{-2s},\gamma  \| u \|_{\dot{H}^s}^2) \right\}.
$$
\end{definition}

In the first subsection, we prove that $\mathcal{U}_{\gamma,N,r}^{\mathcal{E},s} $ are stable by a small relative perturbation of the actions in the $\dot{H}^{s-1}$ topology  while in the second subsection we estimate the probability to draw a function in $\mathcal{U}_{\gamma,N,r}^{\mathcal{E},s}$.

\subsection{Stability by perturbations}\label{sec:stab} In this subsection, we aim at proving the following proposition (its proof is done in the subsection \ref{proof:stab_U}).
\begin{proposition}
\label{prop:stab_U}
 Let $\mathcal{E} \in \{\mathrm{\ref{gBO},\ref{gKdV}}\}$, $\lambda\in(0,1)$, $s\geq 1$, ${ 1\geq \gamma >0}$, $r> 1$, $N\gtrsim_{r,(1-\lambda)^{-1}} 1$. For all $u,u'\in \dot{H}^s$, if 
\begin{equation}
\label{assump:prop:stab_U}
\|u' \|_{\dot{H}^s} \leq 2 \| u \|_{\dot{H}^s}  \ \  \mathrm{and}\ \ \forall \ell \in \mathbb{N}^*, \ |I_\ell - I_\ell'| \ell^{2s-2} \leq \gamma^2 (1-\lambda) N^{-22  r} \| u \|_{\dot{H}^s}^2
\end{equation}
then
$$
u \in \mathcal{U}_{\gamma,N,r}^{\mathcal{E},s} \ \Rightarrow \ u' \in \mathcal{U}_{\lambda \gamma,N,r}^{\mathcal{E},s}.
$$
\end{proposition}
We recall that the coefficients $c_{p,q}^{\mathcal{E}}(\ell)$ of $Z_{6,\leq N}^{\mathcal{E}}(I)$ (see \eqref{def:coeff_Z_6}) satisfy (see Theorem \ref{thm-BNF})
\begin{equation}
\label{est:coeff_Z_6}
|c_{p,q}^{\mathcal{E}}(\ell)| \lesssim \ell.
\end{equation}

For $(m,k) \in \mathcal{MI}_{\mathrm{mult}}$, we introduce the polynomials
\begin{equation}
\label{def_Delta6}
\Delta_{m,k}^{(6),\mathcal{E}}(I) = \sum_{j=1}^{\# k} m_j k_j \sum_{0<p\leq q < k_\last}  c_{p,q}^{\mathcal{E}}(k_j) I_p I_q.
\end{equation}
Note that, roughly speaking, if $k_1\leq N$ then $\Delta_{m,k}^{(6),\mathcal{E}}$ is just the main part of the natural expansion of $\Delta_{m,k,N}^{(6),\mathcal{E}}$ (defined by \eqref{def_Delta6_N}).

\subsubsection{Some preliminary Lemma} First, we introduce some elementary preliminary lemma about the size and the variations of the small denominators.
\begin{lemma}
\label{lem:bound:Delta4}
 For all $u\in \dot{H}^{s}$, all $(m,k)\in \mathcal{MI}_{\mathrm{mult}}$, we have
$$
|\Delta^{(4),\mathcal{E}}_{m,k}(I) | \lesssim_{m} k_1^4 (\kappa_{m,k}^{\mathcal{E}})^{-2s}   \left( \max_{p>0} p^{2s-2} I_p \right).
$$
\end{lemma}
\begin{proof}
In view of the formula giving explicitly $\delta^{\mathcal{E}}_{m,k}$ in Definition \ref{def:Delta4}, it is clear that
$(\delta^{\mathcal{E}}_{m,k})_{p} \lesssim_m p.$
Furthermore, we know that if $p>k_1$ then $(\delta^{\mathcal{E}}_{m,k})_{p}=0$ (see Remark \ref{rem:useful}). Finally, by definition of $\kappa_{m,k}^{\mathcal{E}}$ and $\Delta^{(4),\mathcal{E}}_{m,k}$, we have
$$
|\Delta^{(4),\mathcal{E}}_{m,k}(I)| =| \sum_{p = \kappa_{m,k}^{\mathcal{E}}}^{k_1} (\delta^{\mathcal{E}}_{m,k})_{p} I_{p} |\lesssim_{m} \sum_{p = \kappa_{m,k}^{\mathcal{E}}}^{k_1} p  I_{p} \lesssim_{m} k_1^4 (\kappa_{m,k}^{\mathcal{E}})^{-2s}   \left( \max_{p>0} p^{2s-2} I_p \right).
$$
\end{proof}

\begin{lemma}
\label{lem:Delta-DeltaN}
 For all $u\in \dot{H}^{s}$, all $(m,k)\in \mathcal{MI}_{\mathrm{mult}}$, all $N\gtrsim_{|m|_1,s} 2$ such that 
$k_1 \leq N$ we have
\begin{equation}
\label{res:lem:Delta-DeltaN}
|\Delta^{(6),\mathcal{E}}_{m,k}(I) - \Delta^{(6),\mathcal{E}}_{m,k,N}(I) | \leq N^5 (\kappa_{m,k}^{\mathcal{E}})^{-2s}  \| u \|_{\dot{H}^s}^{4}.
\end{equation}
\end{lemma}
\begin{proof} First note that $ \Delta^{(6),\mathcal{E}}_{m,k,N}(I) - \Delta^{(6),\mathcal{E}}_{m,k}(I) $ can be decomposed as
\begin{equation}
\label{expli:dec}
\sum_{j=1}^{\# k} m_j k_j \left( \sum_{0<k_j\leq q \leq \ell\leq N}  c_{k_j,q}^{\mathcal{E}}(\ell) I_q I_{\ell}  + \sum_{0<p\leq k_j \leq \ell\leq N}  c_{p,k_j}^{\mathcal{E}}(\ell) I_p I_{\ell} + \sum_{\substack{0<p\leq q \leq k_j\\ q\geq k_\last}}  c_{p,k_j}^{\mathcal{E}}(k_j) I_p I_{q} \right).
\end{equation}
As a consequence, since $k_1\leq N$ and $|c_{p,q}^{\mathcal{E}}(\ell)|\lesssim \ell$ (see \eqref{est:coeff_Z_6}), we have
$$
| \Delta^{(6),\mathcal{E}}_{m,k,N}(I) - \Delta^{(6),\mathcal{E}}_{m,k}(I)  | \lesssim_{m} N^4 k_{\last}^{-2s} \| u \|_{\dot{H}^s}^{4}.
$$
Furthermore, by using Lemma \ref{lem:diot}, we know that $\kappa_{m,k}^{\mathcal{E}} \lesssim_{m} k_{\last}$. Consequently, if $N$ is large enough with respect to $|m|_1$ and $s$ we get \eqref{res:lem:Delta-DeltaN}. 
\end{proof}

The following lemma is a straightforward corollary of the proof of Lemma \ref{lem:Delta-DeltaN}.
\begin{lemma}
\label{lem:bound:Delta6}
 For all $u\in \dot{H}^{s}$, all $(m,k)\in \mathcal{MI}_{\mathrm{mult}}$, all $N\geq 2$ such that 
$k_1 \leq N$, we have
$$
| \Delta^{(6),\mathcal{E}}_{m,k,N}(I) | \lesssim_{m} N^4  \| u \|_{\dot{H}^s}^{4}.
$$
\end{lemma}

As a corollary of the explicit decomposition \eqref{expli:dec}, we also control the variations of $\Delta^{(6),\mathcal{E}}_{m,k,N}$.
\begin{lemma}
\label{lem:var:Delta6N}
 For all $u,u'\in \dot{H}^{s}$, all $N\geq 2$, all $(m,k)\in \mathcal{MI}_{\mathrm{mult}}$, such that $k_1 \leq N$ we have
$$
|\Delta^{(6),\mathcal{E}}_{m,k,N}(I) -\Delta^{(6),\mathcal{E}}_{m,k,N}(I')| \lesssim_{m} N^4 \,  \|I-I'\|_{L^\infty} (\|I\|_{L^\infty}+\|I'\|_{L^\infty}).
$$
\end{lemma}

\subsubsection{Proof of Proposition \ref{prop:stab_U}}
\label{proof:stab_U}

Let $\lambda\in(0,1)$ and $u \in \mathcal{U}_{\gamma,N,r}^{\mathcal{E},s}$, $(m,k)\in \mathcal{MI}_{\mathrm{mult}}$ be such that $|m|_1\leq r$ and $k_1\leq N$. We consider $u'\in \dot{H}^s$ such that $\|u'\|_{\dot{H}^s} \leq 2\, \|u\|_{\dot{H}^s}$ and we aim at establishing an uniform upper bound on $p^{2s-2}|I_p - I_p'|$ to have $u'\in \mathcal{U}_{\lambda \gamma,N,r}^{\mathcal{E},s}$.

First, we focus on an upper bound to ensure that $u'\in \mathcal{U}_{\lambda \gamma,N,r}^{(4),\mathcal{E},s}$. Since $\Delta^{(4),\mathcal{E}}_{m,k}$ is a linear function of the actions, applying Lemma \ref{lem:bound:Delta4}, we get
\begin{equation*}
\begin{split}
|\Delta^{(4),\mathcal{E}}_{m,k}(I')| &\geq |\Delta^{(4),\mathcal{E}}_{m,k}(I)| -  |\Delta^{(4),\mathcal{E}}_{m,k}(I-I')|\\
 &\geq \gamma N^{-5  |m|_1} \| u \|_{\dot{H}^s}^2 (\kappa_{m,k}^{\mathcal{E}})^{-2s} - C_r N^4 (\kappa_{m,k}^{\mathcal{E}})^{-2s}   \left( \max_{p>0} p^{2s-2} |I_p - I_p'|\right)
 \end{split}
\end{equation*}
where $C_r>0$ is a constant depending only on $r$. Consequently, if 
\begin{equation}
\label{cfl0}
\forall p \in \mathbb{N}^*, \ p^{2s-2} |I_p - I_p'| \leq \gamma (1-\lambda) \, C_r^{-1} N^{-5  r -4} \| u \|_{\dot{H}^s}^2
\end{equation}
then we have $|\Delta^{(4),\mathcal{E}}_{m,k}(I')| \geq \gamma \lambda N^{-5  |m|_1} \| u \|_{\dot{H}^s}^2 (\kappa_{m,k}^{\mathcal{E}})^{-2s}$, i.e. $u'\in \mathcal{U}_{\lambda \gamma,N,r}^{(4),\mathcal{E},s}$.

Now, we aim at establishing some an upper bounds to ensure that $u'\in \mathcal{U}_{\lambda \gamma,N,r}^{(4,6),\mathcal{E},s}$. Recalling that $\|u'\|_{\dot{H}^s} \leq 2\, \|u\|_{\dot{H}^s}$ and applying the triangle inequality, Lemma \ref{lem:bound:Delta4}, Lemma \ref{lem:var:Delta6N}, we get a constant $K_r$ depending only on $r$ such that 
\begin{equation*}
\begin{split}
|\Delta^{(4,6),\mathcal{E}}_{m,k,N}(I')| \geq& \ |\Delta^{(4,6),\mathcal{E}}_{m,k,N}(I)| - |\Delta^{(4),\mathcal{E}}_{m,k}(I)-\Delta^{(4),\mathcal{E}}_{m,k}(I')| 
-|\Delta^{(6),\mathcal{E}}_{m,k,N}(I)-\Delta^{(6),\mathcal{E}}_{m,k,N}(I')| \\
\geq& \  |\Delta^{(4,6),\mathcal{E}}_{m,k,N}(I)| - K_r N^4 (\kappa_{m,k}^{\mathcal{E}})^{-2s}   \left( \max_{p>0} p^{2s-2} |I_p - I_p'|\right) \\&- K_r N^4 \|u\|_{\dot{H}^s}^2 \|I-I'\|_{L^\infty} 
\end{split}
\end{equation*}
Consequently, if each one of the two last terms of the last estimate are controlled by $(1-\lambda)|\Delta^{(4,6),\mathcal{E}}_{m,k,N}(I)|/2 $, we have $|\Delta^{(4,6),\mathcal{E}}_{m,k,N}(I')|\geq \lambda |\Delta^{(4,6),\mathcal{E}}_{m,k,N}(I')|$ and thus, since $u\in \mathcal{U}_{ \gamma,N,r}^{(4),\mathcal{E},s}$, we have $u'\in \mathcal{U}_{\lambda \gamma,N,r}^{(4,6),\mathcal{E},s}$. 

To ensure these two controls it is clearly enough to have
\begin{equation}
\label{cfl1}
\forall p\in \mathbb{N}^*,\   p^{2s-2} |I_p - I_p'| \leq K_r^{-1}   \frac{(1-\lambda)}2 \gamma N^{-21  r-4} \| u \|_{\dot{H}^s}^2,
\end{equation}
\begin{equation}
\label{cfl2}
\forall p\in \mathbb{N}^*,\    |I_p-I_p'| \leq  K_r^{-1}  \frac{(1-\lambda)}2  \gamma^2 N^{-21r-4}\| u \|_{\dot{H}^s}^2,
\end{equation}

Finally, we notice that the conditions \eqref{cfl0},\eqref{cfl1},\eqref{cfl2} are clearly satisfied if $N$ is large enough with respect to $r$ and $(1-\lambda)^{-1}$ and if \eqref{assump:prop:stab_U} is satisfied, i.e.
$$
\forall p\in \mathbb{N}^*,\   p^{2s-2} |I_p - I_p'| \leq \gamma^2 (1-\lambda) N^{-22  r} \| u \|_{\dot{H}^s}^2.
$$

\subsection{Probability estimates} In this subsection $(I_k)_{k\in \mathbb{N}^{*}}$ denotes a sequence of random variables called \emph{actions}. We assume that
\begin{itemize}
\item the actions are independent
\item $I_k$ is uniformly distributed in $J_k + \sigma(0,k^{-2s-\nu})$
\end{itemize}
where $\nu\in(1,9]$ is a given constant, $J_k\geq 0$ and $\sigma>0$.
In this section we take care to get uniform estimates with respect to $\nu$, $J$ and $\sigma$. Note that the assumption $\nu>1$ only ensures that almost surely we have $u\in \dot{H}^s$ where the random function $u$ is naturally defined by
\begin{equation}\label{uI}
u = \sum_{k = 1}^{\infty} 2\sqrt{I_k} \cos(2\pi k x).
\end{equation}

In this subsection, we aim at establishing the following proposition (that is proven in the subsection \ref{proof:prop:sum:proba}).

\begin{proposition}
\label{prop:sum:proba}
 For all $\gamma\in (0,1)$, $\mathcal{E} \in \{\mathrm{\ref{gKdV},\ref{gBO}}\}$, $r\geq 2$, $\lambda\in (0,1)$, $\sigma \lesssim_{r,s,\sigma,\lambda}1$, if\footnote{Here $\zeta$ denotes the Riemann zeta function.}
 \begin{equation}
 \label{assump:proba}
\|u\|_{{L^\infty\dot{H}^s}}^2 = 2\sum_{k=1}^{\infty} J_k |k|^{2s} + 2\sigma\zeta(\nu) \leq 4 \sigma \zeta(\nu)
 \end{equation}
then
$$
\mathbb{P}\left( 1\lesssim_{r,s,\nu} N \leq (\gamma \| u\|_{\dot{H}^s}^{-2})^{1/(21r+5)} \ \Rightarrow \ u\in \mathcal{U}_{\gamma,N,r}^{\mathcal{E},s}\right) \geq 1 - \lambda \gamma
$$
where $\mathcal{U}_{\gamma,N,r}^{\mathcal{E},s}$ is defined in Definition \ref{U}.
\end{proposition}

This subsection is divided in $4$ parts. First, we introduce some stochastic and diophantine preparatory lemmas. Then, we estimate the probability that $u \in \mathcal{U}_{\gamma,N,r}^{(4),\mathcal{E},s}$. The two last parts are devoted to estimate the probability that $u \in \mathcal{U}_{\gamma,N,r}^{(4,6),\mathcal{E},s}$ and to realize the proof of Proposition \ref{prop:sum:proba}.

From now on, in this section, to avoid any possible confusion, $\dotDelta^{(4),\mathcal{E}}_{m,k},\dotDelta^{(6),\mathcal{E}}_{m,k},\dotDelta^{(4,6),\mathcal{E}}_{m,k}$ denote the random variables defined by 
$$
\dotDelta^{(4),\mathcal{E}}_{m,k} = \Delta^{(4),\mathcal{E}}_{m,k}(I), \ \ \dotDelta^{(6),\mathcal{E}}_{m,k} = \Delta^{(6),\mathcal{E}}_{m,k}(I) \ \ \mathrm{and} \ \ \dotDelta^{(4,6),\mathcal{E}}_{m,k} = \dotDelta^{(4),\mathcal{E}}_{m,k} + \dotDelta^{(6),\mathcal{E}}_{m,k}.
$$

\subsubsection{Some preparatory lemmas} First, we recall some elementary lemmas we also introduced in \cite{KillBill}.
\begin{definition}
If a random variable $X$ has a density with respect to the Lebesgue measure, we denote $f_X$ its density, i.e.
\[  \forall g\in C^0_b(\mathbb{R}), \quad \mathbb{E}\left[ g(X) \right] = \int_{\mathbb{R}} g(x)\, f_X(x)\, \mathrm{d}x. \]
\end{definition}

\begin{lemma}
\label{lem:scale_density}
If a random variable $X$ has a density and $\varepsilon >0$ then $\varepsilon X$ has a density given by $f_{\varepsilon X} = \varepsilon^{-1}f_X(\cdot/\varepsilon)$.
\end{lemma}

\begin{lemma} 
\label{whaou}
Let $X,Y$ be some real independent random variables. If $X$ has a density, then for all $\gamma>0$
\[ \mathbb{P}(|X+Y|<\gamma)\leq 2\, \gamma\, \|f_X\|_{L^{\infty}}. \]
\end{lemma}
\begin{proof} By Tonelli theorem, we have
$$  \mathbb{P}(|X+Y|<\gamma) = \mathbb{E}\left[ \mathbb{1}_{|X+Y|<\gamma} \right]=\mathbb{E}\left[ \int_{Y-\gamma}^{Y+\gamma} f_X(x) \,\mathrm{d}x \right] \leq 2\, \gamma\, \|f_X\|_{L^{\infty}}.  
$$
\end{proof}
\begin{lemma}
\label{lem:proba:poly} Let $X$ be a random variable uniformly distributed in $(0,1)$. If $a,b,c \in \mathbb{R}$ are some real coefficients such that $a\neq 0$, then we have
$$
\forall \gamma>0,\, \mathbb{P}(|a X^2+bX+c|<\gamma)\leq 4 \sqrt{\frac{\gamma}{|a|}}
$$
\end{lemma}
\begin{remark} As a direct corollary, note that this result also holds if $b,c$ are some random variables independent of $X$.
\end{remark}
\begin{proof}[Proof of Lemma \ref{lem:proba:poly}] Without loss of generality, we assume that $a>0$. Acting by translation and dilatation, we have
$$
\mathbb{P}(|a X^2+bX+c|<\gamma) = \int_0^1 \mathbb{1}_{|a x^2+bx+c|<\gamma} \mathrm{d}x \leq \int_{-\infty}^{\infty} \mathbb{1}_{|a x^2+bx+c|<\gamma} \mathrm{d}x = \sqrt{a}^{-1}  \int_{-\infty}^{\infty} \mathbb{1}_{| x^2- \widetilde{c}|<\gamma} \mathrm{d}x
$$
where $\widetilde{c}$ could be computed explicitly as a function of $(a,b,c)$. Then we consider $2$ cases.

\noindent \emph{$\bullet$ Case $| \widetilde{c}|<3\gamma$.} Here we observe that if $| x^2- \widetilde{c}|<\gamma$ then $|x| \leq 2 \sqrt{\gamma}$. Consequently, we have $\mathbb{P}(|a X^2+bX+c|<\gamma) \leq \sqrt{a}^{-1} 4 \sqrt{\gamma}$.

\noindent \emph{$\bullet$ Case $ \widetilde{c}\geq 3\gamma$.} Here we observe that if $| x^2- \widetilde{c}|<\gamma$ then $\sqrt{ \widetilde{c} - \gamma } < |x| < \sqrt{ \widetilde{c} + \gamma }$. Consequently, by the mean value inequality, we have
$$
\mathbb{P}(|a X^2+bX+c|<\gamma) \leq 2\,\sqrt{a}^{-1} (\sqrt{ \widetilde{c} + \gamma } - \sqrt{ \widetilde{c} - \gamma }) \leq \frac{2\, \gamma}{\sqrt{\widetilde{c} - \gamma}}\sqrt{a}^{-1} \leq \sqrt{2 \gamma} \sqrt{a}^{-1} .
$$
\end{proof}

The following lemma is no more about probability but diophantine approximation.
\begin{lemma}
\label{lem:continuedfractions} Let $n\geq 1$ and $P,Q\in \mathbb{Z}_n[X]$ be two polynomials, with integer coefficients, of degrees $n$ or less. If $P$ and $Q$ are not collinear and if there exists $J\in \mathbb{N}^*$  such that $J\geq 2n$ and
$$
\forall j\in \llbracket 1,J\rrbracket, \ |Q(j)| <  2^{\frac{J}{2n}-1}
$$
then, for all $\beta\in \mathbb{R}$ there exists $j_{\star}\in \llbracket 1,J\rrbracket$ such that
$$
|P(j_{\star}) - \beta\, Q(j_{\star}) | \geq \frac1{2|Q(j_{\star})|}.
$$
\end{lemma}
\begin{proof} 
In this proof, $M$ denote the upper bound on $|Q|$, i.e. $M=2^{\frac{J}{2n}-1}$. Let $\mathcal{B}$ be the set of the best rational approximations of $\beta$ by rational numbers :
$$
\mathcal{B} = \{ x\in \mathbb{Q} \ | \ \exists (p,q)\in \mathbb{Z}\times \mathbb{N}^*, \ x=\frac{p}q \ \mathrm{and} \  \ \left| \beta - x\right|< \frac1{2q^2} \}.
$$
Let $\mathcal{H}_M$ be the set of the rational numbers with denominators no larger than $M$
$$
\mathcal{H}_M = \{x\in \mathbb{Q} \ | \   \exists (p,q)\in \mathbb{Z}\times \mathbb{N}^*, \ x=\frac{p}q \ \mathrm{and} \ |q|\leq M \}.
$$
Let $\Psi$ be defined by
$$
\Psi : \left\{ \begin{array}{ccc} \llbracket 1,J \rrbracket &\to& \mathcal{H}_M \cup \{ \infty\} \\
						j &\mapsto & P(j)/Q(j)
\end{array} \right. 
$$
where, by convention if $Q(j)=0$ then $\Psi(j)=\infty$ (even if $P(j)=0$).

With these notations, to prove the lemma, we just have to prove that the image of $\Psi$ is not included in $\mathcal{B}\cap \mathcal{H}_M \cup \{ \infty\} $. We are going to proceed by a cardinality argument proving that 
\begin{equation}
\label{cf:what_we_want}
\# \mathrm{Im} \ \Psi > \# (\mathcal{B}\cap \mathcal{H}_M \cup \{ \infty\}).
\end{equation} 
On the one hand, we prove an upper bound on the cardinal of $\mathcal{B}\cap \mathcal{H}_M$. Indeed, it is known (by applying, for example, the Theorem 19 of \cite{Khinchine}) that $\mathcal{B}$ is only composed of \emph{convergents} of the number $\beta$, i.e. the rational numbers obtained truncating the \emph{continued fraction expansion} of $\beta$. As a consequence, by applying the Theorem 12 of \cite{Khinchine}, we know there exists two sequences $(p_\ell,q_\ell)\in (\mathbb{Z}\times \mathbb{N}^*)^{\mathbb{N}^*}$ such that 
$$
\forall \ell \geq 1, \ p_{\ell} \wedge q_\ell = 1, \ \ q_{\ell} \geq 2^{\frac{\ell-1}2} \ \mathrm{and} \  \mathcal{B}\subset \{ p_{\ell}/q_{\ell} \ | \ \ell\geq 1\} .
$$
As a consequence, observing that by construction $M\geq1$, we have
$$
\# (B\cap \mathcal{H}_M) < 1+2\, \log_2 M.
$$

On the other hand, since $P$ and $Q$ are not collinear and their degrees are no larger than $n$, the cardinal of each fiber of $\Psi$ is not larger than $n$. As a consequence, we have
$$
\# \ \mathrm{Im} \ \Psi \geq \frac{J}n,
$$

Consequently, since $M$ has been chosen such that
$$
2+2\, \log_2 M = \frac{J}n,
$$
the cardinality estimate \eqref{cf:what_we_want} is satisfied, which concludes the proof.
\end{proof}

\subsubsection{Genericity of the first non resonance condition} In this subsection, we estimate the probability that $\varepsilon u \in \mathcal{U}_{\gamma,N,r}^{(4),\mathcal{E},s}$. We refer the reader to the definition \ref{def:kappa} for the definition of $\kappa^{\mathcal{E}}_{m,k}$ and to the definition \ref{def:Delta4} for the definition of $\delta^{\mathcal{E}}_{m,k}$.

The following lemma provides a lower bound for the non-vanishing coefficient of $\Delta^{(4),\mathcal{E}}_{m,k}$ associated with the smallest index.
\begin{lemma} \label{lem:ppcptpt} The following lower bound holds
$$
|(\delta^{\mathcal{E}}_{m,k})_{\kappa^{\mathcal{E}}_{m,k}}| \gtrsim_{m} (\kappa^{\mathcal{E}}_{m,k})^{3-2\, \alpha_{\mathcal{E}}}.
$$
\end{lemma}
\begin{proof} We denote $n=\# m$.\\
 \emph{$\bullet$ Case $\mathcal{E} = \mathrm{\ref{gKdV}}$.} We observe that there exists $\eta>0$ such that
$$
\forall \ell \in \mathbb{N}^{*}, \ 12\, a_4 \, \ell^2 + \frac{3 \, a_3^2}{\pi^2}\neq 0 \ \Rightarrow \ \left| 12\, a_4 \, \ell^2 + \frac{3 \, a_3^2}{\pi^2} \right| \geq\eta.
$$
Consequently, in view of \eqref{deltaKdV}, we have $|(\delta^{\mathrm{\ref{gKdV}}}_{m,k})_{\kappa^{\mathrm{\ref{gKdV}}}_{m,k}}| \geq \eta \,(\kappa^{\mathrm{\ref{gKdV}}}_{m,k})^{-1}.$

\noindent \emph{$\bullet$ Case $\mathcal{E} = \mathrm{\ref{gBO}}$.} If $a_3= 0$ or $m_1+\dots+m_n = 0$ then $\kappa^{\mathrm{\ref{gBO}}}_{m,k} =k_n$ and $(\delta^{\mathrm{\ref{gBO}}}_{m,k})_{k_n} =  -12\, a_4 \, m_n\, k_n.$ Consequently, the lower bound is clear. On the other hand, if $a_3\neq 0$ and $m_1+\dots+m_n \neq 0$, we know by Lemma \ref{lem:diot} that $\kappa^{\mathrm{\ref{gBO}}}_{m,k}\leq 2\, n +2$. So we conclude this proof observing that $n$ and $m$ being fixed, $k\mapsto (\delta^{\mathrm{\ref{gBO}}}_{m,k})_{1\leq \ell \leq 2 n +2}$ can only take a finite number of values.
\end{proof}

As a corollary, we get the following probability estimate.
\begin{lemma} 
\label{lem:proba:Z4}
If $\mathcal{E}\in \{\mathrm{\ref{gBO}},\mathrm{\ref{gKdV}} \}$ and $(k,m)\in \mathcal{MI}_{\mathrm{mult}} $  we have
$$
\forall \gamma>0, \ \mathbb{P}\big(|\dotDelta^{(4),\mathcal{E}}_{m,k}  |<\gamma\big)\lesssim_{m} \gamma \sigma^{-1} \,  \big(\kappa_{k,m}^{\mathcal{E}}\big)^{2s+\nu -3 + 2 \, \alpha_{\mathcal{E}}}.
$$
\end{lemma}
\begin{proof} Applying Lemma \ref{lem:scale_density} and Lemma \ref{whaou}, we have
$$
\mathbb{P}\big(|\dotDelta^{(4),\mathcal{E}}_{m,k}  |<\gamma\big) \leq 2 \inf_{\ell \in \mathbb{N}^*} | (\delta^{\mathcal{E}}_{m,k})_{\ell}^{-1} \| f_{I_{\ell}}\|_{L^{\infty}} | = 2 \inf_{\ell \in \mathbb{N}^*} | \sigma^{-1} (\delta^{\mathcal{E}}_{m,k})_\ell^{-1} \ell^{2s + \nu}| 
$$
A fortiori, for $\ell = \kappa^{\mathcal{E}}_{m,k}$, applying Lemma \ref{lem:ppcptpt}, we get the expected result.
\end{proof}

\begin{proposition}
\label{prop:proba:Z4}
For all $\mathcal{E}\in \{\mathrm{\ref{gBO}},\mathrm{\ref{gKdV}} \}$, we have
$$
\forall \gamma>0, \ \mathbb{P}\big( \forall (m,k)\in \mathcal{MI}_{\mathrm{mult}} , \ |\dotDelta^{(4),\mathcal{E}}_{m,k}  |\gtrsim_{m} \gamma \sigma \, k_1^{-4|m|_1} \, \big(\kappa_{k,m}^{\mathcal{E}}\big)^{-2s} \big) \geq 1- \gamma.
$$
\end{proposition}
\begin{proof} We aim at bounding the probability of the complementary event by $\gamma>0$. By sub-additivity of $\mathbb{P}$, we have
\begin{multline}
\mathbb{P}\big( \exists (m,k)\in \mathcal{MI}_{\mathrm{mult}}, \ |\dotDelta^{(4),\mathcal{E}}_{m,k}  |< C_{m}\sigma \gamma  \, k_1^{-4|m|_1} \, \big(\kappa_{k,m}^{\mathcal{E}}\big)^{-2s} \big) \\ \leq \sum_{(m,k)\in \mathcal{MI}_{\mathrm{mult}}} \mathbb{P}\big( \ |\dotDelta^{(4),\mathcal{E}}_{m,k}  |< C_{m}\sigma \gamma  \, k_1^{-4|m|_1} \, \big(\kappa_{k,m}^{\mathcal{E}}\big)^{-2s} \big)
\end{multline}
where $C_{m}>0$ is a positive constant depending only on $m$ and that will be determined later.

Here, we denote by $K_{m}$ the constant in Lemma \ref{lem:proba:Z4}. Then, we apply Lemma \ref{lem:proba:Z4} with $\gamma$ replaced by $\gamma \ C_{m} \sigma \ k_1^{-4|m|_1} \big(\kappa_{k,m}^{\mathcal{E}}\big)^{-2s}$. As a consequence, since $\nu\leq 9$ and $|m_1|\geq 5$ we get
$$\mathbb{P}\big( \ |\dotDelta^{(4),\mathcal{E}}_{m,k}  |< C_{m} \sigma \gamma  \, k_1^{-4|m|_1} \, \big(\kappa_{k,m}^{\mathcal{E}}\big)^{-2s} \big) \leq C_m K_m k_1^{-4|m|_1} \big(\kappa_{k,m}^{\mathcal{E}}\big)^{10} \leq C_{m} K_{m} k_1^{-2\#m}$$
where we have used that $\kappa_{k,m}^{\mathcal{E}}\leq k_1$ (see \eqref{eq:bienvu}).
Finally, observing that
$$
\sum_{(m,k)\in \mathcal{MI}_{\mathrm{mult}}}  \frac{2^{-|m|_1} k_1^{-2\# m}}{\# m!} \leq \sum_{n\geq 2} \sum_{k\in (\mathbb{N}^*)^n} \sum_{m\in (\mathbb{Z}^*)^n} {\frac{2^{-|m|_1}}{n!}} (k_1\dots k_n)^{-2} \leq \sum_{n\geq 0} \frac1{n!} \left(\frac{\pi^2}{6}\right)^n 2^n = e^{\pi^2/3},
$$
and denoting $C_{m} =  \frac{2^{-|m|_1} }{\# m!} K_{m}^{-1} e^{-\pi^2/3}$ we get the expected result.

\end{proof}

\subsubsection{Genericity of the second non resonance condition} In this subsection, we aim at estimating the probability that $\dotDelta^{(4,6),\mathcal{E}}_{m,k,}$ is not too small.

We recall here that by definition $\dotDelta^{(4,6),\mathcal{E}}_{m,k}=\dotDelta^{(4),\mathcal{E}}_{m,k}+\dotDelta^{(6),\mathcal{E}}_{m,k} $ and 
$$
\dotDelta^{(6),\mathcal{E}}_{m,k} =  \sum_{j=1}^{\# m} m_j\, k_j\, \sum_{0<p\leq q<k_\last} c_{p,q}^{\mathcal{E}}(k_j) I_q I_p 
$$
where the real numbers $c_{p,q}^{\mathcal{E}}(k_j)$ are the coefficient of $Z_{6,\leq N^3}$ (see \eqref{premieredefZ6}). If $ 2p<\, k_\last$, an explicit formula for the coefficients $c_{p,p}^{\mathcal{E}}(k_j)$ is given in Theorem \ref{thm-BNF}.

\begin{lemma}
\label{lem:proba:Z6}
Let $\mathcal{E}\in \{\mathrm{\ref{gBO}},\mathrm{\ref{gKdV}} \}$, $ (m,k)\in \mathcal{MI}_{\mathrm{mult}}$ and  $n=\#m$. In the case $\mathcal{E}=\mathrm{\ref{gBO}}$ we further assume that $a_3 = 0$ or $m_1+\dots+m_n = 0$. Then we have
$$
\forall \gamma \leq 1, \ \mathbb{P}( |\dotDelta^{(4,6),\mathcal{E}}_{m,k}|< \gamma) \lesssim_{m,s} \sigma^{-1} \min\left( \gamma \, \, k_n^{2s+\nu -3 + 2 \, \alpha_{\mathcal{E}}}, \sqrt{\gamma}  \, k_1^{6n}  \right).
$$
\end{lemma}
\begin{proof} 
 Note that, by Lemma \ref{lem:diot}, we have $\kappa^{\mathcal{E}}_{m,k} = k_n$ and that, by  definition, $\dotDelta^{(6),\mathcal{E}}_{m,k}$ is independent of $I_{k_n}$. Consequently, the bound
$$
\mathbb{P}( |\dotDelta^{(4,6),\mathcal{E}}_{m,k}|< \gamma) \lesssim_m \gamma \, \sigma^{-1} \, k_n^{2s+\nu -3 + 2 \, \alpha_{\mathcal{E}}}
$$
can be obtained as in the proof of the Lemma \ref{lem:proba:Z4}. \\
 Furthermore if $k_n \leq 2J_{k_1,m}$, where $J_{k_1,m}$ is defined by
$$
J_{k_1,m} =\lfloor 42\, n \, |m|_1 (1+\log_2 k_1) \rfloor,
$$
 then $ \gamma  \, k_n^{2s+\nu -3 + 2 \, \alpha_{\mathcal{E}}}\lesssim_{m,s} \sqrt{\gamma}  \, k_1^{6n}$ and Lemma's proof is over. Thus from now on we assume $k_n > 2J_{k_1,m}$ and we want to prove that $\mathbb{P}( |\dotDelta^{(4,6),\mathcal{E}}_{m,k}|< \gamma) \lesssim_{m,s}  \sigma^{-1} \,\sqrt{\gamma}  \, k_1^{6n}$.

Now, if $p$ is an integer such that $0<2p<k_n$, $\dotDelta^{(4,6),\mathcal{E}}_{m,k}$ writes
$$
\dotDelta^{(4,6),\mathcal{E}}_{m,k} =   \, (d^{\mathcal{E}}_{k,m})_p I_p^2 +  L^{\mathcal{E}}_{k,m,c}((I_{\ell})_{\ell \neq p}) I_p+ Q^{\mathcal{E}}_{k,m,c}((I_{\ell})_{\ell \neq p})
$$
where $L^{\mathcal{E}}_{k,m,c}((I_{\ell})_{\ell \neq p})$ (resp. $Q^{\mathcal{E}}_{k,m,c,\varepsilon}((I_{\ell})_{\ell \neq p})$) is a linear form (resp. quadratic form) in the actions independent of $I_{p}$ and
$$
(d^{\mathcal{E}}_{k,m})_p = \sum_{j=1}^n m_j\, k_j\, b^{\mathcal{E}}_{p,{k_j}}.
$$
Consequently, applying Lemma \ref{lem:proba:poly}, we get
\begin{equation}
\label{temp:proba:Z6}
 \mathbb{P}(\dotDelta^{(4,6),\mathcal{E}}_{m,k} < \gamma) \lesssim \sigma^{-1} \sqrt{\frac{\gamma}{ p^{-2s - \nu} |(d^{\mathcal{E}}_{k,m})_p|}}
\end{equation}
To conclude this proof, we are going to prove that there exists $p_{\star}\in \mathbb{N}^*$ satisfying $2\,p_{\star}<k_n$ and $p_{\star}\lesssim_m J_{k_1,m}$ such that
$|(d^{\mathcal{E}}_{k,m})_{p_{\star}}|\gtrsim_m k_1^{-10n}$.

As a consequence, by \eqref{temp:proba:Z6}, we will have
$$
 \mathbb{P}(\dotDelta^{(4,6),\mathcal{E}}_{m,k,\varepsilon} < \gamma) \lesssim_{m} \sigma^{-1}  \sqrt{\gamma} k_1^{5n} J_{k_1,m}^{s+\nu/2}  \lesssim_{m,s} \sigma^{-1} \sqrt{\gamma} k_1^{6n}.
$$

To prove the existence of such a $p_{\star}$, we have to distinguish $3$ cases.

\noindent \emph{$*$ Case $\mathcal{E}=\mathrm{\ref{gBO}}$ and $a_3=0$.} Using the zero momentum condition (i.e. $k\cdot m =0$) and the exact formula of $c_{p,p}^{\mathrm{\ref{gBO}}}(q)$ given by Theorem \ref{thm-BNF}, we have

$$
(d^{\mathrm{\ref{gBO}}}_{k,m})_p  =-\frac{288\,a_4^2}{\pi} \sum_{j=1}^n m_j \, k_j\, \frac {(p-k_j)\, }{\left( p+2\,k_j \right)  \left( 3\,p-2\,k_j \right) } = -\frac{288\,a_4^2}{\pi} \frac{P_{k,m}(p)}{Q_{k,m}(p)} 
$$
where $P_{k,m},Q_{k,m}\in \mathbb{Z}[X]$ are the polynomial defined by
\begin{equation*}
\begin{split}
P_{k,m}(X) &= \sum_{j=1}^n m_j\, k_j\, (X-k_j) \prod_{\ell\neq j} (X+2\,k_\ell)(3\,X - 2\,k_\ell)\\
Q_{k,m}(X) &= \prod_{j=1}^n (X+2\,k_j)(3\,X - 2\,k_j)
\end{split}
\end{equation*}

We note that $P_{k,m}$ is of degree $2n-1$ or less and is not identically equal to zero because $P_{k,m}(-2k_n) \neq 0 $. As a consequence, there exists $p_{\star} \in \llbracket 1,2n-1 \rrbracket$ such that $P_{k,m}(p_{\star})\neq 0$ and , since $P_{k,m}\in\mathbb{Z}[X]$, we have $|P_{k,m}(p_{\star})|\geq 1$. Furthermore, since $k_n > 2J_{k_1,m}$, we deduce $2(2n-1) <k_n$ and then by a straightforward estimate we get $|Q_{k,m}(p_{\star})| \leq 6^n k_1^{2n}$. Consequently, we have $\big|(d^{\mathrm{\ref{gBO}}}_{k,m})_{p_{\star}}\big| \gtrsim_{m} k_1^{-2n}.$

\noindent \emph{$*$ Case $\mathcal{E}=\mathrm{\ref{gKdV}}$.} Using the zero momentum condition (i.e. $k\cdot m =0$) and the exact formula of $c_{p,p}^{\mathrm{\ref{gKdV}}}(q)$ given by Theorem \ref{thm-BNF}, we have
\begin{equation}
\begin{split}
(d^{\mathrm{\ref{gKdV}}}_{k,m})_p  &=-\sum_{j=1}^n m_j \, k_j \left( \,{\frac {3\, a_{3}^{4}}{{\pi}^{6} {p}^{2} (p^2-k_j^2)^2}}
-\,{\frac {24\, a_{3}^2\,a_{4}}{{\pi}^{4} (p^2-k_j^2)^2}}-\,{\frac {48\, {p}^{2}\,a_4^{2}}{{\pi}^{2} (p^2-k_j^2)^2}} \right) \\
&=-K(p)\sum_{j=1}^n \frac{m_j \, k_j}{p^2 (p^2-k_j^2)^2}  \ \mathrm{where} \ K(p) = \frac{3 \ a_3^4}{\pi^6} + \frac{24 \ a_3^2 a_4}{\pi^4} p^2 + \frac{48 \, a_4^2}{\pi^2}  p^4\\
&= -\frac{K(p)}{p^2} \frac{P_{k,m}(p)}{Q_{k,m}}
\end{split}
\end{equation}
where $P_{k,m},Q_{k,m}\in \mathbb{Z}[X]$ are the polynomials defined by
$$
Q_{k,m}(X) = \prod_{j=1}^n (X^2 - k_j^2)^2 \ \mathrm{and} \ P_{k,m}(X) = \sum_{\ell=1}^n m_\ell \, k_{\ell} \ \prod_{j\neq \ell} (X^2 - k_j^2)^2.
$$
We note that $K$ is a polynomial of degree $2$ with respect to $p^2$ and since by assumption $a_3\neq 0$ or $a_4\neq 0$, it vanishes at most twice. Furthermore, the polynomial $P_{k,m}$ is of degree $2n-2$ or less and is not identically equal to zero because $P_{k,m}(k_n)\neq 0$. As a consequence, there exists $p_{\star}\in \llbracket 1,2n+1 \rrbracket$ such that $K(p_{\star}) \neq 0$ and $P_{k,m}(p_{\star})\neq 0$. A fortiori, since $P_{k,m}\in \mathbb{Z}[X]$ we have $|P_{k,m}(p_{\star})| \geq 1$. Finally, since by assumption $2(2n+1)<k_n$, we have $|Q_{k,m}(p_{\star})| \leq k_1^{4n}$ and thus $(d^{\mathrm{\ref{gKdV}}}_{k,m})_{p_{\star}} \geq k_1^{-4n}$.

\noindent \emph{$*$ Case $\mathcal{E}=\mathrm{\ref{gBO}}$ and $m_1+\dots+m_n = 0$.} Using the zero momentum condition (i.e. $k\cdot m =0$) and the exact formula of $c_{p,p}^{\mathrm{\ref{gBO}}}(q)$ given by Theorem \ref{thm-BNF}, we have

$$
(d^{\mathrm{\ref{gBO}}}_{k,m})_p  =\frac{a_4}{\pi} \sum_{j=1}^n m_j \, k_j\, \left( {\frac {  p }{ 
 \left( p-k_j \right) k_j^{2}}}\frac{108\, a_{3}^{2}}{\pi} -\frac {(p-k_j)\, }{\left( p+2\,k_j \right)  \left( 3\,p-2\,k_j \right) } 288\, a_{4} \right) .
$$
We denote $\beta = \frac{108 \, a_3^2}{288\, \pi  \, a_4}$ and $\eta =- \frac{288 \, a_4^2}{\pi}$. Consequently, we have
$$
(d^{\mathrm{\ref{gBO}}}_{k,m})_p = \eta \frac{  P_{k,m}(p)- \beta \, Q_{k,m}(p) }{D_{k,m}(p)}
$$
where $P_{k,m},Q_{k,m},D_{k,m}\in \mathbb{Z}[X]$ are the polynomials defined by
\begin{equation*}
\begin{split}
P_{k,m}(X) &= \left(\prod_{j=1}^n k_j^2 (X-k_j)\right)  \sum_{\ell=1}^n m_{\ell}\, k_{\ell} (X-k_{\ell}) \prod_{j\neq \ell}  (X+2\, k_j)(3\,X-2\, k_j)\\
Q_{k,m}(X) &= \left(\prod_{j=1}^n (X+2\, k_j)(3\,X-2\, k_j)\right) X \sum_{\ell=1}^n m_{\ell}\, k_{\ell} \prod_{j\neq \ell} k_j^2 (X-k_j)\\
D_{k,m}(X) &= \prod_{j=1}^n k_j^2 (X-k_j) (X+2\,k_j)(3\,X-2\,k_j)
\end{split}
\end{equation*}
Note that $P_{k,m}$ and $Q_{k,m}$ are of degree $3n$ or less and are not collinear because
$$
Q_{k,m}(k_1)\neq 0 = P_{k,m}(k_1) \ \mathrm{and} \ P_{k,m}(-2\, k_1)\neq 0 = Q_{k,m}(-2\,k_1).
$$
Furthermore, by a straightforward estimate, if $p< k_n/2$, we have
$$
|Q_{k,m}(p)|<  |m|_1 \, 6^n \, k_1^{5n}.
$$

Then we observe that, by assumption $6n\leq J_{k_1,m} <k_n/2$ and $|m|_1 \, 6^n \, k_1^{5n} \leq 2^{\frac{J_{k_1,m}}{6n}-1}$. Consequently, recalling that $P_{k,m},Q_{k,m}$ are of degree $3n$ or less and are not collinear, by applying Lemma \ref{lem:continuedfractions}, we get $p_{\star}\in \llbracket 1, J_{k_1,m} \rrbracket$ such that
$$
 |P_{k,m}(p_{\star})- \beta \, Q_{k,m}(p_{\star})|\geq \frac1{2 |Q_{k,m}(p_{\star})|}.
$$
Consequently, we have $\displaystyle |(d^{\mathrm{\ref{gBO}}}_{k,m})_{p_{\star}} | \geq \frac{|\eta|}{2 |D_{k,m}(p_{\star})| |Q_{k,m}|} \gtrsim_m k_1^{-10n}.$
\end{proof}

In the following proposition, we make the estimates of Lemma \ref{lem:proba:Z6} uniform with respect to $k$ and $m$.
\begin{proposition}
\label{prop:proba:Z6}
For all $\mathcal{E}\in \{\mathrm{\ref{gBO}},\mathrm{\ref{gKdV}} \}$, if $\sigma\leq 1$ we have
$$
\forall \gamma\in(0,1), \ \mathbb{P}\big( \forall (m,k)\in \widetilde{\mathcal{MI}_{\mathrm{mult}}}, \ |\dotDelta^{(4,6),\mathcal{E}}_{m,k}  |\gtrsim_{m,s} \gamma \, \sigma\, k_1^{-20\#m} \, \max(k_\last^{-2s},\gamma \sigma)\big) \geq 1- \gamma ,
$$
where $\widetilde{\mathcal{MI}_{\mathrm{mult}}} = \mathcal{MI}_{\mathrm{mult}}$ excepted if $\mathcal{E} = \mathrm{\ref{gBO}}$ and $a_3 \neq 0$ {in which case,} to get $\widetilde{\mathcal{MI}_{\mathrm{mult}}}$, all the indices such that $m_1+\dots+m_{\last}\neq 0$ have to be removed from $\mathcal{MI}_{\mathrm{mult}}$. 
\end{proposition}
\begin{proof}
We aim at bounding the probability of the complementary event by $\gamma>0$. Denoting by $C_{m,s}\in(0,1)$ the constant in the estimate we aim at proving, by sub-additivity of $\mathbb{P}$, the probability of this complementary event is bounded by
\begin{equation}
\label{eq:tmp:Z6:glob}
\sum_{ (m,k)\in \mathcal{M}_{\mathcal{E}}} \mathbb{P}\big(  |\dotDelta^{(4,6),\mathcal{E}}_{m,k}  |< C_{m,s} \sigma \gamma \, k_1^{-20\#m} \, \max(k_\last^{-2s},\gamma\, \sigma )\big).
\end{equation}

 In order to estimate the probability in the previous sum, we want to apply Lemma \ref{lem:proba:Z6}. It can be done since $C_{m,s},\gamma,\sigma\in (0,1)$ we have $C_{m,s} \sigma \gamma \, k_1^{-20\#m} \, \max(k_{\last}^{-2s},\gamma\ \sigma) \leq 1$. As a consequence, each term of  the sum \eqref{eq:tmp:Z6:glob} is smaller than
$$
 K_{m,s} \sigma^{-1} C_{m,s} \gamma \sigma \, k_1^{-20\#m} \, \max(k_\last^{-2s},\gamma\ \sigma) k_\last^{2s+\nu -3 + 2 \, \alpha_{\mathcal{E}}}
$$
and
$$
 K_{m,s} \sigma^{-1}\sqrt{C_{m,s} \gamma \sigma \, k_1^{-20\#m} \, \max(k_\last^{-2s},\gamma\ \sigma )}  \, k_1^{6\#m} .
$$
where $K_{m,s}$ denotes the constant in Lemma \ref{lem:proba:Z6}.

As a consequence, since $\nu\leq 9$ and $\#m\geq 2$, each probability in the sum \eqref{eq:tmp:Z6:glob} is smaller than 
$$
 \gamma K_{m,s} \sqrt{C_{m,s}}  k_1^{-4\#m} \leq \gamma K_{m,s} \sqrt{C_{m,s}}  k_1^{-2\#m}.
$$
Consequently, proceeding as in the proof Proposition \ref{prop:proba:Z4}, and choosing 
$$
\sqrt{C_{m,s}} \leq \min\left(K_{m,s}^{-1}  \frac{2^{-|m|_1} }{\#m!} e^{-\pi^2/3},1\right),
$$
we get
$$
\mathbb{P}\big( \exists (m,k)\in \widetilde{\mathcal{MI}_{\mathrm{mult}}}, \ |\dotDelta^{(4,6),\mathcal{E}}_{m,k}  |< C_{m,s} \sigma\, \gamma \, k_1^{-20\#m} \, \max(k_\last^{-2s},\gamma \sigma)\big) \leq \gamma.
$$
\end{proof}

\subsubsection{Proof of Proposition \ref{prop:sum:proba}}
\label{proof:prop:sum:proba}

Applying Proposition \ref{prop:proba:Z4} and Proposition \ref{prop:proba:Z6}, with a probability larger than $1 -\lambda \gamma$, { $u$ (see \eqref{uI}) satisfies}
$$
\forall (m,k)\in \mathcal{MI}_{\mathrm{mult}}, \ |\Delta^{(4),\mathcal{E}}_{m,k}(I)  |\gtrsim_m \lambda \, \gamma\, \sigma  \, k_1^{-4|m|_1} \, \big(\kappa_{k,m}^{\mathcal{E}}\big)^{-2s}
$$
and, recalling that $\widetilde{\mathcal{MI}_{\mathrm{mult}}}$ is defined in Proposition \ref{prop:proba:Z6},
$$
\forall (m,k)\in \widetilde{\mathcal{MI}_{\mathrm{mult}}}, \ |\Delta^{(4),\mathcal{E}}_{m,k}(I) + \Delta^{(6),\mathcal{E}}_{m,k}(I)  |\gtrsim_{m,s} \lambda \, \gamma \, \sigma\, k_1^{-20\#m} \, \max(k_\last^{-2s},\gamma\ \sigma).
$$

From now on we assume that $u$ satisfies these $2$ last estimates.

If we consider only the case $k_1 \leq N$ and $|m|_1\leq r$ and if $N$ is large enough with respect to $r,s,(1-\nu)^{-1},\lambda^{-1}$ then we have the estimates
\begin{equation}
\label{proof:pro4:eq1}
|\Delta^{(4),\mathcal{E}}_{m,k}(I)  | \geq 8\zeta(\nu) \gamma \,\sigma  \, N^{-5|m|_1} \, \big(\kappa_{k,m}^{\mathcal{E}}\big)^{-2s}
\end{equation}
and
\begin{equation}
\label{proof:pro4:eq2}
 |\Delta^{(4),\mathcal{E}}_{m,k}(I) + \Delta^{(6),\mathcal{E}}_{m,k}(I)  |\geq 8\zeta(\nu) \gamma \, \sigma\, N^{-21|m|_1} \, \max(k_\last^{-2s},8\zeta(\nu) \gamma\ \sigma).
\end{equation}
From \eqref{proof:pro4:eq1}, since by \eqref{assump:proba} we have $\|u\|_{\dot{H}^s}^2\leq 4\zeta(\nu)\sigma$, we deduce directly that $u\in \mathcal{U}_{\gamma,N,r}^{(4),\mathcal{E},s}$.

Now, we aim at proving that, if $\sigma,\nu,N,\gamma$ satisfy some estimates then $u \in \mathcal{U}_{\gamma,N,r}^{(4,6),\mathcal{E},s}$. 

\noindent \emph{$*$ Case $(m,k)\in \widetilde{\mathcal{MI}_{\mathrm{mult}}}$.} We deduce from \eqref{proof:pro4:eq2} that
$$
 |\Delta^{(4,6),\mathcal{E}}_{m,k,N}( I)  |\geq 8\zeta(\nu) \gamma \, \sigma N^{-21|m|_1} \, \max(k_\last^{-2s}, 8\zeta(\nu) \gamma \, \sigma) -  |\Delta^{(6),\mathcal{E}}_{m,k}(I) - \Delta^{(6),\mathcal{E}}_{m,k,N}(I) |.
$$
As a consequence, recalling that by \eqref{assump:proba} we have $\|u\|_{\dot{H}^s}^2\leq 4\zeta(\nu)\sigma$ and estimating this last term by Lemma \ref{lem:Delta-DeltaN} (here $\kappa_{m,k}^{\mathcal{E}}=k_{\last}$, see Lemma \ref{lem:diot}), if $N$ is large enough with respect to $r$ and $s$, we have
$$
 |\Delta^{(4,6),\mathcal{E}}_{m,k,N}( I)  |\geq 2 \gamma \, \|u\|_{\dot{H}^s}^2 N^{-21|m|_1} \, \max(k_\last^{-2s},2\,\gamma\  \|u\|_{\dot{H}^s}^2) -  N^5 (\kappa_{m,k}^{\mathcal{E}})^{-2s} \| u\|_{\dot{H}^s}^4.
$$
Consequently, if $N^5 \| u\|_{\dot{H}^s}^2 \leq \gamma N^{-21r}$ then 
$$|\Delta^{(4,6),\mathcal{E}}_{m,k,N}(I)  |\geq   \gamma \, \|u\|_{\dot{H}^s}^2 N^{-21|m|_1} \, \max((\kappa_{m,k}^{\mathcal{E}})^{-2s},\gamma\  \|u\|_{\dot{H}^s}^2).$$

\noindent \emph{ $*$ Case $(m,k)\in \mathcal{MI}_{\mathrm{mult}}\setminus \widetilde{\mathcal{MI}_{\mathrm{mult}}}$.} Here by construction of $\widetilde{\mathcal{MI}_{\mathrm{mult}}}$, we have $\mathcal{E}=\mathrm{\ref{gBO}}$, $a_3\neq0$ and $m_1+\dots+m_{\last}\neq 0$. Consequently, by applying Lemma \ref{lem:bound:Delta6}  and using \eqref{proof:pro4:eq1}, we have
\begin{equation*}
\begin{split}
 |\Delta^{(4,6),\mathcal{E}}_{m,k,N}( I)  | \geq   |\Delta^{(4),\mathcal{E}}_{m,k}( I)  | -  |\Delta^{(6),\mathcal{E}}_{m,k,N}( I)  | &\geq 8\zeta(\nu) \gamma \sigma  \, N^{-5|m|_1} \, \big(\kappa_{k,m}^{\mathcal{E}}\big)^{-2s} - C_r \|u\|_{\dot{H}^s}^4 N^4\\
 &\mathop{\geq}^{\eqref{assump:proba}} 2 \|u\|_{\dot{H}^s}^2 \gamma   \, N^{-5|m|_1} \, \big(\kappa_{k,m}^{\mathcal{E}}\big)^{-2s} - C_r \|u\|_{\dot{H}^s}^4 N^4
\end{split}
\end{equation*}
where $C_r$ is a constant depending only on $r$. By applying Lemma \ref{lem:diot} to control $\kappa_{k,m}^{\mathcal{E}}$ by $2\#k-1\leq 2r-1$, if $N$ is large enough with respect to $r$ and $s$ and if $\gamma N^{-6r} \geq \|u\|_{\dot{H}^s}^2 N^5$, we have
$$
 |\Delta^{(4,6),\mathcal{E}}_{m,k,N}(I)  |  \geq  \gamma  \, \|u\|_{\dot{H}^s}^2 N^{-5|m|_1} \, \big(\kappa_{k,m}^{\mathcal{E}}\big)^{-2s}.
$$
Observing that if $\zeta(\nu)\sigma$ (and so $\|u\|_{\dot{H}^s}$) is small enough with respect to a constant depending only on $r$ and $s$ then
$$
\big(\kappa_{k,m}^{\mathcal{E}}\big)^{-2s} \geq \big(2r-1\big)^{-2s} \geq \gamma \|u\|_{\dot{H}^s}^2,
$$
we also have $|\Delta^{(4,6),\mathcal{E}}_{m,k,N}(\varepsilon^2 I)  |\geq  \gamma \, \|u\|_{\dot{H}^s}^2 N^{-21|m|_1} \, \max(k_\last^{-2s},\gamma \|u\|_{\dot{H}^s}^2)$.

\section{The rational Hamiltonians and their properties}\label{ratHam}
In this section we construct and we give the principal properties of the classes of rational Hamiltonians that we will use in section \ref{section6}. As explained in the introduction, these classes are strongly based on those defined in \cite{KillBill}. In fact the general principle remains the same: we build a class which contains all Hamiltonians generated by the iterative resolutions of the homological equations 
$$\{\chi,Z_4^{\E}(I)\}=R,\quad \text{ and }\quad \{\chi,Z_4^{\E}(I)+Z_{6,N}^{\E}(I)\}=R$$
and which allows a good control of the associated vector fields. 

We warn the reader that we index some objects defined in Section \ref{sec:sd} by elements of $\Irr \cap \M$ instead of elements of $\mathcal{MI}_{\mathrm{mult}}$. Nevertheless, as explained in Remark \ref{Il faut faire attention}, it make sense using the correspondance \eqref{rel:l->(m,k)}. 

\subsection{The rational Hamiltonians}
The class of rational Hamiltonian is defined as a sum, over a set of admissible indices (see Definition \ref{def:rat:fract}), of monomials $u^\ell$ divided by a product of small divisors (see Definition \ref{def:ev:frac}). In addition, we provide this somewhat complex structure with a number of control functions, defined in  Definition \ref{control}, which will allow us to estimate these Hamiltonians in different context.

\begin{definition}[Structure of the rational fractions] 
\label{def:rat:fract}
For $\mathcal{E}\in \{\mathrm{\ref{gBO}},\mathrm{\ref{gKdV}}\}$ and $r\geq 2$, $\Gamma \in \mathscr{H}_{r}^{\mathcal{E}}$ if 
 $$\Gamma \subset  (\mathcal{D}\cap \mathcal{R}^{\mathcal{E}})   \times \bigcup_{p\geq 0} (\Irr \cap \mathcal{R}^{\mathcal{E}})^p \times \bigcup_{p\geq 0} (\Irr \cap \mathcal{R}^{\mathcal{E}})^p  \times \mathbb{N} \times \mathbb{C} $$
  satisfies the following conditions :

\begin{enumerate}[i)]

\item \underline{\emph{Finite complexity.}} $\Gamma$ is a finite set, i.e. $\# \Gamma < \infty$.
\item \underline{\emph{Reality condition.}} $\Gamma$ enjoys the following symmetry
$$
(\ell,\mathbf{h},\mathbf{k},n,c) \in \Gamma \ \Rightarrow \ (-\ell,-\mathbf{h},-\mathbf{k},n,\bar{c}) \in \Gamma
$$

\item \underline{\emph{Order $r$.}} For all $(\ell,\mathbf{h},\mathbf{k},n,c) \in \Gamma$ we have $r= \#\ell  -  2\#\mathbf{h}  -4 \# \mathbf{k}$.

\item \underline{\emph{Consistency.}} For all $(\ell,\mathbf{h},\mathbf{k},n,c) \in \Gamma$ we have $0\leq n\leq \# \mathbf{h}$.

\item \underline{\emph{Finite expansion of the denominators.}} For all $(\ell,\mathbf{h},\mathbf{k},n,c) \in \Gamma$, we have
$$
\mathbf{h},\mathbf{k} \in \bigcup_{q \in \mathbb{N}} \big( \bigcup_{2\leq n \leq  \# \ell} \Irr_n \big)^q.
$$
\end{enumerate}

\end{definition}

\begin{definition}[Controls of the rational fractions]\label{control}

Being given $\Gamma \in \mathscr{H}_{r}^{\mathcal{E}}$ we introduce the following controls
\begin{itemize}

\item \underline{\emph{Control of multiplicity.}} 
$$
C^{(m)}_\Gamma:=\max_{k \in \mathcal{R}^{\mathcal{E}}} \# \{ (\ell,\mathbf{h},\mathbf{k},n,c) \in \Gamma \ | \ \ell = k \}.
$$

\item \underline{\emph{Control of the degrees of the numerators.}} 
$$
C^{(de)}_\Gamma:=\max_{(\ell,\mathbf{h},\mathbf{k},n,c) \in \Gamma} \# \ell.
$$

\item \underline{\emph{Control of the distribution of the derivatives.}} 
$$
C^{(di)}_\Gamma:=\max_{(\ell,\mathbf{h},\mathbf{k},n,c) \in \Gamma} \frac{(\kappa_{\mathbf{h}_{1}}^{\mathcal{E}}\dots \kappa_{\mathbf{h}_{\last}}^{\mathcal{E}})^2 }{|\ell_{3}\dots \ell_{\last}|}.
$$

\item \underline{\emph{Control of the old zero momenta.}} 
$$
C^{(om)}_\Gamma:=\max_{(\ell,\mathbf{h},\mathbf{k},n,c) \in \Gamma} \max\left(\max_{j=1,\dots,\# \mathbf{k}} \frac{|\mathbf{k}_{j,1}|}{|\ell_{2}|},\max_{j=1,\dots,\# \mathbf{h}} \frac{|\mathbf{h}_{j,1}|}{|\ell_{2}|} \right).
$$

\item \underline{\emph{Global control of the structure.}} 
$$
C^{(str)}_\Gamma:=\max_{\sigma \in \{m,de,di,om \}} C^{(\sigma)}(\Gamma).
$$

\item \underline{\emph{Control of the existing modes.}} 
$$C^{(em)}_\Gamma = \max_{(\ell,\mathbf{h},\mathbf{k},n,c) \in \Gamma} \max_{\substack{1\leq i\leq \# \mathbf{h}\\1\leq j\leq \# \mathbf{k}}} \max( |\ell_1|, |\mathbf{h}_{i,1}|,|\mathbf{k}_{j,1}| ).$$

\item \underline{\emph{Control of the amplitude.}}
$$
C^{(\infty)}_\Gamma = \max_{(\ell,\mathbf{h},\mathbf{k},n,c) \in \Gamma} |c|.
$$ 
\end{itemize}

\end{definition}

\begin{definition}[Evaluations] \label{def:ev:frac} Being given $\Gamma\in \mathscr{H}_r^{\mathcal{E}}$ and $N\geq C^{(em)}_\Gamma$, $\Gamma_N$ denotes the formal rational fraction defined by
$$
\Gamma_N(u) = \sum_{(\ell,\mathbf{h},\mathbf{k},n,c) \in \Gamma} c \, u^{\ell} \left( \prod_{j=1}^n \Delta^{(4),\mathcal{E}}_{\mathbf{h}_j}(I) \right)^{-1} \left( \prod_{j=n+1}^{\# \mathbf{h}} \Delta^{(4,6),\mathcal{E}}_{\mathbf{h}_j,N}(I) \right)^{-1} \left( \prod_{j=1}^{\# \mathbf{k}} \Delta^{(4,6),\mathcal{E}}_{\mathbf{k}_j,N}(I) \right)^{-1}.
$$
Naturally, we also identify this formal rational fraction with the smooth function defined on the subset of $L^2$ where the denominators do not vanish.
\end{definition}

\begin{remark}
\label{rem:commuteZ2} Note that, since the numerators of the rational Hamiltonian are only resonant monomials, they commute with $Z_2^{\mathcal{E}}$.
\end{remark}
The following proposition establishes the stability of the class $\mathscr{H}_{r}^{\mathcal{E}}$ by Poisson bracket.
\begin{proposition} 
\label{prop:stab_frac}
Being given $r_1,r_2\geq 2$, $\Gamma \in \mathscr{H}_{r_1}^{\mathcal{E}}$, $\Upsilon \in \mathscr{H}_{r_2}^{\mathcal{E}}$ and $N\geq \max(C^{(em)}_\Gamma,C^{(em)}_\Upsilon)$, there exists $\Xi \in \mathscr{H}_{r_1+r_2-2}^{\mathcal{E}}$ verifying the identity
$$
\{ \Gamma_N,\Upsilon_N \} = \Xi_N
$$
and satisfying the controls $C^{(em)}_\Xi \leq N$
\begin{equation}\label{Cinfty}
C^{(\infty)}_\Xi \lesssim_{C^{(str)}_\Gamma\! \!,C^{(str)}_\Upsilon} N^3 \, C^{(\infty)}_\Upsilon \, C^{(\infty)}_\Gamma
\end{equation}
and 
\begin{equation}\label{Cstr}
C^{(str)}_\Xi \lesssim_{C^{(str)}_\Gamma \! \!,C^{(str)}_\Upsilon} 1.
\end{equation}
\end{proposition}
\begin{proof}
The proof is similar to the proof of Lemma 6.6 in \cite{KillBill}, for the reader's convenience we outline it again in this new framework. To compute the poisson bracket between $\Gamma_N$ and $\Upsilon_N$, we only need to calculate the poisson brackets of the summands. Applying the Leibniz's rule we see that, up to combinatorial factors and finite linear combinations depending only on $C^{(str)}_\Gamma$ and $C^{(str)}_\Upsilon$  four kind of terms appear depending on which part of the Hamiltonians the Poisson bracket applies to:

\medskip 
\noindent {\bf \emph{Type I.} }
The first type of terms we consider are those where the derivatives apply only on the  numerators. They are of the form (to simplify the presentation we omit the index $\E$ in all the proof)

\begin{align*}
\frac{cc'}{ \prod_{j=1}^n \Delta^{(4)}_{\mathbf{h}_j}  \prod_{j=n+1}^{\# \mathbf{h}} \Delta^{(4,6)}_{\mathbf{h}_j,N}   \prod_{j=1}^{\# \mathbf{k}} \Delta^{(4,6)}_{\mathbf{k}_j,N}
\prod_{j=1}^{n'} \Delta^{(4)}_{\mathbf{h}'_j}  \prod_{j=n'+1}^{\#\mathbf{h'}} \Delta^{(4,6)}_{\mathbf{h}'_j,N}   \prod_{j=1}^{\#\mathbf{k'}} \Delta^{(4,6)}_{\mathbf{k}'_j,N}}   \{u^\ell,u^{\ell'}\}
\end{align*}
with $(\ell,\mathbf{h},\mathbf{k},n,c) \in \Gamma$ and $(\ell',\mathbf{h'},\mathbf{k'},n',c') \in \Upsilon$. The product $\{u^\ell,u^{\ell'}\}$ is a finite linear combination of terms of the 
form $2i\pi ju^{\ell"}$ where $j$ is an element of the multi-indices $\ell$, $-j$ is an element of the multi-indices $\ell'$ and $\ell''$ is the ordered concatenation of $\ell$ and  $\ell'$ minus the indices $j,-j$. We focus on the worst\footnote{That term will turn out to be the worst when we want to 
control of the distribution of the derivatives, see below. All the other cases are treated in the proof of Lemma 6.6 in \cite{KillBill} } term of this linear combination: when $j=\ell_1=\ell_1'$. The corresponding term reads
\begin{align}\label{type1}
\frac{c'' u^{\ell''}}{
\prod_{j=1}^{n''} \Delta^{(4)}_{\mathbf{h}''_j}  \prod_{j=n''+1}^{\#\mathbf{h''}} \Delta^{(4,6)}_{\mathbf{h}''_j,N}   \prod_{j=1}^{\#\mathbf{k''}} \Delta^{(4,6)}_{\mathbf{k}''_j,N}} 
\end{align}
where $n''=n+n'$, $\hb''$ is the  concatenation of $\hb$ and $\hb'$, $\kb''$ is the concatenation of $\kb$ and $\kb'$, $c''=2i\pi j c c'$. It remains to prove that $(\ell'',\hb'', \kb'',n'',c'')$ satisfies conditions (ii)-(v) of Definition \ref{def:rat:fract}. Conditions (ii) and (iv) are clearly satisfied. Condition (iii) holds true since the new order is $r_1+r_2-2$ and $\#\ell''=\# \ell+\#\ell'-2$, $ \#\hb''_0=\#\hb+\#\hb'$, $\#\kb''=\#\kb+\#\kb'$. Finally all the indices of $\hb''$ and $\kb''$ have a length between $2$ and $\max(\# \ell,\#\ell')\leq  \#\ell''$ so (v) is also satisfied. So the term \eqref{type1} is associated with an element of $\mathscr{H}_{r_1+r_2-2}$ (through the Definition \ref{def:ev:frac}) and satisfies the control of existing modes (they are all of index smaller than $N$ since they have all been created from index mode smaller than $N$). Further since $c''=2i\pi j c c'$ and $|j|\leq N$ the control of the amplitude announced in Proposition \ref{prop:stab_frac}, i.e. \eqref{Cinfty}, is verified (actually, here, the factor $N^3$ could be replaced by $N$). It remains the difficult part : to verify \eqref{Cstr}. The control of multiplicity and the control of the degrees of the numerator is clear for \eqref{type1}. Concerning the control of the old zero momenta we have by construction for all $j=1,\cdots,\#\kb''$
\begin{align*}|\kb''_{j,1}|\leq (C^{(om)}_\Gamma+C^{(om)}_\Upsilon)(|\ell_2|+|\ell'_2|)
&\leq 2(C^{(om)}_\Gamma+C^{(om)}_\Upsilon)|\ell''_1|\\&\leq 2(C^{(om)}_\Gamma+C^{(om)}_\Upsilon) (\# \ell+\#\ell'-3)|\ell''_2|\\
&\leq 2(C^{(om)}_\Gamma+C^{(om)}_\Upsilon)(C^{(de)}_\Gamma+C^{(de)}_\Upsilon)|\ell''_2|\end{align*}
and thus, by doing the same thing with $\hb$ instead of $\kb$, the new "old zero momenta" is $\lesssim_{C^{(str)}_\Gamma \! \!,C^{(str)}_\Upsilon} 1$.\\
 We finish in beauty  with the control of the distribution of the derivatives. We have
\begin{align*}|\ell''_3\cdots\ell''_{\rm last}| = | \ell_3\cdots\ell_{\rm last} \ell'_3\cdots\ell'_{\rm last}|\text{ or }| \ell_4\cdots\ell_{\rm last} \ell'_2\cdots\ell'_{\rm last}|\text{ or }| \ell_2\cdots\ell_{\rm last} \ell'_4\cdots\ell'_{\last}|
\end{align*}
depending of the value of $\ell''_1$ and $\ell''_2$: the first term correspond to the case $\min(|\ell_2|,|\ell'_2|)\geq \max (|\ell_3|,|\ell'_3|)$ (and thus $\{\ell''_1,\ell''_2\}=\{\ell_2 ,\ell'_2\}$) the second one corresponds to $|\ell_2|\geq|\ell_3|\geq|\ell'_2|$ (and thus $\{\ell''_1,\ell''_2\}=\{\ell_2 ,\ell_3\}$) and the third one is symmetrical to the previous one. But in the second case, using the zero momentum of $\ell'$, we have $|\ell'_2|\geq \frac1{r_2-1} |\ell'_1|=\frac1{r_2-1} |\ell_1|\geq \frac1{r_2}|\ell_3|$ thus we get
\begin{align*}|\ell''_3\cdots\ell''_{\rm last}|&\geq \min\big(\frac1{r_1},\frac1{r_2}\big) | \ell_3\cdots\ell_{\rm last}||\ell'_3\cdots\ell'_{\rm last}|\\
&\geq \min\big(\frac1{r_1},\frac1{r_2}\big) \big[C^{(di)}_\Gamma C^{(di)}_\Upsilon \big]^{-1}(\kappa_{\mathbf{h}_{1}}\dots \kappa_{\mathbf{h}_{\last}})^2 (\kappa_{\mathbf{h}_{1}'}\dots \kappa_{\mathbf{h}_{\last}'})^2\\
&= \min\big(\frac1{r_1},\frac1{r_2}\big)  \big[C^{(di)}_\Gamma C^{(di)}_\Upsilon \big]^{-1}(\kappa_{\mathbf{h}_{1}''}\dots \kappa_{\mathbf{h}_{\last}''})^2
\end{align*}
and the new coefficient of distribution of derivatives is controlled by $\max(r_1,r_2) C^{(di)}_\Gamma C^{(di)}_\Upsilon$.

\medskip
\noindent {\bf \emph{Type II.} } The second type of terms we consider are those where one $ \Delta^{(4)}_{\mathbf{h}_j} $ appears in the Poisson bracket. (The case where $ \Delta^{(4)}_{\mathbf{h}'_j} $ appears in the Poisson bracket is treated similarly.) They are of the form
\begin{align*}
\frac{cc' u^\ell}{ \prod_{j=1}^{n-1} \Delta^{(4)}_{\mathbf{h}_j}  \prod_{j=n+1}^{\# \mathbf{h}} \Delta^{(4,6)}_{\mathbf{h}_j,N}   \prod_{j=1}^{\# \mathbf{k}} \Delta^{(4,6)}_{\mathbf{k}_j,N}
\prod_{j=1}^{n'} \Delta^{(4)}_{\mathbf{h}'_j}  \prod_{j=n'+1}^{\#\mathbf{h'}} \Delta^{(4,6)}_{\mathbf{h}'_j,N}   \prod_{j=1}^{\#\mathbf{k'}} \Delta^{(4,6)}_{\mathbf{k}'_j,N}}   \{\frac1{ \Delta^{(4)}_{\mathbf{h}_n}} ,u^{\ell'}\}
\end{align*}
In view of the Definition \ref{def:Delta4}, Remark \ref{rem:useful} and Definition \ref{def:kappa} the Poisson bracket $\{\frac1{ \Delta^{(4)}_{\mathbf{h}_n}} ,u^{\ell'}\}$ vanishes except if  there exits $i\in\{1,\cdots,\#\ell'\}$ such that  $\kappa_{\hb_n}\leq |\ell_i|\leq|\hb_{n,1}|$, so we get finitely many terms. Let us analyse one of this terms: let us assume $\kappa_{\hb_n}\leq \ell'_{i_0}\leq |\hb_{n,1}|$, which leads to the term
\begin{align}\label{type2}
\frac{c'' u^{\ell''}}{
\prod_{j=1}^{n''} \Delta^{(4)}_{\mathbf{h}''_j}  \prod_{j=n''+1}^{\#\mathbf{h''}} \Delta^{(4,6)}_{\mathbf{h}''_j,N}   \prod_{j=1}^{\#\mathbf{k''}} \Delta^{(4,6)}_{\mathbf{k}''_j,N}} 
\end{align}
 where $\ell''$ is the ordered concatenation of $\ell$ and $\ell'$, $n''=n+n'+1$, $\hb''$ is the  concatenation of $\hb$, $\hb_{n}$  and $\hb'$ (with $\hb''_{n''}=\hb_n$), $\kb''$ is the concatenation of $\kb$ and $\kb'$ and
  $c''=2i\pi \ell'_{i_0} cc'\partial_{I_{\ell'_{i_0}}} \Delta^{(4)}_{\mathbf{h}_n}$. \\
 We easily verify  that $(\ell'',\hb'', \kb'',n'',c'')$ satisfies conditions (ii)-(v) of Definition \ref{def:rat:fract}. So the term \eqref{type2} is in $\mathscr{H}_{r_1+r_2-2}$ and it remains to prove that it satisfies the controls announced in \eqref{Cinfty} and \eqref{Cstr}. \\
   Going back to the Definition \ref{def:Delta4} we see that $|\partial_{I_{\ell'_1}} \Delta^{(4)}_{\mathbf{h}_n}|\lesssim N$ so we conclude that  $|c''|\lesssim N^2|cc'|$. Now we focus on the control of distribution of derivatives.  We have
   $$\ell''_3\cdots\ell''_{\rm last}=\ell_3\cdots\ell_{\rm last}\ell'_3\cdots\ell'_{\rm last}ji$$
   where $j=\min (\ell_1,\ell_2,\ell'_1,\ell'_2)$ and $i=\min ((\ell_1,\ell_2,\ell'_1,\ell'_2)\setminus j)$. So we get 
\begin{align*}\frac{(\kappa_{\mathbf{h}''_{1}}\dots \kappa_{\mathbf{h}''_{\last}})^2}{|\ell''_3\cdots\ell''_{\rm last}|}\leq C^{(di)}_\Gamma C^{(di)}_\Upsilon \frac{\kappa_{\mathbf{h}_{n}}^2}{|ij|}.
\end{align*}
By construction, $\kappa_{\mathbf{h}_{n}}\leq |\ell'_{i_0}|\leq |\ell'_1|\leq r_2|\ell'_2|$ and using the control of the old zero momenta we know $\kappa_{\mathbf{h}_{n}}\leq|\hb_{n,1}|\leq C^{(om)}_\Gamma |\ell_2| $ so we have $$\kappa_{\mathbf{h}_{n}}\leq (r_2+C^{(om)}_\Gamma) \min(|\ell_2|, |\ell'_2|).$$
But it is clear that $ \min(|\ell_2|^2, |\ell'_2|^2)\leq ij$ 
thus we get
\begin{align*}\frac{(\kappa_{\mathbf{h}''_{1}}\dots \kappa_{\mathbf{h}''_{\last}})^2}{|\ell''_3\cdots\ell''_{\rm last}|}\leq C^{(di)}_\Gamma C^{(di)}_\Upsilon \big[r_2+C^{(om)}_\Gamma\big]^2 .
\end{align*}

\medskip
\noindent {\bf \emph{Type III.} } The third type of terms we consider are those where one  $\Delta^{(4,6)}_{\mathbf{h}_j,N}$  appears in the Poisson bracket (the case where $ \Delta^{(4,6)}_{\mathbf{h}'_j} $ appears in the Poisson bracket is treated similarly). They are of the form
\begin{align*}
\frac{cc' u^\ell}{ \prod_{j=1}^{n} \Delta^{(4)}_{\mathbf{h}_j}  \prod_{j=n+2}^{\# \mathbf{h}} \Delta^{(4,6)}_{\mathbf{h}_j,N}   \prod_{j=1}^{\# \mathbf{k}} \Delta^{(4,6)}_{\mathbf{k}_j,N}
\prod_{j=1}^{n'} \Delta^{(4)}_{\mathbf{h}'_j}  \prod_{j=n'+1}^{\#\mathbf{h'}} \Delta^{(4,6)}_{\mathbf{h}'_j,N}   \prod_{j=1}^{\#\mathbf{k'}} \Delta^{(4,6)}_{\mathbf{k}'_j,N}}   \{\frac1{ \Delta^{(4,6)}_{\mathbf{h}_{n+1},N}} ,u^{\ell'}\}
\end{align*}
We recall that $\Delta^{(4,6)}_{\mathbf{h}_{n+1},N}= \Delta^{(4)}_{\mathbf{h}_{n+1}}+\Delta^{(6)}_{\mathbf{h}_{n+1},N}$ so that
\begin{equation}\label{ouaiche}\{\frac1{ \Delta^{(4,6)}_{\mathbf{h}_{n+1},N}} ,u^{\ell'}\}=\sum_{i=1}^{\#\ell'}2i\pi\ell'_i\frac{\partial_{I_{\ell'_i}}\Delta^{(4)}_{\mathbf{h}_{n+1}}}{\big[ \Delta^{(4,6)}_{\mathbf{h}_{n+1},N}\big]^2}u^{\ell'}+ \sum_{i=1}^{\#\ell'}2i\pi\ell'_i\frac{\partial_{I_{\ell'_i}}\Delta^{(6)}_{\mathbf{h}_{n+1},N}}{\big[ \Delta^{(4,6)}_{\mathbf{h}_{n+1},N}\big]^2}u^{\ell'}\end{equation}
As explained in the previous paragraph, the terms in the first sum vanish except if  there exits $i\in\{1,\cdots,\#\ell'\}$ such that  $\kappa_{\hb_{n+1}}\leq |\ell_i|\leq|\hb_{{n+1},1}|$. Let us analyse one of this terms: let us assume $\kappa_{\hb_{n+1}}\leq \ell'_{i_0}\leq |\hb_{{n+1},1}|$, which leads to the term
\begin{align}\label{type31}
\frac{c'' u^{\ell''}}{
\prod_{j=1}^{n''} \Delta^{(4)}_{\mathbf{h}''_j}  \prod_{j=n''+1}^{\#\mathbf{h''}} \Delta^{(4,6)}_{\mathbf{h}''_j,N}   \prod_{j=1}^{\#\mathbf{k''}} \Delta^{(4,6)}_{\mathbf{k}''_j,N}} 
\end{align}
 where $\ell''$ is the ordered concatenation of $\ell$ and $\ell'$, $n''=n+n'$, $\hb''$ is the  concatenation of $\hb$, $\hb_{n+1}$ and $\hb'$ (with $\hb''_{n''+1}=\hb_{n+1}$), $\kb''$ is the concatenation of $\kb$ and $\kb'$ and
  $c''=2i\pi \ell'_{i_0} cc'\partial_{I_{\ell'_{i_0}}} \Delta^{(4)}_{\mathbf{h}_{n+1}}$. Clearly this term can be treated in the same way as terms of Type 2 dealt with in the previous paragraph. \\
Now we analyse the terms arising from the second sum in \eqref{ouaiche}. In view of \eqref{def_Delta6_N}, we know that $\partial_{I_{\ell'_i}}\Delta^{(6)}_{\mathbf{h}_{n+1},N}(I)=\sum_{j=1}^N a_j I_j$ is a linear form in the actions whose coefficients are reals and bounded by $\lesssim N^2$. This leads to a sum of  terms of the form
\begin{align}\label{type32}
\frac{c'' u^{\ell''}}{
\prod_{j=1}^{n''} \Delta^{(4)}_{\mathbf{h}''_j}  \prod_{j=n''+1}^{\#\mathbf{h''}} \Delta^{(4,6)}_{\mathbf{h}''_j,N}   \prod_{j=1}^{\#\mathbf{k''}} \Delta^{(4,6)}_{\mathbf{k}''_j,N}} 
\end{align}
 where $\ell''$ is the ordered concatenation of $\ell$, $\ell'$ and $(j,-j)$, $n''=n+n'$, $\hb''$ is the  concatenation of $\hb$ and $\hb'$,  $\kb''$ is the concatenation of $\kb$, $\kb'$ and $\hb_{n+1}$  and
  $c''=2i\pi \ell'_{i_0} cc'a_j$. In particular we see that $|c''|\lesssim N^3|cc'|$.\\
  We easily verify  that $(\ell'',\hb'', \kb'',n'',c'')$ satisfies conditions (ii)-(v) of Definition \ref{def:rat:fract}. So the term \eqref{type2} is in $\mathscr{H}_{r_1+r_2-2}$ and it remains to prove that it satisfies the controls announced in \eqref{Cstr}.  We notice that since we conserved $\ell$ and $\ell'$ and since we did not add new $\hb$ the control of distribution of derivatives of this new term is automatic.  So \eqref{Cstr} is satisfied.

\medskip
\noindent {\bf \emph{Type IV.} } The second type of terms we consider are those where one  $\Delta^{(4,6)}_{\mathbf{k}_j}$  appears in the Poisson bracket and are treated essentially as 
terms of  \emph{Type III} except that to deal with the terms arising from the first sum in \eqref{ouaiche} we distribute the new denominator $ \Delta^{(4,6)}_{\mathbf{k}_{1},N}$ evenly: $\Delta^{(4,6)}_{\mathbf{k}_{1},N} = \Delta^{(4,6)}_{\mathbf{h''_{\#\hb+\#\hb'+1}},N}$.

\end{proof}

In order to optimize the estimates of the different terms that we will encounter by applying a Birkhoff procedure in the next section, we will need subclasses that follow the evolution of the different indices of $\Gamma$ as closely as possible (they have been designed applying the ideas presented in Remark \ref{about3k-9}).

\begin{definition}[Sharp subclasses] \label{sharpclass}
Let $\mathscr{H}_r^{(4),\mathcal{E}},\mathscr{H}_r^{(6),\mathcal{E}},\mathscr{H}_r^{(4),*,\mathcal{E}},\mathscr{H}_r^{(6),*,\mathcal{E}}$ be the subsets of $\mathscr{H}_r^{\mathcal{E}}$ such that
\begin{itemize}
\renewcommand{\labelitemi}{$\bullet$} 
\item if $(\ell,\mathbf{h},\mathbf{k},n,c) \in \Gamma \in \mathscr{H}_r^{(4),\mathcal{E}}$ then
$$
\# \mathbf{k}=0 \hspace{1cm} \mathrm{and} \hspace{1cm} \# \mathbf{h}=n\leq 2r-10.
$$
\item if $(\ell,\mathbf{h},\mathbf{k},n,c) \in \Gamma \in \mathscr{H}_{r-2}^{(4),*,\mathcal{E}}$ then
$$
\# \mathbf{k}=0 \hspace{1cm} \mathrm{and} \hspace{1cm} \# \mathbf{h}=n\leq 2r-10+1.
$$
\item if $(\ell,\mathbf{h},\mathbf{k},n,c) \in \Gamma \in  \mathscr{H}_r^{(6),\mathcal{E}}$ then there exists $\beta \in \mathbb{N}^3$ such that $\beta_1+\beta_2+\beta_3\leq r-7$ and 
$$
n\leq 13 r-87 +\beta_1 \hspace{1cm}  \# \mathbf{h}-n \leq \beta_2 \hspace{1cm}  \# \mathbf{k} \leq 4r-28 + \beta_3.
$$
\item if $(\ell,\mathbf{h},\mathbf{k},n,c) \in \Gamma \in  \mathscr{H}_{r-4}^{(6),*,\mathcal{E}}$ then there exists $\beta \in \mathbb{N}^3$ such that $\beta_1+\beta_2+\beta_3\leq r-7$ and 
$$
n\leq 13 r-87 +  \beta_1 + 3\hspace{1cm}  \# \mathbf{h}-n \leq \beta_2 \hspace{1cm}  \# \mathbf{k} \leq 4r-28 + \beta_3 + 2.
$$
\end{itemize}
\end{definition}

\begin{remark} By definition, if $r\geq 7$, it is clear that $\mathscr{H}_r^{(4),\mathcal{E}} \subset \mathscr{H}_r^{(6),\mathcal{E}}$.
\end{remark}

\begin{remark} 
\label{rem:bound:deg}
If $\Gamma\in \mathscr{H}_r^{(6),\mathcal{E}}$ then the condition $iii)$ gives the following upper bound of $C^{(de)}_\Gamma$ by an affine function of $r$ :
$$
C^{(de)}_\Gamma \leq 47r -314.
$$
\end{remark}

\begin{definition}[Integrable rational fraction]\label{def:Aint} {$\mathscr{A}_r^{\mathcal{E}}$ denotes the set of the integrable rational fractions of order $r$ : $\Gamma \in \mathscr{A}_r^{\mathcal{E}}$ if  $ \Gamma \in \mathscr{H}_r^{\mathcal{E}}$ and for all $(\ell,\mathbf{h},\mathbf{k},n,c) \in \Gamma$ we have $\Irr(\ell)=\emptyset$. Furthermore, $\mathscr{R}_r^{\mathcal{E}} = \mathscr{H}_r^{\mathcal{E}} \setminus \mathscr{A}_r^{\mathcal{E}}$ denotes the complementary of  $\mathscr{A}_r^{\mathcal{E}}$ in $\mathscr{H}_r^{\mathcal{E}}$. \\
Similarly, for $n\in \{4,6\}$, we define $\mathscr{A}_r^{(n),\mathcal{E}}$ as the set of the integrable rational fractions of order $r$ in $ \mathscr{H}_r^{(n),\mathcal{E}}$ and  $\mathscr{R}_r^{(n),\mathcal{E}} $
its  complementary in $\mathscr{H}_r^{(n),\mathcal{E}}$. }
\end{definition}

\begin{remark} 
If $r$ is odd then $\mathscr{A}_r^{\mathcal{E}}=\emptyset$ and $\mathscr{H}_r^{\mathcal{E}} = \mathscr{R}_r^{(n),\mathcal{E}} $.
\end{remark}

\begin{proposition} 
\label{prop:stab:subclasses}
In Proposition \ref{prop:stab_frac}, if $\Gamma \in \mathscr{H}_{r_1}^{(m),*,\mathcal{E}}$, $\Upsilon \in \mathscr{H}_{r_2}^{(m),\mathcal{E}}$ with $m=4$ or $m=6$ then $\Xi \in \mathscr{H}_{r_1+r_2-2}^{(m),\mathcal{E}}$.
\end{proposition} 
\begin{proof} Here, we only have to count the denominators. We recall that, by construction (see the proof of Proposition \ref{prop:stab_frac}), the terms of  $\Xi$ are constructed by distributing the derivatives of the Poisson brackets of the summand of $\Gamma_N$ and $\Upsilon_N$. Consequently, if $(\ell'',\mathbf{h}'',\mathbf{k}'',n'',c'') \in \Xi$ then there exist $(\ell,\mathbf{h},\mathbf{k},n,c) \in \Gamma$, $(\ell',\mathbf{h}',\mathbf{k}',n',c') \in \Upsilon$ such that following the types of the proof of Proposition \ref{prop:stab_frac} we have:

\noindent {\bf \emph{Type I.} } $n''=n+n'$, $\mathbf{h}''=(\mathbf{h}_1,\dots,\mathbf{h}_{n},\mathbf{h}_1',\dots,\mathbf{h}_{n'}',\mathbf{h}_{n+1}, \dots,\mathbf{h}_{\last},\mathbf{h}_{n'+1}', \dots,\mathbf{h}_{\last}' )$ and $\mathbf{k}''=(\mathbf{k}_1,\dots,\mathbf{k}_{\last},\mathbf{k}_1',\dots,\mathbf{h}_\last')$. In that case we have
$$
n''=n+n' \quad \mathrm{and} \quad  \# \mathbf{h}'' = \# \mathbf{h} +\# \mathbf{h}' \quad \mathrm{and} \quad  \# \mathbf{k}'' = \# \mathbf{k} +\# \mathbf{k}'.
$$

\noindent {\bf \emph{Type II.} } $n''=n+n'+1$, $\mathbf{h}''=(\mathbf{h}_1,\dots,\mathbf{h}_{n},\mathbf{h}_1',\dots,\mathbf{h}_{n'}',h,\mathbf{h}_{n+1}, \dots,\mathbf{h}_{\last},\mathbf{h}_{n'+1}', \dots,\mathbf{h}_{\last}' )$ and $\mathbf{k}''=(\mathbf{k}_1,\dots,\mathbf{k}_{\last},\mathbf{k}_1',\dots,\mathbf{h}_\last')$ where $h=\mathbf{h}_{i_0}$ (or $h=\mathbf{h}_{i_0}'$) for some some $i_0\leq n$ (or $i_0\leq n'$).
In that case we have
$$
n''=n+n'+1 \quad \mathrm{and} \quad \# \mathbf{h}'' = \# \mathbf{h} +\# \mathbf{h}'+1 \quad \mathrm{and} \quad  \# \mathbf{k}'' = \# \mathbf{k} +\# \mathbf{k}'.
$$

\noindent {\bf \emph{Type III, first sum in \eqref{ouaiche}} }. $n''=n+n'$, $\mathbf{h}''$ and $\mathbf{k}''$ are given by the same formula as in Type II but $h=\mathbf{h}_{i_0}$ (or $h=\mathbf{h}_{i_0}'$) for some some $i_0> n$ (or $i_0> n'$).
In that case we have
$$
n''=n+n' \quad \mathrm{and} \quad \# \mathbf{h}'' = \# \mathbf{h} +\# \mathbf{h}'+1 \quad \mathrm{and} \quad  \# \mathbf{k}'' = \# \mathbf{k} +\# \mathbf{k}'.
$$

\noindent {\bf \emph{Type III, second sum in \eqref{ouaiche}} }. $n''=n+n'$, $\mathbf{h}''$is given by the same formula as in Type I and $\mathbf{k}''=(\mathbf{k}_1,\dots,\mathbf{k}_{\last},\mathbf{k}_1',\dots,\mathbf{h}_\last',h)$ where $h=\mathbf{h}_{i_0}$ (or $h=\mathbf{h}_{i_0}'$) for some some $i_0> n$ (or $i_0> n'$).
In that case we have
$$
n''=n+n' \quad \mathrm{and} \quad \# \mathbf{h}'' = \# \mathbf{h} +\# \mathbf{h}' \quad \mathrm{and} \quad  \# \mathbf{k}'' = \# \mathbf{k} +\# \mathbf{k}'+1.
$$

\noindent {\bf \emph{Type IV}.} It produce the same kind of denominators as the type III.

Therefore in any case, there exists $\widetilde{\beta}\in \mathbb{N}^3$ such that $\widetilde{\beta}_1+\widetilde{\beta}_2+\widetilde{\beta}_3\leq 1$ and 
\begin{equation}
\label{engeneral}
n''=n+n'+\widetilde{\beta}_1,  \quad \# \mathbf{h}''-n'' = \# \mathbf{h}-n +\# \mathbf{h}'-n'+\widetilde{\beta}_2,  \quad  \# \mathbf{k}'' = \# \mathbf{k} +\# \mathbf{k}'+ \widetilde{\beta}_3.
\end{equation}

Here we have to distinguish the case $m=4$ and $m=6$.

\noindent \emph{$*$ Case $m=6$.} Since $\Gamma \in  \mathscr{H}_{r_1}^{(6),*,\mathcal{E}}$ and $\Upsilon \in  \mathscr{H}_{r_2}^{(6),\mathcal{E}}$, we deduce of \eqref{engeneral} that
\begin{equation}
\label{pilepoil2}
\begin{split}
n'' &\leq [13 (r_1+4) -87 +3 +  \beta_1] + [13 r_2 -87 +  \beta_1'] + \widetilde{\beta}_1\\
&= 13( r_1 + r_2 -2 ) - 87  + \beta_1+ \beta_1' + \widetilde{\beta}_1 - 6,\\
\# \mathbf{h}''-n'' &\leq \beta_1+ \beta_2' + \widetilde{\beta}_2, \\
\# \mathbf{k}'' &\leq (4(r_1+4)-28 + \beta_3 + 2) + (4r_2-28 + \beta_3' ) + \widetilde{\beta}_3 \\
&= 4( r_1 + r_2 -2 ) - 28  + \beta_3+ \beta_3' + \widetilde{\beta}_3 - 2,
\end{split} 
\end{equation}
where $\beta_1+ \beta_2+\beta_3\leq (r_1+4)-7$ and $\beta_1'+ \beta_2'+\beta_3'\leq r_2-7$. Setting $\beta'' = \beta+ \beta' +\widetilde{\beta}$ and observing that 
$$
\beta_1'' + \beta_2'' + \beta_3'' \leq [(r_1+4)-7] + [r_2-7] + 1 = (r_1+r_2-2)-7
$$
we deduce of \eqref{pilepoil2} that $\Xi \in \mathscr{H}_{r_1+r_2-2}^{(6),\mathcal{E}}$.

\noindent \emph{$*$ Case $m=4$.} Since $\Gamma \in  \mathscr{H}_{r_1}^{(4),*,\mathcal{E}}$ and $\Upsilon \in  \mathscr{H}_{r_2}^{(4),\mathcal{E}}$, we know that $\# \mathbf{k}=\# \mathbf{k}' = \# \mathbf{h}-n = \# \mathbf{h}'-n' =0$ and $\widetilde{\beta}_1=1$. Consequently, we deduce of \eqref{engeneral} that $\# \mathbf{k}''= \# \mathbf{h}''-n''=0$ and
$$
n''\leq [(2r_1+2)-10+1] + [2r_2-10] +1 = 2(r_1+r_2-2) - 10
$$
and thus we have $\Xi \in \mathscr{H}_{r_1+r_2-2}^{(4),\mathcal{E}}$.

\end{proof}

\begin{proposition} 
\label{prop:stab:chichi}
If $\chi \in \mathscr{H}_{r}^{(6),*,\mathcal{E}}$, $Z \in \mathscr{A}_{6}^{(4),\mathcal{E}}$ and $\Xi$ is the element of $ \mathscr{H}_{2r+2}^{\mathcal{E}}$ associated with\footnote{for some $N$ whose the value is irrelevant here.} $\{\chi_N,\{\chi_N,Z\} \}$ through Proposition \ref{prop:stab_frac} then $\Xi \in \mathscr{H}_{2r+2}^{(6),\mathcal{E}}$.
\end{proposition}
\begin{proof} Since $[(6+r)-2]+r-2 = 2r+2$, it is clear that, by Proposition \ref{prop:stab_frac}, $\Xi \in \mathscr{H}_{2r+2}^{\mathcal{E}}$.
The only thing we really have to do is to count the number of denominators of $\{\chi_N,\{\chi_N,Z\} \}$. 

First, we recall that by definition of $\mathscr{A}_{6}^{(4),\mathcal{E}}$, each term of $Z$ has at most two denominators of the form $\Delta^{(4)}$. Then it follows of the proof of Proposition \ref{prop:stab_frac} that the denominators of $\{\chi_N,Z\}$ are\footnote{\label{majollienote}proceeding as in the proof of \ref{prop:stab:subclasses} we could make this sentence become more rigorous.} some products of denominators of $\chi_N$ times some of $Z$ plus at most one denominator of the form $\Delta^{(4)}$ (indeed $Z$ is integrable, so the derivative of the Poisson bracket cannot be distributed on a denominator of $\chi_N$).

Consequently, if $\widetilde{\Xi}_N = \{\chi_N,Z\}$ through the  Proposition \ref{prop:stab_frac}  and $(\ell,\mathbf{h},\mathbf{k},n,c) \in \widetilde{\Xi}$ then 
$$
n \leq [13 (r+4) - 87 + \beta_1^{(1)} + 3] + 2 + 1  = 13 r -29 +  \beta_1^{(1)}
$$
and
$$
\# \mathbf{h}-n \leq \beta_2^{(1)} \hspace{1cm}  \# \mathbf{k} \leq 4(r+4)-28 + \beta_3^{(1)} + 2
$$
where $\beta_1^{(1)} + \beta_2^{(1)}+\beta_3^{(1)}\leq r+4-7$.

Similarly, the denominator of $\Xi_N$ are some product of denominators $\chi_N$ times some of $\widetilde{\Xi}_N$ plus at most one denominator of the form $\Delta^{(4)},\Delta^{(4,6)}_\hb$ or $\Delta^{(4,6)}_\kb$. Consequently, if $(\ell,\mathbf{h},\mathbf{k},n,c) \in \widetilde{\Xi}$ then
\begin{equation}
\label{pilepoil}
\begin{split}
n &\leq [13 r -29 +  \beta_1^{(1)}] +  [13 (r+4) - 87 + \beta_1^{(2)} + 3] + \beta_1^{(3)}\\
&= 13(2 r +2) - 87 + \beta_1^{(1)} + \beta_1^{(2)} + \beta_1^{(3)},\\
\# \mathbf{h}-n &\leq \beta_2^{(1)} + \beta_2^{(2)} + \beta_2^{(3)}, \\
\# \mathbf{k} &\leq (4(r+4)-28 + \beta_3^{(1)} + 2) + (4(r+4)-28 + \beta_3^{(2)} + 2) +  \beta_3^{(3)} \\
&= 4(2r+2)  - 28 + \beta_3^{(1)} + \beta_3^{(2)} + \beta_3^{(3)},
\end{split} 
\end{equation}
where $\beta_1^{(2)} + \beta_2^{(2)}+\beta_3^{(2)}\leq r+4-7$ and $\beta_1^{(3)} + \beta_2^{(3)}+\beta_3^{(3)}\leq 1$. Setting $\beta^{(4)} = \beta^{(1)} + \beta^{(2)} + \beta^{(3)}$ and observing that 
$$
\beta_1^{(4)} + \beta_2^{(4)} + \beta_3^{(4)} \leq 2(r+4-7) + 1 = (2r+2) - 7
$$
we deduce of \eqref{pilepoil} that $\Xi \in \mathscr{H}_{2r+2}^{(6),\mathcal{E}}$.
\end{proof}
\subsection{Control of the vector fields and Lie transforms}

First, in the following proposition, we control the $L^2$-gradient of the rational fractions.
\begin{proposition}
\label{prop:VF}
Let $u\in \mathcal{U}_{\gamma,N,\rho}^{\mathcal{E},s}$ and $\Gamma \in \mathscr{H}_{r}^{\mathcal{E}}$ be such that $\|u\|_{\dot{H}^s}\lesssim 1$, $N\geq C^{(em)}_\Gamma$ and $\rho \geq C^{(de)}_\Gamma $ then we have
\begin{equation}
\label{est:rat:VF}
\| \nabla \Gamma_N(u) \|_{\dot{H}^s} \lesssim_{s,C^{(str)}_\Gamma}  C^{(\infty)}_\Gamma  \sqrt{\gamma}^{-\rho+r-2} N^{12 \,\rho^2}  \| u \|_{\dot{H}^s}^{r-1}
\end{equation}
and
\begin{equation}
\label{est:rat:dVF}
\| \mathrm{d}\nabla \Gamma_N(u) \|_{\mathscr{L}(\dot{H}^s)} \lesssim_{s,C^{(str)}_\Gamma}  C^{(\infty)}_\Gamma  \sqrt{\gamma}^{-\rho+r-2} N^{14 \,\rho^2}  \| u \|_{\dot{H}^s}^{r-2}.
\end{equation}
\end{proposition}
\begin{proof} We recall that $\Gamma_N(u)$ is given by the definition \ref{def:ev:frac}. Naturally, $\Gamma_N(u)$ is of the form
$$
\Gamma_N(u) = \sum_{(\ell,\mathbf{h},\mathbf{k},n,c) \in \Gamma} c \, u^{\ell} f_{\ell,\mathbf{h},\mathbf{k},n,c}(I),
$$
where $f_{\ell,\mathbf{h},\mathbf{k},n,c}(I)$ denotes the denominator. Note that since $u\in \mathcal{U}_{\gamma,N,\rho}^{\mathcal{E},s}$ with $\rho \geq C^{(de)}_\Gamma$ and that by $\mathrm{(v)}$ of definition \ref{def:rat:fract}, $\#\mathbf{h}_j,\#\mathbf{k}_j\leq \# \ell \leq C^{(de)}_\Gamma$, we have lower bounds on the denominators.

First, we aim at controlling $\| \nabla \Gamma_N(u) \|_{\dot{H}^s} $. Naturally, for $k\in \mathbb{N}^*$, we have
\begin{multline*}
 (\nabla \Gamma_N(u))_k =   \sum_{(\ell,\mathbf{h},\mathbf{k},n,c) \in \Gamma} c \, \partial_{u_{-k}} ( u^{\ell} f_{\ell,\mathbf{h},\mathbf{k},n,c}(I)) \\
 =  \sum_{(\ell,\mathbf{h},\mathbf{k},n,c) \in \Gamma} c \, (\partial_{u_{-k}}  u^{\ell} )f_{\ell,\mathbf{h},\mathbf{k},n,c}(I) + u_k \sum_{(\ell,\mathbf{h},\mathbf{k},n,c) \in \Gamma} c \,   u^{\ell}  \partial_{I_k}  f_{\ell,\mathbf{h},\mathbf{k},n,c}(I)
 =: y_k^{(1)} + y_k^{(2)}.
\end{multline*}
We are going to control $\| y^{(j)}\|_{\dot{H}^s}$ for $j\in \{1,2\}$.

\noindent $*$ \emph{Control of $\| y^{(2)}\|_{\dot{H}^s}$.} Clearly, we have $\| y^{(2)}\|_{\dot{H}^s} \leq S \| u \|_{\dot{H}^s}$ where
\begin{equation}
\label{def:S}
S:=\sup_{k\in \mathbb{N}^*} \left| \sum_{(\ell,\mathbf{h},\mathbf{k},n,c) \in \Gamma} c \,   u^{\ell} \partial_{I_k}f_{\ell,\mathbf{h},\mathbf{k},n,c}(I) \right|.
\end{equation}
Thus, we just have to establish an upper bound on $S$. By definition of $f_{\ell,\mathbf{h},\mathbf{k},n,c}$, we have
$$
-\frac{\partial_{I_k}f_{\ell,\mathbf{h},\mathbf{k},n,c}(I)}{f_{\ell,\mathbf{h},\mathbf{k},n,c}(I)} =  \sum_{j=1}^n \frac{\partial_{I_k}\Delta^{(4),\mathcal{E}}_{\mathbf{h}_j}(I)}{\Delta^{(4),\mathcal{E}}_{\mathbf{h}_j}(I)} + \sum_{j=n+1}^{\# \mathbf{h}} \frac{\partial_{I_k}\Delta^{(4,6),\mathcal{E}}_{\mathbf{h}_j,N}(I)}{\Delta^{(4,6),\mathcal{E}}_{\mathbf{h}_j,N}(I)} +  \sum_{j=1}^{\# \mathbf{k}} \frac{\partial_{I_k}\Delta^{(4,6),\mathcal{E}}_{\mathbf{k}_j,N}(I)}{\Delta^{(4,6),\mathcal{E}}_{\mathbf{k}_j,N}(I)}.
$$
So, we have to control each one of the terms in these sums (the terms involving $\Delta^{(4,6),\mathcal{E}}_{\mathbf{k}_j,N}$ and $\Delta^{(4,6),\mathcal{E}}_{\mathbf{h}_j,N}$ are treated in the same way).
\begin{itemize}
\item \emph{Upper bound for $\partial_{I_k}\Delta^{(4),\mathcal{E}}_{\mathbf{h}_j}(I)$.} We have
$$
|\partial_{I_k}\Delta^{(4),\mathcal{E}}_{\mathbf{h}_j}(I)| = |(\delta^{\mathcal{E}}_{\mathbf{h}_j})_k| \lesssim_{\#\mathbf{h}_j}  |\mathbf{h}_{j,1}| \lesssim_{\# \ell} C^{(em)}_\Gamma \lesssim_{C^{(de)}_\Gamma} N
$$
where $\delta^{\mathcal{E}}$ is defined in Definition \ref{def:Delta4}.
\item \emph{Upper bound for $\partial_{I_k} \Delta^{(4,6),\mathcal{E}}_{\mathbf{k}_j,N}(I)$.} Since $\Delta^{(4,6),\mathcal{E}}_{\mathbf{k}_j,N}(I) = \Delta^{(4),\mathcal{E}}_{\mathbf{k}_j}(I) + \Delta^{(6),\mathcal{E}}_{\mathbf{k}_j,N}(I)$
,{  it remains to control} $\partial_{I_k}  \Delta^{(6),\mathcal{E}}_{\mathbf{k}_j,N}(I)$. Using the explicit decomposition of $\Delta^{(6),\mathcal{E}}_{\mathbf{k}_j,N}(I)$ given by \eqref{expli:dec}, we clearly have
$$
|\partial_{I_k}  \Delta^{(6),\mathcal{E}}_{\mathbf{k}_j,N}(I)| \lesssim_{\#\mathbf{k}_j} N^4 \|u\|_{\dot{H}^s}^2 \lesssim_{C^{(de)}_\Gamma} N^4.
$$
\item \emph{Lower bound for $\Delta^{(4),\mathcal{E}}_{\mathbf{h}_j}(I)$.} As noticed at the beginning of the proof, since $u\in \mathcal{U}_{\gamma,N,\rho}^{\mathcal{E},s}$ {(see definition \ref{U})} we have lower bounds on the denominators :
$$
|\Delta^{(4),\mathcal{E}}_{\mathbf{h}_j}(I)| \geq  \gamma N^{-5  \#\mathbf{h}_j} \| u \|_{\dot{H}^s}^2 (\kappa_{\mathbf{h}_j}^{\mathcal{E}})^{-2s}.
$$
However, as explained in Remark \ref{rem:pasbete}, we have $\kappa_{\mathbf{h}_j}^{\mathcal{E}} \leq |\mathbf{h}_{j,1}|$. Consequently, we have  $\kappa_{\mathbf{h}_j}^{\mathcal{E}}\leq C^{(om)}_\Gamma |\ell_2|$. So, since $\ell \in \mathcal{D}$, we have  
$$
|\Delta^{(4),\mathcal{E}}_{\mathbf{h}_j}(I)| \gtrsim_{s,C^{(om)}_\Gamma}   \gamma N^{-5  C^{(de)}_\Gamma} \| u \|_{\dot{H}^s}^2 |\ell_1|^{-s} |\ell_2|^{-s}.
$$
\item \emph{Lower bound for $\Delta^{(4,6),\mathcal{E}}_{\mathbf{k}_j,N}(I)$.} The same analysis as for the previous term leads to
$$
|\Delta^{(4,6),\mathcal{E}}_{\mathbf{k}_j,N}(I)| \gtrsim_{s,C^{(om)}_\Gamma}   \gamma \, N^{-21  C^{(de)}_\Gamma} \| u \|_{\dot{H}^s}^2 |\ell_1|^{-s} |\ell_2|^{-s}.
$$
\end{itemize}
Combining the previous estimates gives 
$$
\left|\frac{\partial_{I_k}f_{\ell,\mathbf{h},\mathbf{k},n,c}(I)}{f_{\ell,\mathbf{h},\mathbf{k},n,c}(I)} \right| \lesssim_{s,C^{(str)}_\Gamma} \gamma^{-1} N^{21  C^{(de)}_\Gamma + 4}  |\ell_1|^{s} |\ell_2|^{s} \| u \|_{\dot{H}^s}^{-2}.
$$
Then, we have to establish an upper bound on $f_{\ell,\mathbf{h},\mathbf{k},n,c}(I)$.  {Since $u\in \mathcal{U}_{\gamma,N,\rho}^{\mathcal{E},s}$ (see definition \ref{U})}
 we control each factor of the form $|\Delta^{(4,6),\mathcal{E}}_{\mathbf{h}_j,N}(I)|$ by $\gamma N^{-21  C^{(de)}_\Gamma} \| u \|_{\dot{H}^s}^2 (\kappa_{\mathbf{h}_j}^{\mathcal{E}})^{-2s}$, each factor of the form $|\Delta^{(4),\mathcal{E}}_{\mathbf{h}_j}(I)|$ by $\gamma N^{-5  C^{(de)}_\Gamma} \| u \|_{\dot{H}^s}^2 (\kappa_{\mathbf{h}_j}^{\mathcal{E}})^{-2s}$ and each factor of the form 
$|\Delta^{(4,6),\mathcal{E}}_{\mathbf{k}_j,N}(I)|$ by $\gamma^2 N^{-21  C^{(de)}_\Gamma} \| u \|_{\dot{H}^s}^4$. This leads naturally to the estimate
$$
|f_{\ell,\mathbf{h},\mathbf{k},n,c}(I)| \lesssim_{s,C^{(str)}_\Gamma}   (N^{-21  C^{(de)}_\Gamma  } \gamma \| u \|_{\dot{H}^s}^2)^{-2\# \mathbf{k} -\# \mathbf{h}} (\kappa_{\mathbf{h}_{1}}^{\mathcal{E}}\dots \kappa_{\mathbf{h}_{\last}}^{\mathcal{E}})^{2s}.
$$
Using the condition $\mathrm{(iii)}$ of the definition \ref{def:rat:fract} and the estimate associated with $C^{(di)}_\Gamma$  this leads to
\begin{equation}
\label{eq:prodden}
|f_{\ell,\mathbf{h},\mathbf{k},n,c}(I)| \lesssim_{s,C^{(str)}_\Gamma}   (N^{-21  C^{(de)}_\Gamma  } \gamma \| u \|_{\dot{H}^s}^2)^{-\frac{\# \ell-r}2} |\ell_{3}\dots \ell_{\last}|^{s}.
\end{equation}
Consequently, recalling that $ C^{(em)}_\Gamma \leq N$ and applying the Young inequality, we have
\begin{equation*}
\begin{split}
S &\lesssim_{s,C^{(str)}_\Gamma}  C^{(\infty)}_\Gamma \sum_{(\ell,\mathbf{h},\mathbf{k},n,c) \in \Gamma}    (|u_{\ell_1} |\ell_{1}|^s \dots u_{\ell_\last} |\ell_{\last}|^s |)  (N^{-21  C^{(de)}_\Gamma  } \gamma \| u \|_{\dot{H}^s}^2)^{-\frac{\# \ell-r}2-1} N^4  \\
&\lesssim_{s,C^{(str)}_\Gamma} N^{C^{(de)}_\Gamma+4} C^{(\infty)}_\Gamma \sum_{(\ell,\mathbf{h},\mathbf{k},n,c) \in \Gamma}    (|u_{\ell_1} |\ell_{1}|^{s-1} \dots u_{\ell_\last} |\ell_{\last}|^{s-1} |) (N^{-21  C^{(de)}_\Gamma  } \gamma \| u \|_{\dot{H}^s}^2)^{-\frac{\# \ell-r}2-1}  \\
 &\lesssim_{s,C^{(str)}_\Gamma} N^{C^{(de)}_\Gamma+4} C^{(m)}_\Gamma C^{(\infty)}_\Gamma \sum_{m= r+2}^{C^{(de)}_\Gamma}  (N^{-21  C^{(de)}_\Gamma  } \gamma \| u \|_{\dot{H}^s}^2)^{-\frac{m-r}2-1}  \sum_{\ell_1+\dots+\ell_m =0}  \prod_{j=1}^m  |u_{\ell_j}| |\ell_{j}|^{s-1}\\
 &\lesssim_{s,C^{(str)}_\Gamma}  N^{2C^{(de)}_\Gamma}  C^{(\infty)}_\Gamma \sum_{m= r+2}^{C^{(de)}_\Gamma}  (N^{-21  C^{(de)}_\Gamma  } \gamma \| u \|_{\dot{H}^s}^2)^{-\frac{m-r}2-1}    \| u \|_{\dot{H}^s}^{m}\\
 &\lesssim_{s,C^{(str)}_\Gamma} C^{(\infty)}_\Gamma  \sqrt{\gamma}^{-C^{(de)}_\Gamma+r-2} N^{12 \,(C^{(de)}_\Gamma)^2}  \| u \|_{\dot{H}^s}^{r-2}.
  \end{split}
\end{equation*}
Finally, we deduce that $\| y^{(2)}\|_{\dot{H}^s} \leq S \| u \|_{\dot{H}^s} \lesssim_{s,C^{(str)}_\Gamma} C^{(\infty)}_\Gamma  \sqrt{\gamma}^{-C^{(de)}_\Gamma+r-2} N^{12\, (C^{(de)}_\Gamma)^2}  \| u \|_{\dot{H}^s}^{r-1}.$

\noindent $*$ \emph{Control of $\| y^{(1)}\|_{\dot{H}^s}$.} 
Using the estimate \eqref{eq:prodden} to control $|f_{\ell,\mathbf{h},\mathbf{k},n,c}(I)|$, we have
\begin{multline*}
|(y^{(1)})_k| \lesssim_{s,C^{(str)}_\Gamma}  C^{(\infty)}_\Gamma   \sum_{(\ell,\mathbf{h},\mathbf{k},n,c) \in \Gamma} |\partial_{u_{-k}}  u^{\ell} |  (N^{-21  C^{(de)}_\Gamma  } \gamma \| u \|_{\dot{H}^s}^2)^{-\frac{\# \ell-r}2} |\ell_{3}\dots \ell_{\last}|^{s} \\
  \lesssim_{s,C^{(str)}_\Gamma} C^{(\infty)}_\Gamma C^{(m)}_\Gamma \sum_{m=r}^{C^{(de)}_\Gamma} (N^{-21  C^{(de)}_\Gamma  } \gamma \| u \|_{\dot{H}^s}^2)^{-\frac{m-r}2} \sum_{\substack{\ell \in \mathcal{M}_m \cap \mathcal{D}\\ | \ell_1| \leq N }}  |\partial_{u_{-k}}  u^{\ell} |   |\ell_{3}\dots \ell_{\last}|^{s}.
\end{multline*}
Consequently, applying a triangle inequality for $\|\cdot\|_{\dot{H}^s}$, we have
\begin{equation}
\label{temp:eqy1}
\| y^{(1)}\|_{\dot{H}^s} \lesssim_{s,C^{(str)}_\Gamma}  C^{(\infty)}_\Gamma  \sum_{m=r}^{C^{(de)}_\Gamma} (N^{-21  C^{(de)}_\Gamma  } \gamma \| u \|_{\dot{H}^s}^2)^{-\frac{m-r}2} \|z^{(m)}\|_{L^2}
\end{equation}
 where
$$
z^{(m)}_k = \sum_{\substack{\ell \in \mathcal{M}_m \cap \mathcal{D}\\ | \ell_1| \leq N }}  |\partial_{u_{-k}}  u^{\ell} |   |\ell_{3}\dots \ell_{m}|^{s} |k|^s.
$$
Naturally, we aim at controlling $\|z^{(m)}\|_{L^2}$. We observe that if $-k$ is not one of the coordinates of $\ell$ then $\partial_{u_{-k}}  u^{\ell} =0$. Consequently, we have
$$
|z^{(m)}_k| \leq \sum_{j=1}^m  \sum_{\substack{\ell \in \mathcal{M}_m \cap \mathcal{D}\\ k = -\ell_j}}  |\ell_{3}\dots \ell_{m}|^{s} |k|^{s}  \prod_{i\neq j} |u_{\ell_i}|.
$$
Then, we observe that $ |\ell_{3}\dots \ell_{m}| |k| \leq (m-1) |\ell_1\dots \ell_m|/|\ell_j|$.
Indeed, if $j\neq 1$, it is clear since $\ell\in \mathcal{D}$ whereas if $j=1$ we use that $\ell \in \mathcal{M}_m$. As a consequence, we have
$$
|z^{(m)}_k|  \lesssim_{m, s}  \sum_{\substack{\ell_1+\dots+\ell_{m-1} = k \\ \forall j, |\ell_j| \leq N}}  \prod_{i\neq j} |u^{\ell_i}| |\ell_i|^s  \lesssim_{m,s} N^{m-1}  \sum_{\substack{\ell_1+\dots+\ell_{m-1} = k}}  \prod_{i\neq j} |u_{\ell_i}| |\ell_i|^{s-1}.
$$
Thus, using the Young inequality  leads to
$$
\|z^{(m)}\|_{L^2} \lesssim_m N^{m-1} \| u\|_{\dot{H}^s}^{m-1}.
$$
It follows of \eqref{temp:eqy1} that
\begin{equation*}
\begin{split}
\| y^{(1)}\|_{\dot{H}^s} &\lesssim_{s,C^{(str)}_\Gamma}  C^{(\infty)}_\Gamma  \sum_{m=r}^{C^{(de)}_\Gamma} N^{m-1} (N^{-21  C^{(de)}_\Gamma  } \gamma \| u \|_{\dot{H}^s}^2)^{-\frac{m-r}2} \| u\|_{\dot{H}^s}^{m-1}\\
&\lesssim_{s,C^{(str)}_\Gamma} C^{(\infty)}_\Gamma  \sqrt{\gamma}^{-C^{(de)}_\Gamma+r-2} N^{12 \,(C^{(de)}_\Gamma)^2}  \| u \|_{\dot{H}^s}^{r-1}.
\end{split}
\end{equation*}

These estimates on $\| y^{(1)}\|_{\dot{H}^s}$ and $\| y^{(2)}\|_{\dot{H}^s} $ give the estimate \eqref{est:rat:VF} on $\| \nabla \Gamma_N(u) \|_{\dot{H}^s} $. We don't detail the proof of the estimate \eqref{est:rat:dVF}. Indeed, the number of terms appearing naturally in the expression of $\mathrm{d}\nabla \Gamma_N(u)$ is huge, however, it is clear that all of them can be controlled as we have estimated the terms of $\| \nabla \Gamma_N(u) \|_{\dot{H}^s}$. 
\end{proof}

Let us consider the particular case of integrable vector fields.
\begin{proposition}
\label{prop:VFtheta}
If $u\in \mathcal{U}_{\gamma,N,\rho}^{\mathcal{E},s}$ and $Z \in \mathscr{A}_{r}^{\mathcal{E}}$ are such that $\|u\|_{\dot{H}^s}\lesssim 1$, $N\geq C^{(em)}_Z$ and $\rho \geq C^{(de)}_Z $ then we have
\begin{equation}
\label{est:rat:VFthetha}
\sup_{k\in \mathbb{N}^*} k |\partial_{I_k} Z_N(I)| \lesssim_{s,C^{(str)}_Z}  C^{(\infty)}_Z  \sqrt{\gamma}^{-\rho+r-2} N^{1+12 \,\rho^2}  \| u \|_{\dot{H}^s}^{r-2}.
\end{equation}
\end{proposition}
\begin{proof} Since $N\geq C^{(em)}_Z$, $Z_{ N}$ only depends on the variables $I_1,\dots,I_N$. Thus, if $k>N$ then $\partial_{I_k} Z_{ N}(I)=0$. Consequently, the supremum in \eqref{est:rat:VFthetha} only holds on $k\in \llbracket 1,N\rrbracket$ and it is enough to establish upper bounds on $|\partial_{I_k} Z_{ N}(I)|$  uniformly with respect to $k$.

We recall that $Z_N(u)$ is given by the definition \ref{def:ev:frac}. Naturally, $Z_N(u)$ is of the form
$$
Z_N(u) = \sum_{\substack{(\ell,\mathbf{h},\mathbf{k},n,c) \in Z}} c \, u^{\ell} f_{\ell,\mathbf{h},\mathbf{k},n,c}(I),
$$
where $f_{\ell,\mathbf{h},\mathbf{k},n,c}(I)$ denotes the denominator. Note that the sum holds on indices $\ell$ such that $\Irr\, \ell = \emptyset$ because $Z$ is integrable (i.e. $Z \in \mathscr{A}_{r}^{\mathcal{E}}$ { see Definition \ref{def:Aint}}) and thus $u^{\ell}$ is a polynomial in the actions.

Naturally, for $k\in \mathbb{N}^*$, we have
$$
 (\partial_{I_k} Z_N(u))_k  =  \sum_{(\ell,\mathbf{h},\mathbf{k},n,c) \in \Gamma} c \, (\partial_{I_k}  u^{\ell} )f_{\ell,\mathbf{h},\mathbf{k},n,c}(I) + \sum_{(\ell,\mathbf{h},\mathbf{k},n,c) \in \Gamma} c \,   u^{\ell}  \partial_{I_k}  f_{\ell,\mathbf{h},\mathbf{k},n,c}(I) =: \Theta_k + S_k.
$$
We note that we have already estimated $|S_k|$ uniformly with respect to $k$ in the proof of Proposition \ref{prop:VF} (see the definition of $S$ in \eqref{def:S}). Consequently, we know that
$$
\sup_{k\in \mathbb{N}^*} |S_k| \lesssim_{s,C^{(str)}_Z}  C^{(\infty)}_Z  \sqrt{\gamma}^{-\rho+r-2} N^{12 \,\rho^2}  \| u \|_{\dot{H}^s}^{r-2}
$$
Therefore, we just have to control $\Theta_k$. Using the upper bound \eqref{eq:prodden} on $f_{\ell,\mathbf{h},\mathbf{k},n,c}(I)$ and realizing estimates very similar to the ones of Proposition \ref{prop:VF}, we have
\begin{equation*}
\begin{split}
|\Theta_k| &\lesssim_{s,C^{(str)}_Z} C^{(\infty)}_{Z} \sum_{\substack{r\leq m\leq C_Z^{(de)}\\ m \ \mathrm{even}}} (N^{-21  C^{(de)}_Z  } \gamma \| u \|_{\dot{H}^s}^2)^{-\frac{m-r}2} \sum_{\substack{\ell \in \mathcal{M}_m \cap \mathcal{D}\\ \Irr \, \ell = \emptyset \\ |\ell_1|\leq N}} |\partial_{I_k}  u^{\ell} | |\ell_{3}\dots \ell_{\last}|^{s} \\
&\lesssim_{s,C^{(str)}_Z} C^{(\infty)}_{Z} \sum_{\substack{r\leq m \leq C_Z^{(de)}\\ m \ \mathrm{even}}} (N^{-21  C^{(de)}_Z  } \gamma \| u \|_{\dot{H}^s}^2)^{-\frac{m-r}2} \sum_{k\in \mathcal{D}_{(m-1)/2}} \prod_{j=1}^{(m-1)/2} I_{k_j} |k_j|^{2s} \\
&\lesssim_{s,C^{(str)}_Z} C^{(\infty)}_{Z} \sum_{\substack{r\leq m\leq C_Z^{(de)}\\ m \ \mathrm{even}}} (N^{-21  C^{(de)}_Z  } \gamma \| u \|_{\dot{H}^s}^2)^{-\frac{m-r}2} \|u\|_{\dot{H}^s}^{m-1}\\
&\lesssim_{s,C^{(str)}_Z} C^{(\infty)}_Z  \sqrt{\gamma}^{-\rho+r-2} N^{12 \,\rho^2}  \| u \|_{\dot{H}^s}^{r-2}.
\end{split}
\end{equation*}
\end{proof}

Now, we focus on the control of the Lie transforms associated with rational Hamiltonians. 
\begin{proposition} 
\label{prop:Lie}
Let $r\geq 3$, $s>1$, $\Gamma \in \mathscr{H}_{r}^{\mathcal{E}}$, $\rho \geq C^{(de)}_\Gamma $, $N\gtrsim_{\rho,s} 1$ satisfying $N\geq C^{(em)}_\Gamma$. If $\varepsilon_0>0$ satisfies
\begin{equation}
\label{cfl_Lie}
\varepsilon_0^{r-11/4} \lesssim_{s,C^{(str)}_\Gamma}  (C^{(\infty)}_\Gamma)^{-1}  \sqrt{\gamma}^{\rho-r+2} N^{-1-14 \,\rho^2}  \ \ \mathrm{and} \ \ 4\, \varepsilon_0^{1/4} \leq 2 \gamma^2  N^{-22  \rho}
\end{equation}
the flow of the Hamiltonian system
\begin{equation}
\label{sys_Lie}
\partial_t u = \partial_x \nabla \Gamma_N(u),
\end{equation}
denoted by $\Phi_ {\Gamma_N}^t$, defines, for $0\leq t\leq 1$, a family of symplectic maps from $\mathcal{U}_{\gamma,N,\rho}^{\mathcal{E},s}\cap B_s(0,\varepsilon_0)$ to $\dot{H}^s$. Furthermore, for $u\in \mathcal{U}_{\gamma,N,\rho}^{\mathcal{E},s} \cap B_s(0,\varepsilon_0)$ and $t\in (0,1)$, we have the estimates
$$
\|\Phi_ {\Gamma_N}^t(u) - u\|_{\dot{H}^s} \leq \| u \|_{\dot{H}^s}^{7/4} \ \ \mathrm{and} \ \
\|(\mathrm{d}\Phi_ {\Gamma_N}^t (u))^{-1}\|_{\mathscr{L}(\dot{H}^s)} \leq 2.
$$
\end{proposition}
\begin{proof}
A priori, the system \eqref{sys_Lie} looks like a partial differential equation. However, the Hamiltonian $\Gamma_N$ only involves modes associated with indices smaller than $N$. Thus, \eqref{sys_Lie} is just an ordinary differential equation associated with a smooth vector field (since it is a rational fraction). Consequently, by the Cauchy-Lipschitz theorem, the flow of \eqref{sys_Lie} is obviously locally well defined and is a smooth function. Furthermore, since \eqref{sys_Lie} is a Hamiltonian system, its flow is naturally symplectic. The non obvious fact is that if $u\in \mathcal{U}_{\gamma,N,\rho}^{\mathcal{E},s} \cap B_s(0,\varepsilon_0)$ then $\Phi_ {\Gamma_N}^t(u)$ is well defined until $t=1$. In other words, we have to prove that the solution of \eqref{sys_Lie} cannot explose if $t\in (0,1)$.

We are going to prove that if $u\in \mathcal{U}_{\gamma,N,\rho}^{\mathcal{E},s} \cap B_s(0,\varepsilon_0)$ and $t_0\in (0,1)$ are such that for $t\in (0,t_0)$, $\Phi_ {\Gamma_N}^t(u) \in \mathcal{U}_{\gamma/3,N,\rho}^{\mathcal{E},s}$ and $\|\Phi_ {\Gamma_N}^t(u)\|_{\dot{H}^s}\leq3\|u\|_{\dot{H}^s}$ then for $t\in (0,t_0)$, $\Phi_ {\Gamma_N}^t(u) \in \mathcal{U}_{\gamma/2,N,\rho}^{\mathcal{E},s}$ and $\|\Phi_ {\Gamma_N}^t(u)\|_{\dot{H}^s}\leq2\|u\|_{\dot{H}^s}$. By this bootstrap argument, we will have naturally that $\Phi_ {\Gamma_N}^t(u)$ is well defined for $t\in (0,1)$, $\Phi_ {\Gamma_N}^t(u) \in \mathcal{U}_{\gamma/2,N,\rho}^{\mathcal{E},s}$ and $\|\Phi_ {\Gamma_N}^t(u)\|_{\dot{H}^s}\leq 2\|u\|_{\dot{H}^s}$.

We assume that $t_0\in (0,1)$ is such that for  $t\in (0,t_0)$,  $\|\Phi_ {\Gamma_N}^t(u)\|_{\dot{H}^s}\leq3\|u\|_{\dot{H}^s}$ and $\Phi_ {\Gamma_N}^t(u) \in \mathcal{U}_{\gamma/3,N,\rho}^{\mathcal{E},s}$. Applying the Proposition \ref{prop:VF}, we deduce that for $t\in (0,t_0)$
$$
\|\partial_t \Phi_ {\Gamma_N}^t(u)\|_{\dot{H}^s} \lesssim_{\rho,s,C^{(str)}_\Gamma}  C^{(\infty)}_\Gamma  \sqrt{\gamma}^{-\rho+r-2} N^{1+12 \,\rho^2}  \| u \|_{\dot{H}^s}^{r-1}.
$$
Consequently, we have
\begin{equation}
\label{unepartieduresultat}
\| \Phi_ {\Gamma_N}^t(u) - u \|_{\dot{H}^s}  \lesssim_{\rho,s,C^{(str)}_\Gamma}  C^{(\infty)}_\Gamma  \sqrt{\gamma}^{-\rho+r-2} N^{1+12 \,\rho^2}  \| u \|_{\dot{H}^s}^{r-1}\mathop{\leq}^{\eqref{cfl_Lie}} \| u \|_{\dot{H}^s}^{7/4}.
\end{equation}
By applying the triangle inequality, we deduce that if $\varepsilon_0\leq 1$ then $\|\Phi_ {\Gamma_N}^t(u)\|_{\dot{H}^s}\leq2\|u\|_{\dot{H}^s}$. Furthermore, if $\ell \in \mathbb{N}^*$ we have
\begin{multline*}
\big| |(\Phi_ {\Gamma_N}^t(u))_{\ell}|^2 - |u_\ell|^2 \big| \ell^{2s-2} \leq \| \Phi_ {\Gamma_N}^t(u) - u\|_{\dot{H}^s}(\| \Phi_ {\Gamma_N}^t(u) \|_{\dot{H}^s} +\| u\|_{\dot{H}^s} )\\ \leq 4 \| u \|_{\dot{H}^s}^{11/4} \mathop{\leq}^{\eqref{cfl_Lie}} 2 \gamma^2  N^{-22  \rho} \| u \|_{\dot{H}^s}^2.
 \end{multline*}
Consequently, by applying the Proposition \ref{prop:stab_U} and using that $u\in\mathcal{U}_{\gamma,N,\rho}^{\mathcal{E},s}$, we deduce that $\Phi_ {\Gamma_N}^t(u) \in \mathcal{U}_{\gamma/2,N,\rho}^{\mathcal{E},s}$.

Finally, we have to prove that $\mathrm{d}\Phi_ {\Gamma_N}^t (u)$ is invertible and to estimate its norm. Differentiating \eqref{sys_Lie}, we have
\begin{equation}
\label{eq:temp:Lie}
\partial_t \mathrm{d}\Phi_ {\Gamma_N}^t (u)  = \partial_x \mathrm{d}\nabla \Gamma_N(\Phi_ {\Gamma_N}^t(u))\mathrm{d}\Phi_ {\Gamma_N}^t (u).
\end{equation}
However, since $N\geq C^{(em)}_\Gamma$, $\Gamma_N$ only depends on modes with indices smaller than $N$. Consequently, in \eqref{eq:temp:Lie}, $\partial_x$ can be replaced by $\mathbb{1}_{|\partial_x|\leq N}\partial_x$. Consequently, by applying the estimate on $\mathrm{d}\nabla \Gamma_N(\Phi_ {\Gamma_N}^t(u))$ given by Proposition \ref{prop:VF}, we deduce that
\begin{multline}
\label{est:acc:diff}
\|\partial_t \mathrm{d}\Phi_ {\Gamma_N}^t (u)\|_{\mathscr{L}(\dot{H}^s)}  \lesssim_{\rho,s,C^{(str)}_\Gamma} \| \mathrm{d}\Phi_ {\Gamma_N}^t (u)\|_{\mathscr{L}(\dot{H}^s)}  C^{(\infty)}_\Gamma  \sqrt{\gamma}^{-\rho+r-2} N^{1+14 \,\rho^2}  \| u \|_{\dot{H}^s}^{r-2} \\
\mathop{\leq}^{\eqref{cfl_Lie}} \log(4/3)\| \mathrm{d}\Phi_ {\Gamma_N}^t (u)\|_{\mathscr{L}(\dot{H}^s)}.
\end{multline}
Thus, the Gr\"onwall Lemma proves that
$$
\forall t\in(0,1), \| \mathrm{d}\Phi_ {\Gamma_N}^t (u)\|_{\mathscr{L}(\dot{H}^s)} \leq \frac43. 
$$
As a consequence, we deduce of \eqref{est:acc:diff} that
$$
\| \mathrm{d}\Phi_ {\Gamma_N}^t (u) - \mathrm{Id}_{\dot{H}^s} \|_{\mathscr{L}(\dot{H}^s)} = \| \mathrm{d}\Phi_ {\Gamma_N}^t (u) - \mathrm{d}\Phi_ {\Gamma_N}^0 (u) \|_{\mathscr{L}(\dot{H}^s)} \leq  \frac43\log(\frac43)  \leq \frac49\leq \frac12.
$$
Consequently, since $\mathscr{L}(\dot{H}^s)$ is a Banach space, $\mathrm{d}\Phi_ {\Gamma_N}^t (u)$ is invertible and the norm of its invert is smaller than or equal to $2$. 
\end{proof}

\section{The rational normal form}\label{section6}
This section is devoted to the proof of the following theorem which is our main normal form result. In this section, we set
\begin{equation}
\label{def:rhor}
\rho_r = 47r -314
\end{equation}
the constant which ensures that if $r\geq 7$ and $\Gamma \in \mathscr{H}_r^{(6),\mathcal{E}} \supset \mathscr{H}_r^{(4),\mathcal{E}}$ then $C^{(de)}_{\Gamma} \leq \rho_r$ (see Remark \ref{rem:bound:deg}).

\begin{theorem}[Rational normal form] \label{thm:KtA}
Being given $\mathcal{E}\in \{ \mathrm{\ref{gKdV}}, \mathrm{\ref{gBO}}\}$, $r\gg 7 $, $s\geq 1$, $N\gtrsim_{r,s} 1$, $\gamma \lesssim_{r,s} 1$ and $\varepsilon_0 \lesssim_{r,s} 1$ satisfying 
\begin{equation}
\label{ultimate:CFL} 
\varepsilon_0 \leq 4^{36} \gamma^{35} \ \mathrm{and} \ \varepsilon_0 \leq N^{-10^5 r}
\end{equation}
there exist four symplectic maps $\tau^{(0)},\dots,\tau^{(3)}$ preserving the $L^2$ norm and making the following diagram  commutative
\begin{equation*}
 \xymatrix{ &  & V_{2}  \ar[rd]^{\tau^{(2)}} & & \\
				& V_{1}  \ar[ru]^{\tau^{(1)}} \ar@{^{(}->}[rr]^{\mathrm{id}_{\dot{H}^s}} & &  V_3  \ar[rd]^{\tau^{(3)}} & \\
						V_{0}  \ar[ru]^{\tau^{(0)}} \ar@{^{(}->}[rrrr]^{\mathrm{id}_{\dot{H}^s}}&  & &   & \dot{H}^s}
\end{equation*}
where $V_{\sigma} = B_s(0,2^\sigma \varepsilon_0)\cap \mathcal{U}_{2^{-\sigma} \gamma,N^3,\rho_{2r}}^{\mathcal{E},s}$
and close to the identity 
\begin{equation}
\label{est:close:138}
\forall \sigma\in \{0,\dots,3\}, \forall u\in V_{\sigma}, \  \|\tau^{(\sigma)}(u)-u\|_{\dot{H}^s} \leq \|u\|_{\dot{H}^s}^{13/8}
\end{equation}
such that $H_{\mathcal{E}} \circ \tau^{(3)} \circ \tau^{(2)}$ writes
$$
H_{\mathcal{E}} \circ \tau^{(3)} \circ \tau^{(2)}= Z_2^{\mathcal{E}} + Z_4^{\mathcal{E}} + \sum_{m=6}^r Z^{(m)}_{N^3} + \mathrm{R^{(res)}}\circ \tau^{(2)} + \mathrm{R^{(rat)}}
$$
where $Z_2^{\mathcal{E}}$ is given by \eqref{def:Z2}, $Z_4^{\mathcal{E}}$ is given by \eqref{Z4-KdV} and \eqref{Z4-BO}, $\mathrm{R^{(res)}} = \mathrm{R}^{(\mu_3>N)} + \mathrm{R}^{(I_{> N^3})}+ \mathrm{R}^{(or)}$ is the sum of the remainder terms of the resonant normal form (see Theorem \ref{thm-BNF}), $\mathrm{R^{(rat)}}$ is of order $r+1$, i.e.
\begin{equation}
\label{est:rem:thm:KtA}
\forall u \in V_2, \ \|\partial_x \nabla \mathrm{R^{(rat)}}(u)\|_{\dot{H}^s} \lesssim_{s,r} N^{10^5 r^2} \gamma^{-23 r + 133}  \| u \|_{\dot{H}^s}^{r}
\end{equation}
 and\footnote{Of course $Z^{(m)}$, as well as $\mathrm{R^{(rat)}}$ and $\mathrm{R^{(res)}}$,  depend on $\E$ but the notations are already heavy enough! } $Z^{(6)}\in \mathscr{A}_6^{(4),\mathcal{E}}$, $Z^{(m)}\in \mathscr{A}_m^{(6),\mathcal{E}}$, for $m\geq 7$, are some integrable Hamiltonians such that 
$$
\forall m\geq 6, \ C_{Z^{(m)}}^{(em)} \leq N^3, \ \ \ C_{Z^{(m)}}^{(str)} \lesssim_m 1, \ \ \ C_{Z^{(m)}}^{(\infty)} \lesssim_m N^{321 m}.
$$ 
Furthermore, $\tau^{(2)}$ preserves the high modes, i.e.
\begin{equation}
\label{idhighmodes}
\forall u \in V_2,\forall \ell \in \mathbb{Z}^*, \ |\ell|>N^3  \Rightarrow \ (\tau(u))_\ell = u_{\ell}
\end{equation}
and its differential is invertible and satisfies the estimate
$$
\forall u \in V_2, \ \|(\mathrm{d} \tau^{(2)}(u))^{-1} \|_{\mathscr{L}(\dot{H}^s)} \lesssim_r 1.
$$
\end{theorem}

\begin{remark}
\label{rem:rgg7}
The assumption $r\gg 7 $ means that $r$ has to be larger than a universal constant that we do not try to determine. This assumption is only useful to ensure that if \eqref{ultimate:CFL} is satisfied then many conditions of the kind $\|u\|_{\dot{H}^s} \lesssim_{r,s} N^{-\alpha_r} \gamma^{\beta_r}$ are clearly satisfied (because it is enough to consider the dominant exponent with respect to $r$).
\end{remark}

We are going to prove this Theorem in three steps. The first one, essentially realized in the Theorem \ref{thm-BNF} consists in constructing the maps $\tau^{(0)}$ and $\tau^{(3)}$ to remove all the non-resonant monomials of order less than $r+1$. Then, we will remove the non integrable resonant monomials of order $5$ and $6$ by computing some averages with respect to the dynamics of $Z_4^{\mathcal{E}}$. Finally, we will remove all the resonant non integrable terms of order less than $r+1$ by computing some averages with respect to the dynamics of $Z_4^{\mathcal{E}}+Z_{6,N^3}^{\mathcal{E}}$.

The property \eqref{idhighmodes} is a direct byproduct of the following proof. Indeed, $\tau^{(2)}$ is designed as the composition of Hamiltonian flows with Hamiltonians depending on modes of indices smaller than $N^3$. Consequently, in the proof, we do not pay attention to \eqref{idhighmodes}.

Similarly, in the proof it is clear that the maps $\tau^{(0)},\dots,\tau^{(3)}$ preserve the $L^2$ norm. Indeed, they are constructed by composition of Hamiltonian flows that preserve the $L^2$ norms because the Hamiltonians are polynomials or rational fractions whose numerators are of the form $u^k$ with $k\in \mathcal{M}$.

\subsection{The resonant normal form.} To prove Theorem \ref{thm:KtA}, the first step consist in putting the system in resonant normal form which has been done in Theorem \ref{thm-BNF}. Provided that $\varepsilon_0 \lesssim_{r,s} N^{-3}$ is small enough (which is ensured by the assumption \eqref{ultimate:CFL}), by applying Theorem \ref{thm-BNF}, we get symplectic maps $\tau^{(0)} : B_s(0,4\varepsilon_0) \to B_s(0,8\varepsilon_0)$ and $\tau^{(3)} : B_s(0,8\varepsilon_0) \to \dot{H}^s$ (denoted $\tau^{(3)}$ in Theorem \ref{thm-BNF})  such that $\tau^{(3)}\circ \tau^{(0)} = \mathrm{id}_{\dot{H}^s}$ and
$$
H_{\mathcal{E}} \circ \tau^{(3)} = Z_2^{\mathcal{E}} + Z_4^{\mathcal{E}} + Z_{6,\leq N^3}^{\mathcal{E}} + \mathrm{Res}_{\leq N^3} + \mathrm{R}^{(\mu_3>N)} + \mathrm{R}^{(I_{> N^3})}+ \mathrm{R}^{(or)}
$$
where the different terms are precisely described in the statement of the Theorem \ref{thm-BNF}. Furthermore, by \eqref{tau}, we know that the maps $\tau^{(0)},\tau^{(3)}$ are closed to the identity. Provided that $\varepsilon_0^{3/8} \lesssim_{r} N^{-3}$ (which is ensured by the assumption \eqref{ultimate:CFL}), we deduce that they satisfy \eqref{est:close:138}.

Finally, we have to prove that if $\tau^{(0)}$ is restricted to $V_0$ then its takes its values in $V_1$. On the one hand, since $\varepsilon_0\lesssim 1$, we deduce of \eqref{est:close:138} that for $u\in V_0$, $\|\tau^{(0)}(u)\|_{\dot{H}^s}\leq 2 \|u\|_{\dot{H}^s} \leq 2 \varepsilon_0$. On the over hand, provided that $\varepsilon_0^{2} N^3 \lesssim_{r,s} \gamma^2 N^{-3\cdot 22\rho_{2r}}$, (which is ensured by the assumption \eqref{ultimate:CFL}), we deduce of Proposition \ref{prop:stab_U} that $\tau^{(0)}(u) \in \mathcal{U}_{ \gamma/2,N^3,\rho_{2r}}^{\mathcal{E},s}$ for $u\in V_0$. Consequently, $\tau^{(0)}$ maps $V_0$ into $V_1$.

\subsection{The two first rational steps: resolution of the quintic and sextic terms.}
Let us decompose $\mathrm{Res}_{\leq N^3} + Z_{6,\leq N^3}^{\mathcal{E}} $ as a sum of homogeneous polynomials
$$
\mathrm{Res}_{\leq N^3} + Z_{6,\leq N^3}^{\mathcal{E}}  = P_5 + P_6+ \dots + P_r
$$
where, as stated in Theorem \ref{thm-BNF}, the polynomials $P_m$, $m\geq 5$, are of the form 
\begin{equation}
\label{def:Pm}
P_m(u) = \sum_{\substack{k \in \mathcal{R}^{\mathcal{E}}_m \cap \mathcal D  \\ |k_1| \leq N^3 \\}} c_k^{(m)} \, u^k \ \ \mathrm{with} \  |c_k^{(m)}| \lesssim_{m}   N^{3m}.
\end{equation}
In this subsection and the following, we aim at removing these polynomials.

\subsubsection{Elimination of the quintic term} Following the strategy introduced in \cite{KillBill}, to remove $P_5$, we consider the solution $\chi_5$ of the homological equation $\{ \chi_5 , Z_4^{\mathcal{E}}\} = -P_5$ 
that is
\begin{equation}
\label{def:chi5}
\chi_5(u) := \sum_{\substack{k \in \mathcal{R}^{\mathcal{E}}_5 \cap \mathcal D  \\ |k_1| \leq N^3 \\}} c_k^{(5)} \, \frac{u^k}{2i\pi (k_1 \partial_{I_{k_1}} Z_4^{\mathcal{E}}(I) + \dots + k_5 \partial_{I_{k_5}} Z_4^{\mathcal{E}}(I))} =  \sum_{\substack{k \in \mathcal{R}^{\mathcal{E}}_5 \cap \mathcal D  \\ |k_1| \leq N^3 \\}} c_k^{(5)} \, \frac{u^k}{2i\pi \Delta_k^{(4),\mathcal{E}}}.
\end{equation}
Note that $\chi_5$ is a smooth function well defined on $V_{11/4}$ (we choose $11/4$, but any value strictly less than three would work). Furthermore, naturally, there exists $\Gamma^{(5)} \in \mathscr{H}_3^{(4),*,\mathcal{E}}$ such that 
$$\Gamma^{(5)}_{N^3} = \chi_5 \ \ \  \mathrm{with} \ \ \ C^{(\infty)}_{\Gamma^{(5)}} \lesssim N^{15}, \ C^{(em)}_{\Gamma^{(5)}} \leq N^3, \ C^{(str)}_{\Gamma^{(5)}}\lesssim 1, \ C^{(de)}_{\Gamma^{(5)}} = 5.$$
  
Provided that $\varepsilon_0^{3-11/4} \lesssim_{s}  N^{-15}  \gamma^{2} (N^3)^{-1-14 \cdot 25}$ and $\varepsilon_0^{1/4} \lesssim  \gamma^2  N^{-3\cdot 22 \cdot 5}$ (which is ensured by the assumption \eqref{ultimate:CFL}), Proposition \ref{prop:Lie}  { (applied with $r=3$, $\rho=5$) } proves that the Hamiltonian flows generated by $\pm \chi_5$ are well defined on $V_{11/4}$ for $t\in (0,1)$ and that these flows are closed to the identity
$$
\forall t \in (0,1),\forall u \in V_{11/4}, \ \|\Phi_ {\pm \chi_5}^t(u) - u\|_{\dot{H}^s} \leq \| u \|_{\dot{H}^s}^{7/4} \ \ \mathrm{and} \ \
\|(\mathrm{d}\Phi_ {\pm \chi_5}^t (u))^{-1}\|_{\mathscr{L}(\dot{H}^s)} \leq 2.
$$

Provided that $\varepsilon_0^{3/4} \lesssim \gamma^2  N^{-3\cdot22  \rho_{2r}}$ (which is ensured by the assumption \eqref{ultimate:CFL}), Proposition \ref{prop:stab_U} proves that, for $t\in (0,1)$, $\Phi_ {\chi_5}^t$ maps $V_{11/4}$ in $V_{3}$ and $\Phi_ {- \chi_5}^t$ maps $V_{1}$ in $V_{5/4}$.

Recalling that $Z_2^{\mathcal{E}}$ commutes with $\chi_5$ (see Remark \ref{rem:commuteZ2}), we have that on $V_{11/4}$ 
\begin{equation*}
H_{\mathcal{E}} \circ \tau^{(3)} \circ \Phi_ {\chi_5}^1 = Z_2^{\mathcal{E}}+Z_4^{\mathcal{E}} \circ \Phi_ {\chi_5}^1 + \sum_{m=5}^r P_m\circ \Phi_ {\chi_5}^1 + \mathrm{R^{(res)}} \circ \Phi_ {\chi_5}^1.
\end{equation*}
where $\mathrm{R^{(res)}} = \mathrm{R}^{(\mu_3>N)} + \mathrm{R}^{(I_{> N^3})}+ \mathrm{R}^{(or)}$ (see Theorem \ref{thm-BNF}) is the sum of the remainder terms of the resonant normal form. Then realizing a Taylor expansion of $P_m\circ \Phi_ {\chi_5}^t$ with respect to $t$, we have on $V_{11/4}$ 
$$
P_m\circ \Phi_ {\chi_5}^1 = P_m + \sum_{j=1}^{r-m} \frac1{j!} \mathrm{ad}_{\chi_5}^j P_m + \int_0^1 \frac{(1-t)^{r-m}}{(r-m)!} \mathrm{ad}_{\chi_5}^{r-m+1} P_m \circ \Phi_{\chi_5}^t  \ \mathrm{d}t
$$
and recalling that by construction  $\{ \chi_5 , Z_4^{\mathcal{E}}\} = -P_5$, we have
$$
Z_4^{\mathcal{E}} \circ \Phi_ {\chi_5}^1 = Z_4^{\mathcal{E}} - P_5 - \sum_{j=2}^{r-4} \frac1{j!} \mathrm{ad}_{\chi_5}^{j-1} P_5 - \int_0^1 \frac{(1-t)^{r-4}}{(r-4)!} \mathrm{ad}_{\chi_5}^{r-4} P_5 \circ \Phi_{\chi_5}^t \ \mathrm{d}t.
$$
Consequently, we have
\begin{equation}
\label{normform5}
H_{\mathcal{E}} \circ \tau^{(3)} \circ \Phi_ {\chi_5}^1 = Z_2^{\mathcal{E}}+Z_4^{\mathcal{E}}  + \sum_{m=6}^r Q_m^{(5)} + R^{(rat),5}+ \mathrm{R^{(res)}} \circ \Phi_ {\chi_5}^1
\end{equation}
where we have set
\begin{equation}
\label{def:Qm5}
Q_m^{(5)} = \sum_{p+q = m} \frac1{p !} \mathrm{ad}_{\chi_5}^p P_q - \frac1{(m-4) !} \mathrm{ad}_{\chi_5}^{m-5} P_5
\end{equation}
and
$$
R^{(rat),5} =  \sum_{m=5}^r \int_0^1 \frac{(1-t)^{r-m}}{(r-m)!} \mathrm{ad}_{\chi_5}^{r-m+1} P_m \circ \Phi_{\chi_5}^t \ \mathrm{d}t -  \int_0^1 \frac{(1-t)^{r-4}}{(r-4)!} \mathrm{ad}_{\chi_5}^{r-4} P_5 \circ \Phi_{\chi_5}^t \ \mathrm{d}t.
$$
Since $P_m$ can be identified with a rational fraction $\Upsilon^{(m)} \in \mathscr{H}_m^{(4),\mathcal{E}}$ such that
$$
P_m = \Upsilon^{(m)}_{N^3} \ \mathrm{and} \ C^{(\infty)}_{\Upsilon^{(m)}} \lesssim N^{3m}, \ C^{(em)}_{\Upsilon^{(m)}} \leq N^3, \ C^{(str)}_{\Upsilon^{(m)}}\lesssim_m 1, \ C^{(de)}_{\Upsilon^{(m)}} = m,
$$ 
by applying Proposition \ref{prop:stab_frac} and Proposition \ref{prop:stab:subclasses}, for all $m\geq 6$ there exists $\Xi^{(m),5} \in  \mathscr{H}_m^{(4),\mathcal{E}}$ such that 
\begin{equation}
\label{Q5intheclasse}
Q_m^{(5)}  = \Xi^{(m),5}_{N^3} \ \ \  \mathrm{with} \ \ \ C^{(\infty)}_{\Xi^{(m),5}} \lesssim_m N^{18m}, \ C^{(em)}_{\Xi^{(m),5}} \leq N^3, \ C^{(str)}_{\Xi^{(m),5}}\lesssim_m 1.
\end{equation}
Finally, we refer the reader to the subsection \ref{app:sub:rem} of the Appendix for the control of the remainder terms which leads to\footnote{ We could get a better estimate here, since the orders are smaller than 5, but it wouldn't do any good in the end.}
\begin{equation}
\label{control_R5rat}
\|R^{(rat),5}\|_{\dot{H}^s}\lesssim_{s,r} N^{321 (r+1) -2049}  \sqrt{\gamma}^{-\rho_{r+1}+r-1} N^{12 \,(\rho_{r+1})^2}  \| u \|_{\dot{H}^s}^{r}. 
\end{equation}
\subsubsection{Elimination of the sextic term} Now, we aim at removing the non-integrable part of $Q_6^{(m)}$ in the expansion \eqref{normform5} of $H_{\mathcal{E}} \circ \tau^{(3)} \circ \Phi_ {\chi_5}^1$. 

First, let us detail precisely the structure of $Q_6^{(m)}$. By definition of $Q_6^{(m)}$ (see \eqref{def:Qm5}) and $P_6$ it writes
$$
Q_6^{(m)} = P_6 + \frac12 \{ \chi_5,P_5\}.
$$
By a direct but tedious calculation, $\{ \chi_5,P_5\}$ writes on the form
\begin{equation}
\label{jmelance}
\{ \chi_5,P_5\} = \sum_{\substack{k \in \mathcal{R}^{\mathcal{E}}_8 \cap \mathcal D  \\ |k_1| \leq N^3 \\}} \sum_{\substack{\ell \in \mathcal{R}^{\mathcal{E}}_5 \\ (\ell_j)_{1\leq j\leq 4}\in \mathrm{ss}_4(k)  \\ |\ell_5|\leq N^3}} c_{\ell,k} \frac{u^k}{\Delta_{\ell}^{(4),\mathcal{E}}}  +  \sum_{\substack{\ell,h \in \mathcal{R}^{\mathcal{E}}_5 \cap \mathcal{D}\\ |h_1|\geq \kappa_\ell^{\mathcal{E}} \\ |h_1|,|\ell_1|\leq N^3} }  d_{\ell,h} \frac{u^h u^{\ell}}{(\Delta_{\ell}^{(4),\mathcal{E}})^2}  
\end{equation}
where $\mathrm{ss}_n(k)$ denotes the sub-sequences of $k$ of length $n$ and the coefficient $c_{\ell,k},d_{\ell,k}$ are such that $|c_{\ell,k}|,|d_{\ell,h}| \lesssim N^{3\cdot 3+15+15}=N^{39}$. The relation $|h_1|\geq \kappa_\ell^{\mathcal{E}}$ below could seem strange. Nevertheless, it comes from the computation of a Poisson bracket on the kind $\{ u^h, (\Delta_{\ell}^{(4),\mathcal{E}})^{-1}\}$. Indeed, recalling that by definition of $ \kappa_\ell^{\mathcal{E}}$, $\Delta_{\ell}^{(4),\mathcal{E}}$ is a linear function of actions of indices larger than or equal to  $\kappa_\ell^{\mathcal{E}}$, if we had $|h_1|< \kappa_\ell^{\mathcal{E}}$ then $u^h$ and $(\Delta_{\ell}^{(4),\mathcal{E}})^{-1}\}$ would commute.

In order to remove the non integrable part of $Q_6^{(m)}$, 
 we consider one solution $\chi_6$ of the homological equation $\{ \chi_6 , Z_4^{\mathcal{E}}\} = -\Pi_{\rm NI}Q_6^{(m)}$ where $\Pi_{\rm NI}$ denotes the projection on the non integrable part, 
that is

\begin{multline*}
\chi_6:=   \sum_{\substack{k \in \mathcal{R}^{\mathcal{E}}_6 \cap \mathcal D  \\ |k_1| \leq N^3 \\ \Irr \, k \neq \emptyset}} c_k^{(6)} \, \frac{u^k}{2i\pi \Delta_k^{(4),\mathcal{E}}} +  \sum_{\substack{k \in \mathcal{R}^{\mathcal{E}}_8 \cap \mathcal D  \\ |k_1| \leq N^3 \\ \Irr \ k \neq \emptyset}} \sum_{\substack{\ell \in \mathcal{R}^{\mathcal{E}}_5 \\ (\ell_j)_{1\leq j\leq 4}\in \mathrm{ss}_4(k)  \\ |\ell_5|\leq N^3}} \!\!\! \frac{u^k}{2i\pi \Delta_{\Irr \, k}^{(4),\mathcal{E}} \Delta_{\ell}^{(4),\mathcal{E}}}  c_{\ell,k} \\ +\sum_{\substack{\ell,h \in \mathcal{R}^{\mathcal{E}}_5 \cap \mathcal{D}\\ |h_1|\geq \kappa_\ell^{\mathcal{E}} \\ \Irr(\ell,h)\neq \emptyset \\ |h_1|,|\ell_1|\leq N^3}}  d_{\ell,h} \frac{u^h u^{\ell}}{2i\pi \Delta_{ \Irr(\ell,h)}^{(4),\mathcal{E}}(\Delta_{\ell}^{(4),\mathcal{E}})^2}  
\end{multline*}
where the coefficients $c_k^{(6)}$ are those of $P_6$ (see \eqref{def:Pm}) . Note that we have used that by Lemma \ref{int-order4} { all resonant term of order 4 are integrable and thus } if $k \in \mathcal{R}^{\mathcal{E}}_6$ and $\Irr \, k \neq \emptyset$ then $k$ is irreducible. By construction, it is clear that $\chi_6$ is a smooth function well defined on $V_{5/2}$ {  (since $11/4 >5/2$)}. Furthermore, naturally, there exists $\Gamma^{(6)} \in \mathscr{H}_4^{(4),*,\mathcal{E}}$ such that 
$$\Gamma^{(6)}_{N^3} = \chi_6 \ \ \  \mathrm{with} \ \ \ C^{(\infty)}_{\Gamma^{(6)}} \lesssim N^{39}, \ C^{(em)}_{\Gamma^{(6)}} \leq N^3, \ C^{(de)}_{\Gamma^{(6)}} = 10.$$ Nevertheless, contrary to the previous case, it is not completely obvious that $C^{(str)}_{\Gamma^{(6)}}\lesssim 1$. Indeed, we have to explain why $C^{(di)}_{\Gamma^{(6)}}\lesssim 1$.

First, note that this fact is obvious for the part of $\Gamma^{(6)}$ coming from the resolution of $P_6$. Then, we consider the part associated with $u^k / (\Delta_{\Irr \, k}^{(4),\mathcal{E}} \Delta_{\ell}^{(4),\mathcal{E}})$. Recalling that by Lemma \ref{int-order4} $\Irr \, k\in \mathrm{ss}_6(k)$, we have by Lemma \ref{lem:diot} that $\kappa_{\Irr \, k}^{\mathcal{E}} \lesssim |k_6|$. Furthermore, since $(\ell_j)_{1\leq j\leq 4}\in \mathrm{ss}_4(k)$, by Lemma \ref{lem:diot}, we have $\kappa_{\ell}^{\mathcal{E}} \lesssim |k_4|$. Consequently, we have $(\kappa_{\Irr \, k}^{\mathcal{E}} \kappa_{\ell}^{\mathcal{E}})^2 \lesssim |k_3|\dots |k_6|$. Finally, we have to consider the terms associated with $u^h u^{\ell} / (\Delta_{\Irr \, (h,\ell)}^{(4),\mathcal{E}} (\Delta_{\ell}^{(4),\mathcal{E}})^2)$. We have to consider to cases.

\begin{itemize}
\item \emph{Case $|h_2|>|\ell_1|$.} We have to estimate 
$$
\mathfrak{q}:=\frac{(\kappa_{\ell}^{\mathcal{E}})^4 (\kappa_{\Irr (\ell,h)}^{\mathcal{E}})^2 }{|\ell_1 \dots \ell_{5}||h_3 \dots h_{5}| }.
$$
Recalling that by Lemma \ref{lem:diot}, $\kappa_{\ell}^{\mathcal{E}} \lesssim |\ell_{5}|$, we have $\mathfrak{q}\lesssim  (\kappa_{\Irr (\ell,h)}^{\mathcal{E}})^2 / (|\ell_3||h_3 \dots h_{5}|)$.\\
Since both $h$ and $\ell$ are irreducible and since by Lemma \ref{int-order4}, $\# \Irr (\ell,h) \geq 6$, we have 
$$\kappa_{\Irr (\ell,h)}^{\mathcal{E}} \lesssim |(\Irr (\ell,h))_{\last}| \lesssim \min(|h_3|,|\ell_3|).$$
Consequently, we have $\mathfrak{q}\lesssim 1$.
\item \emph{Case $|h_2|\leq |\ell_1|$.} Denoting by $\mu_1,\mu_2$ the two largest number among $|k_1|,|k_2|,|\ell_1|,|\ell_2|$, we have to estimate
$$
\mathfrak{q}:=\frac{\mu_1 \mu_2(\kappa_{\ell}^{\mathcal{E}})^4 (\kappa_{\Irr (\ell,h)}^{\mathcal{E}})^2 }{|\ell_1 \dots \ell_{5}||h_1 \dots h_{5}| }.
$$
However, since $k,\ell \in \mathcal{M}_5$, we have ${\mu_2\leq |\ell_1|=}\mu_1 \lesssim |\ell_2|$. Consequently, we have 
$$
\mathfrak{q}\lesssim \frac{(\kappa_{\ell}^{\mathcal{E}})^4 (\kappa_{\Irr (\ell,h)}^{\mathcal{E}})^2 }{|\ell_3 \dots \ell_{5}||h_1 \dots h_{5}| }.
$$
Here it is important to recall that by assumption we have $\kappa_{\ell}^{\mathcal{E}}\leq|h_1|$ (see the remark just below \eqref{jmelance}). Consequently, since, as previously, we also have $\kappa_{\ell}^{\mathcal{E}} \lesssim |\ell_{5}|$ and $\kappa_{\Irr (\ell,h)}^{\mathcal{E}} \lesssim |h_3|\lesssim |h_2|$, we get $\mathfrak{q}\lesssim 1$.
\end{itemize}

Naturally, by definition of $\chi_6$ and $ Z_{6,\leq N^3}^{\mathcal{E}}$ (see Theorem \ref{thm-BNF}), $\chi_6$ is a solution of the homological equation
\begin{equation}
\label{def:hom:chi_6}
P_6 + \frac12 \{ \chi_5,P_5\} + \{\chi_6,Z_4^{\mathcal{E}}\} = Z_{6,\leq N^3}^{\mathcal{E}}+ Z_{6,N^3}^{\mathrm{fr}} =: Z^{(6)}_{N^3}
\end{equation}
where $Z^{(6)},Z_{6}^{\mathrm{fr}} \in \mathcal{A}^{(4),\mathcal{E}}_6$ are integrable Hamiltonians such that
$$
\forall \Gamma \in \{ Z^{(6)},Z_{6}^{\mathrm{fr}}\}, \ C^{(\infty)}_{\Gamma} \lesssim N^{39}, \ C^{(em)}_{\Gamma} \leq N^3, \ C^{(de)}_{\Gamma} = 10.
$$ 

Provided that $\varepsilon_0^{4-11/4} \lesssim_{s}  N^{-39} \gamma^{4} (N^3)^{-1-14 \cdot 100}$ and $\varepsilon_0^{1/4} \lesssim 2 \gamma^2  N^{-3\cdot 22 \cdot 10}$ (which is ensured by the assumption \eqref{ultimate:CFL}), Proposition \ref{prop:Lie} { (applied with $r=4$ and $\rho=C^{(de)}_{\Gamma}=10$)}
proves that the Hamiltonian flows generated by $\pm \chi_6$ are well defined on $V_{5/2}$ and that these flows are closed to the identity
$$
\forall t \in (0,1),\forall u \in V_{5/2}, \ \|\Phi_ {\pm \chi_6}^t(u) - u\|_{\dot{H}^s} \leq \| u \|_{\dot{H}^s}^{7/4} \ \ \mathrm{and} \ \
\|(\mathrm{d}\Phi_ {\pm \chi_6}^t (u))^{-1}\|_{\mathscr{L}(\dot{H}^s)} \leq 2.
$$

Provided that $\varepsilon_0^{3/4} \lesssim \gamma^2  N^{-3\cdot 22  \rho_{2r}}$ (which is ensured by the assumption \eqref{ultimate:CFL}), Proposition \ref{prop:stab_U} proves that, for $t\in (0,1)$, $\Phi_ {\chi_6}^t$ maps $V_{5/2}$ in $V_{11/4}$ and $\Phi_ {- \chi_6}^t$ maps $V_{5/4}$ in $V_{3/2}$.

Recalling that  the expansion of $H_{\mathcal{E}} \circ \tau^{(3)} \circ \Phi_ {\chi_5}^1$ is given by \eqref{normform5} and that $Z_2^{\mathcal{E}}$ commutes with $\chi_5$ (see Remark \ref{rem:commuteZ2}), we have that on $V_{5/2}$ 
\begin{equation*}
H_{\mathcal{E}} \circ \tau^{(3)} \circ \Phi_ {\chi_5}^1\circ \Phi_ {\chi_6}^1 = Z_2^{\mathcal{E}}+Z_4^{\mathcal{E}} \circ \Phi_ {\chi_6}^1 + \sum_{m=6}^r Q_m^{(5)}\circ \Phi_ {\chi_5}^1 + \mathrm{R^{(res)}} \circ \Phi_ {\chi_5}^1\circ \Phi_ {\chi_6}^1.
\end{equation*}
 Then, realizing a Taylor expansion of $P_m\circ \Phi_ {\chi_5}^t$ with respect to $t$, we have on $V_{5/2}$ 
$$
Q_m^{(5)}\circ \Phi_ {\chi_6}^1 = Q_m^{(5)} + \sum_{1\leq j\leq (r-m)/2} \frac1{j!} \mathrm{ad}_{\chi_6}^j Q_m^{(5)} + \int_0^1 \frac{(1-t)^{\lfloor(r-m)/2\rfloor}}{\lfloor(r-m)/2\rfloor!} \mathrm{ad}_{\chi_6}^{\lfloor(r-m)/2\rfloor+1} Q_m^{(5)}  \circ \Phi_{\chi_6}^t \ \mathrm{d}t
$$
and recalling that by construction $\{ \chi_6 , Z_4^{\mathcal{E}}\} = -Q_6^{(5)} + Z^{(6)}_{N^3}$ (see \eqref{def:hom:chi_6}), we have
\begin{multline*}
Z_4^{\mathcal{E}} \circ \Phi_ {\chi_6}^1 = Z_4^{\mathcal{E}} - Q_6^{(5)} + Z^{(6)}_{N^3} + \sum_{1\leq j\leq (r-4)/2} \frac1{j!} \mathrm{ad}_{\chi_6}^{j-1}  (Z^{(6)}_{N^3} - Q_6^{(5)} ) \\+ \int_0^1 \frac{(1-t)^{\lfloor r/2 \rfloor -2}}{(\lfloor r/2 \rfloor -2)!} \mathrm{ad}_{\chi_6}^{\lfloor r/2 \rfloor -1} (Z^{(6)}_{N^3} - Q_6^{(5)} )  \circ \Phi_{\chi_6}^t\ \mathrm{d}t.
\end{multline*}
Consequently, we have
\begin{equation}
\label{normform6}
H_{\mathcal{E}} \circ \tau^{(3)} \circ \Phi_ {\chi_5}^1 =Z_2^{\mathcal{E}}+ Z_4^{\mathcal{E}}  +Z^{(6)}_{N^3}+ \sum_{m=7}^r Q_m^{(6)} + R^{(rat),6}+ \mathrm{R^{(res)}} \circ \Phi_ {\chi_5}^1\circ \Phi_ {\chi_6}^1
\end{equation}
where we have set
\begin{equation}
\label{def:Qm6}
Q_m^{(6)} = \sum_{2p+q = m} \frac1{p !} \mathrm{ad}_{\chi_6}^p Q_q^{(6)} - \frac{\mathbb{1}_{m \in 2\mathbb{Z}}}{(m/2-2) !} \mathrm{ad}_{\chi_6}^{m/2-3} (Q_6^{(5)} - Z^{(6)}_{N^3} )
\end{equation}
and
\begin{multline*}
R^{(rat),6} =  R^{(rat),5}\circ  \Phi_ {\chi_6}^1 + \sum_{m=5}^r \int_0^1 \frac{(1-t)^{\lfloor(r-m)/2\rfloor}}{\lfloor(r-m)/2\rfloor!} \mathrm{ad}_{\chi_6}^{\lfloor(r-m)/2\rfloor+1} Q_m^{(5)}  \circ \Phi_{\chi_6}^t \ \mathrm{d}t \\
+   \int_0^1 \frac{(1-t)^{\lfloor r/2 \rfloor -2}}{(\lfloor r/2 \rfloor -2)!} \mathrm{ad}_{\chi_6}^{\lfloor r/2 \rfloor -1} (Z^{(6)}_{N^3} - Q_6^{(5)} )  \circ \Phi_{\chi_6}^t\ \mathrm{d}t.
\end{multline*}
Since $Q_m^{(5)}$ can be identified with a rational fraction $\Xi^{(m),5} \in \mathscr{H}_m^{(4),\mathcal{E}}$ (see \eqref{Q5intheclasse}) and that similarly $Z^{(6)}_{N^3} - Q_6^{(5)}$ can be identify with a rational fraction in $\mathscr{H}_6^{(4),\mathcal{E}}$ satisfying the same bounds as $\Xi^{(6),5}$ (rigorously it is nothing but a subset of $\Xi^{(6),5}$),
by applying Proposition \ref{prop:stab_frac} and Proposition \ref{prop:stab:subclasses}, for all $m\geq 7$, there exists $\Xi^{(m),6} \in  \mathscr{H}_m^{(4),\mathcal{E}}$ such that 
$$Q_m^{(6)}  = \Xi^{(m),6}_{N^3} \ \ \  \mathrm{with} \ \ \ C^{(\infty)}_{\Xi^{(m),6}} \lesssim_m N^{24m}, \ C^{(em)}_{\Xi^{(m),6}} \leq N^3, \ C^{(str)}_{\Xi^{(m),6}}\lesssim_m 1.$$

Finally, we refer the reader to the subsection \ref{app:sub:rem} of the Appendix for the control of the remainder terms which leads to have same control on $R^{(rat),6}$ that we had for $R^{(rat),5}$ (see \eqref{control_R5rat}).

\subsection{The high order rational steps. } Now we aim at removing the non integrable resonant terms of order higher than $7$. We are going to proceed by induction on $\mathfrak{r} \in \llbracket 7,r\rrbracket$ to prove that
there exist $2$ symplectic maps $\tau^{(1),\mathfrak{r}},\tau^{(2),\mathfrak{r}}$ making the following diagram to commute
\begin{equation*}
\xymatrixcolsep{3pc} \xymatrix{ V_{\frac32 + \frac12\frac{ \mathfrak{r}- 7}{r-7}} \ar[r]^{\mathrm{id}_{\dot{H}^s}} & V_{\frac52 - \frac12 \frac{ \mathfrak{r}- 7}{r-7}} \ar[d]^{\tau^{(2),\mathfrak{r}}} \\
 V_{1} \ar[u]^{\tau^{(1),\mathfrak{r}}} \ar[r]^{\mathrm{id}_{\dot{H}^s}} & V_3 },
\end{equation*}
close to the identity
\begin{equation}
\label{est:close:74}
\forall \sigma\in \{1,2\}, \  \|\tau^{(\sigma)}(u)-u\|_{\dot{H}^s} \lesssim_r \|u\|_{\dot{H}^s}^{7/4}
\end{equation}
 such that $H_{\mathcal{E}} \circ \tau^{(3)} \circ \tau^{(2),\mathfrak{r}}$ writes
\begin{equation}
\label{exp:it}
H_{\mathcal{E}} \circ \tau^{(3)} \circ \tau^{(2),\mathfrak{r}}(u) = Z_2^{\mathcal{E}} + Z_4^{\mathcal{E}} + \sum_{m=6}^{\mathfrak{r}} Z^{(m)}_{N^3} +  \sum_{m=\mathfrak{r}+1}^{r}  \Upsilon^{(m),\mathfrak{r}-1}_{N^3} + \mathrm{R^{(res)}}\circ \tau^{(2),\mathfrak{r}} + \mathrm{R^{(rat),\mathfrak{r}}}
\end{equation}
where the integrable Hamiltonians $Z^{(m)}_{N^3}$ are those described in Theorem \ref{thm-BNF} (and $Z^{(6)}_{N^3}=Z_{6,\leq N^3}^{\mathcal{E}}+ Z_{6,N^3}^{\mathrm{fr}}$ is given by \eqref{def:hom:chi_6}) satisfying for $m\geq 8$, $C^{(\infty)}_{Z^{(m)}_{N^3}} \lesssim_m N^{ 321m-2079}$, $ \mathrm{R^{(rat),\mathfrak{r}}}$ satisfies the same estimate \eqref{control_R5rat} as $R^{(rat),5}$, the norm of the invert of the differential of $\tau^{(2),\mathfrak{r}}$ is controlled by $\mathfrak{r}$ and, for $\mathfrak{r}+1\leq m \leq r$, $\Upsilon^{(m),\mathfrak{r}-1}\in \mathscr{H}_m^{(6),\mathcal{E}}$ satisfies
\begin{equation}
\label{est:xi:m}
C^{(\infty)}_{\Upsilon^{(m),\mathfrak{r}-1}} \lesssim_m N^{ 321m-2079}, \ C^{(em)}_{\Upsilon^{(m),\mathfrak{r}-1}} \leq N^3, \ C^{(str)}_{\Upsilon^{(m),\mathfrak{r}-1}}\lesssim_m 1.
\end{equation}

Note that the case $\mathfrak{r} = 6$ have been proven in the previous subsection. Consequently, here, we only focus on proving that if this normal form result holds at an index $\mathfrak{r}-1$ with $\mathfrak{r}\leq r$ then it also holds at the index $\mathfrak{r}$. 

In order to remove the non-integrable part of $\Upsilon^{(\mathfrak{r}),\mathfrak{r}-2}_{N^3}$, we are going to proceed in $3$ steps. In the two first steps, we solve some homological equations associated with $Z_4^{\mathcal{E}}+Z_{6,\leq N^3}^{\mathcal{E}}$. However, due to the Hamiltonian $Z_{6,N^3}^{\mathrm{fr}}$, it makes appear a new non-integrable term of order $\mathfrak{r}$. Therefore, a priori, these normal form steps seem useless. Neverthleless, actually, at each step, the terms of order $\mathfrak{r}$ become smoother (in some unusual sense). Then using this additional smoothness, we convert this term of order $\mathfrak{r}$ in a term of order $\mathfrak{r}+2$ just by transferring a denominator associated with a $\mathbf{k}$ (i.e. of order $4$) to a denominator  associated with a $\mathbf{h}$ (i.e. of order $2$ but with some derivatives to distribute). { We call this new step, the transmutation step.}

Let $Z^{(\mathfrak{r})}\in \mathscr{A}_{\mathfrak{r}}^{(6),\mathcal{E}}$ and $\Gamma^{(\mathfrak{r})}\in \mathscr{R}_{\mathfrak{r}}^{(6),\mathcal{E}}$ denote respectively the integrable and non integrable part of $\Upsilon^{(\mathfrak{r}),\mathfrak{r}-2}$, i.e. 
$$Z^{(\mathfrak{r})} \cup \Gamma^{(\mathfrak{r})} = \Upsilon^{(\mathfrak{r}),\mathfrak{r}-2}.$$

\subsubsection{A first smoothing transformation. } Let $\chi^{(\mathfrak{r}),1} \in \mathscr{R}_{\mathfrak{r}-4}^{(6),*,\mathcal{E}}$ be the solution of the homological equation
\begin{equation}
\label{hom:chir1}
\{  \chi^{(\mathfrak{r}),1}_{N^3},Z_4^{\mathcal{E}}+Z_{6,\leq N^3}^{\mathcal{E}}\} + \Gamma^{(\mathfrak{r})}_{N^3} = 0
\end{equation}
implicitly defined by
\begin{equation}
\label{def:chir1}
 \chi^{(\mathfrak{r}),1}_{N^3} = \sum_{(\ell,\mathbf{h},\mathbf{k},n,c) \in \Gamma^{(\mathfrak{r})}} c \, u^{\ell} \frac{f_{\ell,\mathbf{h},\mathbf{k},n,c}(I)}{2i\pi \Delta^{(4,6),\mathcal{E}}_{\Irr \, \ell,N^3}(I) }
\end{equation}
where $f_{\ell,\mathbf{h},\mathbf{k},n,c}(I)$ denotes the denominator of $\Gamma^{(\mathfrak{r})}_{N^3}$ naturally associated with $(\ell,\mathbf{h},\mathbf{k},n,c){ \in \Gamma^{(\mathfrak{r})}}$ (see Definition \ref{def:ev:frac}). Of course $\chi^{(\mathfrak{r}),1}$ satisfies the same estimates as $\Upsilon^{(\mathfrak{r}),\mathfrak{r}-2}$ (i.e. \eqref{est:xi:m}). Note that, here, the denominator  $ \Delta^{(4,6),\mathcal{E}}_{\Irr \, \ell,N^3}(I)$ is considered as a term of order $4$ (i.e. a $\mathbf{k}$). 

Provided that $\varepsilon_0^{\mathfrak{r}-4-11/4} \lesssim_{s,\mathfrak{r}}  N^{-321 \mathfrak{r}+2079} \sqrt{\gamma}^{\rho_{\mathfrak{r}}-\mathfrak{r}+4+2} (N^3)^{-1-14 \cdot \rho_{\mathfrak{r}}}$ and $\varepsilon_0^{1/4} \lesssim \gamma^2  N^{-3\cdot 22 \cdot \rho_{\mathfrak{r}}}$  (which is ensured by the assumption \eqref{ultimate:CFL}), Proposition \ref{prop:Lie} proves that the Hamiltonian flows generated by $\pm  \chi^{(\mathfrak{r}),1}_{N^3}$ are well defined on $V_{\frac52 - \frac12 \frac{ \mathfrak{r}- 7.5}{r-7}}$ and that these flows are close to the identity
$$
\forall t \in (0,1),\forall u \in V_{\frac52 - \frac12 \frac{ \mathfrak{r}- 7.5}{r-7}}, \ \|\Phi_ {\pm  \chi^{(\mathfrak{r}),1}_{N^3}}^t(u) - u\|_{\dot{H}^s} \leq \| u \|_{\dot{H}^s}^{7/4}, \ \
\|(\mathrm{d}\Phi_ {\pm  \chi^{(\mathfrak{r}),1}_{N^3}}^t (u))^{-1}\|_{\mathscr{L}(\dot{H}^s)} \leq 2.
$$

Provided that $\varepsilon_0^{3/4} \lesssim_{ r} \gamma^2  N^{-3\cdot 22  \rho_{2r}}$ (which is ensured by the assumption \eqref{ultimate:CFL}), Proposition \ref{prop:stab_U} proves that, for $t\in (0,1)$, $\Phi_ {-\chi^{(\mathfrak{r}),1}_{N^3}}^t$ maps $V_{\frac32 + \frac12 \frac{ \mathfrak{r}- 8}{r-7}}$ in $V_{\frac32 + \frac12 \frac{ \mathfrak{r}- 7.5}{r-7}}$ and $\Phi_ {\chi^{(\mathfrak{r}),1}_{N^3}}^t$ maps $V_{\frac52 - \frac12 \frac{ \mathfrak{r}- 7.5}{r-7}}$ in $V_{\frac52 - \frac12 \frac{ \mathfrak{r}- 8}{r-7}}$.

Denoting $\tau^{(2),\mathfrak{r}-1/2} = \tau^{(2),\mathfrak{r}-1}\circ  \Phi_ {\chi^{(\mathfrak{r}),1}_{N^3}}^1$ and recalling that the expansion of $H_{\mathcal{E}} \circ \tau^{(3)} \circ \tau^{(2),\mathfrak{r}-1}$ is given by \eqref{exp:it}, we have that on $V_{\frac52 - \frac12 \frac{ \mathfrak{r}- 7.5}{r-7}}$ 
\begin{multline*} 
H_{\mathcal{E}} \circ \tau^{(3)} \circ \tau^{(2),\mathfrak{r}-1/2}  = Z_2^{\mathcal{E}} + (Z_4^{\mathcal{E}}+Z_{6,\leq N^3}^{\mathcal{E}})\circ  \Phi_ {\chi^{(\mathfrak{r}),1}_{N^3}}^1  + Z_{6,N^3}^{\mathrm{fr}}\circ  \Phi_ {\chi^{(\mathfrak{r}),1}_{N^3}}^1
+ \sum_{m=8}^{\mathfrak{r}} Z^{(m)}_{N^3}\circ  \Phi_ {\chi^{(\mathfrak{r}),1}_{N^3}}^1 \\+ \Gamma^{(\mathfrak{r})}_{N^3}  \circ  \Phi_ {\chi^{(\mathfrak{r}),1}_{N^3}}^1+\sum_{m=\mathfrak{r}+1}^{r}  \Upsilon^{(m),\mathfrak{r}-2}_{N^3}\circ  \Phi_ {\chi^{(\mathfrak{r}),1}_{N^3}}^1+ \mathrm{R^{(res)}}\circ  \tau^{(2),\mathfrak{r}-1/2} + \mathrm{R^{(rat),\mathfrak{r}-1}}\circ  \Phi_ {\chi^{(\mathfrak{r}),1}_{N^3}}^1.
\end{multline*}
Recalling that $\chi^{(\mathfrak{r}),1}_{N^3}$ solves the homological equation \eqref{hom:chir1} and realizing, as previously,  a Taylor expansions of some of these terms, we get
\begin{multline}
\label{exp:tau12}
H_{\mathcal{E}} \circ \tau^{(3)} \circ \tau^{(2),\mathfrak{r}-1/2}  = Z_2^{\mathcal{E}} + Z_4^{\mathcal{E}} + \sum_{m=6}^{\mathfrak{r}} Z^{(m)}_{N^3} + \{ \chi^{(\mathfrak{r}),1}_{N^3}, Z_{6,N^3}^{\mathrm{fr}}\} + \sum_{n=\mathfrak{r}+1}^{r}  Q^{(\mathfrak{r}-1/2)}_{n} \\+ \mathrm{R^{(res)}}\circ\tau^{(2),\mathfrak{r}-1/2}  + \mathrm{R^{(rat),\mathfrak{r}-1/2}}
\end{multline}
where $Q^{(\mathfrak{r}-1/2)}_{n}$ is the Hamiltonian of order $n$ given by\footnote{the following sums hold on the indices $j$ and $m$ satisfying the prescribed conditions.}
\begin{multline*}
Q^{(\mathfrak{r}-1/2)}_{n} = \sum_{ \substack{j(\mathfrak{r} -6)+m =n\\ 8\leq m \leq \mathfrak{r} } } \frac1{j!} \mathrm{ad}^{j}_{\chi^{(\mathfrak{r}),1}_{N^3}} Z^{(m)}_{N^3} + \sum_{ \substack{j(\mathfrak{r} -6)+m =n\\ \mathfrak{r} + 1\leq m \leq r} } \frac1{j!} \mathrm{ad}^{j}_{\chi^{(\mathfrak{r}),1}_{N^3}} \Upsilon^{(m),\mathfrak{r}-2}_{N^3} + \sum_{ j(\mathfrak{r} -6)+6 =n }  \frac1{j!} \mathrm{ad}^{j}_{\chi^{(\mathfrak{r}),1}_{N^3}} Z_{6,N^3}^{\mathrm{fr}} \\
+\sum_{ j(\mathfrak{r} -6)+ \mathfrak{r}=n }  \left(\frac1{j!}-\frac1{(j+1)!}\right) \mathrm{ad}^{j}_{\chi^{(\mathfrak{r}),1}_{N^3}} \Gamma^{(\mathfrak{r})}_{N^3}
\end{multline*}
and $\mathrm{R^{(rat),\mathfrak{r}-1/2}}$ is given by
\begin{multline} 
\label{la plus jolie formule de tout l univers}
\mathrm{R^{(rat),\mathfrak{r}-1/2}}= \mathrm{R^{(rat),\mathfrak{r}-1}}\circ  \Phi_ {\chi^{(\mathfrak{r}),1}_{N^3}}^1  + \sum_{m=8}^{\mathfrak{r}} \int_0^1 \frac{(1-t)^{\lfloor \frac{r-m}{\mathfrak{r}-6} \rfloor}}{\lfloor \frac{r-m}{\mathfrak{r}-6} \rfloor !}\mathrm{ad}^{1+\lfloor \frac{r-m}{\mathfrak{r}-6} \rfloor }_{\chi^{(\mathfrak{r}),1}_{N^3}} Z^{(m)}_{N^3}  \circ  \Phi_ {\chi^{(\mathfrak{r}),1}_{N^3}}^t \mathrm{d}t \\
 + \sum_{m=\mathfrak{r}+1}^r \int_0^1 \frac{(1-t)^{\lfloor \frac{r-m}{\mathfrak{r}-6} \rfloor}}{\lfloor \frac{r-m}{\mathfrak{r}-6} \rfloor !}\mathrm{ad}^{1+\lfloor \frac{r-m}{\mathfrak{r}-6} \rfloor }_{\chi^{(\mathfrak{r}),1}_{N^3}} \Upsilon^{(m),\mathfrak{r}-2}_{N^3}  \circ  \Phi_ {\chi^{(\mathfrak{r}),1}_{N^3}}^t \mathrm{d}t + \int_0^1 \frac{(1-t)^{\lfloor \frac{r-6}{\mathfrak{r}-6} \rfloor}}{\lfloor \frac{r-6}{\mathfrak{r}-6} \rfloor !}\mathrm{ad}^{1+\lfloor \frac{r-6}{\mathfrak{r}-6} \rfloor }_{\chi^{(\mathfrak{r}),1}_{N^3}} Z_{6,N^3}^{\mathrm{fr}}  \circ  \Phi_ {\chi^{(\mathfrak{r}),1}_{N^3}}^t \mathrm{d}t \\
 + \int_0^1 \left[ \frac{(1-t)^{\lfloor \frac{r-\mathfrak{r}}{\mathfrak{r}-6} \rfloor}}{\lfloor \frac{r-\mathfrak{r}}{\mathfrak{r}-6} \rfloor !} - \frac{(1-t)^{1+\lfloor \frac{r-\mathfrak{r}}{\mathfrak{r}-6} \rfloor}}{\lfloor 1+  \frac{r-\mathfrak{r}}{\mathfrak{r}-6} \rfloor !} \right]\mathrm{ad}^{1+\lfloor \frac{r-\mathfrak{r}}{\mathfrak{r}-6} \rfloor }_{\chi^{(\mathfrak{r}),1}_{N^3}} \Gamma^{(\mathfrak{r})}_{N^3}  \circ  \Phi_ {\chi^{(\mathfrak{r}),1}_{N^3}}^t \mathrm{d}t.
\end{multline}

Since $\chi^{(\mathfrak{r}),1} \in \mathscr{R}_{\mathfrak{r}-4}^{(6),*,\mathcal{E}}$, $Z_{6}^{\mathrm{fr}} \in \mathcal{A}^{(4),\mathcal{E}}_6$, $Z^{(m)},\Upsilon^{(m),\mathfrak{r}-2}\in \mathscr{H}_m^{(6),\mathcal{E}}$, we deduce of  Proposition \ref{prop:stab_frac}, Proposition \ref{prop:stab:subclasses} and Proposition \ref{prop:stab:chichi} that there exists $\Upsilon^{(n),\mathfrak{r}-3/2}\in \mathscr{H}_n^{(6),\mathcal{E}}$ such that
$$Q^{(\mathfrak{r}-1/2)}_{n}  = \Upsilon^{(n),\mathfrak{r}-3/2}_{N^3} \ \ \  \mathrm{with} \ \ \ C^{(\infty)}_{\Upsilon^{(n),\mathfrak{r}-3/2}} \lesssim_n N^{ 321n-2079}, \ C^{(em)}_{\Upsilon^{(n),\mathfrak{r}-3/2}} \leq N^3, \ C^{(str)}_{\Upsilon^{(n),\mathfrak{r}-3/2}}\lesssim_n 1.$$

Finally, we refer the reader to the subsection \ref{app:sub:rem} of the Appendix for the control of the remainder terms which leads to have same control on $R^{(rat),\mathfrak{r}-1/2}$ that the one we had for $R^{(rat),5}$ in \eqref{control_R5rat}.

\subsubsection{A second smoothing transformation. } In the expansion \eqref{exp:tau12} of $H_{\mathcal{E}} \circ \tau^{(3)} \circ \tau^{(2),\mathfrak{r}-1/2}$, there is still an non-integrable term of order $\mathfrak{r}$ : $\{ \chi^{(\mathfrak{r}),1}_{N^3}, Z_{6,N^3}^{\mathrm{fr}}\}$. Indeed, applying Proposition \ref{prop:stab_frac}, there exists $\Gamma^{(\mathfrak{r}+1/2)}\in \mathscr{H}_{\mathfrak{r}}^{\mathcal{E}}$ such that 
$$\{ \chi^{(\mathfrak{r}),1}_{N^3}, Z_{6,N^3}^{\mathrm{fr}}\}  = \Gamma^{(\mathfrak{r}+1/2)}$$ and
\begin{equation}
\label{est:gamma:12}
 C^{(\infty)}_{\Gamma^{(\mathfrak{r}+1/2)}} \lesssim_{\mathfrak{r}} N^{ 321\mathfrak{r}-2079 + 48}, \ C^{(em)}_{\Gamma^{(\mathfrak{r}+1/2)}} \leq N^3, \ C^{(str)}_{\Gamma^{(\mathfrak{r}+1/2)}}\lesssim_{\mathfrak{r}} 1.
\end{equation}
Since $Z_{6}^{\mathrm{fr}} \in \mathcal{A}^{(4),\mathcal{E}}_6$ is an integrable Hamiltonian, considering the definition \eqref{def:chir1} of $\chi^{(\mathfrak{r}),1}_{N^3}$, we observe that if $(\ell,\mathbf{h},\mathbf{k},n,c) \in \Gamma^{(\mathfrak{r}+1/2)}$ then $\Irr \, \ell\neq \emptyset$ and there exists $\beta \in \mathbb{N}^3$ such that $\beta_1+\beta_2+\beta_3\leq \mathfrak{r}-7$ and
$$
n\leq 13 \mathfrak{r}-87 +\beta_1 +3 \hspace{1cm}  \# \mathbf{h}-n \leq \beta_2 \hspace{1cm}  \# \mathbf{k} \leq 4\mathfrak{r}-28 + \beta_3 +1.
$$
We refer the reader to the Propositions \ref{prop:stab:subclasses} and \ref{prop:stab:chichi} where similar estimates are explained in details and we also refer to the next subsection \ref{sub:sub:trans} where this term is computed precisely. As a consequence of these bounds on the number of denominators, as previously, we introduce $\chi^{(\mathfrak{r}),2} \in \mathscr{R}_{\mathfrak{r}-4}^{(6),*,\mathcal{E}}$ (see Definition \ref{sharpclass}) the solution of the homological equation
\begin{equation}
\label{hom:chir2}
\{  \chi^{(\mathfrak{r}),2}_{N^3},Z_4^{\mathcal{E}}+Z_{6,\leq N^3}^{\mathcal{E}}\} + \Gamma^{(\mathfrak{r}+1/2)}_{N^3} = 0
\end{equation}
implicitly defined by
\begin{equation}
\label{def:chir2}
 \chi^{(\mathfrak{r}),2}_{N^3} = \sum_{(\ell,\mathbf{h},\mathbf{k},n,c) \in \Gamma^{(\mathfrak{r}+1/2)}} c \, u^{\ell} \frac{f_{\ell,\mathbf{h},\mathbf{k},n,c}(I)}{2i\pi \Delta^{(4,6),\mathcal{E}}_{\Irr \, \ell,N^3}(I) }
\end{equation}
where $f_{\ell,\mathbf{h},\mathbf{k},n,c}(I)$ denotes the denominator of $\Gamma^{(\mathfrak{r}+1/2)}_{N^3}$ naturally associated with $(\ell,\mathbf{h},\mathbf{k},n,c)$ (see Definition \ref{def:ev:frac}). Of course $\chi^{(\mathfrak{r}),2}$ satisfies the same estimates as $\Gamma^{(\mathfrak{r}+1/2)}$ (i.e. \eqref{est:gamma:12}). Note that here the denominator  $ \Delta^{(4,6),\mathcal{E}}_{\Irr \, \ell,N^3}(I)$ is considered as a term of order $4$ (i.e. a $\mathbf{k}$).

Provided that $\varepsilon_0^{\mathfrak{r}-4-11/4} \lesssim_{s,\mathfrak{r}}  N^{ -321\mathfrak{r}+2079 -48} \sqrt{\gamma}^{\rho_{\mathfrak{r}}-\mathfrak{r}+4+2} (N^3)^{-1-14 \cdot \rho_{\mathfrak{r}}}$ and $\varepsilon_0^{1/4} \lesssim \gamma^2  N^{-3\cdot 22 \cdot \rho_{\mathfrak{r}}}$ (which is ensured by the assumption \eqref{ultimate:CFL}), Proposition \ref{prop:Lie} proves that the Hamiltonian flows generated by $\pm  \chi^{(\mathfrak{r}),2}_{N^3} $ are well defined on $V_{\frac52 - \frac12 \frac{ \mathfrak{r}- 7}{r-7}}$ and that these flows are closed to the identity
$$
\forall t \in (0,1),\forall u \in V_{\frac52 - \frac12 \frac{ \mathfrak{r}- 7}{r-7}}, \ \|\Phi_ {\pm  \chi^{(\mathfrak{r}),2}_{N^3}}^t(u) - u\|_{\dot{H}^s} \leq \| u \|_{\dot{H}^s}^{7/4}, \ \
\|(\mathrm{d}\Phi_ {\pm  \chi^{(\mathfrak{r}),2}_{N^3}}^t (u))^{-1}\|_{\mathscr{L}(\dot{H}^s)} \leq 2.
$$

Provided that $\varepsilon_0^{3/4} \lesssim_{ r} \gamma^2  N^{-3\cdot 22  \rho_{2r}}$ (which is ensured by the assumption \eqref{ultimate:CFL}), Proposition \ref{prop:stab_U} proves that, for $t\in (0,1)$, $\Phi_ {-\chi^{(\mathfrak{r}),2}_{N^3}}^t$ maps $V_{\frac32 + \frac12 \frac{ \mathfrak{r}- 7.5}{r-7}}$ in $V_{\frac32 + \frac12 \frac{ \mathfrak{r}- 7}{r-7}}$ and $\Phi_ {\chi^{(\mathfrak{r}),2}_{N^3}}^t$ maps $V_{\frac52 - \frac12 \frac{ \mathfrak{r}- 7}{r-7}}$ in $V_{\frac52 - \frac12 \frac{ \mathfrak{r}- 7.5}{r-7}}$.

Denoting $\tau^{(2),\mathfrak{r}} = \tau^{(2),\mathfrak{r}-1/2}\circ  \Phi_ {\chi^{(\mathfrak{r}),2}_{N^3}}^1$, recalling that the expansion of $H_{\mathcal{E}} \circ \tau^{(3)} \circ \tau^{(2),\mathfrak{r}-1/2}$ is given by \eqref{exp:tau12} and that $\chi^{(\mathfrak{r}),2}_{N^3}$ solves the homological equation \eqref{hom:chir2}, and 
 realizing, as previously,  a Taylor expansions of some of these terms, we get, on $V_{\frac52 - \frac12 \frac{ \mathfrak{r}- 7}{r-7}}$, that
\begin{multline}
\label{exp:tau12bis}
H_{\mathcal{E}} \circ \tau^{(3)} \circ \tau^{(2),\mathfrak{r}}  = Z_2^{\mathcal{E}} + Z_4^{\mathcal{E}} + \sum_{m=6}^{\mathfrak{r}} Z^{(m)}_{N^3} + \{ \chi^{(\mathfrak{r}),2}_{N^3}, Z_{6,N^3}^{\mathrm{fr}}\} + \sum_{n=\mathfrak{r}+1}^{r}  Q^{(\mathfrak{r})}_{n} \\+ \mathrm{R^{(res)}}\circ\tau^{(2),\mathfrak{r}}  + \mathrm{R^{(rat),\mathfrak{r}}}
\end{multline}
where $Q^{(\mathfrak{r})}_{n}$ is the Hamiltonian of order $n$ given by
\begin{multline*}
Q^{(\mathfrak{r})}_{n} = \sum_{ \substack{j(\mathfrak{r} -6)+m =n\\ 8\leq m \leq \mathfrak{r} } } \frac1{j!} \mathrm{ad}^{j}_{\chi^{(\mathfrak{r}),2}_{N^3}} Z^{(m)}_{N^3} + \sum_{ \substack{j(\mathfrak{r} -6)+m =n\\ \mathfrak{r} + 1\leq m \leq r} } \frac1{j!} \mathrm{ad}^{j}_{\chi^{(\mathfrak{r}),2}_{N^3}} \Upsilon^{(m),\mathfrak{r}-3/2}_{N^3} + \sum_{ j(\mathfrak{r} -6)+6 =n }  \frac1{j!} \mathrm{ad}^{j}_{\chi^{(\mathfrak{r}),2}_{N^3}} Z_{6,N^3}^{\mathrm{fr}} \\
+\sum_{ j(\mathfrak{r} -6)+ \mathfrak{r}=n }  \left(\frac1{j!}-\frac1{(j+1)!}\right) \mathrm{ad}^{j}_{\chi^{(\mathfrak{r}),2}_{N^3}} \Gamma^{(\mathfrak{r}+1/2)}_{N^3}
\end{multline*}
and $\mathrm{R^{(rat),\mathfrak{r}}}$ is given by the same formula as $\mathrm{R^{(rat),\mathfrak{r}-1/2}}$ (i.e. \eqref{la plus jolie formule de tout l univers}) but with the change of index $\mathfrak{r} \leftarrow \mathfrak{r}+1/2$.

Since $\chi^{(\mathfrak{r}),2} \in \mathscr{R}_{\mathfrak{r}-4}^{(6),*,\mathcal{E}}$, $Z_{6}^{\mathrm{fr}} \in \mathcal{A}^{(4),\mathcal{E}}_6$, $Z^{(m)},\Upsilon^{(m),\mathfrak{r}-3/2}\in \mathscr{H}_m^{(6),\mathcal{E}}$, we deduce of  Proposition \ref{prop:stab_frac}, Proposition \ref{prop:stab:subclasses} and Proposition \ref{prop:stab:chichi} that there exists $\Upsilon^{(n),\mathfrak{r}-1}\in \mathscr{H}_n^{(6),\mathcal{E}}$ such that
$$Q^{(\mathfrak{r})}_{n}  = \Upsilon^{(n),\mathfrak{r}-1}_{N^3} \ \ \  \mathrm{with} \ \ \ C^{(\infty)}_{\Upsilon^{(n),\mathfrak{r}-1}} \lesssim_n N^{ 321n-2079}, \ C^{(em)}_{\Upsilon^{(n),\mathfrak{r}-1}} \leq N^3, \ C^{(str)}_{\Upsilon^{(n),\mathfrak{r}-1}}\lesssim_n 1.$$

Finally, we refer the reader to the subsection \ref{app:sub:rem} of the Appendix for the control of the remainder terms which leads to have same control on $R^{(rat),\mathfrak{r}}$ that the one we had for $R^{(rat),5}$ in \eqref{control_R5rat}.

\subsubsection{Transmutation of a denominator and conclusion.} \label{sub:sub:trans}
After these two steps of normal form, the expansion \eqref{exp:tau12bis} of $H_{\mathcal{E}} \circ \tau^{(3)} \circ \tau^{(2),\mathfrak{r}} $ seems similar to the expansion $H_{\mathcal{E}} \circ \tau^{(3)} \circ \tau^{(2),\mathfrak{r}-1}$. Nevertheless, the non integrable term of order $\mathfrak{r}$ denoted $\Gamma^{(\mathfrak{r})}$ has been replaced by $ \{ \chi^{(\mathfrak{r}),2}_{N^3}, Z_{6,N^3}^{\mathrm{fr}}\}$, which, following Proposition \ref{prop:stab_frac}, is another term of  order $\mathfrak{r}$. Describing very carefully this term, we are going to explain why one of its denominators of order $4$ can be considered as a denominator of order $2$. It will prove that this term is actually a term of order $\mathfrak{r}+2$ and it will conclude the proof of this induction.

First, we have to describe precisely $Z_{6,N^3}^{\mathrm{fr}}$. By definition (see \eqref{def:hom:chi_6}), it is the integrable part of $\frac12 \{ \chi_5,P_5\}$. Recalling that $P_5$ is given by \eqref{def:Pm} and $\chi_5$ is given by \eqref{def:chi5}, $Z_{6,N^3}^{\mathrm{fr}}$ can be written as
$$
Z_{6,N^3}^{\mathrm{fr}}(I) = \sum_{\substack{k \in \mathcal{R}^{\mathcal{E}}_5 \cap \mathcal D  \\ |k_1| \leq N^3 \\}}  \sum_{\substack{h \in \mathrm{ss}(k)\\4\leq \# h }} c_{k,h} \frac{I^h}{(\Delta_{k}^{(4),\mathcal{E}})^{\# h-3}}
$$
where $c_{k,\ell}$ are some coefficients satisfying $|c_{k,h}|\lesssim N^{39}$ and $\mathrm{ss}(k)$ denotes the set of the subsequences of $k$. Consequently, if $\ell \in \mathcal{M}$, we have
$$
\{u^\ell,Z_{6,N^3}^{\mathrm{fr}}\} = u^\ell (\sum_{j=1}^{\# \ell} 2i\pi\ell_j \partial_{I_{\ell_j}}Z_{6,N^3}^{\mathrm{fr}})\\
= u^{\ell}\! \! \! \! \! \! \sum_{\substack{k \in \mathcal{R}^{\mathcal{E}}_5 \cap \mathcal D  \\ (\Irr \ell)_{\last}\leq |k_1| \leq N^3}}  \sum_{ \substack{h \in \mathrm{ss}(k)\\ 3\leq \# h}} c_{k,h,\ell} \frac{I^h}{(\Delta_{k}^{(4),\mathcal{E}})^{\# h-2}}
$$
where $c_{k,h,\ell}$ are some coefficients such that $|c_{k,h,\ell}|\lesssim_{\mathfrak{r}} N^{48}$. Let us justify the condition $(\Irr \ell)_{\last}\leq |k_1|$ in the sum above. The term of index $k$ refers to a Poisson bracket of the form $\{ u^\ell, I^h /(\Delta_{k}^{(4),\mathcal{E}})^{\# h-3} \}$ where $h$ is a subsequence of $k$.
However, as stated in Remark \ref{rem:useful}, $\Delta_{k}^{(4),\mathcal{E}}$ is a linear function of actions associated with indices no larger than $|k_1|$. Consequently, if $(\Irr \ell)_{\last}> |k_1|$, then $u^\ell$ and $I^h /(\Delta_{k}^{(4),\mathcal{E}})^{\# h-3}$ would { Poisson} commute. That is why such a term do not appear in the expansion of $\{u^\ell,Z_{6,N^3}^{\mathrm{fr}}\}$.

Recalling that $\chi^{(\mathfrak{r}),2}_{N^3}$, defined in \eqref{def:chir2}, is the usual solution of
$$
\{  \chi^{(\mathfrak{r}),2}_{N^3},Z_4^{\mathcal{E}}+Z_{6,\leq N^3}^{\mathcal{E}}\} + \{ \chi^{(\mathfrak{r}),1}_{N^3}, Z_{6,N^3}^{\mathrm{fr}}\} = 0
$$
where $\chi^{(\mathfrak{r}),1}_{N^3}$, defined in \eqref{def:chir1}, is the usual solution of
$$
\{  \chi^{(\mathfrak{r}),1}_{N^3},Z_4^{\mathcal{E}}+Z_{6,\leq N^3}^{\mathcal{E}}\} + \Gamma^{(\mathfrak{r})}_{N^3} = 0.
$$
Since  the denominator of $\Gamma^{(\mathfrak{r})}$ and $Z_{6,N^3}^{\mathrm{fr}}$ { are both  functions of the actions alone },  we never have to derive  the denominator of $\Gamma^{(\mathfrak{r})}$ when calculating  $ \{ \chi^{(\mathfrak{r}),2}_{N^3}, Z_{6,N^3}^{\mathrm{fr}}\} $. Therefore $ \{ \chi^{(\mathfrak{r}),2}_{N^3}, Z_{6,N^3}^{\mathrm{fr}}\} $ can be decomposed as
$$
 \sum_{(\ell,\mathbf{h},\mathbf{k},n,c) \in \Gamma^{(\mathfrak{r})}}  c \, u^{\ell} \frac{f_{\ell,\mathbf{h},\mathbf{k},n,c}(I)}{ (2i\pi \Delta^{(4,6),\mathcal{E}}_{\Irr \, \ell,N^3}(I) )^2}  \prod_{p=1}^2 \! \! \! \! \! \! \sum_{\substack{k^{(p)} \in \mathcal{R}^{\mathcal{E}}_5 \cap \mathcal D  \\ (\Irr\, \ell)_{\last}\leq |k_1^{(p)}| \leq N^3}}  \sum_{ \substack{h^{(p)} \in \mathrm{ss}(k^{(p)})\\ 3\leq \# h^{(p)} }}  c_{k^{(p)},h^{(p)},\ell} \frac{I^{h^{(p)}}}{(\Delta_{k^{(p)}}^{(4),\mathcal{E}})^{ \# h^{(p)}-2}}
$$
where  and $f_{\ell,\mathbf{h},\mathbf{k},n,c}(I)$ denotes the denominator of $\Gamma^{(\mathfrak{r})}_{N^3}$ naturally associated with $(\ell,\mathbf{h},\mathbf{k},n,c)$ (see Definition \ref{def:ev:frac}).

Considering one a the denominator $ \Delta^{(4,6),\mathcal{E}}_{\Irr \, \ell,N^3}$ as a term of order four (i.e. a new index for $\mathbf{k}$) and the other as a term of order two (i.e. a new index for $\mathbf{h}$), $ \{ \chi^{(\mathfrak{r}),2}_{N^3}, Z_{6,N^3}^{\mathrm{fr}}\}$ is naturally associated with $\Lambda^{(\mathfrak{r})} \in \mathscr{H}_{\mathfrak{r}+2}^{(6),\mathcal{E}}$, i.e.
$$
\{ \chi^{(\mathfrak{r}),2}_{N^3}, Z_{6,N^3}^{\mathrm{fr}}\} = \Lambda^{(\mathfrak{r})}_{N^3}
$$
such that $C^{(\infty)}_{\Lambda^{(\mathfrak{r})}} \lesssim_{\mathfrak{r}} N^{ 321\mathfrak{r}-2079 + 2\cdot 48}, \ C^{(em)}_{\Lambda^{(\mathfrak{r})}} \leq N^3$.  Nevertheless, contrary to the previous cases, it is not completely obvious that $C^{(str)}_{\Lambda^{(\mathfrak{r})}}\lesssim_{\mathfrak{r}} 1$. Indeed, we have to explain why $C^{(di)}_{\Lambda^{(\mathfrak{r})}}\lesssim_{\mathfrak{r}} 1$.

By construction, the numerators of $\Lambda^{(\mathfrak{r})}$ are of the form 
$$
u^{\ell'} = u^{\ell} I^{h^{(1)}} I^{h^{(2)}} \ \mathrm{where} \ \ell'\in \mathcal{R}^{\mathcal{E}} \cap \mathcal{D}.
$$
We aim at estimating 
$$
\frac{(\kappa_{\mathbf{h}_{1}}^{\mathcal{E}}\dots \kappa_{\mathbf{h}_{\last}}^{\mathcal{E}})^2 (\kappa_{k^{(1)}}^{\mathcal{E}})^{2 \# h^{(1)}-4}(\kappa_{k^{(2)}}^{\mathcal{E}})^{2 \# h^{(2)}-4} (\kappa_{\Irr \,\ell}^{\mathcal{E}})^2}{|\ell_3'\dots \ell_\last'|}.
$$
We recall that by Lemma \ref{lem:diot}, for all $\ell''\in \Irr$, we have $\kappa_{\ell''}^{\mathcal{E}} \lesssim_{\# \ell''} |\ell''_{\last}|$. Consequently, we only have to estimate
$$
\mathfrak{q}:=\frac{(\kappa_{\mathbf{h}_{1}}^{\mathcal{E}}\dots \kappa_{\mathbf{h}_{\last}}^{\mathcal{E}})^2 (k^{(1)}_{\last})^{2 \# h^{(1)}-4}(k^{(2)}_{\last})^{2 \# h^{(2)}-4} (\Irr \,\ell)_\last ^2}{|\ell_3'\dots \ell_\last'|}.
$$
Up to natural symmetries, we only have to consider two cases.
\begin{itemize}
\item \emph{Case $|\ell'_2|=|\ell_2|$}. Here we also necessary have $|\ell'_1|=|\ell_1|$. Consequently, we have
$$
\mathfrak{q} \leq C^{(di)}_{\Gamma^{(\mathfrak{r})}} \ \frac{ (k^{(1)}_{\last})^{2 \# h^{(1)}-4}(k^{(2)}_{\last})^{2 \# h^{(2)}-4} (\Irr \,\ell)_\last ^2}{(h^{(1)}_1 \dots h^{(1)}_{\last})^2(h^{(2)}_1 \dots h^{(2)}_{\last})^2 } \lesssim_{\mathfrak{r}} \frac{ (k^{(1)}_{\last})^{2 \# h^{(1)}-4}(k^{(2)}_{\last})^{2 \# h^{(2)}-4} (\Irr \,\ell)_\last ^2}{(h^{(1)}_1 \dots h^{(1)}_{\last})^2(h^{(2)}_1 \dots h^{(2)}_{\last})^2 }.
$$
Since $h^{(p)}$ is a subsequence of $k^{(p)}$, we have $|k^{(p)}_{\last}| \leq |h^{(p)}_{\last}|$ and consequently
$$
\mathfrak{q} \lesssim_{\mathfrak{r}}  \frac{  (\Irr \,\ell)_\last ^2}{(h^{(1)}_1 h^{(1)}_2 )^2(h^{(2)}_1 h^{(2)}_2 )^2} {\lesssim_{\mathfrak{r}}} \frac{  (\Irr \,\ell)_\last ^2}{(h^{(1)}_1  )^2(h^{(2)}_1  )^2}.
$$
Recalling that, by construction, for $p\in \{1,2\}$, $|(\Irr \,\ell)_\last| \leq |k^{(p)}_1|$, we get
$$
\mathfrak{q} \lesssim_{\mathfrak{r}} \prod_{p=1}^2 \frac{|k^{(p)}_1|  }{(h^{(p)}_1  )^2}.
$$
However, by construction, $h^{(p)}$ is a subsequence of $k$ with at least $3$ elements. Consequently, we have $|k^{(p)}_3|\leq |h^{(p)}_1|$ and so 
$$
\mathfrak{q} \lesssim_{\mathfrak{r}} \prod_{p=1}^2 \frac{|k^{(p)}_1|  }{(k^{(p)}_3  )^2}.
$$
Finally, recalling that $k^{(p)}\in \mathcal{R}^{\mathcal{E}}_5$ and applying Lemma \ref{lem:origin} { (again this Lemma is the key)}, we have 
$$|k^{(p)}_1| /(k^{(p)}_3  )^2 \leq 9$$
 and so we have proven that $\mathfrak{q} \lesssim_{\mathfrak{r}} 1$.

\item \emph{Case $|\ell'_2|>|\ell_2|$}. This case is much easier. Without loss of generality, we assume that $|\ell'_1| = |\ell'_2| = |h^{(1)}_1|$. Consequently, we have
$$
\mathfrak{q} \leq C^{(di)}_{\Gamma^{(\mathfrak{r})}} \ \frac{ (k^{(1)}_{\last})^{2 \# h^{(1)}-4}(k^{(2)}_{\last})^{2 \# h^{(2)}-4} (\Irr \,\ell)_\last ^2}{|\ell_1||\ell_2|(h^{(1)}_2 \dots h^{(1)}_{\last})^2(h^{(2)}_1 \dots h^{(2)}_{\last})^2 } \lesssim_{\mathfrak{r}} \frac{ (k^{(1)}_{\last})^{2 \# h^{(1)}-4}(k^{(2)}_{\last})^{2 \# h^{(2)}-4}}{(h^{(1)}_2 \dots h^{(1)}_{\last})^2(h^{(2)}_1 \dots h^{(2)}_{\last})^2}. 
$$
Since $h^{(p)}$ is a subsequence of $k^{(p)}$, we have $|k^{(p)}_{\last}| \leq |h^{(p)}_{\last}|$ and consequently,
$$
\mathfrak{q} \lesssim_{\mathfrak{r}}  \frac1{(h^{(1)}_2 )^2(h^{(2)}_1  h^{(2)}_{2})^2} \leq 1.
$$
\end{itemize}

Finally, we conclude this induction step by the change of notation
\begin{equation}
\label{prod:transmutation}
\begin{array}{lcrl}
\Upsilon^{(\mathfrak{r}+2),\mathfrak{r}-1}_{N^3} &\leftarrow& \Upsilon^{(\mathfrak{r}+2),\mathfrak{r}-1}_{N^3}+\Lambda^{(\mathfrak{r})}_{N^3} & \mathrm{if}\ \mathfrak{r}<r-1\\
\mathrm{R^{(rat),\mathfrak{r}}} &\leftarrow& \mathrm{R^{(rat),\mathfrak{r}}} + \Lambda^{(\mathfrak{r})}_{N^3}& \mathrm{else}.
\end{array}
\end{equation}
As before, we refer the reader to the subsection \ref{app:sub:rem} of the Appendix for the control of the remainder term induced by this change of notation which leads to have same control on $R^{(rat),r-1},R^{(rat),r}$ that the one we had for $R^{(rat),5}$ in \eqref{control_R5rat}.

\section{Description of the dynamics} \label{sec:dyn}

\subsection{Dynamical consequences of the rational normal form.} This subsection is devoted to the proof of the following theorem which in nothing but the dynamical part of the Theorem \ref{mainPDE}.
\begin{theorem} 
\label{thm:dyn}
Being given $\mathcal{E}\in \{ \mathrm{\ref{gKdV}}, \mathrm{\ref{gBO}}\}$, $r\gg 7$\footnote{see Remark \ref{rem:rgg7}.}, $s\geq s_0(r) := 10^7 r^2$, $N\gtrsim_{r,s} 1$, $\gamma \lesssim_{r,s} 1$ and $\varepsilon \lesssim_{r,s} 1$ satisfying
\begin{equation}
\label{ultimate:CFL2} 
\varepsilon \leq \gamma^{35} \ \mathrm{and} \ \varepsilon \leq N^{-10^7 r} \ \mathrm{and} \ N^{-s} \leq \varepsilon^r 
\end{equation}
if $u\in \mathcal{U}_{ \gamma,N^3,\rho_{2r}}^{\mathcal{E},s}$ ($\rho_{2r}$ being given by \eqref{def:rhor}) satisfies 
\begin{equation}
\label{lacouronne}
\varepsilon/2 \leq \| u\|_{\dot{H}^s} \leq \varepsilon
\end{equation}
then, as long as $|t|\leq \|u\|_{\dot{H}^s}^{-r/5}$, the solution of $\mathcal{E}$, initially equals to $u$ and denoted $\Phi_t^\E(u)$, exists and satisfies
$$
 \|\Phi_t^\E(u)\|_{\dot H^s}\leq 2 \|u\|_{\dot H^s}.
$$
Furthermore, there exist $C^1$ functions $\theta_k:\R_+\mapsto \R$, $k\in\Z^*,$ such that
\begin{equation}
\label{desc:dyn:2}
 \|\Phi_t^\E(u) -(e^{i\theta_k(t)}u_k)_{k\in\Z^*}\|_{\dot H^{s-1}}\leq  \|u\|^{3/2}_{\dot H^s}
\end{equation}
where $|\dot\theta_k-(2\pi)^{1+\alpha_{\E}} k|k|^{\alpha_{\E}}-2\pi k \partial_{I_k} Z^\E_4(I)|\leq |k| \|u\|^{5/2}_{\dot H^s}.$
\end{theorem}

\subsubsection{Setting of the proof}

We are going to proceed by bootstrap. We denote $V_{\sigma} = B_s(0,2^\sigma \varepsilon)\cap \mathcal{U}_{2^{-\sigma} \gamma,N^3,\rho_{2r}}^{\mathcal{E},s}$. By assumption we know that $u\in V_{0}$. We are going to prove that, assuming $u(t):=\Phi_t^\E(u)$ exists,
$u(t) \in V_2$ and $|t|\leq \|u(0)\|_{\dot{H}^s}^{-r/5}$, we have $u(t)\in V_{1}$ and $u(t)$ is described by \eqref{desc:dyn:2}.

 Of course, such a proof by bootstrap requires a local existence theorem for solutions of $\mathcal{E}$ in $\dot{H}^s(\mathbb{T})$. Even if we do not have found a precise reference of such a theorem in the literature, since $s$ is large, its proof would be classical and could be realized quite directly, adapting, for example, the proof of local well-posedness of the quasi-linear symmetric hyperbolic systems presented by Taylor in the section $1$ chapter $16$ of his book \cite{Tay}.

Naturally this result relies on the rational norm form Theorem \ref{thm:KtA} that we apply with $\varepsilon_0= 4 \varepsilon$ and $\gamma \leftarrow \gamma/4$. Note that with this change of notations, the indices of the set $V_\sigma$ introduced in Theorem \ref{thm:KtA} have to be increased of $2$ (for example, here, $\tau^{(0)}$ maps $V_2$ in $V_3$).

From now, we assume that $0<T\leq \|u(0)\|_{\dot{H}^s}^{-r/5}$ is such that if $|t|<T$ then $u(t)$ exists and belongs to $ V_2$. In this proof, we set
\begin{equation}
\label{eq:vfromu}
v(t) := \tau^{(1)}\circ \tau^{(0)}(u(t)).
\end{equation}
Since, for $|t|<T$, $ u(t) \in V_2$, we also have
\begin{equation}
\label{eq:ufromv}
u(t) = \tau^{(3)}\circ \tau^{(2)}(v(t)).
\end{equation}
Furthermore, since $\tau^{(1)},\tau^{(0)}$ are symplectic, $v(t)$ is solution of the Hamiltonian system
\begin{equation}
\label{sys:v}
\partial_t v(t) = \partial_x \nabla (H_{\E} \circ \tau^{(3)} \circ \tau^{(2)}) (v(t)).
\end{equation}

In order to prove that $u(t)\in V_1$, first we prove that $\|u(t)\|_{\dot{H}^s} \leq 2\|u(0)\|_{\dot{H}^s}$.
Indeed, since $\tau^{(1)}, \tau^{(0)}$ are closed to the identity (in the sense of \eqref{est:close:138}) and $\varepsilon$ is small enough, it follows of \eqref{eq:vfromu} that
$$
\| v(0) \|_{\dot{H}^s} \leq \|u_0\|_{\dot{H}^s}+ 2 \| u(0)\|_{\dot{H}^s}^{13/8} \leq (4/3) \|u_0\|_{\dot{H}^s}.
$$
Consequently, since, in the next subsection \ref{sub:growth:v}, we are going to prove that
\begin{equation}
\label{but_sous_section}
\| v(t)\|_{\dot{H}^s} \leq \| v(0)\|_{\dot{H}^s} + (1/3)\| u(0)\|_{\dot{H}^s},
\end{equation}
it follows of \eqref{eq:ufromv} and \eqref{est:close:138} that, provided that $\varepsilon$ is small enough, we have
$$
\| u(t) \|_{\dot{H}^s} \leq \| v(t) \|_{\dot{H}^s} +2 \| v(t)\|_{\dot{H}^s}^{13/8} \leq 2\|u(0)\|_{\dot{H}^s}.
$$

In the last subsection \ref{sub:dynam:v}, we are going to design $C^1$ functions $\theta_k:\R_+\mapsto \R$, $k\in\Z^*,$ such that
\begin{equation}
\label{whatweexpectaboutv}
 \|v(t) -(e^{i\theta_k(t)}v(0))_{k\in\Z^*}\|_{\dot H^{s-1}}\leq  \|u(0)\|^{10}_{\dot H^s}
\end{equation}
where denoting $J_k(t) = |v_k(t)|^2$
\begin{equation}
\label{passimalpourdesangles}
|\dot\theta_k-(2\pi)^{1+\alpha_{\E}} k|k|^{\alpha_{\E}}-2k\pi\partial_{I_k} Z^\E_4(J(t))|\leq (1/2) |k| \|u(0)\|^{5/2}_{\dot H^s}.
\end{equation}
However, $\tau^{(0)}, \dots,\tau^{(3)}$ being closed to the identity and $\varepsilon$ being small enough, we have
\begin{equation}
\label{uetvetbinysontpasloin}
\| v(t) - u(t)\|_{\dot{H}^s} + \| v(0) - u(0)\|_{\dot{H}^s}  \lesssim \| u(0) \|_{\dot{H}^s}^{13/8}
\end{equation}
and thus
\begin{equation}
\label{cestpourcaquonafaittoutca}
\|u(t) -(e^{i\theta_k(t)}u(0))_{k\in\Z^*}\|_{\dot H^{s-1}}\leq \|u(0)\|^{3/2}_{\dot H^s}.
\end{equation}

We deduce of \eqref{cestpourcaquonafaittoutca} that
\begin{equation*}
\begin{split}
\sup_{k\in \mathbb{N}^*} k^{2s-2}| |u_k(t)|^2 - |u_k(0)|^2| &\leq (\| u(t)\|_{\dot{H}^{s-1}}+\| u(0)\|_{\dot{H}^{s-1}} ) \|u(t) -(e^{i\theta_k(t)}u(0))_{k\in\Z^*}\|_{\dot H^{s-1}} \\
&\leq 3 \|u(0)\|_{\dot H^{s}}^{5/2}.
\end{split}
\end{equation*}
Consequently, provided that $6 \sqrt{\varepsilon} \leq \gamma^2 N^{22 \rho_{2r}}$ (which is ensured by \eqref{ultimate:CFL2}) by applying Proposition \ref{prop:stab_U}, we get that $u(t)\in \mathcal{U}_{ \gamma/2,N^3,\rho_{2r}}^{\mathcal{E},s}$ and so $u(t)\in V_1$ which conclude the bootstrap.

To conclude the proof we just have to establish the bound about the variation of the angles. Somehow, we would like to replace $|v(t)|^2$ by $|u(0)|^2$ in \eqref{passimalpourdesangles}. To do this, we are going to apply the following lemma about the variations of $\partial_{I_{k}} Z^\E_4$ (which is proven in the subsection \ref{proof:lem:var:Z4:sharp} of the Appendix).
\begin{lemma}
\label{lem:var:Z4:sharp}
 Let $u,v \in \dot{H}^1$ and $k \in \mathbb{N}^*$, if $\|u\|_{L^2} = \|v\|_{L^2}$ then we have
$$
|\partial_{I_{k}} Z^\E_4 (I) -  \partial_{I_{k}} Z^\E_4 (J) | \lesssim |k|^{-1}\|u-v\|_{\dot{H}^1} \|u\|_{L^2}
$$
where $I_\ell := |u_\ell|^2$ and $J_\ell := |v_\ell|^2$.
\end{lemma}

Since \ref{gKdV} and \ref{gBO} are homogeneous equations, we know by Noether's Theorem that $\| u(t) \|_{L^2} = \| u(0)\|_{L^2}$. Consequently, since the maps $\tau^{(0)},\dots,\tau^{(3)}$ preserve the $L^2$ norm we have $\| v(t) \|_{L^2} = \| u(t) \|_{L^2} = \| u(0)\|_{L^2} = \| v(0)\|_{L^2} $.

By Lemma \ref{lem:var:Z4:sharp} and the estimate \eqref{whatweexpectaboutv}, we deduce of \eqref{passimalpourdesangles} that
$$
|\dot\theta_k-(2\pi)^{1+\alpha_{\E}} k|k|^{\alpha_{\E}}-2k\pi\partial_{I_k} Z^\E_4(J(0))|\leq (3/4) |k| \|u(0)\|^{5/2}_{\dot H^s}.
$$
Finally applying once again Lemma \ref{lem:var:Z4:sharp} and using the estimate \eqref{uetvetbinysontpasloin}, we deduce that
$$
|\dot\theta_k-(2\pi)^{1+\alpha_{\E}} k|k|^{\alpha_{\E}}-2k\pi\partial_{I_k} Z^\E_4(I(0))|\leq |k| \|u(0)\|^{5/2}_{\dot H^s}.
$$
\subsubsection{Control of the Sobolev norm of $v$}\label{ma jolie sous section} 
\label{sub:growth:v} We aim at proving the estimate \eqref{but_sous_section}, i.e. that
$$
\| v(t)\|_{\dot{H}^s} \leq \| v(0)\|_{\dot{H}^s} + (1/3)\| u(0)\|_{\dot{H}^s}
$$
Since $\varepsilon$ is small, it is enough to prove that
$$
\| v(t)\|_{\dot{H}^s}^2 \leq \| v(0)\|_{\dot{H}^s}^2 + \varepsilon^3.
$$
Recalling that $v$ is a solution of the Hamiltonian system \eqref{sys:v}, we have
$$
\| v(t)\|_{\dot{H}^s}^2 = \| v(0)\|_{\dot{H}^s}^2 + \int_0^t \{ \|\cdot\|_{\dot{H}^s}^2,H_{\E} \circ \tau^{(3)} \circ \tau^{(2)} \}(v(\mathfrak{t})) \mathrm{d}\mathfrak{t}.
$$
Consequently, since by assumption $T<\varepsilon^{-r/5}$ and $\varepsilon$ is such that $\varepsilon \lesssim_{r,s} 1$, it is enough to prove that
\begin{equation}
\label{what_we_really_want_here}
 |\{ \|\cdot\|_{\dot{H}^s}^2,H_{\E} \circ \tau^{(3)} \circ \tau^{(2)} \}(v(t))|\lesssim_{r,s} \varepsilon^{r/5 + 4}.
\end{equation}
Since $\|\cdot\|_{\dot{H}^s}^2$ is integrable, it commutes with the others integrable Hamiltonians. Thus by construction of $H_{\E} \circ \tau^{(3)} \circ \tau^{(2)}$ (see Theorem  \ref{thm:KtA}), we have
\begin{multline*}
\{ \|\cdot\|_{\dot{H}^s}^2,H_{\E} \circ \tau^{(3)} \circ \tau^{(2)} \} = \{ \|\cdot\|_{\dot{H}^s}^2,\mathrm{R}^{(\mu_3>N)} \circ \tau^{(2)}\}+\{ \|\cdot\|_{\dot{H}^s}^2, \mathrm{R}^{(I_{> N^3})}\circ \tau^{(2)}\}\\+\{ \|\cdot\|_{\dot{H}^s}^2, \mathrm{R}^{(or)}\circ \tau^{(2)}\} +\{ \|\cdot\|_{\dot{H}^s}^2, \mathrm{R^{(rat)}}\}
\end{multline*}
where $\mathrm{R}^{(\mu_3>N)}, \mathrm{R}^{(I_{> N^3})}, \mathrm{R}^{(or)}$ are the remainder terms of the resonant normal form (see Theorem \ref{thm-BNF}) and $\mathrm{R^{(rat)}}$ is the remainder term of the rational normal form (see Theorem \ref{thm:KtA}). We are going to prove that the estimate \eqref{what_we_really_want_here} holds for each one of these Poisson brackets.

\noindent \emph{$*$ \emph{Control of $\{ \|\cdot\|_{\dot{H}^s}^2, \mathrm{R^{(rat)}}\}(v(t))$.}} First, let us just recall that by \eqref{eq:vfromu}, since $u(t)\in V_2$, we have $v(t)\in V_4$. Consequently, by \eqref{lacouronne}, we have
\begin{equation}
\label{le reblochon c'est bien bon}
\|v(t)\|_{\dot{H}^s}\lesssim \varepsilon.
\end{equation}
Applying the estimate \eqref{est:rem:thm:KtA} of the Hamiltonian vector field generated by $\mathrm{R^{(rat)}}$, we have directly that
\begin{multline*}
|\{ \|\cdot\|_{\dot{H}^s}^2, \mathrm{R^{(rat)}}\}(v(t))|  \lesssim_{s,r} N^{10^5 r^2} \gamma^{-23 r+133}  \| v(t) \|_{\dot{H}^s}^{r+1} \mathop{\lesssim_{s,r}}^{\eqref{le reblochon c'est bien bon}} N^{10^5 r^2} \gamma^{-23 r }  \varepsilon^{r+1} \\
\mathop{\lesssim_{s,r}}^{\eqref{ultimate:CFL2}} \varepsilon^{r+1-\frac{23 }{35}r- 10^{-2}r} \mathop{\lesssim_{s,r}}^{r \gg 7} \varepsilon^{r/5+4}.
\end{multline*}

To control the other Poisson brackets, we are going to use Proposition \ref{prop-Ns} with $N\leftarrow N^3$ and $\tau \leftarrow \tau^{(2)}$. It follows from Theorem  \ref{thm:KtA} that the assumptions (A1),(A2) and (A3) are satisfied with $\kappa_\tau\lesssim_r 1$.

 \noindent \emph{$*$ \emph{Control of $\{ \|\cdot\|_{\dot{H}^s}^2, \mathrm{R}^{(or)}\circ \tau^{(2)}\}(v(t))$.}} We recall that by Theorem \ref{thm-BNF}, $\mathrm{R}^{(or)}$ writes
$$
\mathrm{R}^{(or)}(u) = \sum_{\substack{k \in \mathcal{M} \\ \#k  \geq r+1 }} c_k \, u^k \ \ \mathrm{with} \ \  |c_k| \leq \rho^{\#k} N^{3 \#k -9}  \ \mathrm{and} \ \rho \lesssim_r 1.
$$
Therefore by (i) of Proposition \ref{prop-Ns}, provided that $\|v(t)\|_{\dot{H}^s} \lesssim_r   N^{-3}$ which is ensured by \eqref{ultimate:CFL2} and \eqref{le reblochon c'est bien bon}, we have
$$
\big|\{\|\cdot\|_{\dot{H}^s}^2,\mathrm{R}^{(or)}\circ\tau^{2}\}(v(t))\big|\lesssim_{s,r,\kappa_\tau} N^3 (N^3 \|v(t)\|_{\dot{H}^s})^{r+1} 
$$
Consequently, we deduce of \eqref{ultimate:CFL2} and \eqref{le reblochon c'est bien bon} that
$\big|\{\|\cdot\|_{\dot{H}^s}^2,\mathrm{R}^{(or)}\circ\tau^{2}\}(v(t))\big| \lesssim_{r,s}  \varepsilon^{r/5+4}.$

\noindent \emph{$*$ \emph{Control of $\{ \|\cdot\|_{\dot{H}^s}^2, \mathrm{R}^{(\mu_3>N)}\}(v(t))$.}} We recall that by Theorem \ref{thm-BNF}, $\mathrm{R}^{(\mu_3>N)}$ writes
$$
\mathrm{R}^{(\mu_3>N)}(u) =  \sum_{\substack{k \in \mathcal{M}\cap \mathcal D \\ 4\leq  \#k  \leq r \\ k_3 \geq N }} c_k\, u^k \ \ \mathrm{with} \ \ |c_k| \lesssim_{\#k }   N^{3{ \#k -9}}.
$$
Then by applying (ii) of Proposition \ref{prop-Ns} (with $K=N$), we get
$$
|\{ \|\cdot\|_{\dot{H}^s}^2, \mathrm{R}^{(\mu_3>N)}\}(v(t))| \lesssim_{s,r}N^3 N^{-s+2} N^{3r-9} \|v(t)\|_{\dot H^s}^4
$$
Since by  \eqref{ultimate:CFL2}, $N^{-s} \leq \varepsilon^r$, we deduce of \eqref{ultimate:CFL2} and \eqref{le reblochon c'est bien bon} that
$$
|\{ \|\cdot\|_{\dot{H}^s}^2, \mathrm{R}^{(\mu_3>N)}\}(v(t))| \lesssim_{r,s}  \varepsilon^{r/5+4}.
$$
\noindent \emph{$*$ \emph{Control of $\{ \|\cdot\|_{\dot{H}^s}^2, \mathrm{R}^{(I_{> N^3})}\circ \tau^{(2)}\}(v(t))$.}} We recall that by Theorem \ref{thm-BNF}, $\mathrm{R}^{(I_{> N^3})}$ writes
$$
\mathrm{R}^{(I_{> N^3})}(u) = \sum_{\ell = N^3+1}^{\infty} \sum_{\substack{k \in \mathcal{M}\cap \mathcal D \\ 3\leq  \#k  \leq r-2 }} c_{\ell,k} \, I_\ell \, u^k \ \ \mathrm{with} \ \  |c_{\ell,k}| \lesssim_{ \#k }   N^{3( \#k +2) -9}.
$$
Consequently, by applying (iii) of Proposition \ref{prop-Ns}, we get
$$
\big|\{ \|\cdot\|_{\dot{H}^s}^2, \mathrm{R}^{(I_{> N^3})}\circ \tau^{(2)}\}(v(t))\big|\lesssim_{s,r}N^{-6(s-1)} N^{3r-9 }\|v(t)\|^{5}_{\dot{H}^s}.
$$
As previously, we deduce that $\big|\{ \|\cdot\|_{\dot{H}^s}^2, \mathrm{R}^{(I_{> N^3})}\circ \tau^{(2)}\}(v(t))\big| \lesssim_{r,s}  \varepsilon^{r/5+4}.$

\subsubsection{Dynamics of the Hamiltonian system in the new variables}
\label{sub:dynam:v}
We are going to design $C^1$ functions $\theta_k:\R_+\mapsto \R$, $k\in\Z^*,$ such that $v(t)$ is closed to $(e^{i\theta_k(t)}v_k(0))_k$ (see \eqref{whatweexpectaboutv}).

We recall that $v$ is solution of the Hamiltonian system \eqref{sys:v}. We note that this system can be rewritten
$$
\partial_t v_{\ell}(t) = 2i\ell\pi \omega_\ell(t) v_{\ell}(t) + R_\ell(t)
$$
where, denoting $Z= Z_2^{\mathcal{E}} + Z_4^{\mathcal{E}} + Z^{(6)}_{N^3} + \dots + Z^{(r)}_{N^3}$ and $c_{\ell,k}$ the coefficients of $\mathrm{R}^{(I_{> N^3})}$ (defined in Theorem \ref{thm-BNF} and satisfying $|c_{\ell,k}|\lesssim_{\#k} N^{3(\# k+2)}$), we have set
$$
\omega_{\ell}(t) = \partial_{I_\ell} Z(J(t)) + \mathbb{1}_{\ell > N^3} P^{(\ell)}\circ \tau^{(2)}(v(t)) 
$$
and $R(t) $ is given by
$$ \partial_x \nabla (\mathrm{R}^{(\mu_3>N)} \circ \tau^{(2)} +   \mathrm{R}^{(or)}\circ \tau^{(2)} + \mathrm{R^{(rat)}})(v(t)) + \sum_{\ell = N^3+1 }^{\infty} |v_{\ell}(t)|^2 \partial_x \nabla (P^{(\ell)}\circ \tau^{(2)})(v(t))$$
with 
\begin{equation}\label{PPP}P^{\ell}(u):=\sum_{\substack{k \in \mathcal{M}\cap \mathcal D \\ 3\leq \# k \leq r-2 }} c_{\ell,k}   u^k.\end{equation}
Note that by construction, since the Hamiltonians are real valued, $\omega_{\ell}(t)\in \mathbb{R}$. 

Applying the Duhamel formula, it comes
\begin{equation}
\label{merciDudu}
\|v(t) -(e^{i\theta_\ell(t)}v(0))_{\ell \in\Z^*}\|_{\dot H^{s-1}} \leq |t| \sup_{|\mathfrak{t}|\leq t} \| R(\mathfrak{t}) \|_{\dot H^{s-1}}
\end{equation}
where we have set\footnote{Here $\theta_{\ell}$ is only well defined for $|t|<T$. Nevertheless, using a localizing function, it could be easily extended.} 
\begin{equation}
\label{def:theta:primi}
\theta_{\ell}(t) = 2\ell\pi \int_0^t \ \omega_{\ell} (\mathfrak{t}) \ \mathrm{d}\mathfrak{t}.
\end{equation}

\noindent \emph{$\bullet$ Step 1 : Control of $\| R(\mathfrak{t}) \|_{\dot H^{s-1}}$. } By construction $R(t) $ is the sum of $4$ kinds of terms. By applying the triangle inequality, we control them one by one.

\medskip

\noindent \emph{$*$ Step 1.1 : Control of $R^{(1)}:=\| \nabla (\mathrm{R}^{(\mu_3>N)} \circ \tau^{(2)})(v(t)) \|_{\dot H^{s}}$. } \\
The map $\tau^{(2)}$ being symplectic, denoting $\nu(t) = \tau^{(2)}(v(t))\in V_3$, we have
$$
R^{(1)} \leq \| (\mathrm{d}\tau^{(2)})^{-1}(v(t))\|_{\mathscr{L}(\dot{H}^s)}  \| \nabla \mathrm{R}^{(\mu_3>N)} (\nu(t)) \|_{\dot H^{s}}\lesssim_r  \| \nabla \mathrm{R}^{(\mu_3>N)} (\nu(t)) \|_{\dot H^{s}}.   
$$
We recall that $\mathrm{R}^{(\mu_3>N)}$ writtes (see \eqref{Rmu})
$$
\mathrm{R}^{(\mu_3>N)}(u) =  \sum_{\substack{k \in \mathcal{M}\cap \mathcal D \\ 4\leq \# k \leq r \\ |k_3| \geq N }} c_k\, u^k \ \ \mathrm{with} \ \ |c_k| \lesssim_{\# k}   N^{3{\# k}}.
$$
Consequently, for $\ell \in \mathbb{N}^*$, we have
\begin{multline*}
\ell^{s}|\partial_{u_{-\ell}}\mathrm{R}^{(\mu_3>N)} (\nu(t))|  \leq  \sum_{n=4}^r N^{3n} \sum_{i=1}^n  \sum_{\substack{k \in \mathcal{M}_n\cap \mathcal D \\ |k_3| \geq N \\ k_i=-\ell}}  \ell^{s} \prod_{j\neq i}|\nu_{k_j}(t)| \\
\leq \sum_{n=4}^r N^{3n} \sum_{i=1}^n \sum_{\substack{k_1+\dots+k_n = \ell \\ k\in \mathcal{D} \\ |k_2| \geq N }}  \ell^{s} |\nu^{k}(t)| \nu_{k_j}(t)| \\
\leq N^{-s+1}  \sum_{n=4}^r n^{s+1} N^{3n} \sum_{\substack{k_1+\dots+k_n = \ell   }}   | k_1|^s |\nu^{k_1}(t)|  | k_2|^{s-1} |\nu^{k_2}(t)| \prod_{j\geq 3}|\nu_{k_j}(t)|.
\end{multline*}
Consequently, recalling that $\nu(t)\in V_3$, applying a Young inequality (and a triangle inequality for the sum for $n=4,\dots,r$), we get 
$$
R^{(1)} \lesssim_{r,s} N^{-s} \sum_{n=4}^r N^{3n+1} \| \nu(t)\|_{\dot{H}^s}^n \mathop{\lesssim_{r,s}}^{\eqref{ultimate:CFL2}} \varepsilon^{r+3} \mathop{\lesssim_{r,s}}^{\eqref{lacouronne}} \| u(0)\|_{\dot{H}^s}^{r+3}. 
$$

\noindent \emph{$*$ Step 1.2 : Control of $R^{(2)}:=\| \nabla (\mathrm{R}^{(or)} \circ \tau^{(2)})(v(t)) \|_{\dot H^{s}}$. } \\
As previously, we naturally have
$$
R^{(2)} \lesssim_r  \| \nabla \mathrm{R}^{(or)}(\nu(t)) \|_{\dot{H}^s}.
$$
We recall that $\mathrm{R}^{(or)}$ writes (see \ref{Ror})
$$
\mathrm{R}^{(or)}(u) = \sum_{n=r+1}^{\infty} \mathrm{R}^{(n)}(u) := \sum_{n=r+1}^{\infty}  \sum_{k \in \mathcal{M}_n } c_k \, u^k \ \ \mathrm{with} \ \  |c_k|  \lesssim_r M^{\#k} N^{3\#k-9} \ \mathrm{and}\ M\lesssim_r 1.
$$
Realizing the same estimates we did at the previous step naturally leads to have
$$
\| \nabla \mathrm{R}^{(n)}(\nu(t))\|_{\dot{H}^s} \lesssim_r n^{s+1} M^n N^{3n} \|\nu(t)\|_{\dot{H}^s}^{n-1} \lesssim_r (8 M \varepsilon N^3)^{n-1}  n^{s+1} N^{-6}.
$$
Consequently, since $N$ is large enough with respect to $r$ and $s$, we have
$$
R^{(2)} \lesssim_r N^{-6} \sum_{n >r}  (8 M \varepsilon N^3)^{n-1}  n^{s+1}   \mathop{\leq}^{\eqref{ultimate:CFL2}} (8 M \varepsilon N^3)^{r} \leq \varepsilon^{r-1}.
$$
Thus by \eqref{lacouronne}, we get $R^{(2)} \lesssim_r \|u(0)\|_{\dot{H}^s}^{r-1}.$

\noindent \emph{$*$ Step 1.3 : Control of $\mathrm{R^{(rat)}}(v(t))$.} Theorem \ref{thm:KtA} states in \eqref{est:rem:thm:KtA} that
$$
 \|\partial_x \nabla \mathrm{R^{(rat)}}(v(t))\|_{\dot{H}^s} \lesssim_{s,r} N^{10^5 r^2} \gamma^{-23 r + 133}  \| v(t) \|_{\dot{H}^s}^{r}.
$$
Consequently, by the assumption \eqref{ultimate:CFL2}, we have 
\begin{multline*}
 \|\partial_x \nabla \mathrm{R^{(rat)}}(v(t))\|_{\dot{H}^s} \lesssim_{s,r} N^{10^5 r^2} \gamma^{-23 r + 133} \varepsilon^r  \lesssim_{s,r} N^{10^5 r^2} \varepsilon^{12r/35} \\ \simeq_{s,r} (N^{10^7 r} \varepsilon)^{r/100}  \varepsilon^{\frac{12\, r}{35} - \frac{r}{100} }  \mathop{\lesssim_{r,s}}^{\eqref{lacouronne}} \| u(0)\|_{\dot{H}^s}^{r/4}.
\end{multline*}
\noindent \emph{$*$ Step 1.4 : Control of the remainder term induced by $\mathrm{R}^{(I_{> N^3})}$.}
Naturally, by applying the triangle inequality, we have
\begin{equation*}
\begin{split}
\| \sum_{\ell = N^3+1 }^{\infty} |v_{\ell}(t)|^2 \partial_x \nabla (P^{(\ell)}\circ \tau^{(2)})(v(t)) \|_{\dot{H}^s} 
&\lesssim \sum_{\ell = N^3+1 }^{\infty} \ell^{-2s} \| v(t)\|_{\dot{H}^s}^2  \| \partial_x \nabla (P^{(\ell)}\circ \tau^{(2)})(v(t)) \|_{\dot{H}^s} \\
&\lesssim N^{-2s+2} \varepsilon^2 \sup_{\ell>N}  \| \partial_x \nabla (P^{(\ell)}\circ \tau^{(2)})(v(t)) \|_{\dot{H}^s} \\
&\lesssim_r N^{-2s+2} \varepsilon^2 \sup_{\ell>N}  \| \partial_x \nabla P^{(\ell)}(\nu(t)) \|_{\dot{H}^s} 
\end{split}
\end{equation*}
Considering the estimates of the two first sub-steps, it is clear that
$$
\| \partial_x \nabla P^{(\ell)}(\nu(t))  \|_{\dot{H}^s} \leq r^{s+1} N^{3r+1} \| \nu(t)\|_{\dot{H}^s}^3.  
$$
It follows that by \eqref{ultimate:CFL2} and \eqref{lacouronne}, we have
$$
\| \sum_{\ell = N^3+1 }^{\infty} |v_{\ell}(t)|^2 \partial_x \nabla (P^{(\ell)}\circ \tau^{(2)})(v(t)) \|_{\dot{H}^s} \lesssim_{r,s} \|u(0)\|_{\dot{H}^s}^{2r}.
$$
\noindent \emph{$*$ Step 1.5 : Conclusion.} We have proven that while $|t|<T$ we have 
$$
 \| R(\mathfrak{t}) \|_{\dot H^{s-1}} \lesssim_{r,s} \| u(0)\|_{\dot{H}^s}^{r/4}.
$$
Consequently, since $T< \|u(0)\|_{\dot{H}^s}^{-r/5}$, \eqref{merciDudu} leads to (for $r\geq 200$)
$$
\|v(t) -(e^{i\theta_\ell(t)}v(0))_{\ell \in\Z^*}\|_{\dot H^{s-1}} \lesssim_{r,s} \|u(0)\|_{\dot{H}^s}^{r/4} \|u(0)\|_{\dot{H}^s}^{-r/5} \leq \|u(0)\|_{\dot{H}^s}^{10}.
$$

\noindent \emph{$\bullet$ Step 2 : Control of $E_{\ell}:=|\dot\theta_\ell-(2\pi)^{1+\alpha_{\E}} \ell|\ell|^{\alpha_{\E}}-2\ell\pi\partial_{I_{\ell}} Z^\E_4(J(t))|$. }
Naturally, we assume that $\ell>0$. By definition of $\theta_{\ell}$ (see \eqref{def:theta:primi}), we have
$$
\frac{E_{\ell}}{2\pi \ell } \leq | \partial_{I_\ell} Z^{(6)}_{N^3}(J(t)) + \dots + \partial_{I_\ell} Z^{(r)}_{N^3}(J(t))  | + |P^{(\ell)}\circ \tau^{(2)}(v(t))|.
$$
On the one hand, in view of \eqref{PPP}, by applying the Young inequality and using \eqref{ultimate:CFL2},\eqref{lacouronne},
 we have
$$
|P^{(\ell)}\circ \tau^{(2)}(v(t)) | \lesssim_r \sum_{n=3}^{r-2} N^{3(n+2)}\| \nu(t)\|_{\dot{H}^s}^n \lesssim_r \| u(0)\|_{\dot{H}^s}^{11/4}
$$
On the other hand, by applying the result of the Proposition \ref{prop:VFtheta}, we have for $n\geq 8$
$$
|\partial_{I_\ell} Z^{(n)}_{N^3}(J(t))| \lesssim_{n,s} N^{321 n}  \sqrt{\gamma}^{-47n + 314 +n-2} N^{1+12 \,(47n)^2}  \| u(0) \|_{\dot{H}^s}^{n-2} \mathop{\lesssim_{n,s}}^{\eqref{ultimate:CFL2} }  \| u(0) \|_{\dot{H}^s}^{3}
$$
and for $n=6$
$$
|\partial_{I_\ell} Z^{(6)}_{N^3}(J(t))| \lesssim_{s} N^{39}  \sqrt{\gamma}^{-6} N^{1+12 \,10^2}  \| u(0) \|_{\dot{H}^s}^{4} \mathop{\lesssim_s}^{\eqref{ultimate:CFL2} }  \| u(0) \|_{\dot{H}^s}^{3}.
$$
Finally, since $\varepsilon$ is small enough, we have $E_{\ell} \leq (1/2) |\ell| \| u(0)\|_{\dot{H}^s}^{5/2}$.

\subsection{Proof of Theorem \ref{mainPDE} : $\mathcal{V}_{r,s}^{\mathcal{E}}$ and its geometry} Note that if $r_1<r_2$ and the Theorem \ref{mainPDE} holds for $r_2$ then it also holds for $r_1$. Hence, without loss of generality, we assume that $r\gg 7$. Considering the Theorem \ref{thm:dyn}, the second part of the Theorem \ref{mainPDE} holds if we set
\begin{equation}
\label{def_of_Vrs}
\mathcal{V}_{r/5,s}^{\mathcal{E}} = \bigcup_{\varepsilon=0}^{\varepsilon_0(r,s) } \bigcup_{\gamma = \varepsilon^{1/35} }^{\gamma_0(r,s) } \bigcup_{ N= \varepsilon^{-r/s}  }^{ \varepsilon^{-1/(10^7 r)}} \mathcal{U}_{ \gamma,N^3,\rho_{2r}}^{\mathcal{E},s} \cap (B_s(0,\varepsilon) \setminus B_s(0,\varepsilon/2)).
\end{equation}
where $\rho_{r}$ is given by \eqref{def:rhor} and $\varepsilon_0(r,s),\gamma_0(r,s)$ are given by the Theorem \ref{thm:dyn}.

Note that by construction it is clear that $\mathcal{V}_{r/5,s}^{\mathcal{E}}$ is invariant by translation of the angles (i.e. it satisfies \eqref{def:inv_trans_ang}).

The other properties we aim at establishing on $\mathcal{V}_{r/5,s}^{\mathcal{E}}$ relies on probabilities estimates.
Consequently, from now $(I_k)_{k\in \mathbb{N}^{*}}$ denotes a sequence of random variables for which we assume that
\begin{itemize}
\item the actions are independent
\item $I_k$ is uniformly distributed in $J_k + \varepsilon_0^2 (4\zeta(\nu))^{-1}(0,k^{-2s-\nu})$
\end{itemize}
where $\nu\in(1,9]$ is a given constant, $J_k\geq 0$ and $\varepsilon_0>0$. We denote by $u$ the random function  defined by
$$
u(x) = \sum_{k = 1}^{\infty} 2\sqrt{I_k} \cos(2\pi k x).
$$

The following proposition is our main probability result even if it is essentially a corollary of the Proposition \ref{prop:sum:proba}.
\begin{proposition}
\label{prop:ultimate:proba}
 For all $s>s_0(r) = 10^{7}r^2$, all $\lambda \in (0,1)$, all $\varepsilon_0 \lesssim_{r,s,\nu,\lambda} 1$, all $\gamma \leq \gamma_0(r,s) $, if
 \begin{equation}
 \label{assump:proba:final}
\|u\|_{L^\infty\dot{H}^s}^2 = 2\sum_{k=1}^{\infty} J_k |k|^{2s} + \frac{\varepsilon_0^2}2 \leq \varepsilon_0^2 \ \ \mathrm{and}\ \ \varepsilon_0 \leq \gamma^{35}
 \end{equation}
then
$$
\mathbb{P}\left( u\in \mathcal{V}_{r/5,s}^{\mathcal{E}} \right) \geq 1 - \lambda \gamma.
$$
\end{proposition}
\begin{proof} 
By applying Proposition \ref{prop:sum:proba}, provided that $\varepsilon_0 \lesssim_{r,s,\nu} 1$, we have a probability larger than $1-\lambda \gamma$ to draw $u$ such that 
\begin{equation}
\label{cestpasmalnon}
1\lesssim_{r,s,\nu} N \leq (\gamma \| u\|_{\dot{H}^s}^{-2})^{1/(63\rho_{2r}+15)} \ \Rightarrow \ u\in \mathcal{U}_{\gamma,N^3,\rho_{2r}}^{\mathcal{E},s}.
\end{equation}
Consequently from now on we assume that $u$ satisfies \eqref{cestpasmalnon}. By construction of $\mathcal{V}_{r/5,s}^{\mathcal{E}}$ and recalling that by assumption $\varepsilon_0\leq \gamma^{35}$ we just have to check that there exists $N$ such that
$$
1\lesssim_{r,s,\nu} N \leq ( \| u\|_{\dot{H}^s}^{-2+\frac1{35}})^{1/(63\rho_{2r}+15)}  \ \mathrm{and}\ (\|u \|_{\dot{H}^s}/2)^{-r/s} \leq N \leq (2\|u \|_{\dot{H}^s})^{-1/(10^7 r)}.
$$
On the one hand, since $10^7 r \geq 63\rho_{2r}+15$, it is clear that 
$$(\|u \|_{\dot{H}^s}/2)^{-1/(10^7 r)} \leq ( \| u\|_{\dot{H}^s}^{-2+\frac1{35}})^{1/(63\rho_{2r}+15)}.$$
On the other hand, since by assumption $s>s_0(r) = 10^{7}r^2$, provided $\varepsilon_0 \lesssim_{r,s} 1$,  
 we have 
$$
\exists N \in ((\|u \|_{\dot{H}^s}/2)^{-r/s} , (2\|u \|_{\dot{H}^s})^{-1/(10^7 r)}) \cap \mathbb{N} .
$$
\end{proof}

In the following corollary, we prove that $\mathcal{V}_{r/5,s}^{\mathcal{E}}$ is \emph{asymptotically of full measure} in the sense of the Theorem \ref{mainPDE} (see \eqref{def:asympt_full_measure} and set $\varepsilon = \varepsilon_0/ (2\sqrt{\zeta(\nu)})$).
\begin{corollary} If $J=0$, $\varepsilon_0 \lesssim_{r,s,\nu} 1$ then
\begin{equation*}
 \mathbb{P}(u \in \mathcal{V}_{r/5,s}^{\mathcal{E}}) \geq 1-\left(\frac{\varepsilon_0}{2\sqrt{\zeta(\nu)}}\right)^{\frac1{35}}.
\end{equation*}
\end{corollary}
\begin{proof}
It is enough to apply the Proposition \ref{prop:ultimate:proba}, with $\lambda= (2\sqrt{\zeta(\nu)})^{-1/35}$.
\end{proof}

In the following corollary of the Proposition \ref{prop:ultimate:proba}, we prove that $\mathcal{V}_{r/5,s}^{\mathcal{E}}$ is \emph{asymptotically dense}.
\begin{corollary} 
If $\|v\|_{\dot{H}^s} \lesssim_{r,s} 1$ then there exists $w\in \mathcal{V}_{r/5,s}^{\mathcal{E}}$ such that
\begin{equation}
\label{plusdurqueprevu}
\| v - w\|_{\dot{H}^s} \leq \frac{\| v \|_{\dot{H}^s}}{|\log(\| v \|_{\dot{H}^s})|}.
\end{equation}
\end{corollary}
\begin{proof} Since $\mathcal{V}_{r/5,s}^{\mathcal{E}}$ is invariant by translation of the angles (i.e. it satisfies \eqref{def:inv_trans_ang}), without loss of generality we can assume that $v$ is of the form
$$
v(x) = \sum_{k = 1}^{\infty} 2\sqrt{J_k} \cos(2\pi k x).
$$
We set $\nu = 9$ and $\varepsilon_0 =\sqrt{2} \|v\|_{\dot{H}^s}$. Provided that $\|v\|_{\dot{H}^s}\lesssim_{r,s} 1$, applying Proposition \ref{prop:ultimate:proba} with $\gamma = \varepsilon_0^{1/{35}}$ and $\lambda = 2^{-1/70}$, we have
$$
\mathbb{P}\left( u\in \mathcal{V}_{r/5,s}^{\mathcal{E}} \right) \geq 1 - \|v\|_{\dot{H}^s}^{1/35}.
$$

Now we aim at estimating the probability that $\| u - v\|_{\dot{H}^s} \leq \| v\|_{\dot{H}^s}/ |\log ( \| v\|_{\dot{H}^s})|$. First, we observe that by applying the Minkowski inequality, we have
$$
\| u - v\|_{\dot{H}^s}^2 = 2\sum_{k=1}^{\infty} k^{2s} |\sqrt{I_k} - \sqrt{J_k}|^2 \leq 2\sum_{k=1}^{\infty} k^{2s}  |I_k-J_k|.
$$
Then we recall that, by construction, there exists some random variables $(X_k)_{k\geq 1}$, independent and uniformly distributed in $(0,1)$ such that
$$
I_k-J_k = 2 \|v\|_{\dot{H}^s}^2 X_k k^{-2s - 9}.
$$
Consequently, defining $\lambda = 2\pi^2/3$ and $\eta = -\log(\| v\|_{\dot{H}^s})$, provided that $\|v\|_{\dot{H}^s}$ is small enough, we have
\begin{multline*}
\mathbb{P}(\| u - v\|_{\dot{H}^s} \leq \| v\|_{\dot{H}^s} \eta) \geq \mathbb{P}(4\sum_{k=1}^{\infty} k^{-2 } X_k\leq \eta^2 ) \geq \mathbb{P}(\forall k\geq 1, \ \lambda X_k \leq k^7\eta^2 ) \\
= \prod_{\lambda \geq k^7 \eta^2} \frac{k^7 \eta^2}{\lambda } 
\geq \big(\frac{\eta^2}{\lambda }\big)^{(\frac{\lambda }{\eta^2})^{1/7}} \geq e^{- (\frac{\lambda }{\eta^2})^{1/4}} \geq e^{-1/(36\eta)} = \| v\|^{1/36}.
\end{multline*}
Finally, provided that $\|v\|_{\dot{H}^s}$ is small enough, we have proven that
$$
\mathbb{P}(\| u - v\|_{\dot{H}^s} \leq \| v\|_{\dot{H}^s} \eta) + \mathbb{P}\left( u\in \mathcal{V}_{r/5,s}^{\mathcal{E}} \right) >1
$$
which ensures that the intersection of these events is not empty and so the existence of $w\in \mathcal{V}_{r/5,s}^{\mathcal{E}}$ satisfying \eqref{plusdurqueprevu}.

\end{proof}

\section{Appendix}
\subsection{Proof of Lemma \ref{lem:diot}}
\label{app:proof:diot}
We denote $n= \#m$. We aim at proving that $\kappa_{m,k}^{\mathrm{\ref{gKdV}}}=k_{n}$
and if $a_3= 0 \ \mathrm{or} \ m_1+\dots+m_n = 0$ then we have $\kappa_{m,k}^{\mathrm{\ref{gBO}}} = k_n$ else we have $\kappa_{m,k}^{\mathrm{\ref{gBO}}}\leq 2 n-1$.

\medskip

In view of \eqref{deltaKdV} and assumption \eqref{assump:gKdV}, the case $\mathcal{E} = \mathrm{\ref{gKdV}}$ is clear. So we only focus on the case $\mathcal{E} = \mathrm{\ref{gBO}}$.  
If $a_3= 0$ or $m_1+\dots+m_n = 0$, then applying \eqref{deltaBO}  we obtain $(\delta^{\mathrm{\ref{gBO}}}_{m,k})_{p}=0$  for all $p < k_n $ while $(\delta^{\mathrm{\ref{gBO}}}_{m,k})_{k_n}=-12a_4 m_nk_n\neq 0$ in view of \eqref{assump:gBO}.    Consequently, we have $\kappa_{m,k}^{\mathrm{\ref{gBO}}} = k_n$. 

Now we assume that we have $a_3 \neq 0$, $m_1+\dots+m_n \neq 0$ and $\kappa_{m,k}^{\mathrm{\ref{gBO}}}> 2 \, n - 2$. As a consequence, for $\ell\in \llbracket1,2\, n-2\rrbracket$, we have $(\delta^{\mathrm{\ref{gBO}}}_{m,k})_{\ell}=0$. In particular for $\ell=1$ we get from \eqref{deltaBO}
 $$
-12 \, a_4\, m_n\,\mathbb{1}_{k_n=1} - \frac{18\, a_3^2}{\pi} (m_1+\dots+m_n)=0
$$
and thus $k_n=1$.\\
We are going to prove that 
\begin{equation}
\label{eq:funfun}
\forall j\in \llbracket 2,n \rrbracket, \ (\delta^{\mathrm{\ref{gBO}}}_{m,k})_{k_j} =  (\delta^{\mathrm{\ref{gBO}}}_{m,k})_{k_j +1} = (\delta^{\mathrm{\ref{gBO}}}_{m,k})_{k_j +2} =0 \ \Rightarrow \ k_{j-1}\in \{ k_j +1,k_j +2 \}.
\end{equation}
Before proving it, let us explain how we get the upper bound on $\kappa_{m,k}^{\mathrm{\ref{gBO}}}$ from \eqref{eq:funfun}.  Since $(\delta^{\mathrm{\ref{gBO}}}_{m,k})_{\ell}=0$ for $\ell\in \llbracket1,2\, n-2\rrbracket$ we deduce from \eqref{eq:funfun} that $k_{j-1}-k_j \leq 2$ for $j\in \llbracket 2,n \rrbracket$. But since $k_n=1$, we get
$$
k_1 = k_n + \sum_{j=2}^n k_{j-1}-k_j \leq 1 + 2(n-1) = 2n-1.
$$
Since by \eqref{eq:bienvu}, we have $\kappa_{m,k}^{\mathrm{\ref{gBO}}} \leq k_1$. We deduce that $\kappa_{m,k}^{\mathrm{\ref{gBO}}} = 2\,n-1$.

It remains to prove \eqref{eq:funfun}. To do that we assume that there exists ${j_{\star}}\in \llbracket 2,n \rrbracket$ such that 
$ (\delta^{\mathrm{\ref{gBO}}}_{m,k})_{k_{j_{\star}} +1}=(\delta^{\mathrm{\ref{gBO}}}_{m,k})_{k_{j_{\star}} +2} =0$ and $k_{{j_{\star}}-1}>k_{j_{\star}} +2$, and we are going  to prove that $(\delta^{\mathrm{\ref{gBO}}}_{m,k})_{k_{j_{\star}}} \neq 0$. By assumption, we have by \eqref{deltaBO}
$$
\forall \ell \in \{ k_{j_{\star}} +1,k_{j_{\star}} +2\},\ \sum_{j>j_{\star}} m_j + \frac1{\ell}\sum_{j\leq j_{\star}} m_j\, k_j =0
$$
As a consequence, we have $\displaystyle \sum_{j>j_{\star}} m_j = \sum_{j\leq j_{\star}} m_j\, k_j  = 0.$ Thus, we have
$$
\sum_{k_j\geq k_{j_\star}} m_j + \frac1{k_{j_\star}}\sum_{k_j< k_{j_\star}} m_j\, k_j = \sum_{k_j> k_{j_\star}} m_j + \frac1{k_{j_\star}}\sum_{k_j\leq k_{j_\star}} m_j\, k_j =0.
$$
Thus, we get $(\delta^{\mathrm{\ref{gBO}}}_{m,k})_{k_{j_{\star}}} = -12\, a_4 \, m_{k_{j_\star}} \, k_{j_\star} \neq 0.$

\subsection{Control of the remainder terms in the rational normal form process}
\label{app:sub:rem}
In this subsection, we explain how we control the remainder terms in the rational normal form process.
The remainder terms we meet are some linear combinations of $3$ kinds of terms.

\begin{itemize}
\item \emph{Type {\rm I}: An old remainder term in new variables.} Such terms are of the form $R\circ \tau$ where $R$ is an Hamiltonian defined on a set of the form $V_{\lambda_2}$ and $\tau$ is a symplectic map from $V_{\lambda_1}$ to $V_{\lambda_2}$ where $2\leq \lambda_1 < \lambda_2\leq 5/2$. Furthermore, the invert of the differential of $\tau$ is invertible and its norm is smaller than $2$.

Since $\tau$ is symplectic, if $u\in V_{\lambda_1}$, we have 
\begin{equation*}
\begin{split}
\| \partial_x \nabla (R\circ \tau(u)) \|_{\dot{H}^s} = \|  \partial_x (\mathrm{d}\tau(u))^* (\nabla R)(\tau(u)) \|_{\dot{H}^s}  &=  \|  (\mathrm{d}\tau(u))^{-1} \partial_x  (\nabla R)(\tau(u)) \|_{\dot{H}^s} \\ &\leq 2 \sup_{v \in V_{\lambda_2}} \| \partial_x \nabla R(v) \|_{\dot{H}^s}
\end{split}
\end{equation*}
where $(\mathrm{d}\tau(u))^*$ denotes the $L^2$ adjoint of $\mathrm{d}\tau(u)$. Consequently, the vector field associated with $R\circ \tau$ is controlled by the vector field associated with $R$.
\item \emph{Type {\rm II}: A remainder term of a Taylor expansion.} Such terms are of the form
$$
R=\int_0^1  \Xi_{N^3} \circ \tau^{(t)}\, g(t) \mathrm{d}t
$$
where $\|g\|_{L^{\infty}}\leq 1$, for $t\in (0,1)$, $\tau^{(t)}$ is a symplectic transformation mapping a set of the form $V_{\lambda_1}$ in a set of the form $V_{\lambda_2}$, the differential of $\tau^{(t)}$ is invertible and the norm of its invert is smaller than $2$, $\Xi\in \mathscr{H}_m^{(6),\mathcal{E}}$ with\footnote{note that it is to obtain this estimate that we have paid a lot of attention to the order of our Taylor expansions} $r+1\leq m \leq 2r$ and such that $C^{(em)}_{\Xi} \leq N^3$, $C^{(str)}_{\Xi}\lesssim_r 1$ and $C^{(\infty)}_{\Xi}\leq N^{321 m -2049}$ (note that all these results on $\Xi$ rely on the application of Proposition \ref{prop:stab_frac}, \ref{prop:stab:subclasses} and \ref{prop:stab:chichi}).

Consequently, by Proposition \ref{prop:VF}, for $u\in V_{\lambda_1}=B_s(0,2^{\lambda_1} \varepsilon_0)\cap \mathcal{U}_{2^{-\lambda_1} \gamma,N^3,\rho_{2r}}^{\mathcal{E},s}$, we have
\begin{equation}
\label{est:rem:sharp}
\begin{split}
\| \partial_x \nabla R(u) \|_{\dot{H}^s} &=\| g\|_{L^{\infty}} \int_0^1 \| (\mathrm{d}\tau^{(t)}(u))^{-1} \partial_x (\nabla \Xi_{N^3}) \circ \tau^{(t)}   \|_{\dot{H}^s} \\ 
&\lesssim_{s,r}  N^{321 m -2049}  \sqrt{\gamma}^{-\rho_{m}+m-2} N^{3\cdot12 \,(\rho_{m})^2}  \| u \|_{\dot{H}^s}^{m-1} \\
&\mathop{\lesssim_{s,r}}^{\eqref{ultimate:CFL}} N^{321 (r+1) -2049}  \sqrt{\gamma}^{-\rho_{r+1}+r-1} N^{3\cdot12 \,(\rho_{r+1})^2}  \| u \|_{\dot{H}^s}^{r}
\end{split}
\end{equation}
Note to apply this proposition it has been crucial to have the index $\rho_{2r}$ in the definition of $V$.

\item \emph{Type {\rm III}: The product of a transmutation.} The last kind of remainder terms are the terms $\Lambda^{(r-1)},\Lambda^{(r)}$ appearing in \eqref{prod:transmutation}. As a straightforward application of Proposition with the estimate we have established on $C^{(\infty)}_{\Lambda^{(\mathfrak{r})}},C^{(em)}_{\Lambda^{(\mathfrak{r})}},C^{(str)}_{\Lambda^{(\mathfrak{r})}}$, $\| \partial_x \nabla \Lambda^{r-1}(u) \|_{\dot{H}^s}$ and $\| \partial_x \nabla \Lambda^{r}(u) \|_{\dot{H}^s}$ satisfy the same estimate as $\| \partial_x \nabla R(u) \|_{\dot{H}^s}$ above.

Finally, note that the final estimate we write in \eqref{est:rem:thm:KtA} for the remainder term is a direct consequence of the estimate \eqref{est:rem:sharp}.
\end{itemize}

\subsection{Proof of Lemma \ref{lem:var:Z4:sharp}.} 
\label{proof:lem:var:Z4:sharp} We use the explicit formula of $Z_4^{\E}$ established in Theorem \ref{thm-BNF}.

\noindent \emph{$\bullet$ Case $\E = \mathrm{\ref{gKdV}}$.} Since $\|u\|_{L^2}=\|v\|_{L^2}$, we have
$$
|\partial_{I_{k}} Z^{\mathrm{\ref{gKdV}}}_4 (I) -  \partial_{I_{k}} Z^{\mathrm{\ref{gKdV}}}_4 (J) |  =  \left|12\, a_4+ \frac{3\, a_3^2}{ \pi^2 \, k^2} \right| |I_k-J_k| \lesssim_a |k|^{-1} \|u-v\|_{\dot{H}^1} \|u\|_{L^2}.
$$
\noindent \emph{$\bullet$ Case $\E = \mathrm{\ref{gBO}}$.} Since $\|u\|_{L^2}=\|v\|_{L^2}$, we have
\begin{multline*}
|\partial_{I_{k}} Z^{\mathrm{\ref{gBO}}}_4 (I) -  \partial_{I_{k}} Z^{\mathrm{\ref{gBO}}}_4 (J) |  \leq \left|12\, a_4+ \frac{18\, a_3^2}{ \pi \, k^2} \right| |I_k-J_k| + \frac{18\, a_3^2}{\pi} \sum_{p\neq k} \frac{|I_p-J_p|}{\min(k,p)} \\
\lesssim_a |k|^{-1} \|u-v\|_{\dot{H}^1} \|u\|_{L^2}.
\end{multline*}

\end{document}